\documentclass[11pt]{article}
\usepackage{amsmath,amssymb,amsthm,amscd,bm,mathrsfs,color,enumerate}
\usepackage[all]{xy}
\usepackage[dvips]{graphics}
\usepackage[nottoc,notlot,notlof]{tocbibind}

\SelectTips{cm}{12}
\objectmargin={1pt}

\usepackage{ytableau}
\ytableausetup{centertableaux,mathmode,smalltableaux}
\def\young(#1){\ytableaushort{#1}}
\def\yng(#1){\ydiagram{#1}}

\allowdisplaybreaks

\numberwithin{equation}{section}

\theoremstyle{definition}
\newtheorem{thm}{Theorem}[section]
\newtheorem{lem}[thm]{Lemma}
\newtheorem{prop}[thm]{Proposition}

\theoremstyle{definition}
\newtheorem{define}[thm]{Definition}

\theoremstyle{remark}
\newtheorem{rem}[thm]{Remark}
\newtheorem{ex}[thm]{Example}

\newcommand{\Fg}{\mathfrak{g}}
\newcommand{\Fh}{\mathfrak{h}}

\newcommand{\BC}{\mathbb{C}}
\newcommand{\BR}{\mathbb{R}}

\newcommand{\BZ}{\mathbb{Z}}

\newcommand{\U}{\mathbf{U}}
\newcommand{\af}{\mathrm{af}}
\newcommand{\lsi}{\ell^{\frac{\infty}{2}}}
\newcommand{\si}{\frac{\infty}{2}}

\begin{document}

\title{\bf
Tableau models for semi-infinite Bruhat order
and 
level-zero representations of quantum affine algebras}
\author{Motohiro Ishii}
\date{}
\maketitle

\begin{abstract} 
We prove that 
semi-infinite Bruhat order on an affine Weyl group is 
completely determined from those on the quotients 
by affine Weyl subgroups
associated with various maximal (standard) parabolic subgroups of finite type. 
Furthermore, 
for an affine Weyl group of classical type, 
we 
give a complete classification 
of
all cover relations of semi-infinite Bruhat order 
(or equivalently, all edges of the quantum Bruhat graphs)
on the quotients in terms of tableaux. 
Combining these 
we obtain a tableau criterion for semi-infinite Bruhat order
on an affine Weyl group of classical type. 
As an application, 
we give new tableau models for the crystal bases of 
a level-zero fundamental representation
and 
a level-zero extremal weight module
over a quantum affine algebra of classical untwisted type,
which we call 
quantum Kashiwara--Nakashima columns
and 
semi-infinite Kashiwara--Nakashima tableaux. 
We give an explicit description of the crystal isomorphisms 
among 
three different realizations of the crystal basis of 
a level-zero fundamental representation
by
quantum Lakshmibai--Seshadri paths, 
quantum Kashiwara--Nakashima columns, 
and 
(ordinary) Kashiwara--Nakashima columns.
\end{abstract}

\tableofcontents

\section{Introduction}

The aim of this paper is to give an explicit description  
of semi-infinite Bruhat order in terms of tableaux
and its application to level-zero representation theory
of quantum affine algebras.

Semi-infinite Bruhat order 
on an affine Weyl group is a 
variant of 
Bruhat order 
on a Coxeter group, 
and it is also an important tool to study 
representation theory of 
algebraic groups, 
quantum groups, 
and 
affine Kac--Moody Lie algebras, 
quantum and affine Schubert calculi, 
and so forth
(see 
\cite{A,BFP,I20,INS16,KNS,Len12,LNSSS15,LNSSS16,LS,Lus,NS16,P}
and the references given there). 
In fact, 
analogously to (ordinary) Bruhat order, 
each of the following structures 
is closely related (or equivalent)
to semi-infinite Bruhat order:
Lusztig's generic Bruhat order
(\cite[\S 1.5]{Lus})
and 
Peterson's stable Bruhat order
(\cite[Lecture 12]{P})
on an affine Weyl group, 
Littelmann's order on the 
affine Weyl group orbit through a level-zero integral weight 
of an affine Kac--Moody Lie algebra
(\cite[\S 4]{L95}),
the quantum Bruhat graphs
(\cite[Definition 6.1]{BFP} and \cite[\S 4]{LNSSS15}),
homomorphisms among Wakimoto modules over an affine Kac--Moody Lie algebra
(\cite[Proposition 4.10]{A}),
the containment relation among (opposite) Demazure subcrystals 
of the crystal basis of a level-zero extremal weight module 
over a quantum affine algebra
(\cite[Corollary 5.2.5]{NS16}), 
and the containment relation among semi-infinite Schubert varieties 
(\cite[\S 4.2]{KNS}); 
see also (i)--(ii) below. 
Also, it is worth pointing out that 
semi-infinite Bruhat order has the lifting property 
(or ``diamond lemma")
with respect to the semi-infinite length function
(\cite[\S 4.1]{INS16}; see also Lemma \ref{lem:diamond}); 
note that Bruhat order on a Coxeter group is characterized uniquely 
by the lifting property with respect to the length function 
(\cite[Theorem 1.1]{D}; see also \cite[Exercise 14 in \S 2]{BB}).
However there indeed exist some differences 
between semi-infinite and ordinary Bruhat orders. 
For example, 
there are no minimal elements in semi-infinite Bruhat order, 
and this order depends on not only the Coxeter system but also the root data. 
Therefore the standard method to study Bruhat order 
(see for instance \cite{BB})
is not well adapted to the study of 
semi-infinite Bruhat order. 

In \cite{I20}, 
we intended to develop tableau combinatorics 
on semi-infinite Bruhat order. 
We introduced
semi-infinite Young tableaux, 
and showed that these tableaux give a combinatorial realization 
of the crystal basis of a level-zero extremal weight module
(\cite{K94,K02})
over the quantum affine algebra of type $A_{n-1}^{(1)}$. 
Our proof of this result was based on 
a tableau criterion for semi-infinite Bruhat order
and 
standard monomial theory 
for semi-infinite Lakshmibai--Seshadri paths
(\cite[\S 3.1]{INS16}).
Therefore we can think of these tableaux 
as a natural generalization of 
(ordinary) Young tableaux.
However 
we restricted the discussion in
\cite[\S 4]{I20} 
only to the affine Weyl group of type $A_{n-1}^{(1)}$, 
and the tableau criterion in 
\cite[Theorem 4.7]{I20} 
is applicable only to 
semi-infinite Bruhat order 
on the quotient by an affine Weyl subgroup
associated with a maximal parabolic subgroup of finite type. 

In this paper, 
we wish to investigate 
semi-infinite Bruhat order 
on an affine Weyl group of all classical untwisted type
via tableaux, 
and aim to give an application to 
level-zero representation theory 
(see for instance \cite{AK,BN04,K02,K05})
of a quantum affine algebra of classical untwisted type.
The main results of this paper are the following
(see Theorems 
\ref{thm:Deo}, 
\ref{thm:tab-cri-A},
\ref{thm:tab-cri-C},
\ref{thm:tab-cri-B},
\ref{thm:tab-cri-D}, 
\ref{thm:SiKN->B-C}, 
\ref{thm:QKN->QLS-B}, 
\ref{thm:SiKN->B-B}, 
\ref{thm:QKN->QLS-D},
and 
\ref{thm:SiKN->B-D}):
\begin{enumerate}[(I)]
\item
a Deodhar-type criterion 
for semi-infinite Bruhat order 
on an affine Weyl group of 
arbitrary untwisted type, 
\item
a tableau criterion for semi-infinite Bruhat order 
on an affine Weyl group of type
$A_{n-1}^{(1)}$,
$B_n^{(1)}$,
$C_n^{(1)}$, 
and 
$D_n^{(1)}$ 
in full generality, 
\item
a tableau model for the crystal basis of 
a level-zero fundamental representation
(resp. 
a level-zero extremal weight module)
over a quantum affine algebra 
of type $B_n^{(1)}$ and $D_n^{(1)}$
(resp. 
$B_n^{(1)},C_n^{(1)}$, and $D_n^{(1)}$), 
which we call 
quantum Kashiwara--Nakashima columns 
(resp. 
semi-infinite Kashiwara--Nakashima tableaux), 
and 
\item
an explicit description of the isomorphisms among 
three different realizations of the crystal basis of 
a level-zero fundamental representation
by
quantum Lakshmibai--Seshadri paths
(\cite{LNSSS16}), 
quantum Kashiwara--Nakashima columns, 
and 
(ordinary) Kashiwara--Nakashima columns
(\cite{KN}).
\end{enumerate}

Let us give more precise explanation of our results.
Let 
$\U$
be a quantum affine algebra of untwisted type, 
and let 
$\U'$
be its derived subalgebra. 
Let 
$W_{\af} 
= 
\langle r_i \mid i \in I_{\af} \rangle$
be the affine Weyl group associated with a finite Weyl group
$W = \langle r_i \mid i \in I \rangle$, 
where 
$I_{\af} = \{ 0 \} \sqcup I$
and 
$r_i$, $i \in I_{\af}$, 
is a simple reflection. 
For 
$J \subset I$, 
let 
$W_J = \langle r_j \mid j \in J \rangle \subset W$
be a (standard) parabolic subgroup, and let 
$W^J \subset W$
be the set of minimal(-length) coset representatives for $W/W_J$. 
Let 
$(W_J)_{\af} \subset W_{\af}$
be the affine Weyl subgroup associated with $W_J$ (see \eqref{eq:W_J_af}). 
Note that 
$(W_J)_{\af}$
is not a parabolic subgroup of 
$W_{\af}$, 
but it is generated by reflections. 
Therefore there exists the subset 
$(W^J)_{\af} \subset W_{\af}$
of minimal coset representatives for $W_{\af} / (W_J)_{\af}$
(see \eqref{eq:W^J_af}); 
note that if $J = \emptyset$, 
then 
$(W_J)_{\af}$ 
is trivial and 
$(W^J)_{\af} = W_{\af}$. 
Let 
$\Pi^J : W_{\af} \rightarrow (W^J)_{\af}$
be the canonical surjection. 
In this paper, 
following 
\cite[\S 2.4]{INS16}
(see also \cite[Lecture 12]{P}), 
we define semi-infinite Bruhat order 
$\preceq$
on 
$(W^J)_{\af}$
by using the semi-infinite length function 
$\ell^{\si} : W_{\af} \rightarrow \BZ$ (see \eqref{eq:sil}). 
But we mainly use the following two realizations 
of semi-infinite Bruhat order 
(see 
Lemmas \ref{lem:Dem} (3) 
and 
\ref{lem:Q=SiB}
for precise formulation): 
\begin{enumerate}[(i)]
\item
the containment relation among 
the path model 
$\mathbb{B}^{\si}_{\succeq x}(\lambda)$, 
$x \in W_{\af}$,  
of the 
(opposite) 
Demazure subcrystals of the crystal basis of 
the extremal weight $\U$-module 
of level-zero extremal weight $\lambda$
(\cite[\S 5]{NS16}),
\item
the ``affinization" of 
the (parabolic) quantum Bruhat graph 
$\mathrm{QB}^J$
for 
$W^J$
(\cite[Appendix A]{INS16}).
\end{enumerate}

In \S \ref{Section:Deodhar}, 
we prove a Deodhar-type criterion 
for semi-infinite Bruhat order
(see (I) above), 
which states that 
for 
$x,y \in (W^J)_{\af}$
we have 
$x \preceq y$
in 
$(W^J)_{\af}$
if and only if 
\begin{align}
\Pi^{I \setminus \{ i \}}(x)
\preceq 
\Pi^{I \setminus \{ i \}}(y)
\ \text{in} \ 
(W^{I \setminus \{ i \}})_{\af}
\ \text{for all} \ 
i \in I \setminus J
\end{align}
(cf. \cite[Theorem 2.6.1]{BB}; see also \cite[Lemma 3.6]{D}).
We prove this for 
$W_{\af}$
of arbitrary untwisted type
(see Theorem \ref{thm:Deo} and Proposition \ref{prop:Deo}). 
This is shown by investigating the path model
$\mathbb{B}^{\si}_{\succeq x}(\lambda)$, 
$x \in W_{\af}$, 
of Demazure subcrystals.
More precisely, 
let 
$\varpi_i$, $i \in I$, 
be a level-zero fundamental weight, 
and let 
$\lambda = \sum_{i \in I} m_i \varpi_i$, 
$m_i \in \BZ_{\geq 0}$, 
$i \in I$. 
Set 
$J_{\lambda} = \{ i \in I \mid m_i = 0 \}$. 
The path model 
$\mathbb{B}^{\si}_{\succeq x}(\lambda)$, $x \in W_{\af}$, 
is defined as a subset of the $\U$-crystal 
$\mathbb{B}^{\si}(\lambda)$
of semi-infinite Lakshmibai--Seshadri paths of shape $\lambda$.
We prove that any extremal element in the tensor product 
$\bigotimes_{i \in I \setminus J_{\lambda}} \mathbb{B}^{\si}_{\succeq \Pi^{I \setminus \{ i \}}(x)}(m_i \varpi_i)$
(see \eqref{eq:tensor-Dem}) 
is in the image of 
$\mathbb{B}^{\si}_{\succeq x}(\lambda)$
under the isomorphism 
$\psi_{\lambda} : \mathbb{B}^{\si}(\lambda) 
\rightarrow 
\bigotimes_{i \in I \setminus J_{\lambda}} \mathbb{B}^{\si}(m_i \varpi_i)$
of $\U$-crystals 
(see Lemma \ref{lem:Psi_lambda} and Proposition \ref{prop:SMT-Dem} (2)). 
Combining this with the realization (i)
yields the Deodhar-type criterion. 

In \S \ref{Section:Tab-cri}, 
for 
$W_{\af}$
of type
$A_{n-1}^{(1)}$,
$B_n^{(1)}$,
$C_n^{(1)}$, 
and 
$D_n^{(1)}$, 
we give a complete classification 
of the cover relations of semi-infinite Bruhat order on
$(W^{I \setminus \{ i \}})_{\af}$, 
$i \in I$,
in terms of tableaux
(Propositions 
\ref{prop:SiB-C},
\ref{prop:SiB-B},
and 
\ref{prop:SiB-D}; 
see also 
\cite[Proposition 4.11]{I20}). 
Moreover, we prove a tableau criterion 
for semi-infinite Bruhat order on 
$(W^{I \setminus \{ i \}})_{\af}$, 
$i \in I$
(Definitions 
\ref{def:SiB-A},
\ref{def:SiB-C},
\ref{def:SiB-B}, 
and 
\ref{def:SiB-D}, 
and 
Propositions 
\ref{prop:tab-cri-A},
\ref{prop:tab-cri-C},
\ref{prop:tab-cri-B},
and 
\ref{prop:tab-cri-D}; 
see also 
\cite[Theorem 4.7]{I20}). 
Combining this with the Deodhar-type criterion we obtain (II). 
We emphasize that 
(II) 
can be thought of as a generalization of the 
tableau criterion for 
Bruhat order on the symmetric group 
(\cite[Theorem 2.6.3 (Tableau Criterion)]{BB}). 
Our main tool in 
\S \ref{Section:Tab-cri}
is the quantum Bruhat graph. 
In fact, we classify all 
(quantum) edges in the quantum Bruhat graph 
$\mathrm{QB}^{I \setminus \{ i \}}$
for 
$W^{I \setminus \{ i \}}$, 
$i \in I$
(Propositions 
\ref{prop:Q=C}, 
\ref{prop:Q=B}, 
and 
\ref{prop:Q=D}; 
see also
\cite[Lemma 4.15]{I20}). 
Combining this with the realization (ii)
yields the classification results above.
We should remark that, 
for 
$W_{\af}$ 
of type 
$A_{n-1}^{(1)}$
and 
$C_n^{(1)}$, 
Lenart's criterion 
(\cite[Propositions 3.6 and 5.7]{Len12})
for the edges of the quantum Bruhat graph 
$\mathrm{QB}^{\emptyset}$ 
for 
$W$
is a necessary condition 
for our classification results. 
Indeed, 
the existence of an edge in 
$\mathrm{QB}^{I \setminus \{ i \}}$
implies that in 
$\mathrm{QB}^{\emptyset}$
(see Lemma \ref{lem:LS} for a precise statement).

We now briefly sketch the tableau criterion for 
semi-infinite Bruhat order on
$(W^{I \setminus \{ i \}})_{\af}$
when 
$W_{\af}$
is of type 
$A_{n-1}^{(1)}$
(\S \ref{subsection:tab-cri-A} and \cite[\S 4]{I20}; 
see also Example \ref{ex:tab-cri}). 
First, we associate each element of 
$W_{\af}$
with a pair 
$(\mathsf{T},c)$
of a column
$\mathsf{T}$
of length $i$
and an integer 
$c$
(see \eqref{eq:map-Y-A}).
Let 
$\mathsf{T}(u)
\in 
\{ 1,2,\ldots ,n \}$
denote the 
$u$-th entry of 
$\mathsf{T}$. 
Let 
$(\mathsf{T},c),
(\mathsf{T}',c')$
be such pairs for 
$x,y \in W_{\af}$, 
respectively. 
Then we have 
$\Pi^{I \setminus \{ i \}}(x) \preceq \Pi^{I \setminus \{ i \}}(y)$
in 
$(W^{I \setminus \{ i \}})_{\af}$
if and only if 
\begin{align} \label{Eq:tab-A}
\left( c \leq c' \right)
\ \text{and} \ 
\left( 
\mathsf{T}(u) \leq \mathsf{T}'(u+c'-c)
\ \text{for} \ 
1 \leq u \leq i-c'+c
\right) .
\end{align}

In \S \ref{Section:Tab-model}, 
we prove (III)--(IV). 
For this purpose, 
we first investigate the subgraph
$\mathrm{QB}(\varpi_i;1/2)$
(see \S\ref{Subsection:QLS}--\S\ref{Subsection:QBG-Maya})
of the quantum Bruhat graph
$\mathrm{QB}^{I \setminus \{ i \}}$, 
where the vertex sets of these graphs are both
$W^{I \setminus \{ i \}}$.
We see that 
$\mathrm{QB}(\varpi_i;1/2)$
defines a partial order 
$\trianglelefteq$
on
$W^{I \setminus \{ i \}}$.
Moreover, 
by the classification results in \S \ref{Section:Tab-cri}, 
we obtain an explicit description of the order 
$\trianglelefteq$
in terms of Maya diagrams
(see Definitions \ref{def:Maya-1/2-J} and \ref{def:Maya-1/2}).
Then the $\U'$-crystal  
$\mathrm{QLS}(\varpi_i)$
of 
quantum Lakshmibai--Seshadri paths 
(QLS paths for short)
of shape $\varpi_i$
is given by
\begin{align}
\mathrm{QLS}(\varpi_i)
=
\{ (v,w) \mid 
w,v \in W^{I \setminus \{ i \}}, \ 
w \trianglelefteq v \} .
\end{align}
It is well-known that 
$W^{I \setminus \{ i \}}$
can be realized as a set of columns
(see Lemmas \ref{lem:Gr-A}, \ref{lem:Gr-C}, and \ref{lem:Gr-D}).
Thus we can think of each QLS path 
$(v,w) \in \mathrm{QLS}(\varpi_i)$
as a pair of columns. 
We know from 
\cite{LNSSS16}
that 
$\mathrm{QLS}(\varpi_i)$
is isomorphic, 
as a $\U'$-crystal, 
to the crystal basis of the 
level-zero fundamental representation 
$W(\varpi_i)$
(\cite{K02,K05}; 
see also 
\S\ref{Subsection:extremal}
and 
\S\ref{Subsection:QLS}). 

Assume that 
$\U$
is of type 
$B_n^{(1)}$. 
In \S\ref{Subsection:tab-type-B}, 
we introduce 
quantum Kashiwara--Nakashima $B_n$-columns
(QKN $B_n$-columns for short)
and a $\U'$-crystal structure on them. 
In particular, 
we define Kashiwara operators 
$e_j,f_j$, 
$j \in I_{\af}$,
acting on them. 
A QKN $B_n$-column 
$\Tilde{\mathsf{C}}$
of shape $\varpi_i$ 
is consisting of 
an (ordinary) Kashiwara--Nakashima $B_n$-column
(KN $B_n$-column for short)
$\mathsf{C}$
of shape 
$\varpi_{i-2m}$
and a multiset 
$\{ \underbrace{\overline{0},\overline{0},\ldots ,\overline{0}}_{2m \ \text{times}} \}$
for some integer
$0 
\leq 
m 
\leq 
\left\lfloor \frac{i}{2} \right\rfloor
=
\max \{ k \in \BZ \mid k \leq i/2 \}$; 
we write
$\Tilde{\mathsf{C}}
=
\mathsf{C}
\cup
\{ \underbrace{\overline{0},\overline{0},\ldots ,\overline{0}}_{2m \ \text{times}} \}$
for brevity.
Let 
$\mathrm{QKN}_{B_n}(\varpi_i)$
(resp.
$\mathrm{KN}_{B_n}(\varpi_i)$)
denote the set of 
QKN $B_n$-columns 
(resp. 
KN $B_n$-columns)
of shape 
$\varpi_i$.
Let us give an example of 
a QKN $B_n$-column.
We have 
\ytableausetup{mathmode,boxsize=5.5mm}
\begin{align} \label{eq:ex-QKN}
\Tilde{\mathsf{C}}
\ = \ 
\begin{ytableau}
2 \\ 3 \\ 0 \\ \overline{9} \\ \overline{3} \\ \overline{0} \\ \overline{0}
\end{ytableau}
\ = \ 
\mathsf{C}
\cup 
\{ \overline{0},\overline{0} \}
\ 
\in 
\mathrm{QKN}_{B_9}(\varpi_7)
\ \ \ \text{and} \ \ \ 
\mathsf{C}
\ = \ 
\begin{ytableau}
2 \\ 3 \\ 0 \\ \overline{9} \\ \overline{3} 
\end{ytableau}
\ 
\in 
\mathrm{KN}_{B_9}(\varpi_5).
\end{align}
For each QKN $B_n$-column 
$\Tilde{\mathsf{C}}$
of shape 
$\varpi_i$, 
we construct a pair 
$(r\Tilde{\mathsf{C}},l\Tilde{\mathsf{C}})$
of columns, 
and show that it is a QLS path
of shape 
$\varpi_i$. 
The construction of 
$(r\Tilde{\mathsf{C}},l\Tilde{\mathsf{C}})$
was motivated by 
\cite[\S 4]{She}
(see also \cite[\S 3]{Le03}), 
and has previously been used by Briggs
(\cite{B}; see also \cite[Algorithm 4.1]{LeSc}). 
For the QKN $B_n$-column 
$\Tilde{\mathsf{C}}$
in 
\eqref{eq:ex-QKN}, 
we have 
\begin{align}
r\Tilde{\mathsf{C}}
\ = \ 
\begin{ytableau}
2 \\ 3 \\ 4 \\ 5 \\ \overline{9} \\ \overline{8} \\ \overline{1}
\end{ytableau}
\ \ \ \text{and} \ \ \ 
l\Tilde{\mathsf{C}}
\ = \ 
\begin{ytableau}
1 \\ 2 \\ 8 \\ \overline{9} \\ \overline{5} \\ \overline{4} \\ \overline{3}
\end{ytableau} \ ;
\end{align}
by using notation in \S \ref{Subsection:tab-type-B}, 
we have 
\begin{enumerate}[(1)]
\item
$I_{\mathsf{C}}
=
\{ 0 \succeq 3 \}$, 
$J_{\mathsf{C}}
=
\{ 8 > 1 \}$, 
$K_{\Tilde{\mathsf{C}}}
=
\{ 4 < 5 \}$, 
\item
$\mathcal{J}(r\Tilde{\mathsf{C}})
=
\mathcal{J}(l\Tilde{\mathsf{C}})
=
(\{ 1,2,3,4,5 \} < \{ 8,9 \})
\in 
\mathcal{S}_7$, 
and 
\item
$\mathcal{M}(r\Tilde{\mathsf{C}})
=
(\{ 1 \} , \{ 8,9 \}), 
\mathcal{M}(l\Tilde{\mathsf{C}})
=
(\{ 3,4,5 \} , \{ 9 \})
\in 
2^{\{ 1,2,3,4,5 \}}
\times 
2^{\{ 8,9 \}}$. 
\end{enumerate}
Then we prove that the map 
$\mathrm{QKN}_{B_n}(\varpi_i)
\to 
\mathrm{QLS}(\varpi_i)$, 
$\Tilde{\mathsf{C}}
\mapsto
(r\Tilde{\mathsf{C}},l\Tilde{\mathsf{C}})$, 
is bijective. 
The important point to note here is that 
the inverse of this map 
can be described explicitly in terms of Maya diagrams 
(see Theorem \ref{thm:QKN->QLS-B} (2)); 
similar results have been obtained independently by 
Lenart--Schulze
(\cite[\S 4]{LeSc}), 
where they used the quantum alcove model (\cite{LL}). 
Thus we obtain the following 
crystal isomorphisms among 
the sets of 
QLS paths, 
QKN $B_n$-columns, 
and 
KN $B_n$-columns
(cf. \cite[Lemma 2.7]{C}).
\begin{align} \label{eq:QLS<->QKN<->KN}
\begin{split}
\mathrm{QLS}(\varpi_i)
&
\leftrightarrow 
\mathrm{QKN}_{B_n}(\varpi_i)
\leftrightarrow 
\bigsqcup_{m=0}^{\left\lfloor \frac{i}{2} \right\rfloor}
\mathrm{KN}_{B_n}(\varpi_{i-2m}), \\[2mm]
(r\Tilde{\mathsf{C}},l\Tilde{\mathsf{C}})
&
\leftrightarrow 
\Tilde{\mathsf{C}}
=
\mathsf{C}
\cup
\{ \underbrace{\overline{0},\overline{0},\ldots ,\overline{0}}_{2m \ \text{times}} \}
\leftrightarrow 
\mathsf{C} .
\end{split}
\end{align}
Consequently, 
the $\U'$-crystal 
$\mathrm{QKN}_{B_n}(\varpi_i)$
is isomorphic to the crystal basis of the 
level-zero fundamental representation 
$W(\varpi_i)$. 
Similar formulation and results also hold 
for $\U$ of type $D_n^{(1)}$
(see \S\ref{Subsection:tab-type-D}).
We should remark that 
the crystal basis of 
$W(\varpi_i)$
is isomorphic to the Kirillov--Reshetikhin crystal 
$B^{i,1}$, 
and that there is another tableau model, 
called 
Kirillov--Reshetikhin tableaux
(\cite{OSS,SS}), 
more generally for the Kirillov--Reshetikhin crystals
$B^{r,s}$
(see 
\cite[\S 2.3]{HKOTY}), 
$r \in I$, 
$s > 0$, 
of non-exceptional type. 
The advantage of using QKN columns lies in 
the explicit description of the isomorphisms \eqref{eq:QLS<->QKN<->KN}. 
Likewise, the Kirillov--Reshetikhin tableaux model is an important ingredient 
to describe the rigged configuration bijections 
(see for instance \cite{OSSS}).
It would be desirable to relate these two tableau models 
but we will not develop this point here.

Assume that 
$\U$
is of type 
$B_n^{(1)}$, 
$C_n^{(1)}$, 
or 
$D_n^{(1)}$. 
In \S\ref{Subsection:tab-type-C}--\S\ref{Subsection:tab-type-D}, 
we introduce 
semi-infinite Kashiwara--Nakashima tableaux
(semi-infinite KN tableaux for short), 
and show that the set of these tableaux 
is isomorphic, 
as a $\U$-crystal, 
to the crystal basis of 
a level-zero extremal weight $\U$-module. 
This is achieved by applying 
standard monomial theory for semi-infinite Lakshmibai--Seshadri paths 
(\cite[Theorem 3.4]{I20})
to the affinizations of 
$\U'$-crystals of (Q)KN columns.
The definition of semi-infinite KN tableaux 
is based on the tableau criterion for semi-infinite Bruhat order
obtained in 
\S \ref{Section:Tab-cri}. 

This paper is organized as follows. 
In \S \ref{Section:Prel}, 
we set up notation and terminology 
on untwisted affine root data, 
crystals, 
semi-infinite Bruhat order, 
extremal weight modules, 
and 
semi-infinite Lakshmibai--Seshadri paths. 
Also, we have compiled some basic facts on these. 
In \S \ref{Section:Deodhar}, 
we prove a Deodhar-type criterion 
for 
semi-infinite Bruhat order
on 
$W_{\af}$
of arbitrary untwisted type. 
In \S \ref{Section:Tab-cri}, 
we prove a tableau criterion for 
semi-infinite Bruhat order
on 
$W_{\af}$
of type
$A_{n-1}^{(1)}$,
$B_n^{(1)}$,
$C_n^{(1)}$, 
and 
$D_n^{(1)}$, 
by classifying all cover relations of 
semi-infinite Bruhat order on 
$(W^{I \setminus \{ i \}})_{\af}$
in terms of tableaux. 
In \S \ref{Section:Tab-model},
we introduce 
the $\U'$-crystal of QKN columns 
and 
the $\U$-crystal of semi-infinite KN tableaux. 
We show that these tableaux give combinatorial models for 
crystal bases of 
level-zero fundamental representations
and 
level-zero extremal weight modules. 
We give an explicit description of the crystal isomorphisms among 
QLS paths, QKN columns, and KN columns.

\paragraph{Acknowledgements.}
The author was supported by JSPS KAKENHI Grant Numbers 16K17577 and 20K14278.

\paragraph{Notation.} 
Let 
$\BZ_{>0}$
(resp. 
$\BZ_{\geq 0}$)
denote the set of 
positive integers
(resp. 
non-negative integers). 
For 
$k \in \BZ$, 
set 
$[k] = \{ 1,2,\ldots ,k \}$
if 
$k \geq 1$, 
and set
$[k] = \emptyset$
if 
$k \leq 0$. 
For integers 
$k \leq l$, 
set 
$[k,l] = \{ k,k+1,\ldots ,l \}$; 
we understand that 
$[k,l] = \emptyset$
if 
$k > l$. 
The disjoint union of two sets 
$A,B$
will be denoted by 
$A \sqcup B$. 
For a (non-empty) set $A$, 
let
$\mathfrak{S}(A)$
be the permutation group of 
$A$. 
For a finite set 
$A$, 
let 
$\# A$
denote the number of elements in $A$.

\section{Preliminaries} \label{Section:Prel}

\subsection{Untwisted affine root data}

Let 
$\Fg_{\af}$ 
be an untwisted affine Kac--Moody Lie algebra over $\BC$ 
with a Cartan subalgebra $\Fh_{\af}$. 
Let 
$\{ \alpha_i \}_{i \in I_{\af}} \subset \Fh_{\af}^* = \mathrm{Hom}_{\BC}(\Fh_{\af},\BC)$ 
and 
$\{ \alpha_i^{\vee} \}_{i \in I_{\af}} \subset \Fh_{\af}$
be the sets of simple roots and simple coroots, respectively. 
Here $I_{\af}$ denotes the vertex set of the (affine) Dynkin diagram of $\Fg_{\af}$. 
Let 
$\langle \cdot , \cdot \rangle : \Fh_{\af} \times \Fh_{\af}^* \rightarrow \BC$ 
be the canonical pairing. 
We take and fix an integral weight lattice 
$P_{\af} \subset \Fh_{\af}^*$
satisfying the conditions that 
$\alpha_i \in P_{\af}$ 
and 
$\alpha_i^{\vee} \in \mathrm{Hom}_{\BZ}(P_{\af} , \BZ)$ 
for all 
$i \in I_{\af}$, 
and for each $i \in I_{\af}$ 
there exists
$\Lambda_i \in P_{\af}$
such that 
$\langle \alpha_j^{\vee} , \Lambda_i \rangle = \delta_{ij}$ 
for all 
$j \in I_{\af}$.
Similarly, let 
$P_{\af}^{\vee} \subset \Fh_{\af}$
be an integral coweight lattice such that 
$\alpha_i^{\vee} \in P_{\af}^{\vee}$ 
and 
$\alpha_i \in \mathrm{Hom}_{\BZ}(P_{\af}^{\vee} , \BZ)$ 
for all 
$i \in I_{\af}$, 
and for each $i \in I_{\af}$ 
there exists
$\Lambda_i^{\vee} \in P_{\af}^{\vee}$
such that 
$\langle \Lambda_i^{\vee} , \alpha_j \rangle = \delta_{ij}$ 
for all 
$j \in I_{\af}$. 
Let 
$\delta = \sum_{i \in I_{\af}} a_i \alpha_i \in \Fh_{\af}^*$ 
and 
$\mathrm{c} = \sum_{i \in I_{\af}} a_i^{\vee} \alpha_i^{\vee} \in \Fh_{\af}$ 
be the null root and the canonical central element, respectively.
For 
$\lambda \in P_{\af}$, 
the integer 
$\langle \mathrm{c} , \lambda \rangle$
is called the level of $\lambda$. 
We take and fix $0 \in I_{\af}$ such that 
$a_0 = a_0^{\vee} = 1$. 
Set 
$I = I_{\af} \setminus \{ 0 \}$; 
note that the subset $I$ of $I_{\af}$ 
corresponds to the vertex set of the Dynkin diagram of a 
complex finite-dimensional simple Lie subalgebra $\Fg$ of $\Fg_{\af}$. 
Fix a section 
$\iota : P_{\af}/(P_{\af} \cap \BC\delta) \to P_{\af}$
(resp. 
$\iota : P_{\af}^{\vee}/(P_{\af}^{\vee} \cap \BC\mathrm{c}) \to P_{\af}^{\vee}$)
of the canonical surjection
$\mathrm{cl} : P_{\af} \to P_{\af}/(P_{\af} \cap \BC\delta)$
(resp.
$\mathrm{cl} : P_{\af}^{\vee} \to P_{\af}^{\vee}/(P_{\af}^{\vee} \cap \BC\mathrm{c})$)
such that 
$(\iota \circ \mathrm{cl})(\alpha_i)
=
\alpha_i$
(resp.
$(\iota \circ \mathrm{cl})(\alpha_i^{\vee})
=
\alpha_i^{\vee}$)
for 
$i \in I$. 
For each 
$i \in I_{\af}$, 
define
$\varpi_i 
= 
(\iota \circ \mathrm{cl})
(\Lambda_i - \langle \mathrm{c} , \Lambda_i \rangle \Lambda_0)$
and 
$\varpi_i^{\vee} 
= 
(\iota \circ \mathrm{cl})
(\Lambda^{\vee}_i - \langle \Lambda_i^{\vee} , \delta \rangle \Lambda^{\vee}_0)$; 
note that 
$\varpi_0 = 0$, 
$\varpi^{\vee}_0 = 0$,  
$\langle \mathrm{c} , \varpi_i \rangle = 
\langle \varpi^{\vee}_i , \delta \rangle = 0$, 
and 
$\langle \alpha_j^{\vee} , \varpi_i \rangle = 
\langle \varpi^{\vee}_i , \alpha_j \rangle 
= \delta_{ij}$
for all $i,j \in I$. 
We call 
$\varpi_i$ 
the $i$-th level-zero fundamental weight of $\Fg_{\af}$.
Set
\begin{align}
Q
= 
\bigoplus_{i \in I} \BZ \alpha_i, &&
Q^{\vee} 
= 
\bigoplus_{i \in I} \BZ \alpha_i^{\vee}, &&
P
= 
\bigoplus_{i \in I} \BZ \varpi_i, &&
P^+ 
= 
\sum_{i \in I} \BZ_{\geq 0} \varpi_i;
\end{align}
note that 
$Q \subset P$
and 
$Q^{\vee} \subset \bigoplus_{i \in I} \BZ \varpi_i^{\vee}$. 
We think of 
$P$
(resp. 
$Q$)
as a weight lattice 
(resp. 
a root lattice)
of 
$\Fg$.

Let 
$W_{\af} = \langle r_i \mid i \in I_{\af} \rangle$ 
be the (affine) Weyl group of $\Fg_{\af}$, 
where $r_i$ denotes the simple reflection with respect to $\alpha_i$.
The subgroup 
$W = \langle r_i \mid i \in I \rangle \subset W_{\af}$ 
is the (finite) Weyl group of $\Fg$. 
Let $\ell : W_{\af} \rightarrow \BZ_{\geq 0}$ be the length function. 
Let $e \in W_{\af}$ be the unit element. 
The action of $W_{\af}$ on $\Fh_{\af}^*$ (resp. $\Fh_{\af}$) is given by
$r_i (\lambda) 
=
\lambda - \langle \alpha_i^{\vee} , \lambda \rangle \alpha_i$ 
(resp. 
$r_i (h) 
=
h - \langle h , \alpha_i \rangle \alpha_i^{\vee}$)
for $i \in I_{\af}$ and $\lambda \in \Fh_{\af}^*$
(resp. $h \in \Fh_{\af}$).
For $\xi \in Q^{\vee}$, 
we denote by 
$t_{\xi} \in W_{\af}$ 
the translation by $\xi$ (see \cite[\S6.5]{Kac}).
We know from 
\cite[Proposition 6.5]{Kac} 
that 
$\{ t_{\xi} \mid \xi \in Q^{\vee} \}$ 
forms an abelian normal subgroup of $W_{\af}$, 
for which 
$t_{\xi} t_{\zeta} = t_{\xi + \zeta}$, 
$\xi , \zeta \in Q^{\vee}$, 
and 
$W_{\af} = W \ltimes \{ t_{\xi} \mid \xi \in Q^{\vee} \}$. 
For $w \in W$ and $\xi \in Q^{\vee}$, we have 
$wt_{\xi}\lambda = w\lambda - \langle \xi,\lambda \rangle \delta$
if 
$\lambda \in \Fh_{\af}^*$
satisfies
$\langle \mathrm{c},\lambda \rangle = 0$.

Let $\Delta$ (resp. $\Delta^{\vee}$) be the root system (resp. the coroot system) of $\Fg$
with a simple root system 
$\Pi = \{ \alpha_i \mid i \in I \}$
(resp. a simple coroot system 
$\Pi^{\vee} = \{ \alpha_i^{\vee} \mid i \in I \}$). 
Set 
$\Delta^+ = \Delta \cap \sum_{i \in I} \BZ_{\geq 0} \alpha_i$
and 
$\Delta^{\vee,+} = \Delta^{\vee} \cap \sum_{i \in I} \BZ_{\geq 0} \alpha_i^{\vee}$. 
For a subset $J \subset I$, set
\begin{align}
Q_J &= \bigoplus_{j \in J} \BZ \alpha_j, & 
\Delta_J &= \Delta \cap Q_J, &
\Delta_J^+ &= \Delta^+ \cap Q_J, \\
Q_J^{\vee} &= \bigoplus_{j \in J} \BZ \alpha_j^{\vee}, & 
\Delta_J^{\vee} &= \Delta^{\vee} \cap Q_J^{\vee}, &
\Delta_J^{\vee,+} &= \Delta^{\vee,+} \cap Q_J^{\vee}.
\end{align}
Denote by $\Delta_{\af}$ the set of real roots of $\Fg_{\af}$, 
and by $\Delta_{\af}^+$ the set of positive real roots of $\Fg_{\af}$; 
we know from \cite[Proposition 6.3]{Kac} that
\begin{align}
\Delta_{\af} 
= 
\{ \alpha + n\delta \mid \alpha \in \Delta , n \in \BZ \} , &&
\Delta_{\af}^+
= 
\Delta^+ 
\sqcup
\{ \alpha + n\delta \mid \alpha \in \Delta , n \in \BZ_{>0} \} .
\end{align}
Let 
$\beta^{\vee} \in \Fh_{\af}$
denote the coroot of 
$\beta \in \Delta_{\af}$. 
Let 
$r_{\beta} \in W_{\af}$ 
be the reflection with respect to $\beta \in \Delta_{\af}$; 
if $\beta = \alpha + n\delta$, $\alpha \in \Delta$ and $n \in \BZ$, then
$r_{\beta} = r_{\alpha} t_{n\alpha^{\vee}}$.

\subsection{Crystals} \label{Subsection:Crystals}

In this subsection, 
we set up notation and terminology on crystals. 
For a fuller treatment, we refer the reader to 
\cite{K94,K02}. 

Let 
$\U$
be the quantized universal enveloping algebra associated with $\Fg_{\af}$. 
Let 
$\U'$
be the subalgebra of $\U$ 
corresponding to the derived subalgebra
$[\Fg_{\af},\Fg_{\af}]$ of $\Fg_{\af}$
(see for instance \cite[\S 2.2]{BN04}). 

A set
$\mathcal{B}$
together with the maps
$\mathrm{wt} : \mathcal{B} \rightarrow P_{\af}$
(resp. $\mathrm{wt} : \mathcal{B} \rightarrow P_{\af}/(P_{\af} \cap \BC \delta)$), 
$e_i , f_i : \mathcal{B} \rightarrow \mathcal{B} \sqcup \{ \bm{0} \}$, 
and 
$\varepsilon_i , \varphi_i : \mathcal{B} \rightarrow \BZ \sqcup \{ -\infty \}$,
$i \in I_{\af}$, 
is called a $\U$-crystal 
(resp. a $\U'$-crystal) 
if the following conditions are satisfied:
\begin{enumerate}[(C1)]
\item
$\varphi_i(b) = \varepsilon_i(b) + \langle \alpha_i^{\vee} , \mathrm{wt}(b) \rangle$ 
for 
$b \in \mathcal{B}$
and 
$i \in I_{\af}$, 
\item
$\mathrm{wt}(e_i b) = \mathrm{wt}(b) + \alpha_i$ 
if 
$e_i b \in \mathcal{B}$, 
\item
$\mathrm{wt}(f_i b) = \mathrm{wt}(b) - \alpha_i$ 
if 
$f_i b \in \mathcal{B}$, 
\item
$\varepsilon_i (e_i b) = \varepsilon_i(b) - 1$
and 
$\varphi_i (e_i b) = \varphi_i(b) + 1$
if 
$e_i b \in \mathcal{B}$, 
\item
$\varepsilon_i (f_i b) = \varepsilon_i(b) + 1$
and 
$\varphi_i (f_i b) = \varphi_i(b) - 1$
if 
$f_i b \in \mathcal{B}$, 
\item
$f_i b = b'$
if and only if 
$b = e_i b'$
for 
$b,b' \in \mathcal{B}$ 
and 
$i \in I_{\af}$, 
\item
if 
$\varphi_i(b) = -\infty$, 
then
$e_i b = f_i b = \bm{0}$.
\end{enumerate}

A set 
$\mathcal{B}$
together with the maps
$\mathrm{wt} : \mathcal{B} \rightarrow P$
and 
$e_i,f_i,\varepsilon_i,\varphi_i$
for 
$i \in I$
as above 
is called a 
$\Fg$-crystal 
if these maps satisfy 
(C1)--(C7). 
We can think of a 
$\U'$-crystal 
$\mathcal{B}$
such that 
$\mathrm{wt}(\mathcal{B})
\subset 
P/(P \cap \BC \delta)
\cong 
P$
as a 
$\Fg$-crystal 
by forgetting the maps 
$e_0,f_0,\varepsilon_0,\varphi_0$. 

Following 
\cite[\S 4.2]{K02}, 
define the affinization 
$\mathcal{B}_{\af}
=
\mathcal{B} \times \BZ$
of a 
$\U'$-crystal
$\mathcal{B}$
to be the 
$\U$-crystal
such that for 
$b \in \mathcal{B}$, 
$c \in \BZ$, 
and 
$i \in I_{\af}$, 
$\mathrm{wt}(b,c)
=
\iota(\mathrm{wt}(b)) - c\delta \in P_{\af}$, 
$e_i(b,c)
=
(e_ib,c-\delta_{i,0})$, 
$f_i(b,c)
=
(f_ib,c+\delta_{i,0})$, 
$\varepsilon_i(b,c)
=
\varepsilon_i(b)$, 
and 
$\varphi_i(b,c)
=
\varphi_i(b)$; 
we understand that 
$(\bm{0},c) = \bm{0}$.

Let 
$\mathcal{B}_1$
and 
$\mathcal{B}_2$
be $\U$-crystals or $\U'$-crystals.
A morphism 
$\Psi : \mathcal{B}_1 \rightarrow \mathcal{B}_2$
is, by definition, a map
$\mathcal{B}_1 \sqcup \{ \bm{0} \} \rightarrow \mathcal{B}_2 \sqcup \{ \bm{0} \}$
such that 
\begin{enumerate}[(CM1)]
\item
$\Psi (\bm{0}) = \bm{0}$, 
\item
if 
$b \in \mathcal{B}_1$ 
and 
$\Psi(b) \in \mathcal{B}_2$, 
then 
$\mathrm{wt}(\Psi(b)) = \mathrm{wt}(b)$, 
$\varepsilon_i (\Psi(b)) = \varepsilon_i (b)$, and
$\varphi_i(\Psi(b)) = \varphi_i(b)$ 
for all 
$i \in I_{\af}$, 
\item
if
$b , b' \in \mathcal{B}_1$, 
$\Psi(b) , \Psi(b') \in \mathcal{B}_2$ and 
$f_i b = b'$, 
then
$f_i \Psi(b) = \Psi(b')$
for all 
$i \in I_{\af}$. 
\end{enumerate}
A morphism
$\Psi : \mathcal{B}_1 \rightarrow \mathcal{B}_2$
is called strict if 
$\Psi(f_i b) = f_i \Psi (b)$
and 
$\Psi(e_i b) = e_i \Psi (b)$
for all 
$b \in \mathcal{B}_1$ and $i \in I_{\af}$.
A morphism
$\Psi : \mathcal{B}_1 \rightarrow \mathcal{B}_2$
is called a strict embedding if it is a strict morphism and the associated map 
$\mathcal{B}_1 \sqcup \{ \bm{0} \} \rightarrow \mathcal{B}_2 \sqcup \{ \bm{0} \}$
is injective. A morphism
$\Psi : \mathcal{B}_1 \rightarrow \mathcal{B}_2$
is called an isomorphism if the associated map 
$\mathcal{B}_1 \sqcup \{ \bm{0} \} \rightarrow \mathcal{B}_2 \sqcup \{ \bm{0} \}$
is bijective. 
In the same manner we define a morphism of $\Fg$-crystals. 

The tensor product
$\mathcal{B}_1 \otimes \mathcal{B}_2$ 
of crystals
$\mathcal{B}_1$
and 
$\mathcal{B}_2$ 
is defined to be the set 
$\{ b_1 \otimes b_2 \mid b_1 \in \mathcal{B}_1 , \, b_2 \in \mathcal{B}_2 \}$ 
whose crystal structure is as follows: 
\begin{enumerate}[(T1)]
\item
$\mathrm{wt}(b_1 \otimes b_2) = \mathrm{wt}(b_1) + \mathrm{wt}(b_2)$, 
\item
$\varepsilon_i (b_1 \otimes b_2) 
= 
\max \{ \varepsilon_i(b_1) , \varepsilon_i(b_2) - \langle \alpha_i^{\vee} , \mathrm{wt}(b_1) \rangle \}$, 
\item
$\varphi_i (b_1 \otimes b_2) 
= 
\max \{ \varphi_i(b_2) , \varphi_i(b_1) + \langle \alpha_i^{\vee} , \mathrm{wt}(b_2) \rangle \}$, 
\item
$e_i (b_1 \otimes b_2) 
=
\begin{cases}
(e_i b_1) \otimes b_2 
&\text{if} \ \varphi_i(b_1) \geq \varepsilon_i(b_2), \\
b_1 \otimes (e_i b_2) 
&\text{if} \ \varphi_i(b_1) < \varepsilon_i(b_2),
\end{cases}$
\item
$f_i (b_1 \otimes b_2) 
=
\begin{cases}
(f_i b_1) \otimes b_2 
&\text{if} \ \varphi_i(b_1) > \varepsilon_i(b_2), \\
b_1 \otimes (f_i b_2) 
&\text{if} \ \varphi_i(b_1) \leq \varepsilon_i(b_2).
\end{cases}$
\end{enumerate}
Here, we understand that
$b_1 \otimes \bm{0} = \bm{0} \otimes b_2 = \bm{0}$. 

Let 
$\mathcal{B}$ 
be a regular $\U$-crystal in the sense of 
\cite[\S 2.2]{K02}. 
It follows that
\begin{align}
\varphi_i(b) = \max \{ k \in \BZ_{\geq 0} \mid f_i^k b \neq \bm{0} \} , 
\hspace{5mm}
\varepsilon_i(b) = \max \{ k \in \BZ_{\geq 0} \mid e_i^k b \neq \bm{0} \}
\end{align}
for 
$b \in \mathcal{B}$ 
and 
$i \in I_{\af}$. 
Define 
$f_i^{\max} b = f_i^{\varphi_i(b)} b \in \mathcal{B}$
and 
$e_i^{\max} b = e_i^{\varepsilon_i(b)} b \in \mathcal{B}$
for 
$b \in \mathcal{B}$ 
and 
$i \in I_{\af}$. 
By \cite[\S 7]{K94}, 
we have a $W_{\af}$-action 
$S : W_{\af} \rightarrow \mathfrak{S}(\mathcal{B})$, 
$x \mapsto S_x$, 
on (the underlying set) 
$\mathcal{B}$
given by
\begin{align} \label{eq:W-action}
S_{r_i} b
=
\begin{cases}
f_i^{\langle \alpha_i^{\vee} , \mathrm{wt}(b) \rangle} b & \text{if}\ \langle \alpha_i^{\vee} , \mathrm{wt}(b) \rangle \geq 0, \\[1mm]
e_i^{-\langle \alpha_i^{\vee} , \mathrm{wt}(b) \rangle} b & \text{if}\ \langle \alpha_i^{\vee} , \mathrm{wt}(b) \rangle \leq 0
\end{cases}
\end{align}
for 
$b \in \mathcal{B}$ 
and 
$i \in I_{\af}$. 
Note that 
$\mathrm{wt}(S_x b) = x \mathrm{wt}(b)$ 
holds for all
$x \in W_{\af}$ and $b \in \mathcal{B}$. 
An element 
$b \in \mathcal{B}$ 
of weight 
$\mathrm{wt}(b) = \lambda \in P_{\af}$
is called an extremal element if 
there exist
$b_x \in \mathcal{B}$, $x \in W_{\af}$, 
such that 
\begin{enumerate}[(E1)]
\item
$b_e = b$, 
\item
if 
$\langle \alpha_i^{\vee} , x\lambda \rangle \geq 0$, 
then
$e_i b_x = \bm{0}$
and 
$f_i^{\langle \alpha_i^{\vee} , x\lambda \rangle} b_x = f_i^{\max}b_x = b_{r_i x}$, 
\item
if 
$\langle \alpha_i^{\vee} , x\lambda \rangle \leq 0$, 
then
$f_i b_x = \bm{0}$
and 
$e_i^{-\langle \alpha_i^{\vee} , x\lambda \rangle} b_x = e_i^{\max}b_x = b_{r_i x}$. 
\end{enumerate}
Then 
$b_x = S_x b$
holds for all 
$x \in W_{\af}$, 
and each 
$b_x$
is an extremal element of 
weight $x\lambda$.
The proof of the next lemma is straightforward.

\begin{lem} \label{lem:ext-elm}
\begin{enumerate}[(1)]
\item
Let 
$\mathcal{B}$ 
be a regular $\U$-crystal, 
and let 
$b \in \mathcal{B}$ 
be an extremal element. 
If there exist 
$i_1 , i_2 , \ldots , i_N \in I_{\af}$
such that 
\begin{align}
\langle \alpha_{i_n}^{\vee} , r_{i_{n-1}} \cdots r_{i_2} r_{i_1} \mathrm{wt}(b) \rangle \geq 0
\ 
\text{for all}
\ 
n \in [N],
\end{align} 
then 
\begin{align}
f_{i_N}^{\max} \cdots f_{i_2}^{\max} f_{i_1}^{\max} b
=
S_{r_{i_N} \cdots r_{i_2} r_{i_1}} b .
\end{align}
\item
Let 
$\mathcal{B}_1 , \mathcal{B}_2 , \ldots , \mathcal{B}_M$
be regular $\U$-crystals, 
and let 
$b_{\nu} \in \mathcal{B}_{\nu}$, 
$\nu \in [M]$, 
be extremal elements such that, 
for every
$\beta \in \Delta_{\af}$, 
$\langle \beta^{\vee} , \mathrm{wt}(b_{\nu}) \rangle$, 
$\nu \in [M]$, 
are all nonnegative or all nonpositive. 
Then the equalities
\begin{align}
f_i^{\max}(b_1 \otimes b_2 \otimes \cdots \otimes b_M)
&= 
f_i^{\max} b_1 \otimes f_i^{\max} b_2 \otimes \cdots \otimes f_i^{\max} b_M , 
\label{eq:f^max}
\\[2mm]
S_x(b_1 \otimes b_2 \otimes \cdots \otimes b_M)
&= 
S_x b_1 \otimes S_x b_2 \otimes \cdots \otimes S_x b_M
\label{eq:S}
\end{align}
hold and \eqref{eq:f^max}--\eqref{eq:S} are extremal elements 
for all 
$i \in I_{\af}$
and 
$x \in W_{\af}$.
\end{enumerate}
\end{lem}

\subsection{Semi-infinite Bruhat order} \label{Subsection:SiBO}

In this subsection, 
we collect some basic facts on 
semi-infinite Bruhat order
on an affine Weyl group
(see \cite{INS16,LS,P} for more details). 
Throughout this subsection, we take and fix 
$J \subset I$. 

Let
$W_J = \langle r_j \mid j \in J \rangle$, 
and let 
$W^J$
be the set of minimal(-length) coset representatives for $W/W_J$.
For 
$w \in W$, 
we denote by 
$\lfloor w \rfloor = \lfloor w \rfloor^J \in W^J$ 
the minimal coset representative for the coset 
$wW_J \in W/W_J$.
Define
\begin{align}
(\Delta_J)_{\af} 
&=
\{ \alpha + n\delta \mid \alpha \in \Delta_J , n \in \BZ \} \subset \Delta_{\af}, \label{eq:Delta_J_af} \\[1mm]
(\Delta_J)_{\af}^+ 
&=
(\Delta_J)_{\af} \cap \Delta_{\af}^+ 
=
\Delta_J^+ 
\sqcup
\{ \alpha + n\delta \mid \alpha \in \Delta_J , n \in \BZ_{>0} \}, \\[1mm]
(W_J)_{\af}
&=
W_J \ltimes \{ t_{\xi} \mid \xi \in Q_J^{\vee} \}
=
\langle r_{\beta} \mid \beta \in (\Delta_J)_{\af}^+ \rangle , \label{eq:W_J_af} \\[1mm]
(W^J)_{\af} 
&=
\{ x \in W_{\af} \mid x\beta \in \Delta_{\af}^+ \ \text{for all} \ \beta \in (\Delta_J)_{\af}^+ \} ; \label{eq:W^J_af}
\end{align}
note that 
$(W_{\emptyset})_{\af} = \{ e \}$ 
and 
$(W^{\emptyset})_{\af} = W_{\af}$.

We see from \cite{P} (see also \cite[Lemma 10.5]{LS}) that, 
for each $x \in W_{\af}$, 
there exist a unique $x_1 \in (W^J)_{\af}$
and a unique $x_2 \in (W_J)_{\af}$ such that
$x = x_1 x_2$. 
Define 
$\Pi^J : W_{\af} \rightarrow (W^J)_{\af}$ 
by $\Pi^J(x) = x_1$ 
if $x = x_1 x_2$ with 
$x_1 \in (W^J)_{\af}$
and 
$x_2 \in (W_J)_{\af}$. 
It follows immediately from 
\eqref{eq:Delta_J_af}--\eqref{eq:W^J_af}
that 
$\Pi^J = \Pi^J \circ \Pi^K$
if 
$K \subset J$. 

Set $\rho_J = (1/2)\sum_{\alpha \in \Delta_J^+} \alpha$; 
we abbreviate $\rho_J$ to $\rho$ if $J = I$. 
Define the semi-infinite length function 
$\lsi : W_{\af} \to \BZ$
by
\begin{align} \label{eq:sil}
\lsi(x) = \ell(w) + 2 \langle \xi,\rho \rangle
\end{align}
for $x = wt_{\xi} \in W_{\af}$ with $w \in W$ and $\xi \in Q^{\vee}$.

Define the (parabolic) semi-infinite Bruhat graph $\mathrm{SiB}^J$ 
to be the $\Delta_{\af}^+$-colored directed graph 
with vertex set $(W^J)_{\af}$ and edges of the form 
$x \xrightarrow{\ \beta \ } r_{\beta} x$ 
for 
$x \in (W^J)_{\af}$ and $\beta \in \Delta_{\af}^+$, 
where
$r_{\beta} x \in (W^J)_{\af}$ 
and 
$\lsi(r_{\beta}x) = \lsi(x) + 1$. 
We know from 
\cite[Appendix A]{INS16}
that the existence of the edge 
$x \xrightarrow{\ \beta \ } \Pi^J(r_{\beta}x)$
in 
$\mathrm{SiB}^J$
implies
$r_{\beta}x = \Pi^J (r_{\beta}x) \in (W^J)_{\af}$.
The semi-infinite Bruhat order is a partial order $\succeq$ on $(W^J)_{\af}$ 
defined as follows: 
for $x , y \in (W^J)_{\af}$, write $x \succeq y$ 
if there exists a directed path from $y$ to $x$ in $\mathrm{SiB}^J$. 
Write 
$x \succ y$ 
if 
$x \succeq y$ 
and 
$x \neq y$.

\begin{lem} \label{lem:act-wt}
Let 
$\lambda \in P^+$
be such that 
$J = \{ i \in I \mid \langle \alpha_i^{\vee} , \lambda \rangle = 0 \}$. 
We have 
$x \lambda = \Pi^J(x) \lambda$
for all 
$x \in W_{\af}$.
\end{lem}

\begin{proof}
The assertion follows from 
$\langle \beta^{\vee} , \lambda \rangle = 0$
for all
$\beta \in (\Delta_J)_{\af}^+$. 
\end{proof}

\begin{lem}[{\cite{P}; see also \cite[Lemma 10.7 and Proposition 10.10]{LS}}] \label{lem:Pi} \ 
\begin{enumerate}[(1)]
\item
$\Pi^J(w) = \lfloor w \rfloor^J$ 
for $w \in W$.
\item
$\Pi^J(x t_{\xi}) = \Pi^J(x) \Pi^J(t_{\xi})$
for $x \in W_{\af}$ and $\xi \in Q^{\vee}$.
%
\end{enumerate}
\end{lem}

For simplicity of notation, we let $T_{\xi}^J$ stand for 
$\Pi^J(t_{\xi}) \in (W^J)_{\af}$ for $\xi \in Q^{\vee}$. 
By Lemma \ref{lem:Pi}, we have
\begin{align} \label{eq:(W^J)_af=wT_xi}
(W^J)_{\af} = \left\{ wT^J_{\xi} \ \vline \ w \in W^J , \ \xi \in Q^{\vee} \right\} .
\end{align}

Let
\begin{align}
[\,\cdot\,]_{I \setminus J} : 
Q^{\vee} = Q^{\vee}_J \oplus Q^{\vee}_{I \setminus J} \rightarrow Q^{\vee}_{I \setminus J}
\end{align}
be the projection. 
For 
$\xi_1 , \xi_2 \in Q^{\vee}$, 
write 
$\xi_1 \succeq \xi_2$
if 
$\xi_1 - \xi_2 \in \sum_{i \in I} \BZ_{\geq 0} \alpha_i^{\vee}$. 

\begin{lem}[{\cite[Lemmas 6.1.1 and 6.2.1]{INS16}}] \label{lem:order} \ 
\begin{enumerate}[(1)]
\item
Let 
$K \subset J$
and 
$x,y \in (W^K)_{\af}$. 
If 
$x \succeq y$
in 
$(W^K)_{\af}$, 
then 
$\Pi^J(x) \succeq \Pi^J(y)$
in 
$(W^J)_{\af}$. 
\item
Let 
$\xi_1 , \xi_2 \in Q^{\vee}$. 
We have 
$\Pi^J(t_{\xi_1}) \succeq \Pi^J(t_{\xi_2})$
if and only if 
$[\xi_1]_{I \setminus J} \succeq [\xi_2]_{I \setminus J}$.
In particular, we have 
$\Pi^J(t_{\xi_1}) \succeq \Pi^J(t_{\xi_2})$
if and only if 
$\Pi^{I \setminus \{ i \}}(t_{\xi_1}) \succeq \Pi^{I \setminus \{ i \}}(t_{\xi_2})$
for all 
$i \in I \setminus J$.
\end{enumerate}
\end{lem}

\begin{lem}[{\cite[Remark 4.1.3]{INS16}}] \label{lem:r_ix}
Let 
$x \in (W^J)_{\af}$, 
$i \in I_{\af}$, 
and let 
$\lambda \in P^+$
be such that 
$J = \{ j \in I \mid \langle \alpha_j^{\vee} , \lambda \rangle = 0 \}$.
\begin{enumerate}[(1)]
\item
$r_i x \in (W^J)_{\af}$
if and only if 
$\langle \alpha_i^{\vee} , x\lambda \rangle \neq 0$.
\item
$\Pi^J(r_i x) = x$
if and only if 
$\langle \alpha_i^{\vee} , x\lambda \rangle = 0$.
\item
$r_i x \xleftarrow{\alpha_i} x$
(resp. 
$x \xleftarrow{\alpha_i} r_i x$)
if and only if 
$\langle \alpha_i^{\vee} , x\lambda \rangle > 0$
(resp.
$\langle \alpha_i^{\vee} , x\lambda \rangle < 0$).
\end{enumerate}
\end{lem}

The next lemma is a reformulation of 
the ``diamond lemma" for semi-infinite Bruhat order 
obtained in \cite[\S 4.1]{INS16}.

\begin{lem}\label{lem:diamond}
Let 
$x,y \in (W^J)_{\af}$ 
and 
$i \in I_{\af}$
be such that 
$\Pi^J(r_ix) \succeq x$ 
and 
$\Pi^J(r_iy) \succeq y$.
\begin{enumerate}[(1)]
\item
If 
$\Pi^J(r_i x) \succeq y$, 
then
$x \succeq y$
and 
$\Pi^J(r_ix) \succeq \Pi^J(r_iy)$.
\item
$x \succeq y$
if and only if 
$\Pi^J(r_ix) \succeq \Pi^J(r_iy)$.
\end{enumerate}
\end{lem}

%
%

\begin{lem} \label{lem:wt-wt}
Let 
$w,v \in W$ 
and 
$\xi , \zeta \in Q^{\vee}$. 
If
$\Pi^J(wt_{\xi}) \succeq \Pi^J(wt_{\zeta})$, 
then
$\Pi^J(vt_{\xi}) \succeq \Pi^J(vt_{\zeta})$. 
\end{lem}

\begin{proof}
It suffices to prove that 
$\Pi^J(wt_{\xi}) \succeq \Pi^J(wt_{\zeta})$
if and only if 
$\Pi^J(t_{\xi}) \succeq \Pi^J(t_{\zeta})$. 
The proof is by induction on $\ell(w)$. 
If $\ell(w) = 0$, then the assertion is obvious. 
Assume that $\ell(w) > 0$.
Let 
$i \in I$ 
be such that 
$\ell(r_i w) < \ell(w)$. 
By induction hypothesis, 
$\Pi^J(r_iwt_{\xi}) \succeq \Pi^J(r_iwt_{\zeta})$
if and only if 
$\Pi^J(t_{\xi}) \succeq \Pi^J(t_{\zeta})$. 
The proof is completed by showing that 
$\Pi^J(r_iwt_{\xi}) \succeq \Pi^J(r_iwt_{\zeta})$
if and only if 
$\Pi^J(wt_{\xi}) \succeq \Pi^J(wt_{\zeta})$.
Let 
$\lambda \in P^+$
be such that 
$J = \{ j \in I \mid \langle \alpha_j^{\vee} , \lambda \rangle = 0 \}$; 
note that 
$0 \geq \langle \alpha_i^{\vee} , w \lambda \rangle 
= \langle \alpha_i^{\vee} , \Pi^J(wt_{\xi}) \lambda \rangle
= \langle \alpha_i^{\vee} , \Pi^J(wt_{\zeta}) \lambda \rangle$. 
We see from Lemma \ref{lem:r_ix} (2)--(3) that 
$\Pi^J(wt_{\xi}) 
\succeq 
\Pi^J(r_iwt_{\xi})$
and 
$\Pi^J(wt_{\zeta}) 
\succeq 
\Pi^J(r_iwt_{\zeta})$.
By Lemma \ref{lem:diamond} (2), 
$\Pi^J(r_iwt_{\xi}) \succeq \Pi^J(r_iwt_{\zeta})$
if and only if 
$\Pi^J(wt_{\xi}) \succeq \Pi^J(wt_{\zeta})$.
\end{proof}

\subsection{Extremal weight modules and their crystal bases} \label{Subsection:extremal}

In this subsection, 
following \cite{BN04,K94,K02}, 
we review some of the standard facts on 
extremal weight modules and their crystal bases.

For 
$\lambda \in P^+$, 
let 
$V(\lambda)$ 
be the extremal weight $\U$-module generated by an extremal weight vector 
$u_{\lambda}$ 
of extremal weight 
$\lambda$, 
and let 
$\mathcal{B}(\lambda)$ 
be the crystal basis of 
$V(\lambda)$
(\cite[Proposition 8.2.2]{K94}; see also \cite[\S 3.2]{K02}). 
Note that 
$\mathcal{B}(\lambda)$ 
is a regular $\U$-crystal in the sense of 
\cite[\S 2.2]{K02} (see \S \ref{Subsection:Crystals}). 
Let $z_i$, $i \in I$, be the $\U'$-linear automorphism of $V(\varpi_i)$ of weight $\delta$ introduced in \cite[\S 5.2]{K02}; $z_i$ sends a (unique) global basis element of weight $\varpi_i$ to a (unique) global basis element of weight $\varpi_i+\delta$. 
Then $z_i$ induces an automorphism of 
$\mathcal{B}(\varpi_i)$ 
as a $\U'$-crystal; 
by abuse of notation, 
we use the same letter 
$z_i$ 
for the automorphism of 
$\mathcal{B}(\varpi_i)$. 
The $\U'$-module 
$W(\varpi_i) = V(\varpi_i) / (z_i - 1)V(\varpi_i)$
is called a level-zero fundamental representation. 
We know from \cite[Theorem 5.17]{K02} that 
$W(\varpi_i)$ is a finite-dimensional irreducible $\U'$-module 
and has a (simple) crystal basis. 

For 
$\lambda = \sum_{i \in I} m_i \varpi_i \in P^+$, 
with 
$m_i \in \BZ_{\geq 0}$, $i \in I$, 
set 
$\Breve{\mathcal{B}}(\lambda) = \bigotimes_{i \in I} \mathcal{B}(\varpi_i)^{\otimes m_i}$. 
For each 
$i \in I$ 
and 
$\nu \in [m_i]$,
let $z_{i,\nu}$ be the automorphism 
of the $\U'$-crystal 
$\Breve{\mathcal{B}}(\lambda)$
obtained by the action of $z_i$ on the $\nu$-th factor $\mathcal{B}(\varpi_i)$ 
of 
$\mathcal{B}(\varpi_i)^{\otimes m_i}$ 
in 
$\Breve{\mathcal{B}}(\lambda)$. 
Set 
\begin{align} \label{eq:Par}
\begin{split}
\mathrm{Par}(\lambda) 
= 
\bigl\{ 
&
\bm{\rho} = (\rho^{(i)})_{i \in I} \\
&
\mid
\rho^{(i)} \ \text{is a partition of length less than} \ m_i \ \text{for} \ i \in I 
\bigr\} ;
\end{split}
\end{align}
we understand that a partition of length less than $1$ is an empty partition $\emptyset$. 
Let 
$\bm{\rho} = (\rho^{(i)})_{i \in I} \in \mathrm{Par}(\lambda)$
and 
$\rho^{(i)} = (\rho^{(i)}_1 \geq \rho^{(i)}_2 \geq \cdots \geq \rho^{(i)}_{m_i -1} \geq 0)$, 
$i \in I$. 
Define the automorphism 
$z^{-\bm{\rho}}$ 
of the $\U'$-crystal 
$\Breve{\mathcal{B}}(\lambda)$
by
\begin{align}
z^{-\bm{\rho}}
= 
\prod_{i \in I}
z_{i,1}^{- \rho^{(i)}_1} 
z_{i,2}^{- \rho^{(i)}_2}
\cdots 
z_{i,m_i-1}^{- \rho^{(i)}_{m_i-1}}.
\end{align}
Let 
$S_{\bm{\rho}}^-$
be the (PBW-type) basis element 
of weight 
\begin{align}
\mathrm{wt}(\bm{\rho})
= -\sum_{i \in I} \sum_{\nu = 1}^{m_i-1} \rho^{(i)}_{\nu} \delta
\end{align}
of the negative imaginary part of $\U$ 
constructed in \cite[the element $S_{\mathbf{c}_0}^-$ in \S 3.1; see also Remark 4.1]{BN04}.
We know from 
{\cite[\S 4.2]{BN04}}
that
\begin{align}
\begin{split}
\mathcal{B}(\lambda)
=
\bigl\{ 
&
g_1 g_2 \cdots g_l S_{\bm{\rho}}^- u_{\lambda} 
\\
&
\mid 
g_k \in \{ e_i , f_i \mid i \in I_{\af} \} , \ 
k \in [l] , \ 
l \in \BZ_{\geq 0}, \ 
\bm{\rho} \in \mathrm{Par}(\lambda) 
\bigr\} 
\setminus
\left\{ 0 \right\} .
\end{split}
\end{align}
Define the map
$\Phi_{\lambda \mid q=0}^{\mathrm{LT}} : 
\mathcal{B}(\lambda) \rightarrow 
\Breve{\mathcal{B}}(\lambda)$
by 
$g_1 g_2 \cdots g_l S_{\bm{\rho}}^- u_{\lambda}
\mapsto
g_1 g_2 \cdots g_l z^{-\bm{\rho}} \Tilde{u}_{\lambda}$, 
where 
$g_k \in \{ e_i , f_i \mid i \in I_{\af} \}$, 
$k \in [l]$, 
$l \in \BZ_{\geq 0}$, and 
$\bm{\rho} \in \mathrm{Par}(\lambda)$. 
We know from 
\cite[Lemma 3.1]{I20}
that the map 
$\Phi_{\lambda \mid q=0}^{\mathrm{LT}}$ 
is well-defined, and is a strict embedding of $\U$-crystals.

The next theorem will be needed in \S \ref{Section:Tab-model}.

\begin{thm}[{\cite[Remark 4.17]{BN04}}; see also {\cite[Conjecture 13.1 (iii)]{K02}}] \label{thm:BN-Kashiwara}
Let 
$\lambda = \sum_{i \in I} m_i \varpi_i \in P^+$. 
We have an isomorphism
$\mathcal{B}(\lambda)
\cong 
\bigotimes_{i \in I} 
\mathcal{B}(m_i \varpi_i)$
of
$\U$-crystals. 
\end{thm}

\subsection{Path model for Demazure crystals} \label{Subsection:Path-Demazure}

In this subsection, 
we give a brief exposition of the 
path model for the crystal bases of 
level-zero extremal weight $\U$-modules 
and their Demazure submodules. 
For a fuller treatment we refer the reader to 
\cite{BN04,INS16,K94,K02,K05,NS16}.

For 
$\lambda = \sum_{i \in I} m_i \varpi_i \in P^+$, 
set 
\begin{align} \label{eq:J_lambda}
J_{\lambda} = \{ j \in I \mid \langle \alpha_j^{\vee} , \lambda \rangle = 0 \} , 
\hspace{1cm}
J_{\lambda}^c = I \setminus J_{\lambda} .
\end{align}
For a rational number 
$0 < a \leq 1$, 
define 
$\mathrm{SiB}(\lambda;a)$
to be the subgraph of 
$\mathrm{SiB}^{J_{\lambda}}$ 
with the same vertex set 
but having only the edges of the form 
\begin{align}
x \xrightarrow{\ \beta \ } y
\ \ \text{with} \ \ 
a \langle \beta^{\vee} , x\lambda \rangle \in \BZ ;
\end{align}
note that 
$\mathrm{SiB}(\lambda ; 1) = \mathrm{SiB}^{J_{\lambda}}$. 
A semi-infinite Lakshmibai--Seshadri path of shape $\lambda$ is, by definition, 
a pair 
$(\bm{x}; \bm{a})$
of a decreasing sequence 
$\bm{x} : x_1 \succeq x_2 \succeq \cdots \succeq x_l$ 
of elements in $(W^{J_{\lambda}})_{\af}$
and an increasing sequence
$\bm{a} : 0 = a_0 < a_1 < \cdots < a_l = 1$ 
of rational numbers such that 
there exists a directed path from
$x_{u+1}$ to $x_u$ in $\mathrm{SiB}(\lambda;a_u)$ 
for each 
$u \in [l-1]$. 
Let 
$\mathbb{B}^{\si}(\lambda)$ 
denote the set of semi-infinite Lakshmibai--Seshadri paths of shape $\lambda$. 

Following 
\cite[\S 3.1]{INS16}, 
we equip the set 
$\mathbb{B}^{\si}(\lambda)$ 
with a $\U$-crystal structure.
For 
$\eta = (x_1 , \ldots , x_l ; a_0 , \ldots , a_l) \in \mathbb{B}^{\si}(\lambda)$, 
define the map 
$\bar{\eta} : [0,1] \rightarrow \BR \otimes_{\BZ} P_{\af}$ 
by
\begin{align} \label{eq:path}
\bar{\eta}(t) 
=
\sum_{p=1}^{u-1} (a_p - a_{p-1}) x_p \lambda + (t - a_{u-1}) x_u \lambda 
\ \ 
\text{for} \ 
t \in [a_{u-1} , a_u] 
\ \text{and} \  
u \in [l].
\end{align}
Define 
$\mathrm{wt} : \mathbb{B}^{\si}(\lambda) \rightarrow P_{\af}$
by
$\mathrm{wt}(\eta) = \bar{\eta}(1) \in P_{\af}$.
Set
\begin{align} \label{eq:path-h-m}
h_i^{\eta}(t) 
= 
\langle \alpha_i^{\vee} , \bar{\eta}(t) \rangle 
\ \ \text{for}\ 
t \in [0,1] , &&
m_i^{\eta} 
= 
\min \{ h_i^{\eta}(t) \mid t \in [0,1] \} .
\end{align}

We define 
$e_i \eta$, $i \in I_{\af}$, 
as follows: 
if $m_i^{\eta} = 0$, 
then we set $e_i \eta = \bm{0}$.
If $m_i^{\eta} \leq -1$, then we set
\begin{align} \label{eq:e_i-t}
\begin{cases}
t_1 = \min \{ t \in [0,1] \mid h_i^{\eta}(t) = m_i^{\eta} \} , \\[1mm]
t_0 = \max \{ t \in [0,t_1] \mid h_i^{\eta}(t) = m_i^{\eta} + 1 \} .
\end{cases}
\end{align}
Let 
$1 \leq p \leq q \leq l$
be such that 
$a_{p-1} \leq t_0 < a_p$ and $t_1 = a_q$. 
Then we define 
\begin{align} \label{eq:e_i}
\begin{split}
e_i \eta 
= 
(x_1 , \ldots , x_p , 
&
r_i x_p , \ldots , r_i x_q , x_{q+1} , \ldots , x_l;\\
&
a_0 , \ldots , a_{p-1} , t_0 , a_p , \ldots , a_q = t_1 , \ldots , a_l) ;
\end{split}
\end{align}
if $t_0 = a_{p-1}$, 
then we drop $x_p$ and $a_{p-1}$, 
and if $r_j x_q = x_{q+1}$, 
then we drop $x_{q+1}$ and $a_q = t_1$.

Next, we define $f_i \eta$, $i \in I_{\af}$, as follows: 
if $m_i^{\eta} = h_i^{\eta}(1)$, then we set $f_i \eta = \bm{0}$. 
If $h_i^{\eta}(1) - m_i^{\eta} \geq 1$, then we set
\begin{align} \label{eq:f_i-t}
\begin{cases}
t_0 = \max \{ t \in [0,1] \mid h_i^{\eta}(t) = m_i^{\eta} \} , \\[1mm]
t_1 = \min \{ t \in [t_0,1] \mid h_i^{\eta}(t) = m_i^{\eta} + 1 \} .
\end{cases}
\end{align}
Let 
$1 \leq p \leq q \leq l-1$
be such that
$t_0 = a_p$ and $a_q < t_1 \leq a_{q+1}$. 
Then we define
\begin{align}
\begin{split}
f_i \eta 
= 
(x_1 , \ldots , x_p , 
&
r_i x_{p+1} , \ldots , r_i x_{q+1} , x_{q+1} , \ldots , x_l ; \\
&
a_0 , \ldots , a_p = t_0 , \ldots , a_q , t_1 , a_{q+1} , \ldots , a_l) ;
\end{split}
\end{align}
if $t_1 = a_{q+1}$, 
then we drop $x_{q+1}$ and $a_{q+1}$, 
and if $x_p = r_i x_{p+1}$, 
then we drop $x_p$ and $a_p = t_0$.

For $\eta \in \mathbb{B}^{\si}(\lambda)$ and $i \in I_{\af}$, define
\begin{align} \label{eq:path-eps-phi}
\begin{cases}
\varepsilon_i(\eta) = -m_i^{\eta} , \\[1mm]
\varphi_i(\eta) = h_i^{\eta}(1) - m_i^{\eta} .
\end{cases}
\end{align}

For 
$\eta = 
(x_1 , x_2 , \ldots , x_l ; \bm{a}) \in \mathbb{B}^{\si}(\lambda)$, 
set 
$\kappa(\eta) = x_l$. 
Following 
\cite[Equation (4.2.1)]{NS16}, 
for each 
$x \in (W^{J_{\lambda}})_{\af}$, 
set 
\begin{align} \label{eq:model-Dem}
\mathbb{B}^{\si}_{\succeq x}(\lambda) 
=
\left\{ 
\eta \in \mathbb{B}^{\si}(\lambda) 
\ \vline \ 
\kappa(\eta) \succeq x 
\right\} .
\end{align}

Following \cite[Equation (7.2.2)]{INS16}, 
we define an extremal element 
$\eta_{\bm{\rho}} \in \mathbb{B}^{\si}(\lambda)$ 
of weight 
$\lambda + \mathrm{wt}(\bm{\rho})$
for each 
$\bm{\rho} = \left( \rho^{(i)} \right)_{i \in I} \in \mathrm{Par}(\lambda)$, 
with 
$\rho^{(i)} = \left(\rho^{(i)}_1 \geq \rho^{(i)}_2 \geq \cdots \geq \rho^{(i)}_{m_i-1} \geq \rho^{(i)}_{m_i} = 0 \right)$.
Let 
$s$ 
be the least common multiple of 
$\{ m_i \mid i \in J_{\lambda}^c \}$. 
Let 
$c_i (\xi) \in \BZ$ 
denote the coefficient of 
$\alpha_i^{\vee}$ 
in 
$\xi \in Q^{\vee}$.
For 
$\xi , \zeta \in Q^{\vee}$, 
write 
$\xi \succ \zeta$
if 
$\xi \succeq \zeta$ 
and 
$\xi \neq \zeta$. 
Let 
$\zeta_1 , \ldots , \zeta_s \in Q^{\vee}$
be such that 
\begin{enumerate}[(i)]
\item
$c_i(\zeta_t) = \rho^{(i)}_u$ 
if $i \in J_{\lambda}^c$ 
and $\dfrac{s(u-1)}{m_i} < t \leq \dfrac{su}{m_i}$, and 
\item
$c_j(\zeta_t) = 0$ 
for all 
$j \in J_{\lambda}$ 
and 
$t \in [s]$; 
\end{enumerate}
note that 
$\zeta_1 \succeq \cdots \succeq \zeta_s$
and 
$\zeta_s = 0$. 
Assume that
\begin{align}
\zeta_1 = \cdots = \zeta_{s_1} \succ \zeta_{s_1 + 1} = \cdots = \zeta_{s_2} \succ \cdots \cdots 
\succ \zeta_{s_{k-1} + 1} = \cdots = \zeta_{s_k}, 
\end{align}
where 
$1 \leq s_1 < \cdots < s_{k-1} < s_k = s$.
Set 
\begin{align} \label{eq:eta_rho}
\eta_{\bm{\rho}} 
=
\left( 
T_{\zeta_{s_1}}^{J_{\lambda}} , T_{\zeta_{s_2}}^{J_{\lambda}} , \ldots , T_{\zeta_{s_{k-1}}}^{J_{\lambda}} , e ; \, 
0 , \dfrac{s_1}{s} , \dfrac{s_2}{s} , \ldots , \dfrac{s_{k-1}}{s} , 1
\right) .
\end{align}
%

\begin{thm}[{\cite[Theorem 3.2.1 and Proposition 7.2.1]{INS16}}] \label{thm:SiLS-isom}
Let 
$\lambda \in P^+$.
\begin{enumerate}[(1)]
\item
For each connected component 
$C$ 
of (the crystal graph of) 
$\mathbb{B}^{\si}(\lambda)$, 
there exists a unique
$\bm{\rho} \in \mathrm{Par}(\lambda)$
such that 
$\eta_{\bm{\rho}} \in C$. 
\item
There exists a unique isomorphism 
$\varphi_{\lambda} : \mathcal{B}(\lambda) \rightarrow \mathbb{B}^{\si}(\lambda)$
of $\U$-crystals sending 
$S_{\bm{\rho}}^- u_{\lambda}$ 
to 
$\eta_{\bm{\rho}}$
for every 
$\bm{\rho} \in \mathrm{Par}(\lambda)$. 
\end{enumerate}
\end{thm}

\begin{rem}
Let 
$\lambda \in P^+$. 
\begin{enumerate}[(1)]
\item
Since 
$\mathcal{B}(\lambda)$
is a regular $\U$-crystal, 
it follows from 
Theorem \ref{thm:SiLS-isom} (2) that 
$\mathbb{B}^{\si}(\lambda)$
is also a regular $\U$-crystal. 
Hence $W_{\af}$ acts on
$\mathbb{B}^{\si}(\lambda)$
as 
\eqref{eq:W-action}. 
By \eqref{eq:eta_rho} and \cite[Remark 3.5.2 (2)]{NS16}, 
\begin{align} \label{eq:S_x-rho}
S_x \eta_{\bm{\rho}}
=
\left( 
\Pi^{J_{\lambda}} \left( xT^{J_{\lambda}}_{\zeta_{s_1}} \right) , 
\ldots , 
\Pi^{J_{\lambda}} \left( xT^{J_{\lambda}}_{\zeta_{s_{k-1}}} \right) , 
\Pi^{J_{\lambda}}(x) ; \,
0 , \dfrac{s_1}{s} , \ldots , \dfrac{s_{k-1}}{s} , 1
\right)
\end{align}
for 
$\bm{\rho} \in \mathrm{Par}(\lambda)$
and 
$x \in W_{\af}$. 
\item
Let 
$\mathcal{B}^-_x(\lambda)$
be the (opposite) Demazure subcrystal of 
$\mathcal{B}(\lambda)$ 
associated with 
$x \in (W^{J_{\lambda}})_{\af}$
(\cite[\S 4.1]{NS16}; see also 
\cite[\S 2.8]{K05}).
We know from 
\cite[Theorem 4.2.1]{NS16}
that 
$\varphi_{\lambda}(\mathcal{B}^-_x(\lambda)) = \mathbb{B}^{\si}_{\succeq x}(\lambda)$. 
However, we will not use this fact in any essential way
in the remainder of this paper.
\end{enumerate}
\end{rem}

\section{Deodhar-type criterion for semi-infinite Bruhat order}
\label{Section:Deodhar}

This section is devoted to the proof of the next theorem.  

\begin{thm} \label{thm:Deo}
Let 
$J , K_1 , K_2 , \ldots , K_s \subset I$
be such that 
$J = \bigcap_{\nu = 1}^s K_{\nu}$, 
and let 
$x,y \in (W^J)_{\af}$. 
We have 
$x \succeq y$
in 
$(W^J)_{\af}$ 
if and only if 
$\Pi^{K_{\nu}}(x) \succeq \Pi^{K_{\nu}}(y)$
in 
$(W^{K_{\nu}})_{\af}$ 
for all 
$\nu \in [s]$.
\end{thm}

This theorem is an analogue of 
Deodhar's criterion for Bruhat order on Coxeter groups
(\cite[Theorem 2.6.1]{BB}; see also \cite[Lemma 3.6]{D}). 
It is easily seen that 
Theorem \ref{thm:Deo} 
is equivalent to the next proposition.

\begin{prop} \label{prop:Deo}
Let 
$J \subset I$
and 
$x,y \in (W^J)_{\af}$. 
We have 
$x \succeq y$
in 
$(W^J)_{\af}$ 
if and only if 
$\Pi^{I \setminus \{ i \}}(x) \succeq \Pi^{I \setminus \{ i \}}(y)$
in 
$(W^{I \setminus \{ i \}})_{\af}$ 
for all 
$i \in I \setminus J$.
\end{prop}
 
We give a proof of Proposition \ref{prop:Deo} in \S \ref{subsection:pr-Deo}. 
For this purpose, 
we first introduce an isomorphism
$\psi_{\lambda} : 
\mathbb{B}^{\si}(\lambda) 
\rightarrow 
\bigotimes_{i \in J_{\lambda}^c} \mathbb{B}^{\si}(m_i \varpi_i)$
of $\U$-crystals in \S\ref{subsection:SMT-Dem}. 
Next, in Proposition \ref{prop:SMT-Dem}, 
we give a partial characterization of the image of 
$\mathbb{B}^{\si}_{\succeq x}(\lambda)$, 
$x \in (W^{J_{\lambda}})_{\af}$, 
under the map $\psi_{\lambda}$. 
Finally, we show that 
Proposition \ref{prop:SMT-Dem} (2) 
implies Proposition \ref{prop:Deo}, 
by using the fact that semi-infinite Bruhat order on 
$(W^{J_{\lambda}})_{\af}$
corresponds to the containment relation among 
$\mathbb{B}^{\si}_{\succeq x}(\lambda)$, 
$x \in (W^{J_{\lambda}})_{\af}$
(see Lemma \ref{lem:Dem} (3)). 

\subsection{The map $\psi_{\lambda}$ and Demazure crystals} \label{subsection:SMT-Dem}

In this subsection, 
we give a partial characterization of the image of 
$\mathbb{B}^{\si}_{\succeq x}(\lambda)$
under the isomorphism 
$\psi_{\lambda} : 
\mathbb{B}^{\si}(\lambda) 
\rightarrow 
\bigotimes_{i \in J_{\lambda}^c} \mathbb{B}^{\si}(m_i \varpi_i)$
of $\U$-crystals obtained in the next lemma.

\begin{lem}[cf. Theorem \ref{thm:BN-Kashiwara}] \label{lem:Psi_lambda}
Let 
$\lambda = \sum_{i \in I} m_i \varpi_i \in P^+$. 
There exists a unique isomorphism
$\psi_{\lambda} : 
\mathbb{B}^{\si}(\lambda) 
\rightarrow 
\bigotimes_{i \in J_{\lambda}^c} \mathbb{B}^{\si}(m_i \varpi_i)$ 
of $\U$-crystals such that 
\begin{align} \label{eq:rho->rho}
\psi_{\lambda} (\eta_{\bm{\rho}})
=
\bigotimes_{i \in J_{\lambda}^c} 
\eta_{\rho^{(i)}}
\ \ 
\text{for all}\ 
\bm{\rho} = \left( \rho^{(i)} \right) 
\in \mathrm{Par}(\lambda), 
\end{align}
where, 
for each 
$i \in J_{\lambda}^c$, 
$\eta_{\rho^{(i)}} \in \mathbb{B}^{\si}(m_i \varpi_i)$
denotes the element 
\eqref{eq:eta_rho}
associated with 
$\rho^{(i)} \in \mathrm{Par}(m_i \varpi_i)$.
\end{lem}

\begin{proof}
We first claim that there exists a unique strict embedding 
$\Breve{\psi}_{\lambda} 
: \mathbb{B}^{\si}(\lambda) \rightarrow 
\bigotimes_{i \in J_{\lambda}^c} \mathbb{B}^{\si}(\varpi_i)^{\otimes m_i}$
of $\U$-crystals such that 
\begin{align} \label{eq:breve-psi}
\Breve{\psi}_{\lambda} (\eta_{\bm{\rho}}) 
=
\bigotimes_{i \in J_{\lambda}^c} \bigotimes_{\nu = 1}^{m_i} 
\left( T^{I \setminus \{ i \}}_{\rho^{(i)}_{\nu} \alpha_i^{\vee}} ; 0,1 \right)
\end{align}
for all
$\bm{\rho} 
= 
\bigl(
\rho^{(i)}_1 \geq \rho^{(i)}_2 \geq \cdots \geq \rho^{(i)}_{m_i-1} \geq \rho^{(i)}_{m_i} = 0
\bigr)_{i \in I} 
\in \mathrm{Par}(\lambda)$; 
the uniqueness follows from Theorem \ref{thm:SiLS-isom} (1). 
Indeed, 
we see from \cite[Lemma 3.8 (1)]{I20} that the map
\begin{align}
\Breve{\psi}_{\lambda} 
:= 
\left( \bigotimes_{i \in J_{\lambda}^c} \varphi_{\varpi_i}^{\otimes m_i} \right) \circ 
\Phi_{\lambda \mid q=0}^{\mathrm{LT}} \circ 
\varphi_{\lambda}^{-1}
: \mathbb{B}^{\si}(\lambda) \rightarrow 
\bigotimes_{i \in J_{\lambda}^c} \mathbb{B}^{\si}(\varpi_i)^{\otimes m_i}
\end{align}
is a strict embedding satisfying \eqref{eq:breve-psi}. 
We can now construct the map $\psi_{\lambda}$ as follows. 
By \eqref{eq:breve-psi} and Theorem \ref{thm:SiLS-isom} (1), 
the image of 
$\Breve{\psi}_{\lambda}$ 
equals that of 
$\bigotimes_{i \in J_{\lambda}^c} \Breve{\psi}_{m_i \varpi_i}$. 
Hence the map 
\begin{align}
\psi_{\lambda}
:=
\left( \bigotimes_{i \in J_{\lambda}^c} \Breve{\psi}_{m_i \varpi_i} \right)^{-1}
\circ 
\Breve{\psi}_{\lambda} : 
\mathbb{B}^{\si}(\lambda) 
\rightarrow 
\bigotimes_{i \in J_{\lambda}^c} \mathbb{B}^{\si}(m_i \varpi_i)
\end{align}
is well-defined and satisfies \eqref{eq:rho->rho}; 
the uniqueness follows from Theorem \ref{thm:SiLS-isom} (1). 
\end{proof}

For 
$\lambda = \sum_{i \in I} m_i \varpi_i \in P^+$
and 
$x \in (W^{J_{\lambda}})_{\af}$, 
set 
\begin{align} \label{eq:tensor-Dem}
\begin{split}
\bigotimes_{i \in J_{\lambda}^c} 
&\,
\mathbb{B}^{\si}_{\succeq \Pi^{I \setminus \{ i \}}(x)}(m_i\varpi_i) \\
&=
\left\{ \bigotimes_{i \in J_{\lambda}^c} \eta^{(i)} 
\in 
\bigotimes_{i \in J_{\lambda}^c} \mathbb{B}^{\si}(m_i\varpi_i)
\ \ \vline \ \ 
\kappa(\eta^{(i)}) \succeq \Pi^{I \setminus \{ i \}}(x) 
\ \text{for all} \ 
i \in J_{\lambda}^c \right\} .
\end{split}
\end{align}
%

\begin{prop} \label{prop:SMT-Dem}
Let 
$\lambda = \sum_{i \in I} m_i \varpi_i \in P^+$
and 
$x \in (W^{J_{\lambda}})_{\af}$. 
\begin{enumerate}[(1)]
\item
$\psi_{\lambda}\left( \mathbb{B}^{\si}_{\succeq x}(\lambda) \right) 
\subset
\bigotimes_{i \in J_{\lambda}^c} \mathbb{B}^{\si}_{\succeq \Pi^{I \setminus \{ i \}}(x)}(m_i\varpi_i)$.
\item
If 
$\pi \in 
\bigotimes_{i \in J_{\lambda}^c} \mathbb{B}^{\si}_{\succeq \Pi^{I \setminus \{ i \}}(x)}(m_i\varpi_i)$
is of the form 
\begin{align} \label{label:y-rho}
\pi 
= 
\bigotimes_{i \in J_{\lambda}^c} S_y \eta_{\rho^{(i)}}
\ \ 
\text{for some}
\ 
y \in W_{\af}
\ 
\text{and}
\
(\rho^{(i)}) \in \mathrm{Par}(\lambda), 
\end{align}
then 
$\pi \in \psi_{\lambda}\left( \mathbb{B}^{\si}_{\succeq x}(\lambda) \right)$.
\end{enumerate}
\end{prop}

\begin{rem}
It follows from 
\eqref{eq:breve-psi}, 
\cite[Lemma 1.6 (2)]{AK} 
and 
\cite[Proposition 5.4 (i)]{K02}
that an element 
$\pi \in \mathbb{B}^{\si}(\lambda)$
is extremal if and only if 
$\pi$ is of the form 
\eqref{label:y-rho}. 
\end{rem}

\subsection{Proof of Proposition \ref{prop:SMT-Dem}} \label{subsection:pr-SMT-Dem}

This subsection is devoted to the proof of 
Proposition \ref{prop:SMT-Dem}. 
We begin by recalling some fundamental properties of 
$\mathbb{B}^{\si}(\lambda)$.

\begin{lem} \label{lem:kappa}
Let 
$\lambda \in P^+$
and 
$\bm{\rho} \in \mathrm{Par}(\lambda)$.
\begin{enumerate}[(1)]
\item
$\kappa(S_x \eta_{\bm{\rho}}) = \Pi^{J_{\lambda}}(x)$
for
$x \in W_{\af}$.
\item
$(W_{J_{\lambda}})_{\af}
=
\{ x \in W_{\af} \mid S_x \eta_{\bm{\rho}} = \eta_{\bm{\rho}} \}$. 
\end{enumerate}
\end{lem}

\begin{proof}
(1) follows from \eqref{eq:S_x-rho}.

By \eqref{eq:eta_rho}--\eqref{eq:S_x-rho}, 
$S_x \eta_{\bm{\rho}} = \eta_{\bm{\rho}}$
if and only if 
$\Pi^{J_{\lambda}} \left( xT^{J_{\lambda}}_{\zeta_{s_p}} \right)
=
T^{J_{\lambda}}_{\zeta_{s_p}}$
for 
$p \in [k]$.
By Lemma \ref{lem:Pi} (2), 
\begin{align*}
\Pi^{J_{\lambda}} \left( xT^{J_{\lambda}}_{\zeta_{s_p}} \right)
&=
\Pi^{J_{\lambda}}\left( x \Pi^{J_{\lambda}}(t_{\zeta_{s_p}})\right)
=
\Pi^{J_{\lambda}}\left( x t_{\zeta_{s_p}} \right) \\[2mm]
&=
\Pi^{J_{\lambda}}(x) \Pi^{J_{\lambda}}(t_{\zeta_{s_p}})
=
\Pi^{J_{\lambda}}(x) T^{J_{\lambda}}_{\zeta_{s_p}}
\end{align*}
for 
$p \in [k]$.
This implies that 
$S_x \eta_{\bm{\rho}} = \eta_{\bm{\rho}}$
if and only if 
$\Pi^{J_{\lambda}}(x) = e$, 
or equivalently, 
$x \in (W_{J_{\lambda}})_{\af}$.
\end{proof}

\begin{lem}[{\cite[\S 5]{NS16}}] \label{lem:Dem}
Let 
$\lambda \in P^+$
and 
$x,y \in (W^{J_{\lambda}})_{\af}$.
\begin{enumerate}[(1)]
\item
$\mathbb{B}^{\si}_{\succeq x}(\lambda) \cup \{ \bm{0} \}$
is stable under the action of $f_j$ for all $j \in I_{\af}$.
\item
$\mathbb{B}^{\si}_{\succeq x}(\lambda) \cup \{ \bm{0} \}$
is stable under the action of 
$e_j$
for 
$j \in I_{\af}$
such that 
$\langle \alpha_j^{\vee} , x\lambda \rangle \geq 0$.
\item
$x \succeq y$ 
if and only if 
$\mathbb{B}^{\si}_{\succeq x}(\lambda) \subset \mathbb{B}^{\si}_{\succeq y}(\lambda)$.
\item
For every 
$\eta \in \mathbb{B}^{\si}_{\succeq x} (\lambda)$, 
we have 
$f_j^{\max} \eta \in \mathbb{B}^{\si}_{\succeq \Pi^{J_{\lambda}}(r_j x)} (\lambda)$
for all 
$j \in I_{\af}$.
\end{enumerate}
\end{lem}

The next lemma is a slight refinement of \cite[Lemma 5.4.1]{NS16}.

\begin{lem} \label{lem:max}
For 
$\eta \in \mathbb{B}^{\si}(\lambda)$ 
$\left( \text{resp.} \ \eta \in \bigotimes_{i \in J_{\lambda}^c}\mathbb{B}^{\si}(m_i \varpi_i) \right)$, 
$x \in W_{\af}$ and 
$w \in W$, 
there exist
$i_1 , i_2 , \ldots , i_N \in I_{\af}$ 
such that 
\begin{enumerate}[(i)]
\item
$\langle \alpha_{i_n}^{\vee} , r_{i_{n-1}} \cdots r_{i_2} r_{i_1} x \lambda \rangle \geq 0$
for all 
$n \in [N]$, 
and
\item
$f_{i_N}^{\max} \cdots f_{i_2}^{\max} f_{i_1}^{\max} \eta = S_{w t_{\xi}} \eta_{\bm{\rho}}$ 
$\left( \text{resp.} \ 
f_{i_N}^{\max} \cdots f_{i_2}^{\max} f_{i_1}^{\max} \eta = \bigotimes_{i \in J_{\lambda}^c} S_{wt_{\xi}} \eta_{\rho^{(i)}} \right)$ 
for some 
$\xi \in Q^{\vee}$ 
and 
$\bm{\rho} = (\rho^{(i)})_{i \in J_{\lambda}^c} \in \mathrm{Par}(\lambda)$.
\end{enumerate}
\end{lem}

\begin{proof}
By Lemma \ref{lem:Psi_lambda}, 
it suffices to prove the assertion only for 
$\mathbb{B}^{\si}(\lambda)$. 
By \cite[Lemma 5.4.1]{NS16}, 
there exist 
$j_1 , j_2 , \ldots , j_p \in I_{\af}$ 
such that 
\begin{enumerate}[(i)]
\item
$\langle \alpha_{j_m}^{\vee} , r_{j_{m-1}} \cdots r_{j_2} r_{j_1} x \lambda \rangle \geq 0$ 
for all 
$m \in [p]$, 
and 
\item
$f_{j_p}^{\max} \cdots f_{j_2}^{\max} f_{j_1}^{\max} \eta 
= S_{t_{\xi}} \eta_{\bm{\rho}}$ 
for some 
$\xi \in Q^{\vee}$ 
and 
$\bm{\rho} \in \mathrm{Par}(\lambda)$.
\end{enumerate}
We see from (the proof of) 
\cite[Lemma 5.4.1]{NS16} 
that 
\begin{align*}
r_{j_p} \cdots r_{j_2} r_{j_1} x \lambda 
\equiv
\lambda 
\mod \BC\delta .
\end{align*}
For a reduced expression 
$w = r_{k_q} r_{k_{q-1}} \cdots r_{k_1}$, 
$k_1 , k_2 , \ldots , k_q \in I$, 
we have
\begin{align*}
\langle \alpha_{k_m}^{\vee} , r_{k_{m-1}} \cdots r_{k_2} r_{k_1} r_{j_p} \cdots r_{j_2} r_{j_1} x \lambda \rangle
= 
\langle \alpha_{k_m}^{\vee} , r_{k_{m-1}} \cdots r_{k_2} r_{k_1} \lambda \rangle
\geq 
0
\end{align*}
for all 
$1 \leq m \leq q$. 
It follows from
$\mathrm{wt}(S_{t_{\xi}} \eta_{\bm{\rho}}) \equiv \lambda \mod \BC \delta$
and 
Lemma \ref{lem:ext-elm} (1)
that 
\begin{align*}
f_{k_q}^{\max} \cdots f_{k_2}^{\max} f_{k_1}^{\max}
\underbrace{f_{j_p}^{\max} \cdots f_{j_2}^{\max} f_{j_1}^{\max} \eta}_{= S_{t_{\xi}} \eta_{\bm{\rho}}} 
=
S_w S_{t_{\xi}} \eta_{\bm{\rho}}
=
S_{wt_{\xi}} \eta_{\bm{\rho}} .
\end{align*}
This completes the proof.
\end{proof}

We are now in a position to prove Proposition \ref{prop:SMT-Dem}. 
In what follows, we write
\begin{align}
\mathbb{B}_{\succeq x} 
=
\mathbb{B}^{\si}_{\succeq x}(\lambda)
\ \ 
\text{and}
\ \ 
\Breve{\mathbb{B}}_{\succeq x} 
=
\bigotimes_{i \in J_{\lambda}^c} \mathbb{B}^{\si}_{\succeq \Pi^{I \setminus \{ i \}}(x)}(m_i\varpi_i)
\ \ 
\text{for} \ 
x \in (W^{J_{\lambda}})_{\af} .
\end{align}
%
\begin{proof}[Proof of Proposition \ref{prop:SMT-Dem} (1)] 
To see 
$\psi_{\lambda}(\mathbb{B}_{\succeq x}) \subset \Breve{\mathbb{B}}_{\succeq x}$, 
let 
$\eta \in \mathbb{B}_{\succeq x}$, 
and show that 
$\psi_{\lambda}(\eta) \in \Breve{\mathbb{B}}_{\succeq x}$.
By Lemma \ref{lem:max}, there exist 
$i_1 , i_2 , \ldots , i_N \in I_{\af}$ 
such that 
\begin{enumerate}[(i)]
\item
$\langle \alpha_{i_n}^{\vee} , r_{i_{n-1}} r_{i_{n-2}} \cdots r_{i_1} x \lambda \rangle \geq 0$ 
for all 
$n \in [N]$, 
and
\item
$f_{i_N}^{\max} \cdots f_{i_2}^{\max} f_{i_1}^{\max} \eta 
= S_{t_{\xi}} \eta_{\bm{\rho}}$ 
for some 
$\xi \in Q^{\vee}$ 
and 
$\bm{\rho} \in \mathrm{Par}(\lambda)$.
\end{enumerate}
The proof is by induction on $N$. 

If $N = 0$, 
then 
$\eta = S_{t_{\xi}} \eta_{\bm{\rho}}$
and 
\begin{align*}
\psi_{\lambda}(\eta) 
=
\psi_{\lambda}(S_{t_{\xi}} \eta_{\bm{\rho}}) 
=
S_{t_{\xi}} \psi_{\lambda}(\eta_{\bm{\rho}}) 
=
S_{t_{\xi}} 
\bigotimes_{i \in J_{\lambda}^c}
\eta_{\rho^{(i)}}
= 
\bigotimes_{i \in J_{\lambda}^c}
S_{t_{\xi}} \eta_{\rho^{(i)}}
\end{align*}
by Lemmas \ref{lem:ext-elm} (2) and \ref{lem:Psi_lambda}.
By Lemma \ref{lem:kappa} (1), 
$\kappa (S_{t_{\xi}} \eta_{\rho^{(i)}}) 
= \Pi^{I \setminus \{ i \}} (t_{\xi})$
for all 
$i \in J_{\lambda}^c$. 
Since 
$\Pi^{J_{\lambda}}(t_{\xi}) = \kappa(\eta) \succeq x$, 
it follows from 
Lemma \ref{lem:order} (1) 
that
\begin{align}
\Pi^{I \setminus \{ i \}} (t_{\xi}) 
= \Pi^{I \setminus \{ i \}} \left( \Pi^{J_{\lambda}}(t_{\xi}) \right) 
\succeq \Pi^{I \setminus \{ i \}} (x)
\ \ 
\text{for all}
\ 
i \in J_{\lambda}^c .
\end{align}
Hence 
$\kappa (S_{t_{\xi}} \eta_{\rho^{(i)}}) 
\succeq \Pi^{I \setminus \{ i \}} (x)$
for all 
$i \in J_{\lambda}^c$, 
which implies
$\psi_{\lambda}(\eta) \in \Breve{\mathbb{B}}_{\succeq x}$.

If $N > 0$, then 
$f_{i_1}^{\max} \eta \in \mathbb{B}_{\succeq \Pi^{J_{\lambda}}(r_{i_1}x)}$ 
by Lemma \ref{lem:Dem} (4).
By induction hypothesis, 
$\psi_{\lambda}(f_{i_1}^{\max} \eta) 
\in 
\Breve{\mathbb{B}}_{\succeq \Pi^{J_{\lambda}}(r_{i_1}x)}$. 
We claim that 
$\psi_{\lambda}(f_{i_1}^{\max} \eta) 
\in 
\Breve{\mathbb{B}}_{\succeq x}$.
Indeed, 
since 
$\langle \alpha_{i_1}^{\vee} , x\lambda \rangle \geq 0$ 
(see (i) above), 
Lemma \ref{lem:r_ix} (2)--(3) gives 
$\Pi^{J_{\lambda}}(r_{i_1}x) \succeq x$. 
By Lemma \ref{lem:order} (1), 
\begin{align}
\Pi^{I \setminus \{ i \}}(r_{i_1}x) 
=
\Pi^{I \setminus \{ i \}}(\Pi^{J_{\lambda}}(r_{i_1}x)) 
\succeq 
\Pi^{I \setminus \{ i \}}(x)
\ 
\text{for all}
\ 
i \in J_{\lambda}^c.
\end{align}
Lemma \ref{lem:Dem} (3) now shows that 
$\Breve{\mathbb{B}}_{\succeq \Pi^{J_{\lambda}}(r_{i_1}x)}
\subset 
\Breve{\mathbb{B}}_{\succeq x}$, 
and hence 
$\psi_{\lambda}(f_{i_1}^{\max} \eta) 
\in 
\Breve{\mathbb{B}}_{\succeq x}$. 
Since 
$\psi_{\lambda}$ 
is an isomorphism of $\U$-crystals, 
there exists 
$k \in \mathbb{Z}_{\geq 0}$ 
such that 
$\psi_{\lambda}(\eta) = e_{i_1}^k \psi_{\lambda}(f_{i_1}^{\max} \eta)$. 
We see from 
$\Pi^{I \setminus \{ i \}}(r_{i_1}x) 
\succeq 
\Pi^{I \setminus \{ i \}}(x)$
and Lemma \ref{lem:r_ix} (2)--(3) 
that 
\begin{align}
\langle \alpha_{i_1}^{\vee} , \Pi^{I \setminus \{ i \}}(x) m_i\varpi_i \rangle
=
m_i \langle \alpha_{i_1}^{\vee} , \Pi^{I \setminus \{ i \}}(x) \varpi_i \rangle
\geq 0
\ 
\text{for all}
\  
i \in J_{\lambda}^c .
\end{align}
By Lemma \ref{lem:Dem} (2)
and tensor product rule, 
we conclude that 
$\psi_{\lambda}(\eta) 
=
e_{i_1}^k \psi_{\lambda}(f_{i_1}^{\max} \eta)
\in \Breve{\mathbb{B}}_{\succeq x}$.
\end{proof}

\begin{proof}[Proof of Proposition \ref{prop:SMT-Dem} (2)]
Assume that 
$\pi = 
\bigotimes_{i \in J_{\lambda}^c} S_y \eta_{\rho^{(i)}} 
\in \Breve{\mathbb{B}}_{\succeq x}$, 
$y \in W_{\af}$
and 
$\bm{\rho} = (\rho^{(i)}) \in \mathrm{Par}(\lambda)$.  
By Lemmas \ref{lem:ext-elm} (2) and \ref{lem:Psi_lambda}, 
\begin{align} \label{eq:pi=y-rho}
\pi 
= 
\bigotimes_{i \in J_{\lambda}^c} S_y \eta_{\rho^{(i)}}
=
S_y \bigotimes_{i \in J_{\lambda}^c} \eta_{\rho^{(i)}}
=
S_y \psi_{\lambda} (\eta_{\bm{\rho}})
=
\psi_{\lambda} (S_y \eta_{\bm{\rho}}) .
\end{align}
Let 
$w_0 \in W$ 
be the longest element.
By Lemma \ref{lem:max}, 
there exist
$i_1 , i_2 , \ldots , i_N \in I_{\af}$ 
such that 
\begin{enumerate}[(i)]
\item
$\langle \alpha_{i_n}^{\vee} , r_{i_{n-1}} r_{i_{n-2}} \cdots r_{i_1} x \lambda \rangle \geq 0$ 
for all 
$n \in [N]$, 
and
\item
$f_{i_N}^{\max} \cdots f_{i_2}^{\max} f_{i_1}^{\max} \pi 
= \bigotimes_{i \in J_{\lambda}^c} S_{w_0 t_{\xi}} \eta_{\rho^{(i)}}$ 
for some 
$\xi \in Q^{\vee}$.
\end{enumerate}
We show by induction on $N$ that 
$\pi \in \psi_{\lambda}(\mathbb{B}_{\succeq x})$, 
or equivalently, 
$\kappa(S_y \eta_{\bm{\rho}}) \succeq x$; 
note that 
$\kappa(S_y \eta_{\bm{\rho}}) = \Pi^{J_{\lambda}}(y)$
by Lemma \ref{lem:kappa} (1).

We first assume that 
$N = 0$. 
In the same manner as in \eqref{eq:pi=y-rho} we see that  
$\pi = \psi_{\lambda}(S_{w_0 t_{\xi}} \eta_{\bm{\rho}})$. 
Since 
$\psi_{\lambda}$ 
is an isomorphism of $\U$-crystals, 
we have 
$S_y \eta_{\bm{\rho}} = S_{w_0 t_{\xi}} \eta_{\bm{\rho}}$, 
and hence 
$\Pi^{J_{\lambda}}(y) = \Pi^{J_{\lambda}}(w_0 t_{\xi})$
by Lemma \ref{lem:kappa} (2). 
Write 
$x = wt_{\zeta}$, 
$w \in W$, $\zeta \in Q^{\vee}$. 
We proceed by induction on 
$\ell(w_0) - \ell(w)$.

If $\ell(w_0) - \ell(w) = 0$, 
then 
$w = w_0$. 
By Lemma \ref{lem:kappa} (1), 
\begin{align}
\Pi^{I \setminus \{ i \}}(w_0 t_{\xi}) 
=
\Pi^{I \setminus \{ i \}}(y) 
=
\kappa(S_y \eta_{\rho^{(i)}})
\ 
\text{for all}
\ 
i \in J_{\lambda}^c . 
\end{align}
Since 
$\pi \in \Breve{\mathbb{B}}_{\succeq x}$, 
we have 
\begin{align}
\kappa(S_y \eta_{\rho^{(i)}})
\succeq 
\Pi^{I \setminus \{ i \}}(x)
=
\Pi^{I \setminus \{ i \}}(w t_{\zeta})
=
\Pi^{I \setminus \{ i \}}(w_0 t_{\zeta})
\ 
\text{for all}
\ 
i \in J_{\lambda}^c.
\end{align}
Hence 
$\Pi^{I \setminus \{ i \}}(w_0 t_{\xi}) 
\succeq 
\Pi^{I \setminus \{ i \}}(w_0 t_{\zeta})$
for all 
$i \in J_{\lambda}^c$. 
By Lemma \ref{lem:wt-wt}, 
we have 
$\Pi^{I \setminus \{ i \}}(t_{\xi}) \succeq \Pi^{I \setminus \{ i \}}(t_{\zeta})$ 
for all 
$i \in J_{\lambda}^c$, 
which implies 
$\Pi^{J_{\lambda}}(t_{\xi}) \succeq \Pi^{J_{\lambda}}(t_{\zeta})$, 
by Lemma \ref{lem:order} (2). 
Again, 
by Lemma \ref{lem:wt-wt}, 
we have 
$\Pi^{J_{\lambda}}(w_0 t_{\xi}) 
\succeq 
\Pi^{J_{\lambda}}(w_0 t_{\zeta})$. 
Since 
$\kappa (S_y \eta_{\bm{\rho}}) 
=
\Pi^{J_{\lambda}}(y) 
=
\Pi^{J_{\lambda}}(w_0 t_{\xi})$
and 
$x = \Pi^{J_{\lambda}}(w_0 t_{\zeta})$, 
we conclude that 
$\kappa (S_y \eta_{\bm{\rho}}) 
\succeq 
x$.

If 
$\ell(w_0) - \ell(w) > 0$, 
then 
$w \neq w_0$. 
There exists 
$j \in I$ 
such that 
$r_j x \succeq x$; 
note that 
$x = r_jw t_{\zeta}$, 
$r_jw \in W$, 
and 
$\ell(w_0) - \ell(r_j w) < \ell(w_0) - \ell(w)$.
Since 
$\langle \alpha_j^{\vee} , y\lambda \rangle 
=
\langle \alpha_j^{\vee} , w_0 t_{\xi}\lambda \rangle
\leq 0$
by Lemma \ref{lem:act-wt}, 
we have 
$y \succeq r_j y$
by Lemma \ref{lem:r_ix} (3).
By Lemma \ref{lem:order} (1), 
\begin{align}
\Pi^{I \setminus \{ i \}}(r_j x) \succeq \Pi^{I \setminus \{ i \}}(x)
\ 
\text{and}
\ 
\Pi^{I \setminus \{ i \}}(y) \succeq \Pi^{I \setminus \{ i \}}(r_j y)
\ 
\text{for all}
\ 
i \in J_{\lambda}^c .
\end{align}
Since 
$\pi \in \Breve{\mathbb{B}}_{\succeq x}$, 
we have 
$\Pi^{I \setminus \{ i \}}(y)
=
\kappa (S_y \eta_{\rho^{(i)}}) 
\succeq 
\Pi^{I \setminus \{ i \}}(x)$
for all 
$i \in J_{\lambda}^c$. 
It follows from Lemma \ref{lem:diamond} (1) that 
$\Pi^{I \setminus \{ i \}}(y) \succeq \Pi^{I \setminus \{ i \}}(r_j x)$ 
for all 
$i \in J_{\lambda}^c$. 
By induction hypothesis, we have 
$\Pi^{J_{\lambda}}(y) \succeq \Pi^{J_{\lambda}}(r_j x)$. 
We see from 
$\Pi^{I \setminus \{ i \}}(r_j x) \succeq \Pi^{I \setminus \{ i \}}(x)$, 
$i \in J_{\lambda}^c$, 
Lemmas \ref{lem:act-wt} and \ref{lem:r_ix} (2)--(3) 
that 
$\langle \alpha_j^{\vee} , x\varpi_i \rangle
=
\langle \alpha_j^{\vee} , \Pi^{I \setminus \{ i \}}(x) \varpi_i \rangle \geq 0$
for all 
$i \in J_{\lambda}^c$, 
which gives 
$\langle \alpha_j^{\vee} , x\lambda \rangle
= 
\sum_{i \in J_{\lambda}^c} 
m_i \langle \alpha_j^{\vee} , x\varpi_i \rangle
\geq 0$.
Again, 
by Lemma \ref{lem:r_ix} (2)--(3), 
we have 
$\Pi^{J_{\lambda}}(r_j x) \succeq x$. 
Thus 
$\kappa(S_y \eta_{\bm{\rho}}) = \Pi^{J_{\lambda}}(y) \succeq x$.

We next assume that 
$N > 0$. 
It follows from Lemmas \ref{lem:ext-elm} and \ref{lem:Dem} (4) that 
\begin{align}
f_{i_1}^{\max} \pi 
= 
f_{i_1}^{\max} \bigotimes_{i \in J_{\lambda}^c} S_{y} \eta_{\rho^{(i)}} 
= 
\bigotimes_{i \in J_{\lambda}^c} 
\underbrace{f_{i_1}^{\max} S_{y} \eta_{\rho^{(i)}}}_{= S_{r_{i_1}y} \eta_{\rho^{(i)}}}
\in 
\Breve{\mathbb{B}}_{\succeq \Pi^{J_{\lambda}}(r_{i_1}x)}
\end{align}
is of the form 
\eqref{label:y-rho}.
By induction hypothesis, 
there exists 
$\eta \in \mathbb{B}_{\succeq \Pi^{J_{\lambda}}(r_{i_1}x)}$
such that 
$f_{i_1}^{\max} \pi = \psi_{\lambda}(\eta)$. 
Since 
$\psi_{\lambda}$ 
is an isomorphism of $\U$-crystals, 
there exists 
$k \in \mathbb{Z}_{\geq 0}$ 
such that 
$\pi = \psi_{\lambda}(e_{i_1}^k \eta)$.  
We see from 
$\langle \alpha_{i_1}^{\vee} , x\lambda \rangle \geq 0$ 
(see (i) above)
and 
Lemma \ref{lem:r_ix} (2)--(3)
that 
$\Pi^{J_{\lambda}}(r_{i_1}x) \succeq x$, 
which gives
$\mathbb{B}_{\succeq \Pi^{J_{\lambda}}(r_{i_1}x)}
\subset 
\mathbb{B}_{\succeq x}$, 
by Lemma \ref{lem:Dem} (3). 
It follows from 
$\langle \alpha_{i_1}^{\vee} , x\lambda \rangle \geq 0$
(see (i) above)
and 
Lemma \ref{lem:Dem} (2) that
\begin{align}
\begin{split}
e_{i_1}^k \eta 
\ 
&\in
\ 
\left\{ e_{i_1}^k \eta' 
\ \vline \ 
\eta' \in \mathbb{B}_{\succeq \Pi^{J_{\lambda}}(r_{i_1}x)} \right\} 
\setminus \{ \bm{0} \} \\[2mm]
&\subset \ 
\left\{ e_{i_1}^k \eta' 
\ \vline \ 
\eta' \in \mathbb{B}_{\succeq x} \right\} 
\setminus \{ \bm{0} \} 
\ \subset \ 
\mathbb{B}_{\succeq x} .
\end{split}
\end{align}
Thus 
$\pi = \psi_{\lambda}(e_{i_1}^k \eta) \in \psi_{\lambda}(\mathbb{B}_{\succeq x})$. 
\end{proof}

\subsection{Proof of Proposition \ref{prop:Deo}} \label{subsection:pr-Deo}

\begin{proof}[Proof of Proposition \ref{prop:Deo}]
The ``only if" part follows immediately from Lemma \ref{lem:order} (1).
Let us show the ``if" part.
Let  
$\lambda = \sum_{i \in I} m_i \varpi_i \in P^+$ 
be such that 
$J_{\lambda} = J$. 
Let
$\bm{\rho} = (\rho^{(i)}) \in \mathrm{Par}(\lambda)$. 
Since 
$\kappa(S_x \eta_{\rho^{(i)}}) 
= 
\Pi^{I \setminus \{ i \}}(x)$, 
by Lemma \ref{lem:kappa} (1), 
and 
$\Pi^{I \setminus \{ i \}}(x)
\succeq 
\Pi^{I \setminus \{ i \}}(y)$
for all 
$i \in I \setminus J = J_{\lambda}^c$, 
it follows that 
$\bigotimes_{i \in J_{\lambda}^c} S_x \eta_{\rho^{(i)}} \in 
\bigotimes_{i \in J_{\lambda}^c} \mathbb{B}^{\si}_{\succeq \Pi^{I \setminus \{ i \}}(y)}(m_i\varpi_i)$. 
Similarly to 
\eqref{eq:pi=y-rho}, 
we have 
$\bigotimes_{i \in J_{\lambda}^c} S_x \eta_{\rho^{(i)}} = \psi_{\lambda}(S_x \eta_{\bm{\rho}})$. 
By Proposition \ref{prop:SMT-Dem} (2), 
$S_x \eta_{\bm{\rho}} \in \mathbb{B}^{\si}_{\succeq y}(\lambda)$, 
which implies 
$\kappa(S_x \eta_{\bm{\rho}}) \succeq y$. 
Since 
$x = \kappa(S_x \eta_{\bm{\rho}})$, 
by Lemma \ref{lem:kappa} (1), 
we conclude that 
$x \succeq y$.
\end{proof}

\section{Tableau criterion for semi-infinite Bruhat order} \label{Section:Tab-cri}

Proposition \ref{prop:Deo} shows that 
the study of semi-infinite Bruhat order on $W_{\af}$ 
is reduced to those on the sets 
$(W^{I \setminus \{ i \}})_{\af}$, 
$i \in I$, 
of 
``semi-infinite Grassmannian elements." 
In this section, we proceed with the study of 
semi-infinite Bruhat order on 
$(W^{I \setminus \{ i \}})_{\af}$, 
$i \in I$, 
for 
$W_{\af}$ 
of type
$A_{n-1}^{(1)}$, 
$B_n^{(1)}$, 
$C_n^{(1)}$, 
and
$D_n^{(1)}$.
The main results of this section are 
Theorems 
\ref{thm:tab-cri-A}, 
\ref{thm:tab-cri-C}, 
\ref{thm:tab-cri-B}, 
and 
\ref{thm:tab-cri-D}
(see also 
Definitions 
\ref{def:SiB-A},
\ref{def:SiB-C},
\ref{def:SiB-B},
and 
\ref{def:SiB-D}), 
which give combinatorial criteria 
for semi-infinite Bruhat order 
in terms of tableaux. 
In order to get these results, 
we give a complete classification of the edges in 
the quantum Bruhat graph
$\mathrm{QB}^{I \setminus \{ i \}}$, 
$i \in I$
(see 
\S \ref{Subsection:QBG}
and 
Propositions 
\ref{prop:Q=C}, 
\ref{prop:Q=B}, 
and 
\ref{prop:Q=D}). 
For combinatorial descriptions of Bruhat order 
on finite Weyl groups of classical type, 
we refer the reader to 
\cite[\S 8]{BB}.

\subsection{Explicit description of $(W^J)_{\af}$}

In this subsection, 
following \cite[\S 3]{LNSSS15}, 
we give an explicit description of 
$(W^J)_{\af}$ 
for later use. 

We take and fix $J = \bigsqcup_{m=1}^k I_m \subset I$, 
where $I_1 , I_2 , \ldots , I_k$ are the sets of vertices  
of the connected components of the Dynkin diagram of $\Delta_J$; 
note that 
$\Delta_J = \bigsqcup_{m=1}^k \Delta_{I_m}$. 
Set 
$(I_m)_{\af} = \{ 0 \} \sqcup I_m \subset I_{\af}$, 
$m \in [k]$.
Set 
\begin{align}
Q^{\vee,J} 
= 
\left\{ 
\xi \in Q^{\vee} 
\ \vline \ 
\langle \xi , \alpha \rangle \in \{ -1,0\} \ \text{for all} \ \alpha \in \Delta_J^+ 
\right\} .
\end{align}
%

\begin{lem}[{\cite[Equation (3.6)]{LNSSS15}}] \label{lem:J-ad}
For each $\xi \in Q^{\vee}$ there exist a unique
$\phi_J(\xi) \in Q_J^{\vee}$ 
and a unique 
$(j_1, j_2 , \ldots , j_k) \in \prod_{m=1}^k (I_m)_{\af}$ 
such that 
\begin{align} \label{eq:J-ad}
\xi + \phi_J(\xi) + \sum_{m=1}^k \varpi_{j_m}^{\vee} \in \bigoplus_{i \in I \setminus J} \BZ \varpi_i^{\vee} .
\end{align}
In particular, $\xi + \phi_J(\xi) \in Q^{\vee,J}$ for any $\xi \in Q^{\vee}$, 
and hence $Q^{\vee,J}$ is a complete system of coset representatives for $Q^{\vee}/Q^{\vee}_J$. 
\end{lem}

For a subset 
$K \subset I$, 
let 
$w_0^K$ 
be the longest element of 
$W_K$.
For $j_m \in (I_m)_{\af}$, set 
\begin{align}
v_{j_m}^{I_m} = w_0^{I_m} w_0^{I_m \setminus \{ j_m \}} \in W_{I_m} \subset W_J; 
\end{align}
note that $v^{I_m}_0 = e$.
For $\xi \in Q^{\vee}$, define
\begin{align} \label{eq:z}
z_{\xi} = z_{\xi}^J = v_{j_1}^{I_1} v_{j_2}^{I_2} \cdots v_{j_k}^{I_k} \in W_J,
\end{align}
where 
$(j_1, j_2 , \ldots , j_k) \in \prod_{m=1}^k (I_m)_{\af}$, 
satisfying \eqref{eq:J-ad} for $\xi$, 
is determined uniquely by Lemma \ref{lem:J-ad}; 
note that 
$z_{\xi} = z_{\zeta}$ 
if 
$\xi \equiv \zeta \mod Q^{\vee}_J$.

\begin{lem}[{\cite[Lemma 3.7]{LNSSS15}}] \label{lem:(W^J)_af}
We have 
$T_{\xi} = \Pi^J(t_{\xi}) = z_{\xi} t_{\xi + \phi_J(\xi)}$
for every $\xi \in Q^{\vee}$. 
Therefore, by Lemma \ref{lem:Pi}, 
$\Pi^J(wt_{\xi}) = \lfloor w \rfloor z_{\xi} t_{\xi + \phi_J(\xi)}$
for every $w \in W$ and $\xi \in Q$, and we have a bijection
$W^J \times Q^{\vee,J} \rightarrow (W^J)_{\af}$, 
$(w,\xi) \mapsto wT_{\xi}$. 
In particular, 
\begin{align}
(W^J)_{\af} 
= 
\left\{ 
wT_{\xi} = w z_{\xi}t_{\xi} 
\ \vline \ 
w \in W^J , \ \xi \in Q^{\vee,J} 
\right\} .
\end{align}
\end{lem}

\subsection{Quantum Bruhat graphs}
\label{Subsection:QBG}

Following
\cite[\S 4]{LNSSS15}
(see also 
\cite{BFP}), 
define the (parabolic) quantum Bruhat graph 
$\mathrm{QB}^J$
to be the $(\Delta^+ \setminus \Delta_J^+)$-colored directed graph 
with vertex set 
$W^J$ 
and edges of the form 
$w \xrightarrow{\ \gamma \ } \lfloor wr_{\gamma} \rfloor$
for 
$w \in W^J$
and 
$\gamma \in \Delta^+ \setminus \Delta_J^+$, 
where 
$\ell(\lfloor wr_{\gamma} \rfloor) - \ell(w) = 1 -2\chi \langle \gamma^{\vee} , \rho - \rho_J \rangle$
and
$\chi \in \{ 0,1 \}$. 
We say that an edge
$w \xrightarrow{\ \gamma \ } \lfloor wr_{\gamma} \rfloor$
in 
$\mathrm{QB}^J$ 
is Bruhat (resp. quantum) if 
$\ell(\lfloor wr_{\gamma} \rfloor) - \ell(w) = 1$
(resp. 
$\ell(\lfloor wr_{\gamma} \rfloor) - \ell(w) = 1 -2\langle \gamma^{\vee} , \rho - \rho_J \rangle$). 
We see that if there exists a Bruhat edge 
$w \xrightarrow{\ \gamma \ } \lfloor wr_{\gamma} \rfloor$
in 
$\mathrm{QB}^J$, 
then 
$wr_{\gamma} = \lfloor wr_{\gamma} \rfloor \in W^J$. 
Note that $\mathrm{QB}^J$ does not define a partial order on $W^J$. 

\begin{lem}[{\cite[Proposition A.1.2]{INS16}}] \label{lem:Q=SiB}
Let $w \in W^J$, $\xi \in Q^{\vee}$ and $\beta \in \Delta_{\af}^+$. 
Write 
$\beta = w\gamma + \chi\delta$ 
with 
$\gamma \in \Delta$ 
and 
$\chi \in \BZ_{\geq 0}$.
Then 
$r_{\beta}wT_{\xi} \in (W^J)_{\af}$
and there exists an edge 
$wT_{\xi} \xrightarrow{\ \beta \ } r_{\beta}wT_{\xi}$ in $\mathrm{SiB}^J$ 
if and only if $\gamma \in \Delta^+ \setminus \Delta_J^+$ 
and one of the following conditions holds:
\begin{enumerate}[(1)]
\item
$\chi = 0$
and 
$w \xrightarrow{\ \gamma \ } wr_{\gamma}$
is a Bruhat edge in 
$\mathrm{QB}^J$,
\item
$\chi = 1$
and 
$w \xrightarrow{\ \gamma \ } \lfloor wr_{\gamma} \rfloor$
is a quantum edge in 
$\mathrm{QB}^J$; 
\end{enumerate}
in these cases, we have 
$r_{\beta}wT_{\xi} 
= 
\lfloor w r_{\gamma} \rfloor T_{\xi + \chi\gamma^{\vee}}
=
\lfloor w r_{\gamma} \rfloor z^J_{\xi + \chi\gamma^{\vee}} t_{\xi + \chi\gamma^{\vee} + \phi_J(\xi + \chi\gamma^{\vee})}$
and 
$c_i(\xi + \chi\gamma^{\vee} + \phi_J(\xi + \chi\gamma^{\vee}))
=
c_i(\xi) + \chi c_i(\gamma^{\vee})$
for all 
$i \in I \setminus J$. 
\end{lem}

\begin{lem}[{\cite[Proof of Theorem 10.16]{LS}}] \label{lem:LS}
Let 
$w \in W^J$
and 
$\gamma \in \Delta^+ \setminus \Delta_J^+$. 
There exists a quantum edge
$w \xrightarrow{\ \gamma \ } \lfloor wr_{\gamma} \rfloor$
in 
$\mathrm{QB}^J$
if and only if 
$\ell(wr_{\gamma}) - \ell(w) = 1 - 2\langle \gamma^{\vee} , \rho \rangle$
and 
$wr_{\gamma}t_{\gamma^{\vee}} \in (W^J)_{\af}$; 
note that 
$wr_{\gamma}t_{\gamma^{\vee}} \in (W^J)_{\af}$
and 
Lemma \ref{lem:(W^J)_af}
imply
$\gamma^{\vee} \in Q^{\vee,J} \cap (\Delta^{\vee,+} \setminus \Delta_J^{\vee,+})$
and 
$\lfloor wr_{\gamma} \rfloor = wr_{\gamma} z_{\gamma^{\vee}}^{-1}$.
\end{lem}

\begin{lem} \label{lem:<g,r>}
Let 
$i \in I$
and 
$\gamma^{\vee} \in Q^{\vee,I\setminus\{i\}} \cap (\Delta^{\vee,+} \setminus \Delta_{I\setminus\{i\}}^{\vee,+})$. 
We have 
$\langle \gamma^{\vee} , \rho - \rho_{I \setminus \{ i \}} \rangle
=
c_i(\gamma^{\vee}) \langle \alpha_i^{\vee} , \rho - \rho_{I \setminus \{ i \}} \rangle
=
c_i(\gamma^{\vee})(1-\langle \alpha_i^{\vee}, \rho_{I \setminus \{ i \}} \rangle)$.
\end{lem}

\begin{proof}
The assertion follows from 
$\langle \xi , \rho - \rho_{I \setminus \{ i \}} \rangle = 0$
for 
$\xi \in Q_{I \setminus \{ i \}}^{\vee}$, 
$\gamma^{\vee} \equiv c_i(\gamma^{\vee})\alpha_i^{\vee} \mod Q_{I \setminus \{ i \}}^{\vee}$, 
and 
$\langle \alpha_i^{\vee} , \rho \rangle = 1$.
\end{proof}

\subsection{Type $A_{n-1}^{(1)}$} \label{subsection:tab-cri-A}

Fix an integer $n \geq 2$. 
Set 
$I = [n-1]$. 
We assume that the labeling of 
the vertices of the Dynkin diagram of type
$A_{n-1}$
is as follows. 
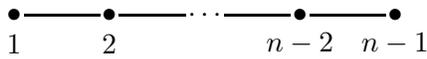
\begin{center}
\ 
\xygraph{
\bullet ([]!{+(0,-.3)} {1}) - [r]
\bullet ([]!{+(0,-.3)} {2}) - [r] \cdots - [r]
\bullet ([]!{+(0,-.3)} {n - 2}) - [r]
\bullet ([]!{+(0,-.3)} {n - 1})}
\end{center}
Let 
$\varepsilon_1,\varepsilon_2,\ldots ,\varepsilon_n$
be an orthonormal basis of an 
$n$-dimensional Euclidean space $\BR^n$. 
Let 
$\Delta = \{ \pm(\varepsilon_s - \varepsilon_t) \mid s,t \in [n], \ s < t \}$
be a root system of type $A_{n-1}$, 
and let
$\Pi 
=
\{ \alpha_s = \varepsilon_s - \varepsilon_{s+1} \mid s \in I \}$
be a simple root system of $\Delta$. 

Let 
$W$
be the Weyl group of 
$\Delta$; 
note that 
$W$
can be described by 
$W = \mathfrak{S}([n])$
as the permutation group of 
$\{ \varepsilon_s \mid s \in [n] \} \subset \BR^n$. 
The longest element of $W$ is given by 
$u \mapsto n-u+1$, 
$u \in [n]$. 
Let 
$\mathrm{CST}_{A_{n-1}}(\varpi_i)$
be the family of $i$-element subsets of $[n]$. 
We identify 
$\mathsf{T} = \{ \mathsf{T}(1) < \mathsf{T}(2) < \cdots < \mathsf{T}(i) \} \in \mathrm{CST}_{A_{n-1}}(\varpi_i)$
with the column-strict tableau
\ytableausetup{mathmode,boxsize=8mm}
\begin{align} \label{eq:column-A}
\begin{ytableau}
\mathsf{T}(1) \\ \mathsf{T}(2) \\ \vdots \\ \mathsf{T}(i)
\end{ytableau}\ .
\end{align}
For 
$w \in W$, 
let 
$\mathsf{T}_w^{(i)} 
\in 
\mathrm{CST}_{A_{n-1}}(\varpi_i)$
be such that 
\begin{align}
\mathsf{T}_w^{(i)} 
=
\left\{ 
\mathsf{T}_w^{(i)}(1) < \mathsf{T}_w^{(i)}(2) < \cdots < \mathsf{T}_w^{(i)}(i) 
\right\} 
=
\{ w(1),w(2),\ldots ,w(i) \}.
\end{align}
The proof of the next lemma is standard (cf. \cite[\S 2.4]{BB}).

\begin{lem} \label{lem:Gr-A}
Let 
$i \in I$. 
We have 
\begin{align*}
W^{I \setminus \{ i \}}
=
\{ w \in W \mid \ 
&w(1) < w(2) < \cdots < w(i), \ \text{and} \\
&w(i+1) < w(i+2) < \cdots < w(n) \}.
\end{align*}
The map 
$W^{I \setminus \{ i \}} \to \mathrm{CST}_{A_{n-1}}(\varpi_i)$, 
$w \mapsto \mathsf{T}_w^{(i)}$, 
is bijective. 
\end{lem}

We see from Lemmas \ref{lem:J-ad}--\ref{lem:(W^J)_af} and \ref{lem:Gr-A}
that the map
\begin{align} \label{eq:map-Y-A}
\mathcal{Y}_i^{A_{n-1}} : 
W_{\af} 
\to 
\mathrm{CST}_{A_{n-1}}(\varpi_i) \times \BZ, \ 
wt_{\xi} \mapsto \left( \mathsf{T}_w^{(i)},c_i(\xi) \right),
\end{align} 
induces a bijection from the subset 
$(W^{I \setminus \{ i \}})_{\af} \subset W_{\af}$
to 
$\mathrm{CST}_{A_{n-1}}(\varpi_i) \times \BZ$.

%
%
%

\begin{define}[{\cite[Definition 4.2 (1)]{I20}}]
\label{def:SiB-A}
Define a partial order 
$\preceq$
on 
$\mathrm{CST}_{A_{n-1}}(\varpi_i) \times \BZ$
as follows: 
for 
$(\mathsf{T},c), (\mathsf{T}',c') \in \mathrm{CST}_{A_{n-1}}(\varpi_i) \times \BZ$, 
set 
$(\mathsf{T},c) \preceq (\mathsf{T}',c')$
if 
\begin{align} \label{eq:tab-A}
\left( c \leq c' \right)
\ \text{and} \ 
\left( 
\mathsf{T}(u) \leq \mathsf{T}'(u+c'-c)
\ \text{for} \ 
u \in [i-c'+c]
\right) .
\end{align}
\end{define}

\begin{prop} \label{prop:tab-cri-A}
Let 
$i \in I$. 
\begin{enumerate}[(1)]
\item
$\mathcal{Y}_i^{A_{n-1}} \circ \Pi^{I \setminus \{ i \}} = \mathcal{Y}_i^{A_{n-1}}$. 
\item 
For 
$x,y \in W_{\af}$, 
we have 
$\Pi^{I \setminus \{ i \}}(x) \preceq \Pi^{I \setminus \{ i \}}(y)$ 
in 
$(W^{I \setminus \{ i \}})_{\af}$
if and only if 
$\mathcal{Y}_i^{A_{n-1}}(x) \preceq \mathcal{Y}_i^{A_{n-1}}(y)$
in 
$\mathrm{CST}_{A_{n-1}}(\varpi_i) \times \BZ$.
\item
Let 
$(\mathsf{T},c), (\mathsf{T}',c') \in \mathrm{CST}_{A_{n-1}}(\varpi_i) \times \BZ$. 
If 
$c' - c \geq \min \{ i, n-i \}$, 
then 
$(\mathsf{T},c) \preceq (\mathsf{T}',c')$. 
\end{enumerate}
\end{prop}

\begin{proof}
(1): 
Let 
$x = wt_{\xi} \in W_{\af}$
and 
$y = \Pi^{I \setminus \{ i \}}(x) = vt_{\zeta}$, 
where 
$w,v \in W$
and 
$\xi , \zeta \in Q^{\vee}$.
We see from Lemma \ref{lem:Pi} that 
$\lfloor w \rfloor^{I \setminus \{ i \}} = \lfloor v \rfloor^{I \setminus \{ i \}}$
and 
$c_i(\xi) = c_i(\zeta)$. 
Since 
$\{ 1,2,\ldots , i \}$
is stable under the action of 
$W_{I \setminus \{ i \}}$, 
$\lfloor w \rfloor^{I \setminus \{ i \}} = \lfloor v \rfloor^{I \setminus \{ i \}}$
implies 
$\mathsf{T}^{(i)}_w = \mathsf{T}^{(i)}_v$.
Thus 
$\mathcal{Y}_i^{A_{n-1}}(x) = \mathcal{Y}_i^{A_{n-1}}(y)$. 

(2): 
The assertion follows from (1) and 
\cite[Theorem 4.7 and Equation (67)]{I20}.

(3): 
We first assume that 
$c'-c \geq i$. 
Obviously, 
$c \leq c'$
holds. 
Since 
$i-c'+c \leq 0$, 
the latter condition in 
\eqref{eq:tab-A}
is trivial. 
Thus
$(\mathsf{T},c) \preceq (\mathsf{T}',c')$. 

We next assume that 
$i > c'-c \geq n-i$. 
Obviously, 
$c \leq c'$
holds. 
Since 
$\mathsf{T}(u) \in [u,u+n-i]$
holds for all 
$u \in [i]$, 
we have
\begin{align}
\mathsf{T}(u) 
\leq 
u+n-i
\leq 
u+c'-c
\leq
\mathsf{T}'(u+c'-c)
\ \text{for} \  
u \in [i - c + c'].
\end{align}
This implies 
$(\mathsf{T},c) \preceq (\mathsf{T}',c')$. 
\end{proof}

By combining Propositions \ref{prop:Deo} and \ref{prop:tab-cri-A} (2), 
we obtain the following tableau criterion for the 
semi-infinite Bruhat order 
on $W_{\af}$ of type $A_{n-1}^{(1)}$.

\begin{thm} \label{thm:tab-cri-A}
Let 
$J \subset I$
and 
$x,y \in (W^J)_{\af}$. 
We have 
$x \preceq y$ 
in $(W^J)_{\af}$ 
if and only if 
$\mathcal{Y}_i^{A_{n-1}}(x) \preceq \mathcal{Y}_i^{A_{n-1}}(y)$
in 
$\mathrm{CST}_{A_{n-1}}(\varpi_i) \times \BZ$ 
for all 
$i \in I \setminus J$.
\end{thm}

For 
$w \in W$, 
the notation
$w = i_1 i_2 \cdots i_n$
means that 
$w(u) = i_u$
for 
$1 \leq u \leq n$.
A column-strict tableau (of skew shape) is called semi-standard 
if its entries are weakly increasing from right to left in each row.

\ytableausetup{mathmode,boxsize=4mm}
\begin{ex} \label{ex:tab-cri}
Assume that 
$n=6$. 
Let 
$w = 564213,\, 
v = 412635 \in W$, 
and 
\begin{align*}
\xi = \alpha_1^{\vee} - \alpha_3^{\vee} + \alpha_4^{\vee} + 2\alpha_5^{\vee}, \ 
\zeta = 2\alpha_1^{\vee} + 3\alpha_2^{\vee} + \alpha_3^{\vee} + 2\alpha_4^{\vee} + 5\alpha_5^{\vee}
\in Q^{\vee} .
\end{align*}
Let us compare
$x = wt_{\xi}$
and 
$y = vt_{\zeta}$
in semi-infinite Bruhat order on $W_{\af}$.
We have
%
\begin{align*}
\mathsf{T}^{(1)}_w \mathsf{T}^{(2)}_w \mathsf{T}^{(3)}_w \mathsf{T}^{(4)}_w \mathsf{T}^{(5)}_w
=
\begin{ytableau}
5 & 5 & 4 & 2 & 1\\
\none & 6 & 5 & 4 & 2\\
\none & \none & 6 & 5 & 4\\
\none & \none & \none & 6 & 5 \\
\none & \none & \none & \none & 6 \\
\end{ytableau}
\end{align*}
and
\begin{align*}
\mathsf{T}^{(1)}_v \mathsf{T}^{(2)}_v \mathsf{T}^{(3)}_v \mathsf{T}^{(4)}_v \mathsf{T}^{(5)}_v
=
\begin{ytableau}
4 & 1 & 1 & 1 & 1\\
\none & 4 & 2 & 2 & 2\\
\none & \none & 4 & 4 & 3\\
\none & \none & \none & 6 & 4 \\
\none & \none & \none & \none & 6 \\
\end{ytableau} \, .
\end{align*}
\begin{enumerate}[(1)]
\item
$\mathcal{Y}_1^{A_{n-1}}(x) \preceq \mathcal{Y}_1^{A_{n-1}}(y)$
since
$c_1(\zeta) - c_1(\xi) = 1 = \min \{ 1,6-1 \}$
(see Proposition \ref{prop:tab-cri-A} (3)).
\item
$\mathcal{Y}_2^{A_{n-1}}(x) \preceq \mathcal{Y}_2^{A_{n-1}}(y)$
since
$c_2(\zeta) - c_2(\xi) = 3 \geq 2 = \min \{ 2,6-2 \}$
(see Proposition \ref{prop:tab-cri-A} (3)).
\item
$\mathcal{Y}_3^{A_{n-1}}(x) \preceq \mathcal{Y}_3^{A_{n-1}}(y)$
since
$c_3(\zeta) - c_3(\xi) = 2$
and 
$\begin{ytableau}
1 & \none \\
2 & \none \\
4 & 4 \\
\none & 5 \\
\none & 6
\end{ytableau}$
is semi-standard
(see \eqref{eq:tab-A}), 
where 
$\mathsf{T}_w^{(3)}
=
\begin{ytableau}
4 \\ 5 \\ 6
\end{ytableau}$
and 
$\mathsf{T}_v^{(3)}
=
\begin{ytableau}
1 \\ 2 \\ 4
\end{ytableau}$.
\item
$\mathcal{Y}_4^{A_{n-1}}(x) \preceq \mathcal{Y}_4^{A_{n-1}}(y)$
since
$c_4(\zeta) - c_4(\xi) = 1$
and 
$\begin{ytableau}
1 & \none \\
2 & 2 \\
4 & 4 \\
6 & 5 \\
\none & 6
\end{ytableau}$
is semi-standard
(see \eqref{eq:tab-A}),
where 
$\mathsf{T}_w^{(4)}
=
\begin{ytableau}
2 \\ 4 \\ 5 \\ 6
\end{ytableau}$
and 
$\mathsf{T}_v^{(4)}
=
\begin{ytableau}
1 \\ 2 \\ 4 \\ 6
\end{ytableau}$.
\item
$\mathcal{Y}_5^{A_{n-1}}(x) \preceq \mathcal{Y}_5^{A_{n-1}}(y)$
since
$c_5(\zeta) - c_5(\xi) = 3 \geq 1 = \min \{ 5,6-5 \}$
(see Proposition \ref{prop:tab-cri-A} (3)).
\end{enumerate}
By Theorem \ref{thm:tab-cri-A}, 
we conclude that 
$x \preceq y$.
\end{ex}

\subsection{Type $C_n^{(1)}$} \label{subsection:tab-cri-C}

Fix an integer $n \geq 2$. 
Set 
$I = [n]$. 
We assume that the labeling of the vertices of the Dynkin diagram of type $C_n$ is as follows.
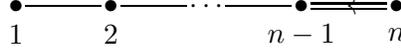
\begin{center}\ 
\xygraph{!~:{@{=}|@{<}}
\bullet ([]!{+(0,-.3)} {1}) - [r]
\bullet ([]!{+(0,-.3)} {2}) - [r] \cdots - [r]
\bullet ([]!{+(0,-.3)} {n - 1}) : [r]
\bullet ([]!{+(0,-.3)} {n})}
\end{center}
Let 
$\varepsilon_1,\varepsilon_2,\ldots ,\varepsilon_n$
be an orthonormal basis of an 
$n$-dimensional Euclidean space $\BR^n$. 
Let 
$\Delta = \{ \pm(\varepsilon_s \pm \varepsilon_t) \mid s,t \in [n], \ s < t \} \sqcup \{ \pm 2\varepsilon_s \mid s \in [n] \}$
be a root system of type $C_n$, 
and let 
$\Pi 
=
\{ \alpha_s = \varepsilon_s - \varepsilon_{s+1} \mid s \in [n-1] \} 
\sqcup 
\{ \alpha_n = 2\varepsilon_n \}$
be a simple root system of 
$\Delta$. 

Let 
$W$
be the Weyl group of 
$\Delta$. 
Note that 
$W$ 
acts faithfully on 
$\{ \pm\varepsilon_s \mid s \in [n] \} \subset \BR^n$. 
Define a totally ordered set 
$\mathcal{C}_n$
by 
\begin{align}
\mathcal{C}_n 
=
\{ 1 \prec 2 \prec \cdots \prec n-1 \prec n \prec \overline{n} \prec \overline{n-1} \prec \cdots \prec \overline{2} \prec \overline{1} \}.\end{align}
Let 
$\sigma : \mathcal{C}_n \to \mathcal{C}_n$
be the bijection defined by 
$s \leftrightarrow \overline{s}$
for 
$s \in [n]$. 
If we identify 
$\mathcal{C}_n$
with 
$\{ \pm\varepsilon_s \mid s \in [n] \}$
by 
$s = \varepsilon_s$
and 
$\overline{s} = -\varepsilon_s$
for 
$s \in [n]$, 
then $W$ can be described as follows:
\begin{align}
W 
=
\{ w \in \mathfrak{S}(\mathcal{C}_n) \mid 
w(\sigma(s)) = \sigma(w(s))
\ \text{for} \ 
s \in [n] \} .
\end{align}
Let 
$(s_1 \ s_2 \ \cdots \ s_l) \in \mathfrak{S}(\mathcal{C}_n)$
denote the cyclic permutation 
$s_1 \mapsto s_2 \mapsto \cdots \mapsto s_l \mapsto s_1$, 
where 
$l \geq 1$
and 
$s_1, s_2, \ldots , s_l \in \mathcal{C}_n$
are all distinct. 
For 
$s,t \in [n]$, $s<t$, 
we have 
$r_{\pm(\varepsilon_s - \varepsilon_t)}
= 
(s \ t)(\overline{s} \ \overline{t})$, 
$r_{\pm2\varepsilon_s}
=
(s \ \overline{s})$, 
and 
$r_{\pm(\varepsilon_s + \varepsilon_t)} 
=
(s \ \overline{t})(\overline{s} \ t)$
in 
$\mathfrak{S}(\mathcal{C}_n)$. 

For 
$w \in \mathfrak{S}(\mathcal{C}_n)$
and 
$s \in [n]$, 
set
\begin{align}
\mathsf{A}_s(w)
&=
\{ t \in [s+1,n] \mid w(s) \succ w(t) \ \text{in} \ \mathcal{C}_n \}, \ \ 
\mathsf{a}_s(w) = \# \mathsf{A}_s(w), \\
\mathsf{B}_s(w)
&=
\{ t \in [s+1,n] \mid w(s) \succ \sigma(w(t)) \ \text{in} \ \mathcal{C}_n \}, \ \ 
\mathsf{b}_s(w) = \# \mathsf{B}_s(w), \\
\mathsf{e}_s(w) 
&=
\begin{cases}
0 & \text{if} \ w(s) \preceq n,  \\
1 & \text{if} \ w(s) \succeq \overline{n};
\end{cases}
\end{align}
note that 
$\mathsf{a}_n(w) = \mathsf{b}_n(w) = 0$
for 
$w \in \mathfrak{S}(\mathcal{C}_n)$. 
The length function
$\ell : W \to \BZ_{\geq 0}$
is given by 
$\ell(w)
=
\sum_{s = 1}^n (\mathsf{a}_s(w) + \mathsf{b}_s(w) + \mathsf{e}_s(w))$
for 
$w \in W$. 
The longest element of $W$ is given by 
$u \mapsto \overline{u}$, 
$u \in [n]$.

Let 
$\| \cdot \| : \mathcal{C}_n \to [n]$
be the map defined by 
$\|s\| = s$
and 
$\|\overline{s}\| = s$
for 
$s \in [n]$.
We identify an $i$-element subset 
$\mathsf{T} = \{ \mathsf{T}(1) \prec \mathsf{T}(2) \prec \cdots \prec \mathsf{T}(i) \} 
\subset 
\mathcal{C}_n$
with the column-strict tableau 
of the form 
\eqref{eq:column-A}.
Set
\begin{align}
\mathrm{CST}_{C_n}(\varpi_i)
=
\{ \mathsf{T} \mid 
\mathsf{T} \subset \mathcal{C}_n, \ 
\# \mathsf{T} = i, \ 
\text{and} \ 
\|\mathsf{T}(u)\|, \ u \in [i], \ 
\text{are all distinct} \} .
\end{align}
For 
$w \in W$, 
let 
$\mathsf{T}_w^{(i)} 
\in 
\mathrm{CST}_{C_n}(\varpi_i)$
be such that 
\begin{align}
\mathsf{T}_w^{(i)}
=
\left\{ \mathsf{T}_w^{(i)}(1) \prec \mathsf{T}_w^{(i)}(2) \prec \cdots \prec \mathsf{T}_w^{(i)}(i) \right\}
=
\{ w(1), w(2), \ldots , w(i) \}.
\end{align}
The proof of the next lemma is standard 
(cf. \cite[\S 8.1]{BB}).

\begin{lem} \label{lem:Gr-C}
Let 
$i \in I$. 
We have 
\begin{align*}
W^{I \setminus \{ i \}}
=
\{ w \in W \mid \ 
&w(1) \prec w(2) \prec \cdots \prec w(i), \ \text{and} \\
&w(i+1) \prec w(i+2) \prec \cdots \prec w(n) \preceq n \}.
\end{align*} 
If 
$w \in W^{I \setminus \{ i \}}$, 
then
$\ell(w) = \sum_{s=1}^i(\mathsf{a}_s(w) + \mathsf{b}_s(w) + \mathsf{e}_s(w))$
and 
$\mathsf{A}_s(w) \subset [i+1,n]$
for 
$s \in [i]$. 
The map 
$W^{I \setminus \{ i \}} \to \mathrm{CST}_{C_n}(\varpi_i)$, 
$w \mapsto \mathsf{T}_w^{(i)}$, 
is bijective. 
\end{lem}

We see from Lemmas \ref{lem:J-ad}--\ref{lem:(W^J)_af} and \ref{lem:Gr-C}
that the map
\begin{align} 
\mathcal{Y}_i^{C_n} : 
W_{\af}
\to 
\mathrm{CST}_{C_n}(\varpi_i) \times \BZ, \ 
wt_{\xi} \mapsto \left( \mathsf{T}_w^{(i)}, c_i(\xi) \right),
\end{align} 
induces a bijection from the subset 
$(W^{I \setminus \{ i \}})_{\af} \subset W_{\af}$
to 
$\mathrm{CST}_{C_n}(\varpi_i) \times \BZ$.

\begin{define} \label{def:SiB-C}
Define a partial order 
$\preceq$
on 
$\mathrm{CST}_{C_n}(\varpi_i) \times \BZ$
as follows: 
for 
$(\mathsf{T},c) , (\mathsf{T}',c')
\in 
\mathrm{CST}_{C_n}(\varpi_i) \times \BZ$, 
set
$(\mathsf{T},c) \preceq (\mathsf{T}',c')$
if 
\begin{align}
(c \leq c')
\ \text{and} \ 
\left(
\mathsf{T}(u) \preceq \mathsf{T}'(u+c'-c)
\ \text{in} \ 
\mathcal{C}_n
\ \text{for} \ 
u \in [i-c'+c]
\right).
\end{align}
\end{define}

\begin{prop} \label{prop:tab-cri-C}
Let 
$i \in I$. 
\begin{enumerate}[(1)]
\item
$\mathcal{Y}_i^{C_n} \circ \Pi^{I \setminus \{ i \}} = \mathcal{Y}_i^{C_n}$.
\item
For 
$x,y \in W_{\af}$, 
we have 
$\Pi^{I \setminus \{ i \}}(x) \preceq \Pi^{I \setminus \{ i \}}(y)$
in 
$(W^{I \setminus \{ i \}})_{\af}$
if and only if 
$\mathcal{Y}_i^{C_n}(x) \preceq \mathcal{Y}_i^{C_n}(y)$
in
$\mathrm{CST}_{C_n}(\varpi_i) \times \BZ$. 
\item
Let 
$(\mathsf{T},c), (\mathsf{T}',c') \in \mathrm{CST}_{C_n}(\varpi_i) \times \BZ$. 
If 
$c' - c \geq i$, 
then 
$(\mathsf{T},c) \preceq (\mathsf{T}',c')$. 
\end{enumerate}
\end{prop}

By combining Propositions \ref{prop:Deo} and \ref{prop:tab-cri-C} (2), 
we obtain the following tableau criterion for the 
semi-infinite Bruhat order
on $W_{\af}$ of type $C_n^{(1)}$.

\begin{thm} \label{thm:tab-cri-C}
Let 
$J \subset I$. 
For 
$x,y \in (W^J)_{\af}$, 
we have 
$x \preceq y$
in 
$(W^J)_{\af}$
if and only if 
$\mathcal{Y}_i^{C_n}(x) \preceq \mathcal{Y}_i^{C_n}(y)$
in 
$\mathrm{CST}_{C_n}(\varpi_i) \times \BZ$
for all 
$i \in I \setminus J$. 
\end{thm}

The remainder of this subsection is devoted to 
the proof of Proposition \ref{prop:tab-cri-C}. 

The proofs of 
Lemmas \ref{lem:J-ad-C}--\ref{lem:<g,r>-C} below 
are straightforward.

\begin{lem} \label{lem:J-ad-C}
Let 
$i \in I$
and 
$\gamma \in \Delta^+ \setminus \Delta_{I \setminus \{ i \}}^+$. 
We have 
$\gamma^{\vee} \in Q^{\vee,I \setminus \{ i \}}$
if and only if one of the following conditions holds:
\begin{enumerate}[(1)]
\item
$\gamma^{\vee}
=
\varepsilon_i 
= 
\alpha_i^{\vee} + \alpha_{i+1}^{\vee} + \cdots + \alpha_n^{\vee}$.
\item
$\gamma^{\vee}
=
\varepsilon_{i-1} + \varepsilon_i 
= 
\alpha_{i-1}^{\vee} + 2\alpha_i^{\vee} + \cdots + 2\alpha_n^{\vee}$.
\end{enumerate}
\end{lem}

\begin{lem} \label{lem:<g,r>-C}
Let $i \in I$. 
We have
$2\langle \alpha_i^{\vee} , \rho - \rho_{I \setminus \{ i \}} \rangle = 2n-i+1$.
\end{lem}

\begin{prop}[cf. {\cite[\S 8.1]{BB}}] \label{prop:Bruhat-C}
Let 
$i \in I$, 
$w \in W^{I \setminus \{ i \}}$, 
and 
$\gamma \in \Delta^+$. 
There exists a Bruhat edge 
$w \xrightarrow{\ \gamma \ } \lfloor wr_{\gamma} \rfloor = wr_{\gamma}$
in 
$\mathrm{QB}^{I \setminus \{ i \}}$
if and only if 
$\gamma \in \Delta^+ \setminus \Delta_{I \setminus \{ i \}}^+$
and one of the following statements holds.
\begin{enumerate}[(b-C1)]
\item
$i \in [n-1]$, 
$c_i(\gamma^{\vee}) = 1$, 
and there exists 
$s \in [i]$
such that  
$wr_{\gamma}(u) = w(u)$
for 
$u \in [i] \setminus \{ s \}$, 
$1 \preceq w(s) \prec n$, 
and 
$wr_{\gamma}(s) = 
\min([w(s)+1,n] \setminus \{ \|w(u)\| \mid u \in [i], \ w(u) \succeq \overline{n} \})$;
in this case, 
we have 
$\gamma^{\vee} 
= 
\varepsilon_s - \varepsilon_t
=
\alpha_s^{\vee} + \alpha_{s+1}^{\vee} + \cdots + \alpha_t^{\vee}$
for some 
$t \in [i+1,n]$.
\item
$i \in [n-1]$, 
$c_i(\gamma^{\vee}) = 1$, 
and there exists
$s \in [i]$
such that  
$wr_{\gamma}(u) = w(u)$
for 
$u \in [i] \setminus \{ s \}$, 
$\overline{n} \preceq w(s) \prec \overline{1}$, 
and 
$wr_{\gamma}(s)
=
\sigma 
\bigl(
\max ([1,\|w(s)\|-1] \setminus \{ w(u) \mid u \in [i], \ w(u) \preceq n \})
\bigr)$;
in this case, 
we have 
$\gamma^{\vee} 
= 
\varepsilon_s - \varepsilon_t
=
\alpha_s^{\vee} + \alpha_{s+1}^{\vee} + \cdots + \alpha_t^{\vee}$
for some 
$t \in [i+1,n]$.
\item
$i \in [2,n]$, 
$c_i(\gamma^{\vee}) = 2$, 
and there exist 
$s,t \in [i]$
such that 
$s < t$, 
$wr_{\gamma}(u) = w(u)$
for 
$u \in [i] \setminus \{ s,t \}$, 
and 
$wr_{\gamma}(s) = w(s) + 1 = \sigma(w(t)) = \sigma(wr_{\gamma}(t)) + 1 \preceq n$;
in this case, 
we have 
$\gamma^{\vee} 
= 
\varepsilon_s + \varepsilon_t
=
\alpha_s^{\vee} + \cdots + \alpha_{t-1}^{\vee}
+ 2\alpha_t^{\vee} + \cdots + 2\alpha_n^{\vee}$.
\item
$c_i(\gamma^{\vee}) = 1$, 
and there exists
$s \in [i]$
such that 
$wr_{\gamma}(u) = w(u)$
for 
$u \in [i] \setminus \{ s \}$
and 
$\sigma(wr_{\gamma}(s)) = w(s) = n$;
in this case, 
we have 
$\gamma^{\vee} 
= 
\varepsilon_s
=
\alpha_s^{\vee} + \alpha_{s+1}^{\vee} + \cdots + \alpha_n^{\vee}$.
\end{enumerate}
Moreover, for 
$w,v \in W^{I \setminus \{ i \}}$, 
we have 
$w \preceq v$
if and only if 
$w(u) \preceq v(u)$
in 
$\mathcal{C}_n$
for 
$u \in [i]$. 
\end{prop}

\begin{prop} \label{prop:Q=C}
Let 
$i \in I$, 
$w \in W^{I \setminus \{ i \}}$
and 
$\gamma \in \Delta^+$. 
There exists a quantum edge 
$w \xrightarrow{\ \gamma \ } \lfloor wr_{\gamma} \rfloor$
in 
$\mathrm{QB}^{I \setminus \{ i \}}$
if and only if 
$\gamma \in \Delta^+ \setminus \Delta_{I \setminus \{ i \}}^+$
and the following statement holds:
\begin{enumerate}
\item[(q-C)]
$c_i(\gamma^{\vee}) = 1$, 
$\lfloor wr_{\gamma} \rfloor (1) = 1$, 
$\lfloor wr_{\gamma} \rfloor (u) = w(u-1)$
for 
$u \in [2,i]$, 
and 
$w(i) = \overline{1}$;
in this case, 
we have 
$\gamma^{\vee} = \varepsilon_i =  \alpha_i^{\vee} + \alpha_{i+1}^{\vee} + \cdots + \alpha_n^{\vee}$.
\end{enumerate}
\end{prop}

Before starting the proof of Proposition \ref{prop:Q=C}, 
we mention a consequence of 
Lemma \ref{lem:Q=SiB} and Propositions \ref{prop:Bruhat-C}--\ref{prop:Q=C}.

\begin{prop} \label{prop:SiB-C}
Let 
$i \in I$, 
$x,y \in (W^{I \setminus \{ i \}})_{\af}$, 
$\mathcal{Y}_i^{C_n}(x) = (\mathsf{T},c)$, 
and 
$\mathcal{Y}_i^{C_n}(y) = (\mathsf{T}',c')$. 
There exists an edge 
$x \xrightarrow{\ \beta \ } y$
in 
$\mathrm{SiB}^{I \setminus \{ i \}}$
for some 
$\beta \in \Delta_{\af}^+$
if and only if one of the following conditions holds:
\begin{enumerate}[($\si$-C1)]
\item
$i \in [n-1]$, 
$c' = c$, 
and there exists 
$s \in [i]$
such that 
$\mathsf{T}'(u) = \mathsf{T}(u)$
for 
$u \in [i] \setminus \{ s \}$, 
$1 \preceq \mathsf{T}(s) \prec n$, 
and 
$\mathsf{T}'(s) = 
\min([\mathsf{T}(s)+1,n] \setminus \{ \|\mathsf{T}(u)\| \mid u \in [i], \ \mathsf{T}(u) \succeq \overline{n} \})$. 
\item
$i \in [n-1]$, 
$c' = c$, 
and there exists 
$s \in [i]$
such that 
$\mathsf{T}'(u) = \mathsf{T}(u)$
for 
$u \in [i] \setminus \{ s \}$,  
$\overline{n} \preceq \mathsf{T}(s) \prec \overline{1}$, 
and 
$\mathsf{T}'(s)
=
\sigma 
\bigl(
\max ([1,\|\mathsf{T}(s)\|-1] \setminus \{ \mathsf{T}(u) \mid u \in [i], \ \mathsf{T}(u) \preceq n \})
\bigr)$. 
\item
$i \in [2,n]$, 
$c' = c$, 
and there exist 
$s,t \in [i]$
such that 
$s < t$, 
$\mathsf{T}'(u) = \mathsf{T}(u)$
for 
$u \in [i] \setminus \{ s,t \}$,  
and 
$\mathsf{T}'(s) 
= 
\mathsf{T}(s)+1
=
\sigma(\mathsf{T}(t))
=
\sigma(\mathsf{T}'(t))+1
\preceq 
n$.
\item
$c' = c$, 
and there exists 
$s \in [i]$
such that 
$\mathsf{T}'(u) = \mathsf{T}(u)$
for 
$u \in [i] \setminus \{ s \}$
and 
$\sigma(\mathsf{T}'(s)) = \mathsf{T}(s) = n$.
\item
$c' = c+1$, 
$\mathsf{T}'(1) = 1$, 
$\mathsf{T}'(u) = \mathsf{T}(u-1)$
for 
$u \in [2,i]$, 
and 
$\mathsf{T}(i) = \overline{1}$.
\end{enumerate}
\end{prop}

For 
$i \in I$, 
$w \in W^{I \setminus \{ i \}}$
and 
$\gamma \in \Delta^+ \setminus \Delta_{I \setminus \{ i \}}^+$, 
let 
$\mathrm{Q}(i,w,\gamma)$
denote the following statement.
\begin{description}
\item[$\mathrm{Q}(i,w,\gamma):$]
There exists a quantum edge 
$w \xrightarrow{\ \gamma \ } \lfloor wr_{\gamma} \rfloor$
in 
$\mathrm{QB}^{I \setminus \{ i \}}$.
\end{description}

\begin{proof}[Proof of Proposition \ref{prop:Q=C}]
By Lemmas \ref{lem:LS} and \ref{lem:J-ad-C}, 
we may assume that 
$\gamma^{\vee} \in Q^{\vee,I \setminus \{ i \}}$, 
$\lfloor wr_{\gamma} \rfloor = wr_{\gamma}(z_{\gamma^{\vee}}^{I \setminus \{ i \}})^{-1}$
and 
$c_i(\gamma^{\vee}) \in \{ 1,2 \}$. 
Let 
$I \setminus \{ i \}
=
I_1 \sqcup I_2$, 
where 
$I_1 = [i-1]$
is of type $A_{i-1}$
and 
$I_2 = [i+1,n]$
is of type $C_{n-i}$.
The proof will be divided into three steps. 

\begin{proof}[Step 1.] 
We show that 
$c_i(\gamma^{\vee}) = 1$
and 
$\mathrm{Q}(i,w,\gamma)$ 
imply 
(q-C). 
It follows from Lemmas \ref{lem:<g,r>} and \ref{lem:<g,r>-C} that 
$c_i(\gamma^{\vee}) = 1$
and 
$\mathrm{Q}(i,w,\gamma)$ 
are equivalent to 
$\ell(\lfloor wr_{\gamma} \rfloor) - \ell(w)
=
i-2n$. 
By Lemma \ref{lem:J-ad-C}, 
$\gamma^{\vee} = \varepsilon_i$
and 
$r_{\gamma} = (i \ \overline{i})$. 
We see that 
$(i-1,0) \in (I_1)_{\af} \times (I_2)_{\af}$
satisfies the condition for 
$\gamma^{\vee} \in Q^{\vee}$
in 
Lemma \ref{lem:J-ad}; 
note that 
$I_1 \setminus \{ i-1 \} = [i-2]$
is of type 
$A_{i-2}$. 
Hence 
$z_{\gamma^{\vee}}^{I \setminus \{ i \}} 
=
w_0^{I_1}w_0^{I_1 \setminus \{ i-1 \}}
=
(1 \ 2 \ \cdots \ i)(\overline{1} \ \overline{2} \ \cdots \ \overline{i})$
and 
$\lfloor wr_{\gamma} \rfloor = w(i \ \overline{i})(i \ \cdots \ 2 \ 1)(\overline{i} \ \cdots \ \overline{2} \ \overline{1})$. 
We have
$\lfloor wr_{\gamma} \rfloor (1) = w(\overline{i})$, 
$\lfloor wr_{\gamma} \rfloor (u) = w(u-1)$
for 
$u \in [2,i]$, 
and 
$\lfloor wr_{\gamma} \rfloor (u) = w(u)$
for 
$u \in [i+1,n]$.
It follows from Lemma \ref{lem:Gr-C} that 
\begin{align} \label{eq:seq-C1}
\underbrace{w(\overline{i})}_{\preceq n} \prec w(1) \prec w(2) \prec \cdots \prec w(i-1) \prec \underbrace{w(i)}_{\succeq \overline{n}}.
\end{align}
It remains to prove that 
$w(i) = \overline{1}$. 
If we prove that 
\begin{enumerate}[(1)]
\item
$\mathsf{a}_1(\lfloor wr_{\gamma} \rfloor) 
=
\mathsf{b}_1(\lfloor wr_{\gamma} \rfloor)
=
\mathsf{e}_1(\lfloor wr_{\gamma} \rfloor) 
= 
0$,
\item
$\mathsf{a}_i(w) 
=
n-i$, 
$\mathsf{b}_i(w)
=n-i-(w(\overline{i})-1)$, 
$\mathsf{e}_i(w)
=
1$,
\item
$\mathsf{e}_s(\lfloor wr_{\gamma} \rfloor) 
=
\mathsf{e}_{s-1}(w)$
for 
$s \in [2,i]$, 
\item
$\mathsf{a}_s(\lfloor wr_{\gamma} \rfloor) 
=
\mathsf{a}_{s-1}(w)$
and 
$\mathsf{b}_s(\lfloor wr_{\gamma} \rfloor) 
=
\mathsf{b}_{s-1}(w)-1$
for 
$s \in [2,i]$, 
\end{enumerate}
then the assertion follows. 
Indeed, 
by Lemma \ref{lem:Gr-C}, 
we have
\begin{align*}
\ell
(\lfloor wr_{\gamma} \rfloor) - \ell(w) 
\ =& \ 
\underbrace{
\mathsf{a}_1(\lfloor wr_{\gamma} \rfloor) 
+
\mathsf{b}_1(\lfloor wr_{\gamma} \rfloor)
+
\mathsf{e}_1(\lfloor wr_{\gamma} \rfloor)}_{=0} \\
&-
\underbrace{
\left(
\mathsf{a}_i(w)
+
\mathsf{b}_i(w)
+
\mathsf{e}_i(w)
\right)}_{= 1+2(n-i)-(w(\overline{i})-1)}
\\
&+
\sum_{s=2}^i
\bigl(
\underbrace{
\mathsf{a}_s(\lfloor wr_{\gamma} \rfloor) - \mathsf{a}_{s-1}(w)}_{=0}
\bigr) \\
&+
\sum_{s=2}^i
\bigl(
\underbrace{
\mathsf{b}_s(\lfloor wr_{\gamma} \rfloor) - \mathsf{b}_{s-1}(w)}_{=-1}
\bigr) \\
&+
\sum_{s=2}^i
\bigl(
\underbrace{
\mathsf{e}_s(\lfloor wr_{\gamma} \rfloor) - \mathsf{e}_{s-1}(w)}_{=0}
\bigr) \\
=& \ 
i - 2n + (w(\overline{i}) - 1).
\end{align*}
Since 
$\ell(\lfloor wr_{\gamma} \rfloor) - \ell(w)
=
i - 2n$, 
we get 
$w(\overline{i}) = 1$
and 
$w(i) = \overline{1}$.

We prove (1)--(4) as follows. 

(1) follows from 
$\lfloor wr_{\gamma} \rfloor(1) = 1$. 

(2): 
Since 
$\overline{n} \preceq w(i)$, 
we have 
$\mathsf{e}_i(w) = 1$, 
$\mathsf{A}_i(w) = [i+1,n]$, 
and hence
$\mathsf{a}_i(w) = n-i$. 
We see from 
$w(\overline{i}) \prec w(1)$
that 
$\mathsf{B}_i(w) = [i+w(\overline{i}),n]$
and 
$\mathsf{b}_i(w) = n-i-(w(\overline{i})-1)$. 

(3) follows from 
$\lfloor wr_{\gamma} \rfloor(s) = w(s-1)$
for 
$s \in [2,i]$. 

(4): 
Let
$s \in [2,i]$. 
We see from Lemma \ref{lem:Gr-C}
that 
\begin{align*}
\mathsf{A}_s(\lfloor wr_{\gamma} \rfloor)
=
\{ t \in [i+1,n] \mid 
\underbrace{\lfloor wr_{\gamma} \rfloor(s)}_{=w(s-1)} 
\succ 
\underbrace{\lfloor wr_{\gamma} \rfloor(t)}_{=w(t)} \}
=
\mathsf{A}_{s-1}(w), 
\end{align*}
which implies 
$\mathsf{a}_s(\lfloor wr_{\gamma} \rfloor)
=
\mathsf{a}_{s-1}(w)$
for 
$s \in [2,i]$. 
Similarly, we deduce from 
$i \in \mathsf{B}_{s-1}(w)$
that the map
\begin{align*}
\mathsf{B}_s(\lfloor wr_{\gamma} \rfloor) \to \mathsf{B}_{s-1}(w) \setminus \{ i \}, \ 
t \mapsto 
\begin{cases}
t-1 & \text{if} \ t \in [s+1,i], \\
t & \text{if} \ t \in [i+1,n],
\end{cases}
\end{align*}
is bijective, 
which implies 
$\mathsf{b}_s(\lfloor wr_{\gamma} \rfloor)
=
\mathsf{b}_{s-1}(w) - 1$
for 
$s \in [2,i]$. 
\end{proof}

\begin{proof}[Step 2.]
We show that 
$c_i(\gamma^{\vee}) = 2$ 
and
$\mathrm{Q}(i,w,\gamma)$ 
lead to a contradiction; 
the computation below will be used again in the proofs of 
Lemmas \ref{lem:Q->B3}--\ref{lem:B3->Q}
in 
\S \ref{subsection:tab-cri-B}. 
It follows from Lemmas \ref{lem:<g,r>} and \ref{lem:<g,r>-C} that 
$c_i(\gamma^{\vee}) = 2$ 
and 
$\mathrm{Q}(i,w,\gamma)$ 
are equivalent to 
$\ell(\lfloor wr_{\gamma} \rfloor) - \ell(w)
=
2i-4n-1$. 
By Lemma \ref{lem:J-ad-C}, 
$\gamma^{\vee} = \varepsilon_{i-1} + \varepsilon_i$
and 
$r_{\gamma} = (i-1 \ \overline{i})(\overline{i-1} \ i)$. 
We see that 
$(i-2,0) \in (I_1)_{\af} \times (I_2)_{\af}$
satisfies the condition for 
$\gamma^{\vee} \in Q^{\vee}$
in 
Lemma \ref{lem:J-ad}; 
note that 
$I_1 \setminus \{ i-2 \}
=[i-3] \sqcup \{ i-1 \}$
is of type
$A_{i-3} \times A_1$. 
Hence
$z_{\gamma^{\vee}}^{I \setminus \{ i \}} 
=
w_0^{I_1}w_0^{I_1 \setminus \{ i-2 \}}$
is given by 
$u \mapsto u+2$
for 
$u \in [i-2]$, 
$i-1 \mapsto 1$, 
and 
$i \mapsto 2$.
We have 
$\lfloor wr_{\gamma} \rfloor(1) = w(\overline{i})$, 
$\lfloor wr_{\gamma} \rfloor(2) = w(\overline{i-1})$, 
$\lfloor wr_{\gamma} \rfloor(u) = w(u-2)$
for 
$u \in [3,i]$, 
and 
$\lfloor wr_{\gamma} \rfloor(u) = w(u)$
for 
$u \in [i+1,n]$. 
It follows from Lemma \ref{lem:Gr-C} that 
\begin{align} \label{eq:seq-C2}
\underbrace{w(\overline{i}) \prec w(\overline{i-1})}_{\preceq n} 
\prec w(1) \prec w(2) \prec \cdots \prec w(i-2) \prec 
\underbrace{w(i-1) \prec w(i)}_{\succeq \overline{n}}; 
\end{align}
note that 
$w(\overline{i})-1 \geq 0$
and 
$w(\overline{i-1})-2 \geq 0$.
If we prove that 
\begin{enumerate}[(1)]
\item
$\mathsf{e}_1(\lfloor wr_{\gamma} \rfloor)
=
\mathsf{e}_2(\lfloor wr_{\gamma} \rfloor)
=
0$, 
$\mathsf{e}_s(\lfloor wr_{\gamma} \rfloor)
=
\mathsf{e}_{s-2}(w)$
for 
$s \in [3,i]$, 
$\mathsf{e}_{i-1}(w)
=
\mathsf{e}_i(w)
=
1$,
\item
$\mathsf{a}_1(\lfloor wr_{\gamma} \rfloor)
=
w(\overline{i})-1$, 
$\mathsf{a}_2(\lfloor wr_{\gamma} \rfloor)
=
w(\overline{i-1})-2$, 
$\mathsf{a}_{i-1}(w)
=
\mathsf{a}_i(w)
=
n-i$, 
\item
$\mathsf{b}_1(\lfloor wr_{\gamma} \rfloor)
=
\mathsf{b}_2(\lfloor wr_{\gamma} \rfloor)
=
0$, 
$\mathsf{b}_{i-1}(w)
=
n-i-(w(\overline{i-1})-2)+1$, 
$\mathsf{b}_i(w)
=
n-i-(w(\overline{i})-1)$, 
\item
$\mathsf{a}_s(\lfloor wr_{\gamma} \rfloor)
=
\mathsf{a}_{s-2}(w)$
and 
$\mathsf{b}_s(\lfloor wr_{\gamma} \rfloor)
=
\mathsf{b}_{s-2}(w) - 2$
for 
$s \in [3,i]$,
\end{enumerate}
then the assertion follows. 
Indeed, (1)--(4) and Lemma \ref{lem:Gr-C} imply
\begin{align*}
\ell(\lfloor wr_{\gamma} \rfloor) - \ell(w)
&=
2i - 4n + 1 + 
2\underbrace{(w(\overline{i})-1)}_{\geq 0} 
+ 
2\underbrace{(w(\overline{i-1})-2)}_{\geq 0} \\
&>
2i - 4n - 1,
\end{align*}
contrary to 
$\ell(\lfloor wr_{\gamma} \rfloor) - \ell(w)
=
2i-4n-1$. 

We prove (1)--(4) as follows. 

(1) follows from \eqref{eq:seq-C2}. 

(2): 
We see from Lemma \ref{lem:Gr-C} and \eqref{eq:seq-C2} that 
\begin{align*}
\mathsf{A}_1(\lfloor wr_{\gamma} \rfloor)
=
\{ t \in [i+1,n] \mid 
\underbrace{\lfloor wr_{\gamma} \rfloor(1)}_{= w(\overline{i})} 
\succ 
\underbrace{\lfloor wr_{\gamma} \rfloor(t)}_{= w(t)} \}
=
[i+1,i+w(\overline{i})-1],
\end{align*}
which implies 
$\mathsf{a}_1(\lfloor wr_{\gamma} \rfloor) = w(\overline{i}) - 1$. 
Similarly, we have 
\begin{align*}
\mathsf{A}_2(\lfloor wr_{\gamma} \rfloor)
&=
[i+1,i+w(\overline{i-1})-2], \\[2mm]
\mathsf{A}_{i-1}(\lfloor wr_{\gamma} \rfloor)
&=
\mathsf{A}_i(\lfloor wr_{\gamma} \rfloor)
=
[i+1,n], 
\end{align*}
which imply
$\mathsf{a}_2(\lfloor wr_{\gamma} \rfloor) = w(\overline{i-1}) - 2$
and 
$\mathsf{a}_{i-1}(\lfloor wr_{\gamma} \rfloor) 
=
\mathsf{a}_i(\lfloor wr_{\gamma} \rfloor) 
=
n-i$. 

(3): 
We claim that 
$\mathsf{B}_1(\lfloor wr_{\gamma} \rfloor) = \emptyset$. 
Suppose that 
$2 \in \mathsf{B}_1(\lfloor wr_{\gamma} \rfloor)$. 
Then 
$w(\overline{i}) = \lfloor wr_{\gamma} \rfloor(1) \succ \sigma(\lfloor wr_{\gamma} \rfloor(2)) = w(i-1)$, 
contrary to 
\eqref{eq:seq-C2}. 
Suppose that 
$t \in \mathsf{B}_1(\lfloor wr_{\gamma} \rfloor) \cap [3,i]$. 
Then 
$w(\overline{i}) = \lfloor wr_{\gamma} \rfloor(1) \succ \sigma(\lfloor wr_{\gamma} \rfloor(t)) = \sigma(w(t-2))$ 
and hence 
$w(i) \prec w(t-2)$,  
contrary to 
\eqref{eq:seq-C2}. 
Suppose that 
$t \in \mathsf{B}_1(\lfloor wr_{\gamma} \rfloor) \cap [i+1,n]$. 
Then 
$w(\overline{i}) = \lfloor wr_{\gamma} \rfloor(1) \succ \sigma(\lfloor wr_{\gamma} \rfloor(t)) = \sigma(w(t))$
and hence
$\overline{n} \preceq w(i) \prec w(t)$,   
contrary to 
Lemma \ref{lem:Gr-C}. 
Consequently, we have
$\mathsf{B}_1(\lfloor wr_{\gamma} \rfloor) = \emptyset$
and 
$\mathsf{b}_1(\lfloor wr_{\gamma} \rfloor) = 0$
as claimed. 
Similarly, 
we have 
$\mathsf{B}_2(\lfloor wr_{\gamma} \rfloor) = \emptyset$
and 
$\mathsf{b}_2(\lfloor wr_{\gamma} \rfloor) = 0$. 
We next claim that 
$\mathsf{B}_i(w) = [i+1,n] \setminus \mathsf{A}_1(\lfloor wr_{\gamma} \rfloor)$. 
Indeed, 
\begin{align*}
\mathsf{B}_i(w) 
&
= 
\{ t \in [i+1,n] \mid w(i) \succ \sigma(w(t)) \} \\
&=
\{ t \in [i+1,n] \mid w(\overline{i}) \prec w(t) \} \\
&=
[i+1,n] \setminus \mathsf{A}_1(\lfloor wr_{\gamma} \rfloor).
\end{align*}
This implies 
$\mathsf{b}_i(w) = n-i-(w(\overline{i})-1)$. 
Similarly, we have 
$\mathsf{B}_{i-1}(w) = [i,n] \setminus \mathsf{A}_2(\lfloor wr_{\gamma} \rfloor)$
and 
$\mathsf{b}_{i-1}(w) = n-i-(w(\overline{i-1})-2)+1$. 

(4): 
Let
$s \in [3,i]$. 
We see from Lemma \ref{lem:Gr-C} that 
\begin{align*}
\mathsf{A}_s(\lfloor wr_{\gamma} \rfloor)
=
\{ t \in [i+1,n] \mid 
\underbrace{\lfloor wr_{\gamma} \rfloor(s)}_{=w(s-2)} 
\succ 
\underbrace{\lfloor wr_{\gamma} \rfloor(t)}_{=w(t)} \}
=
\mathsf{A}_{s-2}(w), 
\end{align*}
which implies 
$\mathsf{a}_s(\lfloor wr_{\gamma} \rfloor)
=
\mathsf{a}_{s-2}(w)$
for 
$s \in [3,i]$. 
Similarly, we deduce from 
$i-1,i \in \mathsf{B}_{s-2}(w)$
that the map
\begin{align*}
\mathsf{B}_s(\lfloor wr_{\gamma} \rfloor) 
\to 
\mathsf{B}_{s-2}(w) \setminus \{ i-1,i \}, \ 
t \mapsto 
\begin{cases}
t-2 & \text{if} \ t \in [s+1,i], \\
t & \text{if} \ t \in [i+1,n],
\end{cases}
\end{align*}
is bijective, 
which implies 
$\mathsf{b}_s(\lfloor wr_{\gamma} \rfloor)
=
\mathsf{b}_{s-2}(w) - 2$
for 
$s \in [3,i]$. 
\end{proof}

\begin{proof}[Step 3.]
We show that 
(q-C) 
implies 
$\mathrm{Q}(i,w,\gamma)$. 
Assume that 
(q-C) 
is true. 
Since $c_i(\gamma^{\vee}) = 1$, 
$\mathrm{Q}(i,w,\gamma)$ 
is equivalent to 
$\ell(\lfloor wr_{\gamma} \rfloor) - \ell(w)
=
i - 2n$, 
by Lemmas \ref{lem:<g,r>} and \ref{lem:<g,r>-C}. 
We see from 
(q-C)
that $w$ and $\lfloor wr_{\gamma} \rfloor$ satisty 
(1)--(4) in Step 1.
As in Step 1, this gives 
$\ell(\lfloor wr_{\gamma} \rfloor) - \ell(w)
=
i - 2n + (w(\overline{i}) - 1)$. 
Since 
$w(\overline{i}) = 1$
by 
(q-C), 
we conclude that 
$\ell(\lfloor wr_{\gamma} \rfloor) - \ell(w)
=
i - 2n$.
\end{proof}

The proof of Proposition \ref{prop:Q=C} is complete.
\end{proof}

\begin{proof}[Proof of Proposition \ref{prop:tab-cri-C}]
(1) and (3) follow by the same method as in the proof of Proposition \ref{prop:tab-cri-A}.

We prove (2). 
Let 
$x,y \in (W^{I \setminus \{ i \}})_{\af}$, 
$\mathcal{Y}_i^{C_n}(x) = (\mathsf{T},c)$, 
and 
$\mathcal{Y}_i^{C_n}(y) = (\mathsf{T}',c')$.
By Proposition \ref{prop:SiB-C}, 
we may assume that 
$d := c'-c \geq 0$. 
The proof is by induction on $d$. 

If $d = 0$, 
the assertion follows from 
Propositions \ref{prop:Bruhat-C} and \ref{prop:SiB-C}. 

Assume that $d > 0$. 
It follows immediately from 
Proposition \ref{prop:SiB-C} that 
$x \preceq y$
implies 
$(\mathsf{T},c) \preceq (\mathsf{T}',c')$.
Conversely, we prove that 
$(\mathsf{T},c) \preceq (\mathsf{T}',c')$
implies 
$x \preceq y$; 
assume that 
$\mathsf{T}(u) \preceq \mathsf{T}'(u+d)$
for 
$u \in [i-d]$. 
To this end, 
we construct 
$x_1,x_2 \in (W^{I \setminus \{ i \}})_{\af}$
and 
$\mathsf{T}_1,\mathsf{T}_2 \in \mathrm{CST}_{C_n}(\varpi_i)$
such that 
$\mathcal{Y}_i^{C_n}(x_1) = (\mathsf{T}_1,c')$, 
$\mathcal{Y}_i^{C_n}(x_2) = (\mathsf{T}_2,c'-1)$, 
and 
$x \preceq x_2 \prec x_1 \preceq y$ 
as follows.
Let 
$s \in [0,i]$
be such that 
$\mathsf{T}'(u) = \overline{i-u+1}$
for 
$u \in [s+1,i]$, 
and 
$\mathsf{T}'(s) \prec \overline{i-s+1}$
if 
$s \neq 0$. 
Define 
$\mathsf{T}_1, \mathsf{T}_2 \in \mathrm{CST}_{C_n}(\varpi_i)$
by 
\begin{align*}
\mathsf{T}_1(u)
&=
\begin{cases}
1 & \text{if} \ u = 1, \\
\mathsf{T}'(u) & \text{if} \ u \in [2,s], \\
\overline{i-u+2} & \text{if} \ u \in [\max\{2,s+1\},i],
\end{cases} \\
\mathsf{T}_2(u)
&=
\begin{cases}
\mathsf{T}_1(u+1) & \text{if} \ u \in [i-1], \\
\overline{1} & \text{if} \ u=i .
\end{cases}
\end{align*}
To see that 
$\mathsf{T}_1$
is well-defined, 
it suffices to show that 
$\|\mathsf{T}_1(u)\|$, 
$u \in [i]$, 
are all distinct. 
This follows from 
$\|\mathsf{T}_1(1)\| = 1 
< 
\|\mathsf{T}_1(u)\| 
\leq 
i-s+1 
< 
\|\mathsf{T}_1(v)\|$
for 
$u \in [\max\{2,s+1\},i]$
and 
$v \in [2,s]$. 
Let 
$x_1,x_2 \in (W^{I \setminus \{ i \}})_{\af}$
be such that 
$\mathcal{Y}_i^{C_n}(x_1) = (\mathsf{T}_1,c')$
and 
$\mathcal{Y}_i^{C_n}(x_2) = (\mathsf{T}_2,c'-1)$. 
Since 
$\mathsf{T}_1(u) \preceq \mathsf{T}'(u)$
for 
$u \in [i]$, 
we see from the assertion for $d = 0$ that 
$x_1 \preceq y$. 
By Proposition \ref{prop:SiB-C} ($\si$-C5), 
we have 
$x_2 \prec x_1$. 
By induction hypothesis, 
it remains to prove that 
$\mathsf{T}(u) \preceq \mathsf{T}_2(u+d-1)$
for 
$u \in [i-d+1]$. 
Let 
$u \in [i-d+1]$. 
If 
$u \in [i-1] \cap [2,s]$, 
then 
$\mathsf{T}_2(u+d-1) 
= 
\mathsf{T}_1(u+d) 
= 
\mathsf{T}'(u+d)
\succeq 
\mathsf{T}(u)$. 
If 
$u \in [i-1] \cap [\max\{2,s+1\},i]$, 
then 
$\mathsf{T}_2(u+d-1) 
= 
\mathsf{T}_1(u+d) 
= 
\overline{i-(u+d)+1}
\succeq 
\mathsf{T}(u+d)
\succ 
\mathsf{T}(u)$. 
We have 
$\mathsf{T}_2(i)
=
\overline{1}
\succeq 
\mathsf{T}(i-d+1)$.
\end{proof}

\subsection{Type $B_n^{(1)}$} \label{subsection:tab-cri-B}

Fix an integer $n \geq 3$. 
Set 
$I = [n]$. 
We assume that the labeling of the vertices of the Dynkin diagram of type $B_n$ is as follows.
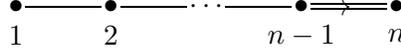
\begin{center}\ 
\xygraph{!~:{@{=}|@{>}}
\bullet ([]!{+(0,-.3)} {1}) - [r]
\bullet ([]!{+(0,-.3)} {2}) - [r] \cdots - [r]
\bullet ([]!{+(0,-.3)} {n - 1}) : [r]
\bullet ([]!{+(0,-.3)} {n})}
\end{center}
Let 
$\varepsilon_1,\varepsilon_2,\ldots ,\varepsilon_n$
be an orthonormal basis of an 
$n$-dimensional Euclidean space $\BR^n$. 
Let 
$\Delta = \{ \pm(\varepsilon_s \pm \varepsilon_t) \mid s,t \in [n], \ s < t \} \sqcup \{ \pm \varepsilon_s \mid s \in [n] \}$
be a root system of type $B_n$, 
and let 
$\Pi 
=
\{ \alpha_s = \varepsilon_s - \varepsilon_{s+1} \mid s \in [n-1] \} 
\sqcup 
\{ \alpha_n = \varepsilon_n \}$
be a simple root system of 
$\Delta$. 

Let 
$W$
be the Weyl group of 
$\Delta$; 
we see from 
\S \ref{subsection:tab-cri-C}
that 
\begin{align}
W 
=
\{ w \in \mathfrak{S}(\mathcal{C}_n) \mid 
w(\sigma(s)) = \sigma(w(s))
\ \text{for} \ 
s \in [n] \} .
\end{align}
Note that 
$r_{\pm(\varepsilon_s - \varepsilon_t)}
= 
(s \ t)(\overline{s} \ \overline{t})$, 
$r_{\pm\varepsilon_s}
=
(s \ \overline{s})$, 
and 
$r_{\pm(\varepsilon_s + \varepsilon_t)} 
=
(s \ \overline{t})(\overline{s} \ t)$
for 
$s,t \in [n]$, $s<t$. 
Write 
$\mathrm{CST}_{B_n}(\varpi_i) = \mathrm{CST}_{C_n}(\varpi_i)$. 
We know from 
Lemmas \ref{lem:J-ad}--\ref{lem:(W^J)_af}
and 
\ref{lem:Gr-C}
that the map
\begin{align} 
\mathcal{Y}_i^{B_n} : 
W_{\af}
\to 
\mathrm{CST}_{B_n}(\varpi_i) \times \BZ, \ 
wt_{\xi} \mapsto \left( \mathsf{T}_w^{(i)}, c_i(\xi) \right),
\end{align} 
induces a bijection from the subset 
$(W^{I \setminus \{ i \}})_{\af} \subset W_{\af}$
to 
$\mathrm{CST}_{B_n}(\varpi_i) \times \BZ$.

\begin{define} \label{def:SiB-B}
Let 
$i \in I$, 
$(\mathsf{T},c) , (\mathsf{T}',c')
\in 
\mathrm{CST}_{B_n}(\varpi_i) \times \BZ$, 
and 
$d := c'-c$. 
Define a partial order 
$\preceq$
on 
$\mathrm{CST}_{B_n}(\varpi_i) \times \BZ$
as follows.
\begin{enumerate}[(1)]
\item
Assume that $i = 1$. 
Set 
$(\mathsf{T},c) \preceq (\mathsf{T}',c')$
if either of the following holds:
\begin{align}
\begin{split}
&(d \geq 2), \ 
(d = 1, \ 
\mathsf{T}(1) = \overline{1}, \ 
\text{and} \ 
\mathsf{T}'(1) \neq 1), \\
&(d = 1 \ 
\text{and} \ 
\mathsf{T}(1) \neq \overline{1}), \ 
\text{or} \ 
(d=0 \ 
\text{and} \ 
\mathsf{T}(1) \preceq \mathsf{T}'(1) \ 
\text{in} \ 
\mathcal{C}_n).
\end{split}
\end{align} 
\item
Assume that 
$i \in [2,n-1]$. 
Set
$(\mathsf{T},c) \preceq (\mathsf{T}',c')$
if 
\begin{align}
(d \geq 0), \ 
(\mathsf{T}(u) \preceq \mathsf{T}'(u+d) \ 
\text{in} \ 
\mathcal{C}_n \ 
\text{for} \ 
u \in [i-d]),
\end{align} 
and one of the following conditions holds: 
\begin{enumerate}[(i)]
\item
$d$ is even.
\item
$d$
is odd. 
If 
$d \in [i-1]$, 
then 
$\mathsf{T}(i-d) \preceq n$. 
If 
$d \in [i]$, 
$\mathsf{T}(i-d+1) \succeq \overline{n}$, 
$\{ a_1 < a_2 < \cdots < a_{n-i+d} \}
=
[n] \setminus \{ \mathsf{T}(u) \mid u \in [i-d] \}$, 
and 
$\mathsf{T}(i-d+1) = \overline{a_d}$, 
then 
$1 \prec a_d \prec \mathsf{T}'(a_d)$; 
note that 
$a_d \in [d,i]$
since 
$a_u \in [u,u+i-d]$
for 
$u \in [n-i+d]$. 
\end{enumerate}
\item
Assume that 
$i = n$.
Set
$(\mathsf{T},c) \preceq (\mathsf{T}',c')$
if 
\begin{align}
&
(d \geq 0)
\ \text{and} \ 
\left(
\mathsf{T}(u) \preceq \mathsf{T}'(u+2d)
\ \text{in} \ 
\mathcal{C}_n
\ \text{for} \ 
u \in [n-2d]
\right).
\end{align}
\end{enumerate}
\end{define}

\begin{prop} \label{prop:tab-cri-B}
Let 
$i \in I$. 
\begin{enumerate}[(1)]
\item
$\mathcal{Y}_i^{B_n} \circ \Pi^{I \setminus \{ i \}} = \mathcal{Y}_i^{B_n}$.
\item
For 
$x,y \in W_{\af}$, 
we have 
$\Pi^{I \setminus \{ i \}}(x) \preceq \Pi^{I \setminus \{ i \}}(y)$
in 
$(W^{I \setminus \{ i \}})_{\af}$
if and only if 
$\mathcal{Y}_i^{B_n}(x) \preceq \mathcal{Y}_i^{B_n}(y)$
in
$\mathrm{CST}_{B_n}(\varpi_i) \times \BZ$. 
\item
Let 
$i \in [n-1]$
and 
$(\mathsf{T},c) , (\mathsf{T}',c')
\in 
\mathrm{CST}_{B_n}(\varpi_i) \times \BZ$. 
If 
$c' - c > i$, 
then
$(\mathsf{T},c) \preceq (\mathsf{T}',c')$.
\item
Let 
$i = n$
and 
$(\mathsf{T},c) , (\mathsf{T}',c')
\in 
\mathrm{CST}_{B_n}(\varpi_n) \times \BZ$. 
If 
$2(c' - c) \geq i$, 
then
$(\mathsf{T},c) \preceq (\mathsf{T}',c')$.
\end{enumerate}
\end{prop}

By combining Propositions \ref{prop:Deo} and \ref{prop:tab-cri-B} (2), 
we obtain the following tableau criterion for the 
semi-infinite Bruhat order
on $W_{\af}$ of type $B_n^{(1)}$.

\begin{thm} \label{thm:tab-cri-B}
Let 
$J \subset I$. 
For 
$x,y \in (W^J)_{\af}$, 
we have 
$x \preceq y$
in 
$(W^J)_{\af}$
if and only if 
$\mathcal{Y}_i^{B_n}(x) \preceq \mathcal{Y}_i^{B_n}(y)$
in 
$\mathrm{CST}_{B_n}(\varpi_i) \times \BZ$
for all 
$i \in I \setminus J$. 
\end{thm}

The remainder of this subsection is devoted to 
the proof of Proposition \ref{prop:tab-cri-B}. 

The proofs of 
Lemmas \ref{lem:J-ad-B}--\ref{lem:<g,r>-B} below 
are straightforward.

\begin{lem} \label{lem:J-ad-B}
Let 
$i \in I$
and 
$\gamma \in \Delta^+ \setminus \Delta_{I \setminus \{ i \}}^+$. 
We have 
$\gamma^{\vee} \in Q^{\vee,I \setminus \{ i \}}$
if and only if one of the following conditions holds:
\begin{enumerate}[(1)]
\item
$i = 1$
and 
$\gamma^{\vee} 
= 
\varepsilon_1 - \varepsilon_2 
= 
\alpha_1^{\vee}$.
\item
$i \in [2,n-1]$
and 
$\gamma^{\vee}
=
\varepsilon_i - \varepsilon_{i+1} 
=
\alpha_i^{\vee}$.
\item
$i \in [2,n-1]$
and
$\gamma^{\vee}
=
\varepsilon_{i-1} + \varepsilon_i 
= 
\alpha_{i-1}^{\vee} + 2\alpha_i^{\vee} + \cdots + 2\alpha_{n-1}^{\vee} + \alpha_n^{\vee}$.
\item
$i = n$
and
$\gamma^{\vee}
=
\varepsilon_{n-1} + \varepsilon_n
= 
\alpha_{n-1}^{\vee} + \alpha_n^{\vee}$.
\end{enumerate}
\end{lem}

\begin{lem} \label{lem:<g,r>-B}
Let $i \in I$. 
We have
\begin{align}
2\langle \alpha_i^{\vee} , \rho - \rho_{I \setminus \{ i \}} \rangle 
= 
\begin{cases}
2n-i & \text{if} \ 
i \in [n-1], \\
2n & \text{if} \ 
i = n.
\end{cases}
\end{align}
\end{lem}

\begin{prop}[cf. {\cite[\S 8.1]{BB}}] \label{prop:Bruhat-B}
Let 
$i \in I$, 
$w \in W^{I \setminus \{ i \}}$, 
and 
$\gamma \in \Delta^+$. 
There exists a Bruhat edge 
$w \xrightarrow{\ \gamma \ } \lfloor wr_{\gamma} \rfloor = wr_{\gamma}$
in 
$\mathrm{QB}^{I \setminus \{ i \}}$
if and only if 
$\gamma \in \Delta^+ \setminus \Delta_{I \setminus \{ i \}}^+$
and one of the following statements holds.
\begin{enumerate}[(b-B1)]
\item
$i \in [n-1]$, 
$c_i(\gamma^{\vee}) = 1$, 
and there exists 
$s \in [i]$
such that 
$wr_{\gamma}(u) = w(u)$
for 
$u \in [i] \setminus \{ s \}$, 
$1 \preceq w(s) \prec n$, 
and 
$wr_{\gamma}(s) = 
\min([w(s)+1,n] \setminus \{ \|w(u)\| \mid u \in [i], \ w(u) \succeq \overline{n} \})$;
in this case, 
we have 
$\gamma^{\vee} 
= 
\varepsilon_s - \varepsilon_t
=
\alpha_s^{\vee} + \alpha_{s+1}^{\vee} + \cdots + \alpha_t^{\vee}$
for some 
$t \in [i+1,n]$.
\item
$i \in [n-1]$, 
$c_i(\gamma^{\vee}) = 1$, 
and there exists 
$s \in [i]$
such that 
$wr_{\gamma}(u) = w(u)$
for 
$u \in [i] \setminus \{ s \}$, 
$\overline{n} \preceq w(s) \prec \overline{1}$, 
and 
$wr_{\gamma}(s)
=
\sigma 
\bigl(
\max ([1,\|w(s)\|-1] \setminus \{ w(u) \mid u \in [i], \ w(u) \preceq n \})
\bigr)$;
in this case, 
we have 
$\gamma^{\vee} 
= 
\varepsilon_s - \varepsilon_t
=
\alpha_s^{\vee} + \alpha_{s+1}^{\vee} + \cdots + \alpha_t^{\vee}$
for some 
$t \in [i+1,n]$.
\item
$i \in [2,n-1]$,
$c_i(\gamma^{\vee}) = 2$, 
and there exist 
$s,t \in [i]$
such that 
$s < t$, 
$wr_{\gamma}(u) = w(u)$
for 
$[i] \setminus \{ s,t \}$
and 
$wr_{\gamma}(s) = w(s) + 1 = \sigma(w(t)) = \sigma(wr_{\gamma}(t)) + 1 \preceq n$;
in this case, 
we have  
$\gamma^{\vee} 
= 
\varepsilon_s + \varepsilon_t
=
\alpha_s^{\vee} + \cdots + \alpha_{t-1}^{\vee}
+ 2\alpha_t^{\vee} + \cdots + 2\alpha_{n-1}^{\vee} + \alpha_n^{\vee}$.
\item
$i = n$, 
$c_n(\gamma^{\vee}) = 1$, 
and there exist 
$s,t \in [i]$
such that 
$s < t$, 
$wr_{\gamma}(u) = w(u)$
for 
$[i] \setminus \{ s,t \}$
and 
$wr_{\gamma}(s) = w(s) + 1 = \sigma(w(t)) = \sigma(wr_{\gamma}(t)) + 1 \preceq n$;
in this case, 
we have  
$\gamma^{\vee} 
= 
\varepsilon_s + \varepsilon_t
=
\alpha_s^{\vee} + \cdots + \alpha_{t-1}^{\vee}
+ 2\alpha_t^{\vee} + \cdots + 2\alpha_{n-1}^{\vee} + \alpha_n^{\vee}$.
\item
$i \in [n-1]$, 
$c_i(\gamma^{\vee}) = 2$, 
and there exists 
$s \in [i]$
such that 
$wr_{\gamma}(u) = w(u)$
for 
$[i] \setminus \{ s \}$
and 
$\sigma(wr_{\gamma}(s)) = w(s) = n$;
in this case, 
we have 
$\gamma^{\vee} 
= 
\varepsilon_s
=
2\alpha_s^{\vee} + 2\alpha_{s+1}^{\vee} + \cdots + 2\alpha_{n-1}^{\vee} + \alpha_n^{\vee}$.
\item
$i = n$, 
$c_n(\gamma^{\vee}) = 1$, 
and there exists 
$s \in [i]$
such that 
$wr_{\gamma}(u) = w(u)$
for 
$[i] \setminus \{ s \}$
and 
$\sigma(wr_{\gamma}(s)) = w(s) = n$;
in this case, 
we have 
$\gamma^{\vee} 
= 
\varepsilon_s
=
2\alpha_s^{\vee} + 2\alpha_{s+1}^{\vee} + \cdots + 2\alpha_{n-1}^{\vee} + \alpha_n^{\vee}$.
\end{enumerate}
Moreover, for 
$w,v \in W^{I \setminus \{ i \}}$, 
we have 
$w \preceq v$
if and only if 
$w(u) \preceq v(u)$
for 
$u \in [i]$. 
\end{prop}

For 
$i \in I$, 
$w \in W^{I \setminus \{ i \}}$
and 
$\gamma \in \Delta^+ \setminus \Delta_{I \setminus \{ i \}}^+$, 
let 
$\mathrm{Q}(i,w,\gamma)$
denote the following statement.
\begin{description}
\item[$\mathrm{Q}(i,w,\gamma):$]
There exists a quantum edge 
$w \xrightarrow{\ \gamma \ } \lfloor wr_{\gamma} \rfloor$
in 
$\mathrm{QB}^{I \setminus \{ i \}}$.
\end{description}

\begin{prop} \label{prop:Q=B}
Let 
$i \in I$, 
$w \in W^{I \setminus \{ i \}}$, 
and 
$\gamma \in \Delta^+ \setminus \Delta_{I \setminus \{ i \}}^+$. 
Then 
$\mathrm{Q}(i,w,\gamma)$
is true if and only if one of the following statements holds.
\begin{enumerate}[(q-B1)]
\item
$i = 1$, 
$c_1(\gamma^{\vee}) = 1$, 
and 
$(w(1),\lfloor wr_{\gamma} \rfloor (1))
\in 
\{ (\overline{1},2), (\overline{2},1) \}$;
in this case, 
we have 
$\gamma^{\vee} = \varepsilon_1 - \varepsilon_2 = \alpha_1^{\vee}$.
\item
$i \in [2,n-1]$, 
$c_i(\gamma^{\vee}) = 1$, 
and 
$1 \prec w(\overline{i}) \preceq i+1$. 
If we set 
$k = w(\overline{i})$, 
then 
$w(u) = u+1$
for 
$u \in [1,k-2]$, 
$k \prec w(u) \preceq n$
for 
$u \in [k-1,i-1]$,
$\lfloor wr_{\gamma} \rfloor (1) = 1$, 
and 
$\lfloor wr_{\gamma} \rfloor (u) = w(u-1)$
for 
$u \in [2,i]$;
in this case, 
we have  
$\gamma^{\vee}
=
\varepsilon_i - \varepsilon_{i+1} 
=
\alpha_i^{\vee}$.
\item
$i \in [2,n-1]$, 
$c_i(\gamma^{\vee}) = 2$, 
$\lfloor wr_{\gamma} \rfloor (1) = w(\overline{i}) = 1$, 
$\lfloor wr_{\gamma} \rfloor (2) = w(\overline{i-1}) = 2$, 
and 
$\lfloor wr_{\gamma} \rfloor (u) = w(u-2)$
for 
$u \in [3,i]$;
in this case, 
we have 
$\gamma^{\vee}
=
\varepsilon_{i-1} + \varepsilon_i 
= 
\alpha_{i-1}^{\vee} + 2\alpha_i^{\vee} + \cdots + 2\alpha_{n-1}^{\vee} + \alpha_n^{\vee}$.
\item
$i = n$, 
$c_n(\gamma^{\vee}) = 1$, 
$\lfloor wr_{\gamma} \rfloor (1) = w(\overline{n}) = 1$, 
$\lfloor wr_{\gamma} \rfloor (2) = w(\overline{n-1}) = 2$, 
and 
$\lfloor wr_{\gamma} \rfloor (u) = w(u-2)$
for 
$u \in [3,n]$;
in this case, 
we have 
$\gamma^{\vee}
=
\varepsilon_{n-1} + \varepsilon_n
= 
\alpha_{n-1}^{\vee} + \alpha_n^{\vee}$.
\end{enumerate}
\end{prop}

Before starting the proof of Proposition \ref{prop:Q=B}, 
we mention a consequence of 
Lemma \ref{lem:Q=SiB} and Propositions \ref{prop:Bruhat-B}--\ref{prop:Q=B}.

\begin{prop} \label{prop:SiB-B}
Let 
$i \in I$, 
$x,y \in (W^{I \setminus \{ i \}})_{\af}$, 
$\mathcal{Y}_i^{B_n}(x) = (\mathsf{T},c)$, 
and 
$\mathcal{Y}_i^{B_n}(y) = (\mathsf{T}',c')$. 
There exists an edge 
$x \xrightarrow{\ \beta \ } y$
in 
$\mathrm{SiB}^{I \setminus \{ i \}}$
for some 
$\beta \in \Delta_{\af}^+$
if and only if one of the following conditions holds:
\begin{enumerate}[($\si$-B1)]
\item
$i \in [n-1]$, 
$c' = c$, 
and 
there exists
$s \in [i]$
such that 
$\mathsf{T}'(u) = \mathsf{T}(u)$
for 
$u \in [i] \setminus \{ s \}$, 
$1 \preceq \mathsf{T}(s) \prec n$, 
and 
$\mathsf{T}'(s) = 
\min([\mathsf{T}(s)+1,n] \setminus \{ \|\mathsf{T}(u)\| \mid u \in [i], \ \mathsf{T}(u) \succeq \overline{n} \})$.
\item
$i \in [n-1]$, 
$c' = c$, 
and there exists 
$s \in [i]$
such that 
$\mathsf{T}'(u) = \mathsf{T}(u)$
for 
$u \in [i] \setminus \{ s \}$, 
$\overline{n} \preceq \mathsf{T}(s) \prec \overline{1}$, 
and 
$\mathsf{T}'(s)
=
\sigma 
\bigl(
\max ([1,\|\mathsf{T}(s)\|-1] \setminus \{ \mathsf{T}(u) \mid u \in [i], \ \mathsf{T}(u) \preceq n \})
\bigr)$.
\item
$i \in [2,n]$, 
$c' = c$, 
and there exist 
$s,t \in [i]$
such that 
$s < t$, 
$\mathsf{T}'(u) = \mathsf{T}(u)$
for 
$[i] \setminus \{ s,t \}$, 
and 
$\mathsf{T}'(s) 
= 
\mathsf{T}(s) + 1 
= 
\sigma(\mathsf{T}(t)) 
= 
\sigma(\mathsf{T}'(t)) + 1 
\preceq 
n$.
\item
$c' = c$, 
and there exists 
$s \in [i]$
such that 
$\mathsf{T}'(u) = \mathsf{T}(u)$
for 
$[i] \setminus \{ s \}$
and 
$\sigma(\mathsf{T}'(s)) = \mathsf{T}(s) = n$.
\item
$i = 1$, 
$c' = c + 1$, 
and 
$(\mathsf{T}(1),\mathsf{T}'(1))
\in 
\{ (\overline{1},2), (\overline{2},1) \}$. 
\item
$i \in [2,n-1]$, 
$c' = c + 1$, 
and 
$1 \prec \sigma(\mathsf{T}(i)) \preceq i+1$. 
If we set 
$k = \sigma(\mathsf{T}(i))$, 
then 
$\mathsf{T}(u) = u+1$
for 
$u \in [k-2]$, 
$k \prec \mathsf{T}(u) \preceq n$
for 
$u \in [k-1,i-1]$,
$\mathsf{T}'(1) = 1$, 
and 
$\mathsf{T}'(u) = \mathsf{T}(u-1)$
for 
$u \in [2,i]$.
\item
$i \in [2,n-1]$, 
$c' = c + 2$, 
$\mathsf{T}'(1) = \sigma(\mathsf{T}(i)) = 1$, 
$\mathsf{T}'(2) = \sigma(\mathsf{T}(i-1)) = 2$, 
and 
$\mathsf{T}'(u) = \mathsf{T}(u-2)$
for 
$u \in [3,i]$.
\item
$i = n$, 
$c' = c + 1$, 
$\mathsf{T}'(1) = \sigma(\mathsf{T}(n)) = 1$, 
$\mathsf{T}'(2) = \sigma(\mathsf{T}(n-1)) = 2$, 
and 
$\mathsf{T}'(u) = \mathsf{T}(u-2)$
for 
$u \in [3,n]$.
\end{enumerate}
\end{prop}

We have divided the proof of 
Proposition \ref{prop:Q=B} 
into a sequence of lemmas.

\begin{lem} \label{lem:Q->B1}
$\mathrm{Q}(1,w,\gamma)$
implies 
(q-B1).
\end{lem}

\begin{proof}
Assume that 
$\mathrm{Q}(1,w,\gamma)$
is true. 
By Lemmas \ref{lem:LS} and \ref{lem:J-ad-B}, 
we have 
$\gamma^{\vee} = \varepsilon_1 - \varepsilon_2 \in Q^{\vee,I\setminus \{ 1 \}}$, 
$r_{\gamma} = (1 \ 2)(\overline{1} \ \overline{2})$, 
and 
$\lfloor wr_{\gamma} \rfloor = wr_{\gamma}(z_{\gamma^{\vee}}^{I \setminus \{ 1 \}})^{-1}$. 
Set 
$J = I \setminus \{ 1 \} = [2,n]$; 
note that 
$J$
is of type $B_{n-1}$. 
We see that 
$2 \in J_{\af}$
satisfies the condition for 
$\gamma^{\vee} \in Q^{\vee}$
in Lemma \ref{lem:J-ad}; 
note that 
$J \setminus \{ 2 \} = [3,n]$ 
is of type $B_{n-2}$.
Hence 
$z_{\gamma^{\vee}}^J = w_0^Jw_0^{J \setminus \{ 2 \}} = (2 \ \overline{2})$
and 
$\lfloor wr_{\gamma} \rfloor 
=
w(1 \ 2)(\overline{1} \ \overline{2})(2 \ \overline{2})
=
w(1 \ 2 \ \overline{1} \ \overline{2})$.
From this we obtain
$\lfloor wr_{\gamma} \rfloor (1) = w(2)$, 
$\lfloor wr_{\gamma} \rfloor (2) = w(\overline{1})$, 
and 
$\lfloor wr_{\gamma} \rfloor (u) = w(u)$
for 
$u \in [3,n]$. 
It follows from 
Lemma \ref{lem:Gr-C}
that 
\begin{align} \label{eq:seq-B1}
\max \{ w(\overline{1}) , w(2) \}
\prec
w(3)
\prec
w(4)
\prec
\cdots 
\prec
w(n)
\preceq
n.
\end{align}
Hence 
$\lfloor wr_{\gamma} \rfloor (u) = w(u) = u$
for 
$u \in [3,n]$. 
If 
$w(\overline{1}) \prec w(2)$, 
then
$(w(1),\lfloor wr_{\gamma} \rfloor (1))
=
(\overline{1},2)$.
If 
$w(2) \prec w(\overline{1})$, 
then 
$(w(1),\lfloor wr_{\gamma} \rfloor (1))
=
(\overline{2},1)$.
\end{proof}

\begin{lem} \label{lem:B1->Q}
(q-B1)
implies
$\mathrm{Q}(1,w,\gamma)$.
\end{lem}

\begin{proof}
Assume that 
(q-B1)
is true. 
We see from 
Lemmas \ref{lem:<g,r>} and \ref{lem:J-ad-B}--\ref{lem:<g,r>-B} that 
$\mathrm{Q}(1,w,\gamma)$ 
is equivalent to 
$\ell(\lfloor wr_{\gamma} \rfloor) - \ell(w)
=2-2n$. 
Note that 
(q-B1)
and 
Lemma \ref{lem:Gr-C}
yield \eqref{eq:seq-B1}. 

If 
$(w(1),\lfloor wr_{\gamma} \rfloor (1))
=
(\overline{1},2)$, 
then
$w = (1 \ \overline{1})$
and 
$\lfloor wr_{\gamma} \rfloor = (1 \ 2)(\overline{1} \ \overline{2}) = r_1$, 
by Lemma \ref{lem:Gr-C}. 
Since 
$\ell(\lfloor wr_{\gamma} \rfloor) = 1$
and
$\ell(w) 
= 
\mathsf{a}_1(w) + \mathsf{b}_1(w) + \mathsf{e}_1(w)
=
(n-1) + (n-1) + 1
=
2n-1$, 
we have
$\ell(\lfloor wr_{\gamma} \rfloor) - \ell(w)
=
2-2n$. 

If 
$(w(1),\lfloor wr_{\gamma} \rfloor (1))
=
(\overline{2},1)$, 
then
$w = (2 \ \overline{2} \ \overline{1} \ 2)$
and 
$\lfloor wr_{\gamma} \rfloor = e$, 
by Lemma \ref{lem:Gr-C}. 
Since 
$\ell(\lfloor wr_{\gamma} \rfloor) = 0$
and
$\ell(w) 
= 
\mathsf{a}_1(w) + \mathsf{b}_1(w) + \mathsf{e}_1(w)
=
(n-1) + (n-2) + 1
=
2n-2$, 
we have
$\ell(\lfloor wr_{\gamma} \rfloor) - \ell(w)
=
2-2n$.  
\end{proof}

\begin{lem} \label{lem:Q->B2}
$i \in [2,n-1]$, 
$c_i(\gamma^{\vee}) = 1$, 
and 
$\mathrm{Q}(i,w,\gamma)$
imply
(q-B2).
\end{lem}

\begin{proof}
Assume that 
$i \in [2,n-1]$
and 
$c_i(\gamma^{\vee}) = 1$, 
and that 
$\mathrm{Q}(i,w,\gamma)$
is true. 
By Lemmas \ref{lem:LS} and \ref{lem:J-ad-B}, 
we have
$\gamma^{\vee} = \varepsilon_i - \varepsilon_{i+1} \in Q^{\vee,I\setminus \{ i \}}$, 
$r_{\gamma} = (i \ i+1)(\overline{i} \ \overline{i+1})$, 
and 
$\lfloor wr_{\gamma} \rfloor 
= 
wr_{\gamma}(z_{\gamma^{\vee}}^{I \setminus \{ i \}})^{-1}$. 
We see from Lemmas \ref{lem:<g,r>} and \ref{lem:<g,r>-B} that 
$\mathrm{Q}(i,w,\gamma)$ 
is equivalent to 
$\ell(\lfloor wr_{\gamma} \rfloor) - \ell(w)
=i - 2n + 1$. 
Let 
$I \setminus \{ i \}
=
I_1 \sqcup I_2$, 
where
$I_1 = [i-1]$
is of type 
$A_{i-1}$
and 
$I_2 = [i+1,n]$
is of type 
$B_{n-i}$.
We see that 
$(i-1,i+1) \in (I_1)_{\af} \times (I_2)_{\af}$
satisfies the condition for 
$\gamma^{\vee} \in Q^{\vee}$
in 
Lemma \ref{lem:J-ad}; 
note that 
$I_1 \setminus \{ i-1 \} = [i-2]$
is of type 
$A_{i-2}$
and 
$I_2 \setminus \{ i+1 \} = [i+2,n]$
is of type 
$B_{n-i-1}$. 
Hence 
$z_{\gamma^{\vee}}^{I \setminus \{ i \}} 
=
w_0^{I_1}w_0^{I_1 \setminus \{ i-1 \}}
w_0^{I_2}w_0^{I_2 \setminus \{ i+1 \}}
=
(1 \ 2 \ \cdots \ i)(\overline{1} \ \overline{2} \ \cdots \ \overline{i})
(i+1 \ \overline{i+1})$
and 
$\lfloor wr_{\gamma} \rfloor = w(i \ i+1)(\overline{i} \ \overline{i+1})(i+1 \ \overline{i+1})
(i \ \cdots \ 2 \ 1)(\overline{i} \ \cdots \ \overline{2} \ \overline{1})
=
w(i \ i+1 \ \overline{i} \ \overline{i+1})
(i \ \cdots \ 2 \ 1)(\overline{i} \ \cdots \ \overline{2} \ \overline{1})$. 
We have 
$\lfloor wr_{\gamma} \rfloor (1) = w(i+1)$, 
$\lfloor wr_{\gamma} \rfloor (u) = w(u-1)$
for 
$u \in [2,i]$, 
$\lfloor wr_{\gamma} \rfloor (i+1) = w(\overline{i})$,
and 
$\lfloor wr_{\gamma} \rfloor (u) = w(u)$
for 
$u \in [i+2,n]$.
It follows from Lemma \ref{lem:Gr-C} that 
\begin{align} \label{eq:seq-B3}
\begin{split}
&w(i+1) \prec w(1) \prec w(2) \prec \cdots \prec w(i-1) \prec w(i), 
\\
&\max \{ w(\overline{i}),w(i+1) \} \prec w(i+2) \prec w(i+3) \prec \cdots \prec w(n) \preceq n \prec w(i).
\end{split}
\end{align}
Hence  
$w(\overline{i}) \preceq i+1$.
Set 
$k = w(\overline{i}) = \lfloor wr_{\gamma} \rfloor(i+1)$.
The rest of the proof will be divided into four steps.

\begin{proof}[Step 1]
We claim that 
$k \succ 1$
and 
(1)--(5)
below imply 
(q-B2).
Let 
$l \in [k-1,i]$
be such that 
$w(l-1) \preceq n \prec w(l)$. 
\begin{enumerate}[(1)]
\item
$\mathsf{a}_1(\lfloor wr_{\gamma} \rfloor)
=
\mathsf{b}_1(\lfloor wr_{\gamma} \rfloor)
=
\mathsf{e}_1(\lfloor wr_{\gamma} \rfloor)=0$, 
\item
$\mathsf{a}_s(\lfloor wr_{\gamma} \rfloor)
=
\mathsf{a}_{s-1}(w)-1$
for 
$s \in [2,k-1]$, 
$\mathsf{a}_s(\lfloor wr_{\gamma} \rfloor)
=
\mathsf{a}_{s-1}(w)$
for 
$s \in [k,i]$, 
\item
$\mathsf{b}_s(\lfloor wr_{\gamma} \rfloor)
=
\mathsf{b}_{s-1}(w)$
for 
$s \in [2,k-1] \cup [l+1,i]$,
$\mathsf{b}_s(\lfloor wr_{\gamma} \rfloor)
=
\mathsf{b}_{s-1}(w)-1$
for 
$s \in [k,l]$, 
\item
$\mathsf{e}_s(\lfloor wr_{\gamma} \rfloor)
=
\mathsf{e}_{s-1}(w)$
for 
$s \in [2,i]$, 
\item
$\mathsf{a}_i(w)
=
n-i$,
$\mathsf{b}_i(w)= n-i-1$, 
$\mathsf{e}_i(w)=1$.
\end{enumerate}
Indeed, 
if 
$k \succ 1$, 
then 
$\lfloor wr_{\gamma} \rfloor(u) = u$
for 
$u \in [k-1]$, 
by 
\eqref{eq:seq-B3}, 
Lemma \ref{lem:Gr-C}, 
and 
$\lfloor wr_{\gamma} \rfloor(i+1) = k$.
Therefore 
$w(u) = \lfloor wr_{\gamma} \rfloor(u+1) = u+1$
for 
$u \in [k-2]$. 
It follows from 
$w(1) = 2$
and 
\eqref{eq:seq-B3} that 
$\lfloor wr_{\gamma} \rfloor(1)
=
w(i+1)
=
1$. 
Also, 
by
\eqref{eq:seq-B3},  
we have 
$k-1 = w(k-2) \prec w(u) \prec w(i) = \overline{k}$
for 
$u \in [k-1,i-1]$, 
which implies 
$\|w(u)\| > k$
for 
$u \in [k-1,i-1]$. 
It remains to prove that 
$w(u) \preceq n$
for 
$u \in [k-1,i-1]$; 
we only need to show that 
$l = i$. 
It follows from 
Lemma \ref{lem:Gr-C} and (1)--(5) above
that
$\ell(\lfloor wr_{\gamma} \rfloor) - \ell(w)
=
i-2n+1+(i-l)$.
Since 
$\ell(\lfloor wr_{\gamma} \rfloor) - \ell(w)
=
1-2n+1$, 
we get 
$l = i$. 
\end{proof}

\begin{proof}[Step 2]
We prove (1)--(5) in Step 1 under the assumption that $k \succ 1$. 

(1) follows from
$\lfloor wr_{\gamma} \rfloor (1) = 1$. 

(2): 
If
$s \in [2,k-1]$, 
then 
$\mathsf{A}_s(\lfloor wr_{\gamma} \rfloor) = \emptyset$
and 
$\mathsf{A}_{s-1}(w) = \{ i+1 \}$, 
which implies 
$\mathsf{a}_s(\lfloor wr_{\gamma} \rfloor)
=
\mathsf{a}_{s-1}(w)-1$.
If 
$s \in [k,i]$, 
then
$\mathsf{A}_s(\lfloor wr_{\gamma} \rfloor)
=
\mathsf{A}_s(w)$, 
which implies 
$\mathsf{a}_s(\lfloor wr_{\gamma} \rfloor)
=
\mathsf{a}_{s-1}(w)$. 

(3): 
If 
$s \in [2,k-1]$, 
then 
$\mathsf{B}_s(\lfloor wr_{\gamma} \rfloor)
=
\mathsf{B}_{s-1}(w)
=
\emptyset$, 
which implies 
$\mathsf{b}_s(\lfloor wr_{\gamma} \rfloor)
= 
\mathsf{b}_{s-1}(w)$.
If 
$s \in [k,l]$, 
then 
$k \prec \lfloor wr_{\gamma} \rfloor (s) = w(s-1) \preceq n$, 
$i \in \mathsf{B}_{s-1}(w)$, 
and the map 
$\mathsf{B}_s(\lfloor wr_{\gamma} \rfloor)
\to 
\mathsf{B}_{s-1}(w) \setminus \{ i \}$, 
$t \mapsto t-1$, 
is bijective.
This implies 
$\mathsf{b}_s(\lfloor wr_{\gamma} \rfloor)
=
\mathsf{b}_{s-1}(w)-1$. 
If 
$s \in [l+1,i]$, 
then 
$\overline{n} \preceq \lfloor wr_{\gamma} \rfloor (s) = w(s-1)$
and 
$\mathsf{B}_s(\lfloor wr_{\gamma} \rfloor)
= 
\mathsf{B}_{s-1}(w)$, 
which implies 
$\mathsf{b}_s(\lfloor wr_{\gamma} \rfloor)
= 
\mathsf{b}_{s-1}(w)$. 

(4) follows from 
$\lfloor wr_{\gamma} \rfloor (s) = w(s-1)$
for 
$s \in [2,i]$. 

(5): 
Lemma \ref{lem:Gr-C}
and 
$w(i) = \overline{k} \succeq \overline{n}$
show that
$\mathsf{a}_i(w)
=
n-i$
and 
$\mathsf{e}_i(w)=1$. 
We see from 
$\|w(s)\| > k = \|w(i)\|$
for 
$s \in [k-1,i-1]$
that 
$\mathsf{b}_i(w)
= 
n-i-1$. 
\end{proof}

\begin{proof}[Step 3]
It remains to prove that 
$k \succ 1$.
On the contrary, 
suppose that 
$k = 1$. 
Then 
$w(\overline{i}) = k \prec w(i+1) \preceq n$. 
If we prove that
\begin{enumerate}[(1)]
\setcounter{enumi}{5}
\item
$\mathsf{e}_1(\lfloor wr_{\gamma} \rfloor) = 0$, 
$\mathsf{e}_s(\lfloor wr_{\gamma} \rfloor) 
=
\mathsf{e}_{s-1}(w)$
for 
$s \in [2,i]$, 
and 
$\mathsf{e}_i(w) = 1$, 
\item
$\mathsf{a}_i(w) = n-i$, 
\item
$\mathsf{a}_1(\lfloor wr_{\gamma} \rfloor) = 1$, 
$\mathsf{b}_1(\lfloor wr_{\gamma} \rfloor) = w(i+1) - 2$, 
\item
$\mathsf{a}_s(\lfloor wr_{\gamma} \rfloor)
=
\mathsf{a}_{s-1}(w)$
for 
$s \in [2,i]$, 
\item 
$\mathsf{b}_s(\lfloor wr_{\gamma} \rfloor)
=
\mathsf{b}_{s-1}(w)-1$  
for 
$s \in [2,i]$, 
\item
$\mathsf{b}_i(w) = n-i$, 
\end{enumerate}
then 
$\ell(\lfloor wr_{\gamma} \rfloor) - \ell(w)
=
i-2n+1
+
(w(i+1)+2i-4)$. 
Since 
$\ell(\lfloor wr_{\gamma} \rfloor) - \ell(w)
=
i-2n+1$
and 
$i \geq 2$, 
we have 
$w(i+1) = 4-2i \leq 0$, 
a contradiction. 
\end{proof}

\begin{proof}[Step 4]
We prove (6)--(11) in Step 3 under the assumption that $k=1$; 
in fact, 
we do not use 
$k = 1$
to prove 
(6)--(7).

(6) follows from 
\eqref{eq:seq-B3}
and 
$\lfloor wr_{\gamma} \rfloor (u) = w(u-1)$
for 
$u \in [2,i]$. 

(7): 
We deduce from 
\eqref{eq:seq-B3}
that 
$\mathsf{A}_i(w) = [i+1,n]$, 
which gives
$\mathsf{a}_i(w) = n-i$.

(8): 
By Lemma \ref{lem:Gr-C}, 
$\mathsf{A}_1(\lfloor wr_{\gamma} \rfloor)
\subset 
[i+1,n]$. 
Since 
$\lfloor wr_{\gamma} \rfloor(1)
=
w(i+1)
\succ 
w(\overline{i})
=
\lfloor wr_{\gamma} \rfloor(i+1)$, 
we have 
$i+1 \in \mathsf{A}_1(\lfloor wr_{\gamma} \rfloor)$. 
Since 
$\lfloor wr_{\gamma} \rfloor(t) = w(t)$
for 
$t \in [i+2,n]$, 
we have 
$\mathsf{A}_1(\lfloor wr_{\gamma} \rfloor) \cap [i+2,n] = \emptyset$.
Hence 
$\mathsf{A}_1(\lfloor wr_{\gamma} \rfloor)
=
\{ i+1 \}$
and 
$\mathsf{a}_1(\lfloor wr_{\gamma} \rfloor)=1$
as claimed.
Since 
$\lfloor wr_{\gamma} \rfloor(1) 
=
w(i+1)
\prec
\overline{1}
=
w(i)
=
\sigma(\lfloor wr_{\gamma} \rfloor(i+1))$, 
we have 
$i+1 
\notin 
\mathsf{B}_1(\lfloor wr_{\gamma} \rfloor)$. 
Therefore
\begin{align*}
\mathsf{B}_1(\lfloor wr_{\gamma} \rfloor)
=
& \ 
\{ t \in [2,i] \mid 
\underbrace{\lfloor wr_{\gamma} \rfloor(1)}_{=w(i+1)}
\succ
\underbrace{\sigma(\lfloor wr_{\gamma} \rfloor(t))}_{=\sigma(w(t-1))}
\} \\
&\sqcup 
\{ t \in [i+2,n] \mid 
\underbrace{\lfloor wr_{\gamma} \rfloor(1)}_{=w(i+1)}
\succ
\underbrace{\sigma(\lfloor wr_{\gamma} \rfloor(t))}_{=\sigma(w(t))}
\} .
\end{align*}
It follows that the map 
\begin{align*}
[2,n]
\to
[n], \ 
t \mapsto 
\begin{cases}
t-1 & \text{if} \ t \in [2,i], \\
t & \text{if} \ t \in [i+1,n],
\end{cases}
\end{align*}
induces a bijection from 
$\mathsf{B}_1(\lfloor wr_{\gamma} \rfloor)$
to
$\{ t \in [n] \mid w(i+1) \succ \sigma(w(t)) \} \setminus \{ i \}$.
Since 
$w(i+1) = \min \{ w(u) \mid u \in [n] \}$, 
the latter set equals 
\begin{align*}
\{ w^{-1}(\overline{2}),
w^{-1}(\overline{3}), 
\ldots , 
w^{-1}(\overline{w(i+1)-1}) \} .
\end{align*} 
This proves 
$\mathsf{b}_1(\lfloor wr_{\gamma} \rfloor) = w(i+1)-2$. 

(9): 
Let 
$s \in [2,i]$. 
Since 
$\lfloor wr_{\gamma} \rfloor(i+1)
=
1$
and 
$w(i+1) = \min \{ w(u) \mid u \in [n] \}$, 
we have 
$i+1 \in \mathsf{A}_s(\lfloor wr_{\gamma} \rfloor)$
and 
$i+1 \in \mathsf{A}_{s-1}(w)$. 
It follows from Lemma \ref{lem:Gr-C} that 
\begin{align*}
\mathsf{A}_s(\lfloor wr_{\gamma} \rfloor)
=
\{ i+1 \} 
\sqcup 
\{ t \in [i+2,n] 
\mid 
\underbrace{\lfloor wr_{\gamma} \rfloor(s)}_{=w(s-1)}
\succ 
\underbrace{\sigma(\lfloor wr_{\gamma} \rfloor(t))}_{=w(t)}\}
=
\mathsf{A}_{s-1}(w),
\end{align*}
which implies 
$\mathsf{a}_s(\lfloor wr_{\gamma} \rfloor)
=
\mathsf{a}_{s-1}(w)$.

(10): 
Let 
$s \in [2,i]$. 
Since 
$\lfloor wr_{\gamma} \rfloor(i+1)
=
w(\overline{i})
=
\overline{1}$, 
we have
$\sigma(\lfloor wr_{\gamma} \rfloor(i+1))
\succ
\lfloor wr_{\gamma} \rfloor(s)$. 
Therefore
$i+1 \notin \mathsf{B}_s(\lfloor wr_{\gamma} \rfloor)$
and 
\begin{align*}
\mathsf{B}_s(\lfloor wr_{\gamma} \rfloor)
&=
\left(
\mathsf{B}_s(\lfloor wr_{\gamma} \rfloor)
\cap 
[s+1,i]
\right)
\sqcup 
\left(
\mathsf{B}_s(\lfloor wr_{\gamma} \rfloor)
\cap 
[i+2,n]
\right) \\
&=
\{ t \in [s+1,i] \mid 
\underbrace{\lfloor wr_{\gamma} \rfloor(s)}_{=w(s-1)} 
\prec 
\underbrace{\sigma(\lfloor wr_{\gamma} \rfloor(t))}_{=\sigma(w(t-1))} \} \\
&\hspace{5mm}
\sqcup 
\{ t \in [i+2,n] \mid 
\underbrace{\lfloor wr_{\gamma} \rfloor(s)}_{=w(s-1)}
\prec 
\underbrace{\sigma(\lfloor wr_{\gamma} \rfloor(t))}_{=\sigma(w(t))} \} .
\end{align*}
Also, 
$w(\overline{1}) = 1$
implies 
$i \in \mathsf{B}_{s-1}(w)$. 
It follows that the map
\begin{align*}
[s+1,n] \to [s,n], \ 
t 
\mapsto 
\begin{cases}
t-1 & \text{if} \ t \in [s+1,i], \\
t & \text{if} \ t \in [i+1,n], 
\end{cases}
\end{align*}
induces a bijection from 
$\mathsf{B}_s(\lfloor wr_{\gamma} \rfloor)$
to 
$\mathsf{B}_{s-1}(w) \setminus \{ i \}$, 
which implies 
$\mathsf{b}_s(\lfloor wr_{\gamma} \rfloor)
=
\mathsf{b}_{s-1}(w)-1$.

(11): 
Since 
$w(i) = \overline{1}$, 
we have
$\mathsf{B}_i(w)
=
\{ t \in [i+1,n] \mid w(i) \succ \sigma(w(t)) \}
=
[i+1,n]$.
This proves 
$\mathsf{b}_i(w) = n-i$
as claimed. 
\end{proof}

The proof of Lemma \ref{lem:Q->B2} is complete.
\end{proof}

\begin{lem} \label{lem:B2->Q}
(q-B2)
implies 
$\mathrm{Q}(i,w,\gamma)$. 
\end{lem}

\begin{proof}
Assume that 
(q-B2)
is true; 
note that 
$c_i(\gamma^{\vee}) = 1$
and 
\eqref{eq:seq-B3}
in the proof of Lemma \ref{lem:Q->B2}
holds. 
We see from Lemmas \ref{lem:<g,r>} and \ref{lem:<g,r>-B} that 
$\mathrm{Q}(i,w,\gamma)$ 
is equivalent to 
$\ell(\lfloor wr_{\gamma} \rfloor) - \ell(w)
=i - 2n + 1$. 
If we prove that 
\begin{enumerate}[(1)]
\item
$\mathsf{e}_1(\lfloor wr_{\gamma} \rfloor) = 0$, 
$\mathsf{e}_s(\lfloor wr_{\gamma} \rfloor) 
=
\mathsf{e}_{s-1}(w)$
for 
$s \in [2,i]$, 
and 
$\mathsf{e}_i(w) = 1$, 
\item
$\mathsf{a}_i(w) = n-i$, 
\item
$\mathsf{a}_1(\lfloor wr_{\gamma} \rfloor) = 
\mathsf{b}_1(\lfloor wr_{\gamma} \rfloor) = 0$, 
\item 
$\mathsf{b}_s(\lfloor wr_{\gamma} \rfloor)
=
\begin{cases}
\mathsf{b}_{s-1}(w) & \text{if} \ w(s-1) \prec w(\overline{i}), \\
\mathsf{b}_{s-1}(w)-1 & \text{if} \ w(s-1) \succ w(\overline{i})
\end{cases}$
\ \ 
for 
$s \in [2,i]$, 
\item 
$\mathsf{a}_s(\lfloor wr_{\gamma} \rfloor)
=
\begin{cases}
\mathsf{a}_{s-1}(w)-1 & \text{if} \ w(s-1) \prec w(\overline{i}), \\
\mathsf{a}_{s-1}(w) & \text{if} \ w(s-1) \succ w(\overline{i})
\end{cases}$
\ \ 
for 
$s \in [2,i]$, 
\item
$\mathsf{b}_i(w) = n-i+1$, 
\end{enumerate}
then
$\ell(\lfloor wr_{\gamma} \rfloor) - \ell(w)
=
i - 2n + 1$
by 
Lemma \ref{lem:Gr-C}, 
which is our assertion. 
We prove (1)--(6) as follows. 

(1)--(2) follow by the same method as in 
Step 4 of the proof of Lemma \ref{lem:Q->B2}.

(3): 
Since 
$\lfloor wr_{\gamma} \rfloor(1) = 1$, 
we have 
$\mathsf{A}_s(\lfloor wr_{\gamma} \rfloor)
=
\mathsf{B}_s(\lfloor wr_{\gamma} \rfloor)
=
\emptyset$
and 
$\mathsf{a}_s(\lfloor wr_{\gamma} \rfloor)
=
\mathsf{b}_s(\lfloor wr_{\gamma} \rfloor)
=
0$. 

(4): 
Let 
$s \in [2,i]$. 
We first claim that 
$i+1 \notin \mathsf{B}_s(\lfloor wr_{\gamma} \rfloor)$. 
Indeed, we see from 
(q-B2)
that 
$[k-1] \subset \{ \lfloor wr_{\gamma} \rfloor(u) \mid u \in [i] \}$
and 
$k \notin \{ \lfloor wr_{\gamma} \rfloor(u) \mid u \in [i] \}$. 
By Lemma \ref{lem:Gr-C}, we have 
\begin{align*}
\lfloor wr_{\gamma} \rfloor(i+1)
=
\min \left( [n] \setminus \{ \lfloor wr_{\gamma} \rfloor(u) \mid u \in [i] \} \right)
=
k
=
w(\overline{i}).
\end{align*}
Since 
$\lfloor wr_{\gamma} \rfloor(s) = w(s-1) \preceq n \prec \overline{k} = w(i) = \sigma(\lfloor wr_{\gamma} \rfloor(i+1))$, 
we have 
$i+1 \notin \mathsf{B}_s(\lfloor wr_{\gamma} \rfloor)$
as claimed. 
It follows that 
\begin{align*}
\mathsf{B}_s(\lfloor wr_{\gamma} \rfloor)
=& 
\{ t \in [s+1,i] \mid 
\underbrace{\lfloor wr_{\gamma} \rfloor(s)}_{=w(s-1)}
\succ 
\underbrace{\sigma(\lfloor wr_{\gamma} \rfloor(t)}_{=\sigma(w(t-1))} \} \\
&
\sqcup
\{ t \in [i+2,n] \mid 
\underbrace{\lfloor wr_{\gamma} \rfloor(s)}_{=w(s-1)}
\succ 
\underbrace{\lfloor wr_{\gamma} \rfloor(t)}_{=\sigma(w(t))} \} .
\end{align*}
We next claim that 
$i+1 \notin \mathsf{B}_{s-1}(w)$.
We see from Lemma \ref{lem:Gr-C} that 
$w(i+1) = 1$. 
Hence 
$\sigma(w(i+1)) = \overline{1} \succ w(s-1)$, 
which implies 
$i+1 \notin \mathsf{B}_{s-1}(w)$
as claimed.
Note that 
$i \in \mathsf{B}_{s-1}(w)$
if and only if 
$w(s-1) \succ w(\overline{i})$. 
It follows that the map
\begin{align*}
[s+1,n] \to [s,n], \ 
t \mapsto 
\begin{cases}
t-1 & \text{if} \ t \in [s+1,i], \\
t & \text{if} \ t \in [i+1,n],
\end{cases}
\end{align*}
induces a bijection 
from 
$\mathsf{B}_s(\lfloor wr_{\gamma} \rfloor)$
to
$\mathsf{B}_{s-1}(w)$
(resp. 
from 
$\mathsf{B}_s(\lfloor wr_{\gamma} \rfloor)$
to 
$\mathsf{B}_{s-1}(w) \setminus \{ i \}$)
if 
$w(s-1) \prec w(\overline{i})$
(resp. 
if
$w(s-1) \succ w(\overline{i})$). 
This proves (4). 

(5): 
Let $s \in [2,i]$. 
Since 
$\lfloor wr_{\gamma} \rfloor(i+1) = w(\overline{i})$, 
we have 
$i+1 \in \mathsf{A}_s(\lfloor wr_{\gamma} \rfloor)$
if and only if 
$w(s-1) \succ w(\overline{i})$. 
It follows from Lemma \ref{lem:Gr-C} and 
$\lfloor wr_{\gamma} \rfloor(t) = w(t)$
for 
$t \in [i+2,n]$
that 
$\mathsf{A}_s(\lfloor wr_{\gamma} \rfloor)
=
\mathsf{A}_{s-1}(w) \setminus \{ i-1 \}$
(resp. 
$\mathsf{A}_s(\lfloor wr_{\gamma} \rfloor)
=
\mathsf{A}_{s-1}(w)$)
if 
$w(s-1) \prec w(\overline{i})$
(resp. 
if
$w(s-1) \succ w(\overline{i})$). 
This prove (5). 

(6): 
By Lemma \ref{lem:Gr-C} and \eqref{eq:seq-B3}, 
we have 
$w(i+1) = 1$. 
Hence 
$w(i) \prec \overline{1} = \sigma(w(i+1))$
and 
$i+1 \notin \mathsf{B}_s(w)$.
Also, we see from \eqref{eq:seq-B3} that 
$w(i) \succ \sigma(w(t))$
for 
$t \in [i+2,n]$, 
and consequently
$\mathsf{B}_s(w) = [i+2,n]$. 
This proves (6). 
\end{proof}

\begin{lem} \label{lem:Q->B3}
$i \in [2,n-1]$, 
$c_i(\gamma^{\vee}) = 2$, 
and 
$\mathrm{Q}(i,w,\gamma)$
imply
(q-B3).
\end{lem}

\begin{proof}
Assume that 
$i \in [2,n-1]$
and 
$c_i(\gamma^{\vee}) = 2$, 
and that 
$\mathrm{Q}(i,w,\gamma)$
is true.
We see from Lemmas \ref{lem:<g,r>} and \ref{lem:<g,r>-B} that 
$\mathrm{Q}(i,w,\gamma)$ 
is equivalent to 
$\ell(\lfloor wr_{\gamma} \rfloor) - \ell(w)
=2i - 4n + 1$. 
By the same argument as in Step 2
of the proof of Proposition \ref{prop:Q=C} 
in \S \ref{subsection:tab-cri-C}, 
we have 
$\lfloor wr_{\gamma} \rfloor(1) = w(\overline{i})$, 
$\lfloor wr_{\gamma} \rfloor(2) = w(\overline{i-1})$, 
$\lfloor wr_{\gamma} \rfloor(u) = w(u-2)$
for 
$u \in [3,i]$, 
$\gamma^{\vee} = \varepsilon_{i-1} + \varepsilon_i$, 
and 
$\ell(\lfloor wr_{\gamma} \rfloor) - \ell(w)
=
2i - 4n + 1 + 
2(w(\overline{i})-1)
+ 
2(w(\overline{i-1})-2)$.
Hence 
$w(\overline{i}) = 1$
and 
$w(\overline{i-1}) = 2$. 
This implies
(q-B3).
\end{proof}

\begin{lem} \label{lem:B3->Q}
(q-B3)
implies
$\mathrm{Q}(i,w,\gamma)$. 
\end{lem}

\begin{proof}
Assume that 
(q-B3)
is true; 
note that 
$c_i(\gamma^{\vee}) = 2$. 
We see from Lemmas \ref{lem:<g,r>} and \ref{lem:<g,r>-B} that 
$\mathrm{Q}(i,w,\gamma)$ 
is equivalent to 
$\ell(\lfloor wr_{\gamma} \rfloor) - \ell(w)
=2i - 4n + 1$. 
By the same argument as in Step 2
of the proof of Proposition \ref{prop:Q=C}
in \S \ref{subsection:tab-cri-C}, 
we have 
$\ell(\lfloor wr_{\gamma} \rfloor) - \ell(w)
=
2i - 4n + 1 + 
2(w(\overline{i})-1)
+ 
2(w(\overline{i-1})-2)$.
Since 
$w(\overline{i}) = 1$
and 
$w(\overline{i-1}) = 2$
by 
(q-B3), 
we conclude that 
$\ell(\lfloor wr_{\gamma} \rfloor) - \ell(w)
=
2i - 4n + 1$. 
\end{proof}

\begin{lem} \label{lem:Q->B4}
$\mathrm{Q}(n,w,\gamma)$
is equivalent to 
(q-B4).
\end{lem}

\begin{proof}
By Lemmas \ref{lem:LS} and \ref{lem:J-ad-B}, 
we may assume that 
$\gamma^{\vee} = \varepsilon_{n-1} + \varepsilon_n \in Q^{\vee,I\setminus \{ n \}}$
and 
$\lfloor wr_{\gamma} \rfloor 
= 
wr_{\gamma}(z_{\gamma^{\vee}}^{I \setminus \{ n \}})^{-1}$. 
We see from Lemmas \ref{lem:<g,r>} and \ref{lem:<g,r>-B} that 
$\mathrm{Q}(n,w,\gamma)$ 
is equivalent to 
$\ell(\lfloor wr_{\gamma} \rfloor) - \ell(w)
=1-2n$. 

We first show that 
$\mathrm{Q}(n,w,\gamma)$
implies
(q-B4). 
Set 
$J = I \setminus \{ n \}$; 
note that 
$J$
is of type $A_{n-1}$. 
We see that 
$n-2 \in J_{\af}$
satisfies the condition for 
$\gamma^{\vee} \in Q^{\vee}$
in Lemma \ref{lem:J-ad}; 
note that 
$J \setminus \{ n-2 \} 
=
[n-3] \sqcup \{ n-1 \}$
is of type 
$A_{n-3} \times A_1$. 
Hence 
$z_{\gamma^{\vee}}^J
=
w_0^J w_0^{J\setminus \{n-2\}}
=
w_0^J w_0^{[n-3]} w_0^{\{n-1\}}$
is given by 
$u \mapsto u+2$
for 
$u \in [n-2]$, 
$n-1 \mapsto 1$, 
and 
$n \mapsto 2$. 
Then 
$\lfloor wr_{\gamma} \rfloor$
is given by 
$1 \mapsto w(\overline{n})$, 
$2 \mapsto w(\overline{n-1})$, 
and 
$u \mapsto w(u-2)$
for 
$u \in [3,n]$. 
It follows from Lemma \ref{lem:Gr-C} that
\begin{align} \label{eq:seq-B2}
\underbrace{w(\overline{n}) \prec w(\overline{n-1})}_{\preceq n} 
\prec w(1) \prec w(2) \prec \cdots \prec \ w(n-2) \prec 
\underbrace{w(n-1) \prec w(n)}_{\succeq \overline{n}}.
\end{align}
Hence 
$\|w(\overline{n})\| < \|w(\overline{n-1})\| < \|w(u)\|$
for 
$u \in [n-2]$, 
which implies 
$\lfloor wr_{\gamma} \rfloor (1) = w(\overline{n}) = 1$
and 
$\lfloor wr_{\gamma} \rfloor (2) = w(\overline{n-1}) = 2$.
This proves 
(q-B4). 

We next show that 
(q-B4)
implies 
$\mathrm{Q}(n,w,\gamma)$.
We see that 
(q-B4)
yields \eqref{eq:seq-B2}. 
If we prove that
\begin{enumerate}[(1)]
\item
$\mathsf{a}_s(\lfloor wr_{\gamma} \rfloor) = \mathsf{a}_s(w) = 0$
for 
$s \in [n]$, 
$\mathsf{b}_1(\lfloor wr_{\gamma} \rfloor) 
=
\mathsf{b}_2(\lfloor wr_{\gamma} \rfloor)
=
\mathsf{b}_n(w)
=0$,
$\mathsf{b}_{n-1}(w) = 1$, 
\item
$\mathsf{e}_1(\lfloor wr_{\gamma} \rfloor)
=
\mathsf{e}_2(\lfloor wr_{\gamma} \rfloor)
=
0$, 
$\mathsf{e}_{n-1}(w)
=
\mathsf{e}_n(w)
=
1$, 
$\mathsf{e}_s(\lfloor wr_{\gamma} \rfloor)
=
\mathsf{e}_{s-2}(w)$
for 
$s \in [3,n]$, 
\item
$\mathsf{b}_s(\lfloor wr_{\gamma} \rfloor)
=
\mathsf{b}_{s-2}(w) - 2$
for 
$s \in [3,n]$,
\end{enumerate}
then 
$\ell(\lfloor wr_{\gamma} \rfloor) - \ell(w) = 1-2n$, 
which implies 
$\mathrm{Q}(n,w,\gamma)$. 
We prove (1)--(3) as follows. 
(1)--(2) follows from \eqref{eq:seq-B2}. 
We deduce from 
$n-1,n \in \mathsf{B}_{s-2}(w)$
that the map 
\begin{align*}
\mathsf{B}_s(\lfloor wr_{\gamma} \rfloor)
&=
\{ 
t \in [s+1,n] \mid 
\underbrace{\lfloor wr_{\gamma} \rfloor(s)}_{=w(s-2)}
\succ 
\underbrace{\sigma(\lfloor wr_{\gamma} \rfloor(t))}_{=\sigma(w(t-2))}
\} \\
&\to 
\mathsf{B}_{s-2}(w) \setminus \{ n-1,n \}, \ 
t \mapsto t-2
\end{align*}
is bijective, 
which implies 
$\mathsf{b}_s(\lfloor wr_{\gamma} \rfloor) = \mathsf{b}_{s-2}(w) - 2$
for 
$s \in [3,n]$. 
This proves (3). 
\end{proof}

\begin{proof}[Proof of Proposition \ref{prop:tab-cri-B}]
(1) and (3)--(4) follow by the same method as in the proof 
of Proposition \ref{prop:tab-cri-A}. 

We prove (2). 
The assertion for 
$i = 1$
follows immediately from 
Proposition \ref{prop:SiB-B}. 
Also, 
we can prove the assertion for 
$i = n$
by a similar argument to the proof of 
Proposition \ref{prop:tab-cri-C}. 
Assume that
$i \in [2,n-1]$.
Let 
$x,y \in (W^{I \setminus \{ i \}})_{\af}$, 
$\mathcal{Y}_i^{B_n}(x) = (\mathsf{T},c)$, 
and 
$\mathcal{Y}_i^{B_n}(y) = (\mathsf{T}',c')$.
It follows immediately from 
Propositions \ref{prop:Bruhat-B}--\ref{prop:SiB-B} that 
$x \preceq y$
implies 
$c \leq c'$
and 
$\mathsf{T}(u) \preceq \mathsf{T}'(u+c'-c)$
for 
$u \in [i-c'+c]$. 
Hence we may assume that 
$d := c' - c \geq 0$
and 
$\mathsf{T}(u) \preceq \mathsf{T}'(u+d)$
for 
$u \in [i-d]$.
The proof is by induction on $d$. 
The assertion for $d = 0$ follows immediately from 
Propositions \ref{prop:Bruhat-B} and \ref{prop:SiB-B}. 
Assume that $d \geq 1$. 
In what follows, we write 
$e = 
\#\{ u \in [i] \mid \mathsf{T}(u) \succeq \overline{n} \}$
and 
$\{ a_1 < a_2 < \cdots < a_{n-i+e} \}
= [n] 
\setminus 
\{ \mathsf{T}(u) \mid u \in [i] \}$. 
We have 
$\overline{n} \preceq \mathsf{T}(i-u+1) \preceq \overline{a_u}$
for 
$u \in [e]$. 
Note that 
$\mathsf{T}(i-d) \preceq n$
is equivalent to 
$e \leq d$. 
We have divided the proof into seven steps.

\begin{proof}[Step 1]
We prove that if 
$d \geq 2$ is even, 
then 
$x \prec y$. 
Let 
$s \in [i+1]$
be such that 
$\mathsf{T}(s) \succeq s+2$
and 
$\mathsf{T}(u) \preceq u+1$
for 
$u \in [s-1]$. 
Define 
$\mathsf{T}_1, \mathsf{T}_2 \in \mathrm{CST}_{B_n}(\varpi_i)$
by 
\begin{align*}
\mathsf{T}_1(u)
&=
\begin{cases}
u+2 & \text{if} \ u \in [\min\{s-1,i-2\}], \\
\mathsf{T}(u) & \text{if} \ u \in [s,i-2], \\
\overline{2} & \text{if} \ u = i-1, \\
\overline{1} & \text{if} \ u = i, 
\end{cases} \\
\mathsf{T}_2(u)
&=
\begin{cases}
1 & \text{if} \ u=1, \\
2 & \text{if} \ u=2, \\
\mathsf{T}_1(u-2) & \text{if} \ u \in [3,i].
\end{cases}
\end{align*}
Let 
$x_1,x_2 \in (W^{I \setminus \{ i \}})_{\af}$
be such that 
$\mathcal{Y}_i^{B_n}(x_1) = (\mathsf{T}_1,c)$
and 
$\mathcal{Y}_i^{B_n}(x_2) = (\mathsf{T}_2,c+2)$. 
We have 
$\mathsf{T}(u) \preceq \mathsf{T}_1(u)$
for 
$u \in [i]$. 
Hence 
$x \preceq x_1$
by the assertion for $d = 0$. 
We have 
$x_1 \prec x_2$
by Proposition \ref{prop:SiB-B} ($\si$-B7). 
It remains to prove that 
$x_2 \preceq y$. 
By induction hypothesis, 
it suffices to show that 
$\mathsf{T}_2(u) \preceq \mathsf{T}'(u+d-2)$
for 
$u \in [i-d+2]$; 
note that 
$c'-(c+2) = d-2 \geq 0$
is even. 
We have 
$\mathsf{T}_2(1)
=
1
\preceq 
\mathsf{T}'(1+d-2)$
and 
$\mathsf{T}_2(2)
=
2
\preceq 
\mathsf{T}'(2+d-2)$. 
Let 
$u \in [i-d+2] \cap [3,i]$. 
If 
$u-2 \in [\min\{s-1,i-2\}]$, 
then 
$\mathsf{T}_2(u)
=
u
\preceq 
\mathsf{T}'(u)
\prec
\mathsf{T}'(u+d-2)$. 
If 
$u-2 \in [s,i-2]$, 
then 
$\mathsf{T}_2(u)
=
\mathsf{T}(u-2)
\preceq
\mathsf{T}'(u+d-2)$. 
\end{proof}

\begin{proof}[Step 2]
We prove that if 
$x \prec y$, 
$d$ is odd, 
and 
$d \in [i-1]$, 
then 
$\mathsf{T}(i-d) \preceq n$, 
or equivalently, 
$e \leq d$.
Since $d$ is odd, 
we see from Proposition \ref{prop:SiB-B} that 
there exists an edge 
$x_1 \xrightarrow{\ \beta \ } x_2$, 
$x_1,x_2 \in (W^{I \setminus \{ i \}})_{\af}$, 
$\beta \in \Delta_{\af}^+$, 
in 
$\mathrm{SiB}^{I \setminus \{ i \}}$
of type ($\si$-B6) 
such that 
$x \preceq x_1$
and 
$x_2 \preceq y$; 
we may assume that there is no edges 
of type ($\si$-B6) 
in a directed path 
$\mathbf{p}$
from $x$ to $x_1$
in $\mathrm{SiB}^{I \setminus \{ i \}}$. 
Write 
$\mathcal{Y}_i^{B_n}(x_1) = (\mathsf{T}_1,c_1)$
and 
$\mathcal{Y}_i^{B_n}(x_2) = (\mathsf{T}_2,c_2)$; 
note that 
$c_1 - c$
is even.
Set 
$e_1 = \# \{ u \in [i] \mid \mathsf{T}_1(u) \succeq \overline{n} \}$; 
we have 
$e_1 = 1$
by ($\si$-B6). 
It follows from 
Proposition \ref{prop:SiB-B}
that there exist 
$(c_1 - c)/2$ 
edges of type ($\si$-B7) 
in the directed path 
$\mathbf{p}$, 
and hence 
$1 = e_1 \geq e - (c_1 - c)$.
Since 
$c' - c_2 \geq 0$
and 
$c_2 - c_1 = 1$, 
we have 
$d 
= 
c' - c 
= 
(c' - c_2) + (c_2 - c_1) + (c_1 - c) 
\geq 
1 + (c_1 - c) 
\geq 
e$. 
\end{proof}

\begin{proof}[Step 3]
We prove that if 
$x \prec y$, 
$d$ is odd, 
$d \in [i]$, 
and 
$\mathsf{T}(i-d+1) = \overline{a_d}$, 
then 
$1 \prec a_d \prec \mathsf{T}'(a_d)$; 
we see from 
Step 2 and 
$\mathsf{T}(i-d+1) = \overline{a_d}$
that 
$e = d$
and 
$a_d \in [d,i]$. 
We proceed by induction on $d$. 

Assume that 
$d = 1$. 
We first prove that 
$1 \prec a_1$. 
It follows from 
$d = 1$, 
$x \prec y$, 
and 
Proposition \ref{prop:SiB-B}
that there exists an edge
$x_1 \xrightarrow{\ \beta \ } x_2$, 
$x_1,x_2 \in (W^{I \setminus \{ i \}})_{\af}$, 
$\beta \in \Delta_{\af}^+$, 
in 
$\mathrm{SiB}^{I \setminus \{ i \}}$
of type ($\si$-B6) 
such that 
$x \preceq x_1$
and 
$x_2 \preceq y$. 
If we write 
$\mathcal{Y}_i^{B_n}(x_1) = (\mathsf{T}'',c'')$, 
then 
$c''=c$, 
$\mathsf{T}(u) \preceq \mathsf{T}''(u)$
for 
$u \in [i]$, 
and 
$\mathsf{T}''(i) \prec \overline{1}$
(see ($\si$-B6)). 
Hence 
$\overline{a_1}
=
\mathsf{T}(i)
\preceq 
\mathsf{T}''(i)
\prec
\overline{1}$, 
which gives 
$1 \prec a_1$. 

Assume that 
$d = 1$. 
We next prove that 
$a_1 \prec \mathsf{T}'(a_1)$. 
Suppose, contrary to our claim, that 
$\mathsf{T}'(a_1) \preceq a_1$.
By a similar argument above, 
we see from
Proposition \ref{prop:SiB-B} 
(see ($\si$-B6))
that 
there exist
$l \in [2,a_1]$
and 
$\mathsf{T}_1,\mathsf{T}_2 \in \mathrm{CST}_{B_n}(\varpi_i)$
such that 
$\mathsf{T}(u) \preceq \mathsf{T}_1(u)$
for 
$u \in [i]$, 
$\mathsf{T}_1(u) = u+1$
for 
$u \in [l-2]$, 
$\mathsf{T}_1(i) = \overline{l}$, 
$\mathsf{T}_2(1) = 1$, 
$\mathsf{T}_2(u) = \mathsf{T}_1(u-1)$
for 
$u \in [2,i]$, 
and 
$\mathsf{T}_2(u) \preceq \mathsf{T}'(u)$
for 
$u \in [i]$; 
it follows from Proposition \ref{lem:Gr-C} that 
$\mathsf{T}_1(u) \succeq u+2$
for 
$u \in [l-1,i-1]$. 
Since $a_1-1 \in [l-1,i-1]$, 
we have 
$\mathsf{T}_2(a_1) 
= 
\mathsf{T}_1(a_1-1) 
\succeq 
a_1+1
\succ 
a_1
\succeq 
\mathsf{T}'(a_1)$, 
contrary to 
$\mathsf{T}_2(a_1) \preceq \mathsf{T}'(a_1)$. 

Assume that 
$d \geq 3$. 
Since 
$a_d \in [d,i]$, 
we have 
$1 \prec a_d$. 
It remains to prove that 
$a_d \prec \mathsf{T}'(a_d)$.
Define 
$\mathsf{T}_3,\mathsf{T}_4 \in \mathrm{CST}_{B_n}(\varpi_i)$
by 
\begin{align*}
\mathsf{T}_3(u)
&=
\begin{cases}
u+2 & \text{if} \ u \in [a_2-2], \\
\mathsf{T}(u) & \text{if} \ u \in [a_2-1,i-2], \\
\overline{2} & \text{if} \ u = i-1, \\
\overline{1} & \text{if} \ u = i,
\end{cases} \\
\mathsf{T}_4(u)
&=
\begin{cases}
1 & \text{if} \ u = 1, \\
2 & \text{if} \ u = 2, \\
\mathsf{T}_3(u-2) & \text{if} \ u \in [3,i]; 
\end{cases}
\end{align*}
note that if 
$u \in [i-2]$
and 
$\mathsf{T}(u) \succ a_2$, 
then 
$u \in [a_2-1,i-2]$
and 
$\mathsf{T}_3(u) = \mathsf{T}(u)$. 
Let 
$x_4 \in (W^{I \setminus \{ i \}})_{\af}$
be such that 
$\mathcal{Y}_i^{B_n}(x_4) = (\mathsf{T}_4,c+2)$. 
We see that 
$c' - (c+2) = d-2$
is odd, 
$d-2 \in [i-2]$, 
$\mathsf{T}_4(i-(d-2))
\preceq 
n$
(by Step 2), 
and 
$\mathsf{T}_4(i-(d-2)+1)
=
\mathsf{T}_3(i-d+1)
=
\mathsf{T}(i-d+1)
=
\overline{a_d}$; 
note that 
$i-d+1 \in [i-2]$
and 
$[n] \setminus \{ \mathsf{T}_4(u) \mid u \in [i-d+2]\}
=
\{ a_3 < \cdots < a_{n-i+1} \}$. 
Therefore, 
if we prove that 
$x_4 \preceq y$, 
then 
$a_d \prec \mathsf{T}'(a_d)$
by induction hypothesis. 

Let us prove that 
$x_4 \preceq y$. 
It suffices to show that  
$\mathsf{T}'' \in \mathrm{CST}_{B_n}(\varpi_i)$
and 
$(\mathsf{T},c) \prec (\mathsf{T}'',c+2)$
imply
$(\mathsf{T}_4,c+2) \preceq (\mathsf{T}'',c+2)$. 
Indeed, we see from 
Proposition \ref{prop:SiB-B}
that there exists 
$x'' \in (W^{I \setminus \{ i \}})_{\af}$
such that 
$x \prec x'' \preceq y$
and 
$\mathcal{Y}_i^{B_n}(x'')
=
(\mathsf{T}'',c+2)$
for some 
$\mathsf{T}'' \in \mathrm{CST}_{B_n}(\varpi_i)$. 
Since 
$(c+2)-c = 2$
is even, 
$x \prec x''$
is equivalent to 
$(\mathsf{T},c) \prec (\mathsf{T}'',c+2)$
by Step 1. 
Also, 
$x_4 \preceq x''$
is equivalent to 
$(\mathsf{T}_4,c+2) \preceq (\mathsf{T}'',c+2)$. 
Therefore, 
$(\mathsf{T}_4,c+2) \preceq (\mathsf{T}'',c+2)$
(and 
$x'' \preceq y$)
yield
$x_4 \preceq y$. 

Assume that 
$\mathsf{T}'' \in \mathrm{CST}_{B_n}(\varpi_i)$
and 
$(\mathsf{T},c) \prec (\mathsf{T}'',c+2)$. 
We have 
$\mathsf{T}_4(u)
=
u
\preceq 
\mathsf{T}''(u)$
for 
$u \in [a_2]$, 
and 
$\mathsf{T}_4(u)
=
\mathsf{T}(u-2)
\preceq 
\mathsf{T}''(u)$
for 
$u \in [a_2+1,i]$.
Hence 
$(\mathsf{T}_4,c+2) \preceq (\mathsf{T}'',c+2)$. 
\end{proof}

\begin{proof}[Step 4]
We prove that 
$d = 1$
and 
(ii) in Definition \ref{def:SiB-B} (2)
for 
$(\mathsf{T},c)$ 
and 
$(\mathsf{T}',c')$
imply
$x \prec y$. 

We first assume that 
($i \neq n-1$
and 
$\mathsf{T}(i) \preceq n$)
or 
($\overline{n} \preceq \mathsf{T}(i) \prec \overline{a_1}$). 
Then 
$n-i+e \geq 2$
and 
$\mathsf{T}(i) \preceq \overline{a_2}$. 
Define
$\mathsf{T}_1, \mathsf{T}_2 \in \mathrm{CST}_{B_n}(\varpi_i)$
by
\begin{align*}
\mathsf{T}_1(u)
&=
\begin{cases}
u+1 & \text{if} \ u \in [a_2-2], \\
\mathsf{T}(u) & \text{if} \ u \in [a_2-1,i-1], \\
\overline{a_2} & \text{if} \ u = i,
\end{cases} \\
\mathsf{T}_2(u)
&=
\begin{cases}
1 & \text{if} \ u = 1, \\
\mathsf{T}_1(u-1) & \text{if} \ u \in [2,i].
\end{cases}
\end{align*}
Let 
$x_1,x_2 \in (W^{I \setminus \{ i \}})_{\af}$
be such that 
$\mathcal{Y}_i^{B_n}(x_1) = (\mathsf{T}_1,c)$
and 
$\mathcal{Y}_i^{B_n}(x_2) 
= (\mathsf{T}_2,c')
= (\mathsf{T}_2,c+1)$. 
We have 
$\mathsf{T}(u)
=
u 
\prec
u+1
= 
\mathsf{T}_1(u)$
for 
$u \in [a_1-1]$, 
$\mathsf{T}(u)
= 
\mathsf{T}_1(u)$
for 
$u \in [a_1,i-1]$, 
and 
$\mathsf{T}(i) 
\preceq 
\overline{a_2} 
= 
\mathsf{T}_1(i)$. 
Hence 
$x \preceq x_1$
by the assertion for $d = 0$. 
We have 
$x_1 \preceq x_2$
by Proposition \ref{prop:SiB-B} ($\si$-B6). 
We have 
$\mathsf{T}_2(u) 
= 
u 
\preceq 
\mathsf{T}'(u)$
for 
$u \in [a_2-1]$, 
and 
$\mathsf{T}_2(u) 
=
\mathsf{T}(u-1) 
\preceq 
\mathsf{T}'(u)$
for 
$u \in [a_2,i]$. 
Hence 
$x_2 \preceq y$
by the assertion for $d = 0$. 
Consequently, 
we have 
$x \prec y$.

We next assume that 
$i = n-1$
and 
$\mathsf{T}(n-1) \preceq n$; 
note that 
$n-i+e = 1$. 
Define
$\mathsf{T}_1,\mathsf{T}_2 \in \mathrm{CST}_{B_n}(\varpi_{n-1})$
to be such that 
$\mathsf{T}_1(u)
=
\mathsf{T}(u)$
for 
$u \in [n-2]$,  
$\mathsf{T}_1(n-1)
=
\overline{a_1}$, 
$\mathsf{T}_2(1)
=
1$, 
and 
$\mathsf{T}_2(u)
=
\mathsf{T}_1(u-1)$
for 
$u \in [2,n-1]$. 
Let 
$x_1,x_2 \in (W^{I \setminus \{ i \}})_{\af}$
be such that 
$\mathcal{Y}_i^{B_n}(x_1) = (\mathsf{T}_1,c)$
and 
$\mathcal{Y}_i^{B_n}(x_2) 
= (\mathsf{T}_2,c')
= (\mathsf{T}_2,c+1)$. 
By a similar argument above, 
we can see that 
$x \preceq x_1 \prec x_2 \preceq y$.

We next assume that 
$\mathsf{T}(i) = \overline{a_1}$
and 
$1 \prec a_1 \prec \mathsf{T}'(a_1)$. 
Then 
$e = d$, 
$n-i+e \geq 2$, 
$a_2 \in [3,i+1]$, 
and 
$u+1 \preceq \mathsf{T}'(u)$
for 
$u \in [a_1,i]$
(see
(ii) in Definition \ref{def:SiB-B} (2)). 
Define
$\mathsf{T}_3, \mathsf{T}_4 \in \mathrm{CST}_{B_n}(\varpi_i)$
by
\begin{align*}
\mathsf{T}_3(u)
&=
\begin{cases}
u+1 & \text{if} \ u \in [a_1-2], \\
u+2 & \text{if} \ u \in [a_1-1,a_2-2], \\
\mathsf{T}(u) & \text{if} \ u \in [a_2-1,i], 
\end{cases} \\
\mathsf{T}_4(u)
&=
\begin{cases}
1 & \text{if} \ u = 1, \\
\mathsf{T}_3(u-1) & \text{if} \ u \in [2,i].
\end{cases}
\end{align*}
Let 
$x_3,x_4 \in (W^{I \setminus \{ i \}})_{\af}$
be such that 
$\mathcal{Y}_i^{B_n}(x_3) = (\mathsf{T}_3,c)$
and 
$\mathcal{Y}_i^{B_n}(x_4) 
= (\mathsf{T}_4,c')
= (\mathsf{T}_4,c+1)$. 
We have 
$\mathsf{T}(u)
=
u
\prec 
u+1
= 
\mathsf{T}_3(u)$
for 
$u \in [a_1-2]$, 
$\mathsf{T}(a_1-1)
=
a_1-1
\prec 
a_1+1
= 
\mathsf{T}_3(a_1-1)$, 
$\mathsf{T}(u)
=
u+1
\prec 
u+2
= 
\mathsf{T}_3(u)$
for 
$u \in [a_1,a_2-2]$,
and 
$\mathsf{T}(u)
=
\mathsf{T}_3(u)$
for 
$u \in [a_2-1,i]$. 
Hence 
$x \preceq x_3$
by the assertion for $d = 0$. 
We have 
$x_3 \preceq x_4$
by Proposition \ref{prop:SiB-B} ($\si$-B6). 
We have 
$\mathsf{T}_4(u)
= 
u
\preceq 
\mathsf{T}'(u)$
for 
$u \in [a_1-1]$, 
$\mathsf{T}_4(u) 
= 
\mathsf{T}_3(u-1)
=
u+1
\preceq 
\mathsf{T}'(u)$
for 
$u \in [a_1,a_2-1]$, 
and 
$\mathsf{T}_4(u) 
= 
\mathsf{T}(u-1)
\preceq 
\mathsf{T}'(u)$
for 
$u \in [a_2,i]$.
Hence 
$x_4 \preceq y$
by the assertion for $d = 0$. 
Consequently, 
we have 
$x \prec y$.
\end{proof}

\begin{proof}[Step 5]
We prove that if 
$d$ is odd, 
$d = e \in [3,i]$,
$\mathsf{T}(i-d+1) = \overline{a_d}$, 
and 
$1 \prec a_d \prec \mathsf{T}'(a_d)$, 
then 
$x \prec y$. 
Let 
$\mathsf{T}_3,\mathsf{T}_4 \in \mathrm{CST}_{B_n}(\varpi_i)$
be as in Step 3, 
and let
$x_3,x_4 \in (W^{I \setminus \{ i \}})_{\af}$
be such that  
$\mathcal{Y}_i^{B_n}(x_3)
=
(\mathsf{T}_3, c)$
and 
$\mathcal{Y}_i^{B_n}(x_4)
=
(\mathsf{T}_4, c+2)$. 
We have 
$\mathsf{T}(u)
\preceq 
\mathsf{T}_3(u)$
for 
$u \in [i]$, 
and hence 
$x \preceq x_3$
by the assertion for 
$d = 0$. 
We have 
$x_3 \prec x_4$
by Proposition \ref{prop:SiB-B} ($\si$-B7). 
Hence 
$x \prec x_4$. 
By the same method as in Step 3, 
we see that 
$c' - (c+2)
= 
d-2$
is odd, 
$d-2 \in [i-2]$, 
$\mathsf{T}_4(i-(d-2))
\preceq 
n$, 
$\mathsf{T}_4(i-(d-2)+1)
=
\overline{a_d}$, 
and 
$\{ a_3 < \cdots < a_{n-i+d} \}
=
[n] \setminus \{ \mathsf{T}_4(u) \mid u \in [i]\}$. 
Also, we have 
$\mathsf{T}_4(u)
\preceq 
\mathsf{T}'(u+d-2)$
for 
$u \in [i-d+2]$, 
because
$\mathsf{T}_4(u)
=
u
\preceq
\mathsf{T}'(u)
\preceq
\mathsf{T}'(u+d-2)$
for 
$u \in [\min \{ a_2,i-d+2 \}]$, 
and
$\mathsf{T}_4(u)
=
\mathsf{T}(u-2)
\preceq
\mathsf{T}'(u+d-2)$
for 
$u \in [a_2+1,i-d+2]$. 
Consequently, 
$(\mathsf{T}_4,c+2)
\preceq 
(\mathsf{T}',c')$
(see
(ii) in Definition \ref{def:SiB-B} (2)). 
By induction hypothesis (and Step 4), 
we have 
$x_4 \preceq y$. 
Hence 
$x \prec y$. 
\end{proof}

\begin{proof}[Step 6]
We prove that if 
$d$ is odd, 
$d \in [3,i]$, 
and 
$\mathsf{T}(i-d+1) \preceq n$, 
then 
$x \prec y$. 
We proceed by induction on $e$; 
note that 
$e \in [0,d-1]$. 

Assume that 
$e = 0$. 
By the same method as in Step 4, 
we can find 
$x_1,x_2 \in (W^{I \setminus \{ i \}})_{\af}$
and 
$\mathsf{T}_1,\mathsf{T}_2 \in \mathrm{CST}_{B_n}(\varpi_i)$
such that 
$\mathcal{Y}_i^{B_n}(x_1) = (\mathsf{T}_1,c)$, 
$\mathcal{Y}_i^{B_n}(x_2) = (\mathsf{T}_2,c+1)$, 
$x \preceq x_1 \prec x_2$, 
and 
$\mathsf{T}_2(u) \preceq \mathsf{T}'(u+d-1)$
for 
$u \in [u-d+1]$. 
Since 
$c'-(c+1) = d-1$
is even, 
we have 
$x_2 \preceq y$
by Step 1. 
Hence 
$x \prec y$. 

Assume that 
$e \in [d-1]$. 
Let
$\mathsf{T}_4 \in \mathrm{CST}_{B_n}(\varpi_i)$
and 
$x_4 \in (W^{I \setminus \{ i \}})_{\af}$
be as in Step 3; 
since 
$n-i+e \geq 2$, 
$\mathsf{T}_4$
is well-defined. 
By the same method as in Step 5, 
we see that 
$x \prec x_4$, 
$c' - (c+2)
= 
d-2$
is odd, 
and 
$\mathsf{T}_4(u)
\preceq
\mathsf{T}'(u+d-2)$
for 
$u \in [i-d+2]$. 
It remains to prove that 
$x_4 \preceq y$.
We have 
$\mathsf{T}_4(i-d+3) 
\preceq 
\max\{ i-d+3,\mathsf{T}(i-d+1)\}
\preceq n$
and 
$\#\{ u \in [i] \mid \mathsf{T}_4(u) \succeq \overline{n} \} 
=
\max\{ 0,e-2\} < e$. 
If $d \geq 5$, 
then 
$c' - (c+2) \in [3,i]$, 
and hence 
$x_4 \preceq y$
by induction hypothesis. 
If 
$d = 3$, 
then 
$e \in [2]$, 
$c' - (c+2) = 1$, 
$\mathsf{T}_4(i) \preceq n$, 
and hence 
$x_4 \preceq y$
by Step 4.
Thus 
$x \prec y$. 
\end{proof}

\begin{proof}[Step 7]
We prove that if 
$d$ is odd
and 
$d > i$, 
then 
$x \prec y$. 
Note that 
$d \geq 3$
and 
$n-i+e \geq 1$. 

We first assume that 
$n-i+e = 1$. 
Then 
$i=n-1$
and 
$e=0$.
In a way similar to the case of $e=0$ in Step 6, 
we can see that 
$x \prec y$.

We next assume that 
$n-i+e \geq 2$. 
We proceed by induction on $d$. 
Let 
$\mathsf{T}_4 \in \mathrm{CST}_{B_n}(\varpi_i)$
and 
$x_4 \in (W^{I \setminus \{ i \}})_{\af}$
be as in Step 3; 
since 
$n-i+e \geq 2$, 
$\mathsf{T}_4$
is well-defined. 
By a similar argument to Step 5, 
we have 
$x \prec x_4$
and 
$\mathsf{T}_4(u) \preceq \mathsf{T}'(u+d-2)$
for 
$u \in [i-d+2]$. 
Note that 
$c' - (c+2) = d-2$
is odd. 
If 
$d = 3 = i+1$, 
then 
$d-2 = 1$, 
$\mathsf{T}_4(i) = 2 \preceq n$, 
and hence 
$x_4 \preceq y$
by Step 4.
If 
$d \geq 5$
and 
$d \in \{ i+1,i+2 \}$, 
then 
$d-2 \in [3,i]$, 
$\mathsf{T}_4(i-d+3)
\preceq 
\mathsf{T}_4(2)
= 2 \preceq n$, 
and hence 
$x_4 \preceq y$
by Step 6. 
If 
$d > i+2$, 
then 
$d - 2 > i$, 
and hence 
$x_4 \preceq y$
by induction hypothesis. 
Thus 
$x \prec y$. 
\end{proof}
The proof of Proposition \ref{prop:tab-cri-B} is complete. 
\end{proof}

\subsection{Type $D_n^{(1)}$} \label{subsection:tab-cri-D}

Fix an integer $n \geq 4$. 
Set 
$I = [n]$. 
We assume that the labeling of the vertices 
of the Dynkin diagram of type $D_n$ is as follows.
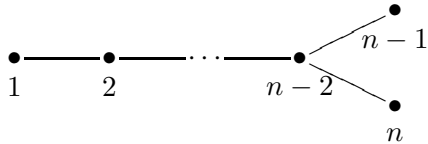
\begin{center}\ 
\xygraph{
\bullet ([]!{+(0,-.3)} {1}) - [r]
\bullet ([]!{+(0,-.3)} {2}) - [r] \cdots - [r]
\bullet ([]!{+(0,-.3)} {n - 2}) (
- []!{+(1,.5)} \bullet ([]!{+(0,-.3)} {n - 1}),
- []!{+(1,-.5)} \bullet ([]!{+(0,-.3)} {n}))}
\end{center}
Let 
$\varepsilon_1,\varepsilon_2,\ldots ,\varepsilon_n$
be an orthonormal basis of an 
$n$-dimensional Euclidean space $\BR^n$. 
Let 
$\Delta 
= 
\{ \pm(\varepsilon_s \pm \varepsilon_t) \mid s,t \in [n],\ s < t \}$
be a root system of type $D_n$, 
and let 
$\Pi 
=
\{ \alpha_s = \varepsilon_s - \varepsilon_{s+1} \mid s \in [n-1] \}
\sqcup
\{ \alpha_n = \varepsilon_{n-1} + \varepsilon_n \}$
be a simple root system of $\Delta$. 

Let 
$W$
be the Weyl group of 
$\Delta$. 
Note that 
$W$ 
acts faithfully on 
$\{ \pm\varepsilon_s \mid s \in [n] \} \subset \BR^n$. 
Define a partially ordered set 
$\mathcal{D}_n$
by 
\begin{align}
\mathcal{D}_n 
=
\left\{ 
1 \prec 2 \prec \cdots \prec n-1 \prec \begin{array}{c} n \\ \overline{n} \end{array} \prec \overline{n-1} \prec \cdots \prec \overline{2} \prec \overline{1}
\right\}.
\end{align}
Let 
$\sigma : \mathcal{D}_n \to \mathcal{D}_n$
be the bijection defined by 
$s \leftrightarrow \overline{s}$
for 
$s \in [n]$. 
If we identify 
$\mathcal{D}_n$
with 
$\{ \pm\varepsilon_s \mid s \in [n] \}$
by 
$s = \varepsilon_s$
and 
$\overline{s} = -\varepsilon_s$
for 
$s \in [n]$, 
then $W$ can be described as follows:
\begin{align} \label{eq:Weyl-grp-D}
\begin{split}
W 
=
\{ w \in \mathfrak{S}(\mathcal{D}_n) \mid \ 
& 
w(\sigma(s)) = \sigma(w(s))
\ \text{for} \ 
s \in [n], \ \text{and} \\ 
&\#\{ s \in [n] \mid w(s) \succeq \overline{n} \} \
\text{is even} \} .
\end{split}
\end{align}
Note that 
$r_{\pm(\varepsilon_s - \varepsilon_t)}
= 
(s \ t)(\overline{s} \ \overline{t})$
and 
$r_{\pm(\varepsilon_s + \varepsilon_t)} 
=
(s \ \overline{t})(\overline{s} \ t)$
for 
$s,t \in [n]$, $s<t$. 

For 
$w \in \mathfrak{S}(\mathcal{D}_n)$
and 
$s \in [n]$, 
set
\begin{align}
\mathsf{A}_s(w)
&=
\{ t \in [s+1,n] \mid w(s) \succ w(t) \ \text{in} \ \mathcal{D}_n \}, \ \ 
\mathsf{a}_s(w) = \# \mathsf{A}_s(w), \\
\mathsf{B}_s(w)
&=
\{ t \in [s+1,n] \mid w(s) \succ \sigma(w(t)) \ \text{in} \ \mathcal{D}_n \}, \ \ 
\mathsf{b}_s(w) = \# \mathsf{B}_s(w); 
\end{align}
note that 
$\mathsf{a}_n(w) = \mathsf{b}_n(w) = 0$. 
The length function
$\ell : W \to \BZ_{\geq 0}$
is given by 
$\ell(w)
=
\sum_{s = 1}^n (\mathsf{a}_s(w) + \mathsf{b}_s(w))$
for 
$w \in W$. 
The longest element of $W$ is given by
$u \mapsto \overline{u}$, 
$u \in [n]$, 
if
$n$
is even, 
and 
$u \mapsto \overline{u}$, 
$u \in [n-1]$, 
if
$n$
is odd. 

Let 
$\| \cdot \| : \mathcal{D}_n \to [n]$
be the map defined by 
$\|s\| = s$
and 
$\|\overline{s}\| = s$
for 
$s \in [n]$.
We identify a totally ordered $i$-element subset 
$\mathsf{T} = \{ \mathsf{T}(1) \prec \mathsf{T}(2) \prec \cdots \prec \mathsf{T}(i) \} 
\subset 
\mathcal{D}_n$
with the column-strict tableau of the form
\eqref{eq:column-A}. 
For 
$i \in I \setminus \{ n-1 \}$
and 
$w \in W$, 
let 
$\mathsf{T}_w^{(i)} 
\subset
\mathcal{D}_n$
be the totally ordered $i$-element subset such that 
\begin{align}
\mathsf{T}_w^{(i)}
=
\left\{ \mathsf{T}_w^{(i)}(1) \prec \mathsf{T}_w^{(i)}(2) \prec \cdots \prec \mathsf{T}_w^{(i)}(i) \right\}
=
\{ w(1), w(2), \ldots , w(i) \}.
\end{align}
For 
$w \in W$, 
let 
$\mathsf{T}_w^{(n-1)} 
\subset
\mathcal{D}_n$
be the totally ordered $n$-element subset such that 
\begin{align}
\begin{split}
\mathsf{T}_w^{(n-1)}
&=
\left\{ \mathsf{T}_w^{(n-1)}(1) \prec \mathsf{T}_w^{(n-1)}(2) \prec \cdots \prec \mathsf{T}_w^{(n-1)}(n) \right\} \\
&=
\{ w(1), w(2), \ldots , w(n-1) , w(\overline{n}) \}.
\end{split}
\end{align}
For $i \in [n-2]$, 
let 
$\mathrm{CST}_{D_n}(\varpi_i)$
be the family of 
totally ordered $i$-element subsets
$\mathsf{T}$ 
of
$\mathcal{D}_n$
such that 
$\|\mathsf{T}(u)\|$, 
$u \in [i]$, 
are all distinct. 
Let 
$\mathrm{CST}_{D_n}(\varpi_{n-1})$
(resp. $\mathrm{CST}_{D_n}(\varpi_n)$)
be the family of totally ordered 
$n$-element subsets
$\mathsf{T}$
of 
$\mathcal{D}_n$
such that 
$\|\mathsf{T}(u)\|$, 
$u \in [n]$, 
are all distinct, 
and 
$\#\{ u \in [n] \mid \mathsf{T}(u) \succeq \overline{n} \}$
is odd 
(resp. even). 
The proof of the next lemma is standard 
(cf. \cite[\S 8.1]{BB}).

\begin{lem} \label{lem:Gr-D}
Let 
$i \in I$. 
We have
\begin{align*}
&W^{I \setminus \{ i \}}
= 
\\
&
\begin{cases}
\{ w \in W \mid 
w(1) \prec \cdots \prec w(i), \ \text{and} \\
\hspace{1.7cm}
w(i+1) \prec \cdots \prec w(n-1) \prec w(n) \prec w(\overline{n-1}) \} 
\ \text{if} \ i \in [n-2], \\[1mm]
\{ w \in W \mid 
w(1) \prec \cdots \prec w(n-1) \prec w(\overline{n}) \} 
\ \text{if} \ i = n-1, \\[1mm]
\{ w \in W 
\mid 
w(1) \prec \cdots \prec w(n-1) \prec w(n) \} 
\ \text{if} \ i = n.
\end{cases}
\end{align*}
If 
$w \in W^{I \setminus \{ i \}}$, 
then 
$\ell(w) = \sum_{s=1}^i(\mathsf{a}_s(w) + \mathsf{b}_s(w))$
and 
$\mathsf{A}_s(w) \subset [i+1,n]$
for 
$s \in [i]$.
The map 
$W^{I \setminus \{ i \}} \to \mathrm{CST}_{D_n}(\varpi_i)$, 
$w \mapsto \mathsf{T}_w^{(i)}$, 
is bijective. 
\end{lem}

We see from Lemmas \ref{lem:J-ad}--\ref{lem:(W^J)_af} and \ref{lem:Gr-D}
that the map
\begin{align}
\mathcal{Y}^{D_n}_i : 
W_{\af}
\to 
\mathrm{CST}_{D_n}(\varpi_i) \times \BZ, \ 
wt_{\xi} \mapsto 
\left( \mathsf{T}^{(i)}_w,c_i(\xi) \right), 
\end{align}
induces a bijection from the subset 
$(W^{I \setminus \{ i \}})_{\af} \subset W_{\af}$
to 
$\mathrm{CST}_{D_n}(\varpi_i) \times \BZ$.

\begin{define} \label{def:SiB-D}
Let 
$i \in I$, 
$(\mathsf{T},c) , (\mathsf{T}',c')
\in 
\mathrm{CST}_{D_n}(\varpi_i) \times \BZ$, 
and 
$d := c'-c$. 
Define a partial order 
$\preceq$
on 
$\mathrm{CST}_{D_n}(\varpi_i) \times \BZ$
as follows.
\begin{enumerate}[(1)]
\item
Assume that 
$i \in [n-2]$. 
Set 
$(\mathsf{T},c) \preceq (\mathsf{T}',c')$
if 
\begin{align}
\label{eq:T(u)<T'(u+d)-D}
(d \geq 0), \ \ 
\left(
\mathsf{T}(u) \preceq \mathsf{T}'(u+d)
\ \text{in} \ 
\mathcal{D}_n
\ \text{for} \ 
u \in [i-d]
\right)
\end{align}
and one of the following conditions holds: 
\begin{enumerate}[(i)]
\item
$d$ is even.
\item
$d$ is odd 
and 
$\mathsf{T}(i) \preceq n$. 
\item
$d$ is odd 
and 
$\mathsf{T}(i) \succeq \overline{n}$. 
Let 
\begin{align} \label{eq:SiB-D-n-2}
\begin{split}
a
&=
\min
\left(
[n]
\setminus 
\{ \| \mathsf{T}(u) \| 
\mid 
u \in [i], \ 
\mathsf{T}(u) \succeq \overline{n} \}
\right), \\
b
&=
\min
\left(
[n]
\setminus 
\{ \| \mathsf{T}(u) \| 
\mid 
u \in [i] \}
\right); 
\end{split}
\end{align}
note that 
$a \leq b$. 
If 
$d \in [i]$, 
then 
$a \preceq \mathsf{T}'(d)$. 
If 
$a < b$
and 
$k \in [i]$
satisfies 
$\mathsf{T}(k) \prec b \prec \mathsf{T}(k+1)$,  
then 
($\mathsf{T}(u) \preceq \mathsf{T}'(u+d-1)$
for 
$u \in [2, \min \{ k , i-d+1 \}]$)
and 
($b \preceq \mathsf{T}'(k+d)$
if 
$k \in [i-d]$). 
\end{enumerate}

\item
Assume that 
$i \in \{ n-1,n \}$. 
Set 
$(\mathsf{T},c) \preceq (\mathsf{T}',c')$
if 
\begin{align}
(d \geq 0)
\ \text{and} \  
\left(
\mathsf{T}(u) \preceq \mathsf{T}'(u+2d)
\ \text{in} \ 
\mathcal{D}_n
\ \text{for} \ 
u \in [n-2d]
\right) .
\end{align}
\end{enumerate}
\end{define}

\begin{prop} \label{prop:tab-cri-D}
Let 
$i \in I$. 
\begin{enumerate}[(1)]
\item
$\mathcal{Y}_i^{D_n} \circ \Pi^{I \setminus \{ i \}} = \mathcal{Y}_i^{D_n}$.
\item
For 
$x,y \in W_{\af}$, 
we have 
$\Pi^{I \setminus \{ i \}}(x) \preceq \Pi^{I \setminus \{ i \}}(y)$
in 
$(W^{I \setminus \{ i \}})_{\af}$
if and only if 
$\mathcal{Y}_i^{D_n}(x) \preceq \mathcal{Y}_i^{D_n}(y)$
in
$\mathrm{CST}_{D_n}(\varpi_i) \times \BZ$. 
\item
Let 
$i \in [n-2]$
and 
$(\mathsf{T},c) , (\mathsf{T}',c')
\in 
\mathrm{CST}_{D_n}(\varpi_i) \times \BZ$. 
If 
$c' - c > i$, 
then
$(\mathsf{T},c) \preceq (\mathsf{T}',c')$.

\item
Let 
$i \in \{ n-1,n \}$
and 
$(\mathsf{T},c) , (\mathsf{T}',c')
\in 
\mathrm{CST}_{D_n}(\varpi_i) \times \BZ$. 
If 
$2(c' - c) > n$, 
then
$(\mathsf{T},c) \preceq (\mathsf{T}',c')$.

\end{enumerate}
\end{prop}

By combining Propositions \ref{prop:Deo} and \ref{prop:tab-cri-D} (2), 
we obtain the following tableau criterion for the 
semi-infinite Bruhat order
on $W_{\af}$ of type $D_n^{(1)}$.

\begin{thm} \label{thm:tab-cri-D}
Let 
$J \subset I$. 
For 
$x,y \in (W^J)_{\af}$, 
we have 
$x \preceq y$
in 
$(W^J)_{\af}$
if and only if 
$\mathcal{Y}_i^{D_n}(x) \preceq \mathcal{Y}_i^{D_n}(y)$
in 
$\mathrm{CST}_{D_n}(\varpi_i) \times \BZ$
for all 
$i \in I \setminus J$. 
\end{thm}

The remainder of this subsection is devoted to 
the proof of Proposition \ref{prop:tab-cri-D}. 

The proofs of 
Lemmas \ref{lem:J-ad-D}--\ref{lem:<g,r>-D} below 
are straightforward.

\begin{lem} \label{lem:J-ad-D}
Let 
$i \in I$
and 
$\gamma \in \Delta^+ \setminus \Delta_{I \setminus \{ i \}}^+$. 
We have 
$\gamma^{\vee} \in Q^{\vee,I \setminus \{ i \}}$
if and only if one of the following conditions holds:
\begin{enumerate}[(1)]
\item
$i = 1$
and 
$\gamma^{\vee} 
=
\varepsilon_1 - \varepsilon_2 
=
\alpha_1^{\vee}$.
\item
$i \in [2,n-2]$
and 
$\gamma^{\vee} 
=
\varepsilon_i - \varepsilon_{i+1} 
=
\alpha_i^{\vee}$.
\item
$i \in [2,n-2]$
and 
$\gamma^{\vee}
=
\varepsilon_{i-1} + \varepsilon_i
= \alpha_{i-1}^{\vee} 
+ 
2\alpha_i^{\vee}
+
2\alpha_{i+1}^{\vee}
+
\cdots 
+
2\alpha_{n-2}^{\vee}
+ 
\alpha_{n-1}^{\vee} 
+ 
\alpha_n^{\vee}$.
\item
$i = n-1$
and 
$\gamma^{\vee} 
=
\varepsilon_{n-1} - \varepsilon_n 
=
\alpha_{n-1}^{\vee}$.
\item
$i = n$
and 
$\gamma^{\vee} 
=
\varepsilon_{n-1} + \varepsilon_n
=
\alpha_n^{\vee}$.
\end{enumerate}
\end{lem}

\begin{lem} \label{lem:<g,r>-D}
Let $i \in I$. 
We have
\begin{align}
2\langle \alpha_i^{\vee} , \rho - \rho_{I \setminus \{ i \}} \rangle 
= 
\begin{cases}
2n - i - 1
&
\text{if} \ i \in [2,n-2], \\
2n-2 
&
\text{if} \ i \in \{1,n-1,n\}.
\end{cases}
\end{align}
\end{lem}

\begin{prop}[cf. {\cite[\S 8.2]{BB}}] \label{prop:Bruhat-D}
Let 
$i \in I$, 
$w \in W^{I \setminus \{ i \}}$, 
and 
$\gamma \in \Delta^+$. 
There exists a Bruhat edge 
$w \to \lfloor wr_{\gamma} \rfloor = wr_{\gamma}$
in 
$\mathrm{QB}^{I \setminus \{ i \}}$
if and only if 
$\gamma \in \Delta^+ \setminus \Delta_{I \setminus \{ i \}}^+$
and one of the following statements holds.
\begin{enumerate}[(b-D1)]
\item
$i \in [n-2]$, 
$c_i(\gamma^{\vee}) = 1$, 
and there exists 
$s \in [i]$
such that 
$wr_{\gamma}(u) = w(u)$
for 
$u \in [i] \setminus \{ s \}$, 
$1 \preceq w(s) \prec n$, 
and 
$wr_{\gamma}(s)
=
\min
([w(s)+1,n]
\setminus 
\{ \|w(u)\| \mid u \in [i], \ w(u) \succeq \overline{n} \})$;
in this case, 
we have  
$\gamma^{\vee} 
= 
\varepsilon_s - \varepsilon_t
=
\alpha_s^{\vee} + \alpha_{s+1}^{\vee} + \cdots + \alpha_t^{\vee}$
for some 
$t \in [i+1,n]$. 
\item
$i \in [n-2]$, 
$c_i(\gamma^{\vee}) = 1$, 
and there exists 
$s \in [i]$
such that 
$wr_{\gamma}(u) = w(u)$
for 
$u \in [i] \setminus \{ s \}$, 
$\overline{n} \prec w(s) \preceq \overline{1}$, 
and 
$wr_{\gamma}(s)
=
\sigma
(\max
([1,\|w(s)\|-1]
\setminus 
\{ w(u) \mid u \in [i], \ w(u) \preceq n \}))$;
in this case, 
we have  
$\gamma^{\vee} 
= 
\varepsilon_s - \varepsilon_t
=
\alpha_s^{\vee} + \alpha_{s+1}^{\vee} + \cdots + \alpha_t^{\vee}$
for some 
$t \in [i+1,n]$. 
\item
$i \in [n-2]$, 
$c_i(\gamma^{\vee}) = 1$, 
and there exists 
$s \in [i]$
such that 
$wr_{\gamma}(u) = w(u)$
for 
$u \in [i] \setminus \{ s \}$, 
and 
$(w(s),wr_{\gamma}(s))
\in 
\{(n-1,\overline{n}), 
(n,\overline{n-1})\}$;
in this case, 
we have  
$\gamma^{\vee} 
= 
\varepsilon_s + \varepsilon_t
=
\alpha_s^{\vee} + \cdots + \alpha_{t-1}^{\vee} 
+ 
2\alpha_t^{\vee} 
\cdots 
+ 
2\alpha_{n-2}^{\vee}
+
\alpha_{n-1}^{\vee}
+
\alpha_n^{\vee}$
for some 
$t \in [i+1,n-1]$, 
or
$\gamma^{\vee} 
= 
\varepsilon_s + \varepsilon_n
=
\alpha_s^{\vee} + \cdots + \alpha_{n-2}^{\vee} 
+
\alpha_n^{\vee}$. 
\item
$i \in [2,n-2]$,
$c_i(\gamma^{\vee}) = 2$, 
and there exist 
$s,t \in [i]$
such that 
$s < t$, 
$wr_{\gamma}(u) = w(u)$
for 
$[i] \setminus \{ s,t \}$
and either 
($wr_{\gamma}(s) = w(s) + 1 = \sigma(w(t)) = \sigma(wr_{\gamma}(t)) + 1 \preceq n$)
or 
($w(s) = \sigma(wr_{\gamma}(t)) = n-1$
and 
$w(t) = \sigma(wr_{\gamma}(s)) = n$)
holds;
in this case, 
we have  
$\gamma^{\vee} 
= 
\varepsilon_s + \varepsilon_t
=
\alpha_s^{\vee} + \cdots + \alpha_{t-1}^{\vee}
+ 2\alpha_t^{\vee} + \cdots + 2\alpha_{n-2}^{\vee} + \alpha_{n-1}^{\vee} + \alpha_n^{\vee}$.

\item
$i \in \{n-1,n\}$,
$c_i(\gamma^{\vee}) = 1$, 
and there exist 
$s,t \in [i]$
such that 
$s < t$, 
$wr_{\gamma}(u) = w(u)$
for 
$[i] \setminus \{ s,t \}$
and 
$wr_{\gamma}(s) = w(s) + 1 = \sigma(w(t)) = \sigma(wr_{\gamma}(t)) + 1 \preceq n$;
in this case, 
we have  
$\gamma^{\vee} 
= 
\varepsilon_s + \varepsilon_t
=
\alpha_s^{\vee} + \cdots + \alpha_{t-1}^{\vee}
+ 2\alpha_t^{\vee} + \cdots + 2\alpha_{n-2}^{\vee} + \alpha_{n-1}^{\vee} + \alpha_n^{\vee}$.
\item
$i = n-1$,
$c_{n-1}(\gamma^{\vee}) = 1$, 
and there exists 
$s \in [n-2]$
such that 
$wr_{\gamma}(u) = w(u)$
for 
$[n] \setminus \{ s,s+1 \}$
and 
$\sigma(wr_{\gamma}(s))
=
\sigma(wr_{\gamma}(s+1))+1
=
w(s+1) 
= 
w(s) + 1
\preceq n$;
in this case, 
we have  
$\gamma^{\vee} 
= 
\varepsilon_s + \varepsilon_{s+1}
=
\alpha_s^{\vee}
+ 2\alpha_{s+1}^{\vee} + \cdots + 2\alpha_{n-2}^{\vee} + \alpha_{n-1}^{\vee} + \alpha_n^{\vee}$
if 
$s \in [n-3]$, 
and 
$\gamma^{\vee} 
= 
\varepsilon_{n-2} + \varepsilon_{n-1}
=
\alpha_{n-2}^{\vee} + \alpha_{n-1}^{\vee} + \alpha_n^{\vee}$
if 
$s = n-2$.
\item
$i = n$,
$c_n(\gamma^{\vee}) = 1$, 
and there exists 
$s \in [n-1]$
such that 
$wr_{\gamma}(u) = w(u)$
for 
$[n] \setminus \{ s,s+1 \}$
and 
$\sigma(wr_{\gamma}(s))
=
\sigma(wr_{\gamma}(s+1))+1
=
w(s+1) 
= 
w(s) + 1
\preceq n$;
in this case, 
we have  
$\gamma^{\vee} 
= 
\varepsilon_s + \varepsilon_{s+1}
=
\alpha_s^{\vee}
+ 2\alpha_{s+1}^{\vee} + \cdots + 2\alpha_{n-2}^{\vee} + \alpha_{n-1}^{\vee} + \alpha_n^{\vee}$
if 
$s \in [n-3]$, 
$\gamma^{\vee} 
= 
\varepsilon_{n-2} + \varepsilon_{n-1}
=
\alpha_{n-2}^{\vee} + \alpha_{n-1}^{\vee} + \alpha_n^{\vee}$
if 
$s = n-2$, 
and 
$\gamma^{\vee} 
= 
\varepsilon_{n-1} + \varepsilon_n
=
\alpha_n^{\vee}$
if 
$s = n-1$.
\end{enumerate}
Moreover, for 
$i \in I \setminus \{n-1\}$
and 
$w,v \in W^{I \setminus \{ i \}}$,  
we have 
$w \preceq v$
if and only if 
$w(u) \preceq v(u)$
in 
$\mathcal{D}_n$
for 
$u \in [i]$. 
For 
$w,v \in W^{I \setminus \{ n-1 \}}$, 
we have 
$w \preceq v$
if and only if 
$w(\overline{n}) \preceq v(\overline{n})$
in 
$\mathcal{D}_n$
and 
$w(u) \preceq v(u)$
in 
$\mathcal{D}_n$
for 
$u \in [n-1]$. 
\end{prop}

For 
$i \in I$, 
$w \in W^{I \setminus \{ i \}}$
and 
$\gamma \in \Delta^+ \setminus \Delta_{I \setminus \{ i \}}^+$, 
let 
$\mathrm{Q}(i,w,\gamma)$
denote the following statement.
\begin{description}
\item[$\mathrm{Q}(i,w,\gamma):$]
There exists a quantum edge 
$w \xrightarrow{\ \gamma \ } \lfloor wr_{\gamma} \rfloor$
in 
$\mathrm{QB}^{I \setminus \{ i \}}$.
\end{description}

\begin{prop} \label{prop:Q=D}
Let 
$i \in I$, 
$w \in W^{I \setminus \{ i \}}$, 
and 
$\gamma \in \Delta^+ \setminus \Delta_{I \setminus \{ i \}}^+$. 
Then 
$\mathrm{Q}(i,w,\gamma)$
is true if and only if one of the following statements holds.
\begin{enumerate}[(q-D1)]
\item
$i \in [n-2]$
and 
$c_i(\gamma^{\vee}) = 1$. 
If we write 
$\{ a_1 < a_2 < \cdots < a_{n-i+1} \}
=
[n] \setminus \{ \|w(u)\| \mid u \in [i-1] \}$, 
then 
$a_1 = 1$, 
$\{ w(\overline{i}), \lfloor wr_{\gamma} \rfloor(1) \}
=
\{ 1,a_2 \}$
and 
$\lfloor wr_{\gamma} \rfloor(u) = w(u-1)$
for 
$u \in [2,i]$;
in this case, 
we have  
$\gamma^{\vee} 
=
\varepsilon_i - \varepsilon_{i+1}
=
\alpha_i^{\vee}$.
\item
$i \in [2,n-2]$, 
$c_i(\gamma^{\vee}) = 2$, 
$\lfloor wr_{\gamma} \rfloor (1) = w(\overline{i}) = 1$, 
$\lfloor wr_{\gamma} \rfloor (2) = w(\overline{i-1}) = 2$, 
and 
$\lfloor wr_{\gamma} \rfloor (u) = w(u-2)$
for 
$u \in [3,i]$;
in this case, 
we have  
$\gamma^{\vee} = 
\varepsilon_{i-1} + \varepsilon_i 
= 
\alpha_{i-1}^{\vee} 
+ 
2\alpha_i^{\vee}
+
2\alpha_{i+1}^{\vee}
+
\cdots 
+
2\alpha_{n-2}^{\vee}
+ 
\alpha_{n-1}^{\vee} 
+ 
\alpha_n^{\vee}$. 



\item
$i = n-1$, 
$c_{n-1}(\gamma^{\vee}) = 1$, 
$\lfloor wr_{\gamma} \rfloor (1) = w(n) = 1$, 
$\lfloor wr_{\gamma} \rfloor (2) = \sigma(w(n-1)) = 2$, 
$\lfloor wr_{\gamma} \rfloor (s) = w(s-2)$
for 
$s \in [3,n-1]$, 
and 
$\lfloor wr_{\gamma} \rfloor (n) = \sigma(w(n-2))$;
in this case, 
we have 
$\gamma^{\vee}
=
\varepsilon_{n-1} - \varepsilon_n
=
\alpha_{n-1}^{\vee}$.
\item
$i = n$, 
$c_n(\gamma^{\vee}) = 1$, 
$\lfloor wr_{\gamma} \rfloor (1) = \sigma(w(n)) = 1$, 
$\lfloor wr_{\gamma} \rfloor (2) = \sigma(w(n-1)) = 2$, 
and
$\lfloor wr_{\gamma} \rfloor (s) = w(s-2)$
for 
$s \in [3,n]$;
in this case, 
we have 
$\gamma^{\vee}
=
\varepsilon_{n-1} + \varepsilon_n
=
\alpha_n^{\vee}$.
\end{enumerate}
\end{prop}

Before starting the proof of Proposition \ref{prop:Q=D}, 
we mention a consequence of 
Lemma \ref{lem:Q=SiB} and Propositions \ref{prop:Bruhat-D}--\ref{prop:Q=D}.

\begin{prop} \label{prop:SiB-D}
Let 
$i \in I$, 
$x,y \in (W^{I \setminus \{ i \}})_{\af}$, 
$\mathcal{Y}_i^{D_n}(x) = (\mathsf{T},c)$, 
and 
$\mathcal{Y}_i^{D_n}(y) = (\mathsf{T}',c')$. 
There exists an edge 
$x \xrightarrow{\ \beta \ } y$
in 
$\mathrm{SiB}^{I \setminus \{ i \}}$
for some 
$\beta \in \Delta_{\af}^+$
if and only if one of the following conditions holds:

\begin{enumerate}[($\si$-D1)]
\item
$i \in [n-2]$, 
$c' = c$, 
and there exists 
$s \in [i]$
such that 
$\mathsf{T}'(u) = \mathsf{T}(u)$
for 
$u \in [i] \setminus \{ s \}$, 
$1 \preceq \mathsf{T}(s) \prec n$, 
and 
$\mathsf{T}'(s)
=
\min
([\mathsf{T}(s)+1,n]
\setminus 
\{ \|\mathsf{T}(u)\| \mid u \in [i], \ \mathsf{T}(u) \succeq \overline{n} \})$.
\item
$i \in [n-2]$, 
$c' = c$, 
and there exists 
$s \in [i]$
such that 
$\mathsf{T}'(u) = \mathsf{T}(u)$
for 
$u \in [i] \setminus \{ s \}$, 
$\overline{n} \prec \mathsf{T}(s) \preceq \overline{1}$, 
and 
$\mathsf{T}'(s)
=
\sigma
(\max
([1,\|\mathsf{T}(s)\|-1]
\setminus 
\{ \mathsf{T}(u) \mid u \in [i], \ \mathsf{T}(u) \preceq n \}))$.
\item
$i \in [n-2]$, 
$c' = c$, 
and there exists 
$s \in [i]$
such that 
$\mathsf{T}'(u) = \mathsf{T}(u)$
for 
$u \in [i] \setminus \{ s \}$, 
and 
$(\mathsf{T}(s),\mathsf{T}'(s))
\in 
\{(n-1,\overline{n}), 
(n,\overline{n-1})\}$.
\item
$i \in [2,n-2]$,
$c' = c$, 
and there exist 
$s,t \in [i]$
such that 
$s < t$, 
$\mathsf{T}'(u) = \mathsf{T}(u)$
for 
$[i] \setminus \{ s,t \}$
and 
$\mathsf{T}'(s) 
= 
\mathsf{T}(s) + 1 
= 
\sigma(\mathsf{T}(t)) 
= 
\sigma(\mathsf{T}'(t)) + 1 
\preceq n$.
\item
$i \in \{n-1,n\}$,
$c' = c$, 
and there exist 
$s,t \in [i]$
such that 
$s < t$, 
$\mathsf{T}'(u) = \mathsf{T}(u)$
for 
$[i] \setminus \{ s,t \}$
and 
$\mathsf{T}'(s) 
= 
\mathsf{T}(s) + 1 
= 
\sigma(\mathsf{T}(t)) 
= 
\sigma(\mathsf{T}'(t)) + 1 
\preceq n$.
\item
$i = n-1$,
$c' = c$, 
and there exists 
$s \in [n-2]$
such that 
$\mathsf{T}'(u) = \mathsf{T}(u)$
for 
$[n] \setminus \{ s,s+1 \}$
and 
$\sigma(\mathsf{T}'(s))
=
\sigma(\mathsf{T}'(s+1))+1
=
\mathsf{T}(s+1) 
= 
\mathsf{T}(s) + 1
\preceq n$.
\item
$i = n$,
$c' = c$, 
and there exists 
$s \in [n-1]$
such that 
$\mathsf{T}'(u) = \mathsf{T}(u)$
for 
$[n] \setminus \{ s,s+1 \}$
and 
$\sigma(\mathsf{T}'(s))
=
\sigma(\mathsf{T}'(s+1))+1
=
\mathsf{T}(s+1) 
= 
\mathsf{T}(s) + 1
\preceq n$.
\item
$i \in [n-2]$
and 
$c' = c+1$. 
If we write 
$\{ a_1 < a_2 < \cdots < a_{n-i+1} \}
=
[n] \setminus \{ \|\mathsf{T}(u)\| \mid u \in [i-1] \}$, 
then 
$a_1 = 1$, 
$\{ \sigma(\mathsf{T}(i)), \mathsf{T}'(1) \}
=
\{ 1,a_2 \}$
and 
$\mathsf{T}'(u) = \mathsf{T}(u-1)$
for 
$u \in [2,i]$.
\item
$i \in [2,n-2]$, 
$c' = c+2$, 
$\mathsf{T}'(1) = \sigma(\mathsf{T}(i)) = 1$, 
$\mathsf{T}'(2) = \sigma(\mathsf{T}(i-1)) = 2$, 
and 
$\mathsf{T}'(u) = \mathsf{T}(u-2)$
for 
$u \in [3,i]$.

\item
$i \in \{ n-1,n \}$, 
$c' = c+1$, 
$\mathsf{T}'(1) = \sigma(\mathsf{T}(n)) = 1$, 
$\mathsf{T}'(2) = \sigma(\mathsf{T}(n-1)) = 2$, 
and
$\mathsf{T}'(s) = \mathsf{T}(s-2)$
for 
$s \in [3,n]$.
\end{enumerate}
\end{prop}

\ytableausetup{mathmode,boxsize=5.5mm}
\begin{ex}

\begin{enumerate}[(1)]
\item
Let 
$i = 1$. 
($\si$-D8)
is equivalent to 
$c' = c+1$
and 
$\left(
\mathsf{T}, 
\mathsf{T}'
\right)
\in 
\left\{ 
\left(
\begin{ytableau} \overline{1} \end{ytableau},
\begin{ytableau} 2 \end{ytableau}
\right),
\left(
\begin{ytableau} \overline{2} \end{ytableau},
\begin{ytableau} 1 \end{ytableau}
\right)
\right\}$.
\item
Let 
$n = 4$
and 
$i = 2$.
($\si$-D8)
is equivalent to the condition that
$c' = c+1$
and 
$(\mathsf{T}, 
\mathsf{T}')$
equals one of the following:
\begin{align*} 
\left(
\begin{ytableau}
2 \\ \overline{3}
\end{ytableau}, 
\begin{ytableau}
1 \\ 2 
\end{ytableau}
\right),
\left(
\begin{ytableau}
3 \\ \overline{2}
\end{ytableau}, 
\begin{ytableau}
1 \\ 3
\end{ytableau}
\right), 
\left(
\begin{ytableau}
4 \\ \overline{2}
\end{ytableau}, 
\begin{ytableau} 
1 \\ 4
\end{ytableau}
\right),
\left(
\begin{ytableau}
\overline{4} \\ \overline{2}
\end{ytableau}, 
\begin{ytableau}
1 \\ \overline{4}
\end{ytableau}
\right), 
\left(
\begin{ytableau}
\overline{3} \\ \overline{2}
\end{ytableau}, 
\begin{ytableau}
1 \\ \overline{3}
\end{ytableau}
\right), \\
\left(
\begin{ytableau}
3 \\ \overline{1}
\end{ytableau}, 
\begin{ytableau}
2 \\ 3
\end{ytableau}
\right), 
\left(
\begin{ytableau}
4 \\ \overline{1}
\end{ytableau}, 
\begin{ytableau}
2 \\ 4
\end{ytableau}
\right),
\left(
\begin{ytableau}
\overline{4} \\ \overline{1}
\end{ytableau}, 
\begin{ytableau}
2 \\ \overline{4}
\end{ytableau}
\right),
\left(
\begin{ytableau}
\overline{3} \\ \overline{1}
\end{ytableau}, 
\begin{ytableau}
2 \\ \overline{3}
\end{ytableau}
\right), 
\left(
\begin{ytableau}
\overline{2} \\ \overline{1}
\end{ytableau}, 
\begin{ytableau} 
3 \\ \overline{2}
\end{ytableau}
\right) .
\end{align*}
($\si$-D9)
is equivalent to 
$c' = c+2$, 
$\mathsf{T}
=
\begin{ytableau}
\overline{2} \\ \overline{1}
\end{ytableau}$
and 
$\mathsf{T}'
=
\begin{ytableau}
1 \\ 2
\end{ytableau}$. 
\end{enumerate}
\end{ex}

We have divided the proof of 
Proposition \ref{prop:Q=D} 
into a sequence of lemmas.

\begin{lem} \label{lem:Q-i=1->D1}
$\mathrm{Q}(1,w,\gamma)$
implies 
(q-D1).
\end{lem}

\begin{proof}
Assume that 
$\mathrm{Q}(1,w,\gamma)$
is true.
By Lemmas \ref{lem:LS} and \ref{lem:J-ad-D}, 
we have 
$\gamma^{\vee} 
= 
\varepsilon_1 - \varepsilon_2 
\in 
Q^{\vee,I\setminus \{ 1 \}}$, 
$r_{\gamma} = (1 \ 2)(\overline{1} \ \overline{2})$, 
and 
$\lfloor wr_{\gamma} \rfloor = wr_{\gamma}(z_{\gamma^{\vee}}^{I \setminus \{ 1 \}})^{-1}$. 
Set 
$J = I \setminus \{ 1 \}$; 
note that 
$J$
is of type 
$D_{n-1}$.
We see that 
$2 \in J_{\af}$
satisfies the condition for 
$\gamma^{\vee} \in Q^{\vee}$
in 
Lemma \ref{lem:J-ad}; 
note that 
$J \setminus \{ 2 \}$
is of type 
$D_{n-2}$. 
Hence 
$z_{\gamma^{\vee}}^{I \setminus \{ 1 \}}
=
w_0^J w_0^{J \setminus \{ 2 \}}$
is given by 
$2 \mapsto \overline{2}$, 
$n \mapsto \overline{n}$, 
and 
$u \mapsto u$
for 
$u \in [n] \setminus \{ 2,n \}$. 
Then 
$\lfloor wr_{\gamma} \rfloor$
is given by 
$1 \mapsto w(2)$, 
$2 \mapsto w(\overline{1})$, 
$u \mapsto w(u)$
for 
$u \in [3,n-1]$, 
and 
$n \mapsto w(\overline{n})$. 
It follows from 
\eqref{eq:Weyl-grp-D}
and 
Lemma \ref{lem:Gr-D} that 
\begin{align} \label{eq:seq-D1-i=1}
\max\{ w(2), w(\overline{1})\}
\prec 
w(3) 
\prec 
w(4) 
\prec \cdots \prec 
w(n-1)
\prec 
\underbrace{w(\overline{n})}_{= n}
\prec 
w(\overline{n-1}).
\end{align}
Hence 
$\{ \lfloor wr_{\gamma} \rfloor(1) = w(2), 
w(\overline{1}) \}
=
\{ a_1 = 1, a_2 = 2 \}$. 
This implies (q-D1). 
\end{proof}

\begin{lem} \label{lem:D1-i=1->Q}
$i = 1$
and 
(q-D1)
imply
$\mathrm{Q}(1,w,\gamma)$.
\end{lem}

\begin{proof}
Assume that (q-D1) and $i = 1$; 
we have 
$a_1 = 1$, 
$a_2 = 2$, 
and 
$\{ 
w(\overline{i}),\lfloor wr_{\gamma} \rfloor (1)
\}
= 
\{ 1,2 \}$. 
We see from 
Lemmas \ref{lem:<g,r>} and \ref{lem:J-ad-D}--\ref{lem:<g,r>-D} that 
$\mathrm{Q}(1,w,\gamma)$ 
is equivalent to 
$\ell(\lfloor wr_{\gamma} \rfloor) - \ell(w)
=
3-2n$. 
We check at once that 
(q-D1), $i=1$, and Lemma \ref{lem:Gr-D}
yield
\eqref{eq:seq-D1-i=1}. 
If $w(1) = \overline{1}$, 
then 
$\ell(w) = 2n-2$, 
$\ell(\lfloor wr_{\gamma} \rfloor) = 1$, 
and hence 
$\ell(\lfloor wr_{\gamma} \rfloor) - \ell(w)
=
3-2n$. 
If $w(1) = \overline{2}$, 
then 
$\ell(w) = 2n-3$, 
$\ell(\lfloor wr_{\gamma} \rfloor) = 0$, 
and hence 
$\ell(\lfloor wr_{\gamma} \rfloor) - \ell(w)
=
3-2n$. 
\end{proof}

\begin{lem} \label{lem:Q-n-3->D1}
$n \geq 5$, 
$i \in [2,n-3]$, 
$c_i(\gamma^{\vee}) = 1$, 
and 
$\mathrm{Q}(i,w,\gamma)$
imply
(q-D1).
\end{lem}

\begin{proof}
Assume that 
$n \geq 5$, 
$i \in [2,n-3]$, 
and 
$c_i(\gamma^{\vee}) = 1$, 
and that
$\mathrm{Q}(i,w,\gamma)$
is true. 
By Lemmas \ref{lem:LS} and \ref{lem:J-ad-D}, 
we have
$\gamma^{\vee} 
= 
\varepsilon_i - \varepsilon_{i+1} 
\in 
Q^{\vee,I\setminus \{ i \}}$, 
$r_{\gamma} = (i \ i+1)(\overline{i} \ \overline{i+1})$, 
and 
$\lfloor wr_{\gamma} \rfloor = wr_{\gamma}(z_{\gamma^{\vee}}^{I \setminus \{ i \}})^{-1}$. 
Let 
$I \setminus \{ i \} = I_1 \sqcup I_2$, 
where 
$I_1 = [i-1]$
is of type 
$A_{i-1}$
and 
$I_2 = [i+1,n]$
is of type 
$D_{n-i}$.
We see that 
$(i-1,i+1) \in (I_1)_{\af} \times (I_2)_{\af}$
satisfies the condition for 
$\gamma^{\vee} \in Q^{\vee}$
in 
Lemma \ref{lem:J-ad}; 
note that 
$I_1 \setminus \{ i-1 \} = [i-2]$
is of type 
$A_{i-2}$
and 
$I_2 \setminus \{ i+1 \} = [i+2,n]$
is of type 
$D_{n-i-1}$. 
Hence 
$z_{\gamma^{\vee}}^{I \setminus \{ i \}} 
=
w_0^{I_1}w_0^{I_1 \setminus \{ i-1 \}}
w_0^{I_2}w_0^{I_2 \setminus \{ i+1 \}}
=
(1 \ 2 \ \cdots \ i)
(\overline{1} \ \overline{2} \ \cdots \ \overline{i})
(i+1 \ \overline{i+1})
(n \ \overline{n})$. 
Then 
$\lfloor wr_{\gamma} \rfloor$
is given by 
$1 \mapsto w(i+1)$, 
$u \mapsto w(u-1)$
for 
$u \in [2,i]$, 
$i+1 \mapsto w(\overline{i})$, 
$u \mapsto w(u)$
for 
$u \in [i+2,n-1]$, 
and 
$n \mapsto w(\overline{n})$. 
It follows from Lemma \ref{lem:Gr-D} that 
\begin{align} \label{eq:seq-D1-n-3}
\begin{split}
&w(i+1) 
\prec 
w(1)
\prec
w(2)
\prec \cdots \prec
w(i-1)
\prec
w(i), 
\\
&
\max\{
w(\overline{i}), 
w(i+1) 
\}
\prec
w(i+2)
\prec \cdots \prec
\underbrace{w(n-1)}_{\preceq n}
\prec
w(\overline{n})
\prec
w(\overline{n-1}).
\end{split}
\end{align}
Let 
$\{ a_1 < a_2 < \cdots < a_{n-i+1} \}
=
[n] \setminus \{ \|w(u)\| \mid u \in [i-1] \}$. 
By \eqref{eq:seq-D1-n-3}, 
we have 
$w(i+1) = \min \{ w(u) \mid u \in [n] \}$, 
which implies 
$a_1 = 1$. 
Since 
$[n] \setminus \{ \|w(u)\| \mid u \in [i-1] \}
=
\{ w(\overline{i}),w(i+1),w(i+2),\ldots ,w(n-1),\|w(n)\| \}$, 
we have 
$\{ w(\overline{i}), \lfloor wr_{\gamma} \rfloor(1) = w(i+1) \}
=
\{ a_1 = 1,a_2 \}$. 
\end{proof}

\begin{lem} \label{lem:D1-n-3->Q}
$n \geq 5$, 
$i \in [2,n-3]$, 
and 
(q-D1)
imply
$\mathrm{Q}(i,w,\gamma)$. 
\end{lem}

\begin{proof}
Assume that 
$n \geq 5$, 
$i \in [2,n-3]$, 
and 
(q-D1)
hold.
We see from 
Lemmas \ref{lem:<g,r>} and \ref{lem:J-ad-D}--\ref{lem:<g,r>-D} that 
$\mathrm{Q}(i,w,\gamma)$ 
is equivalent to 
$\ell(\lfloor wr_{\gamma} \rfloor) - \ell(w)
=
2-2n+i$. 
We check at once that 
(q-D1)
yields 
\eqref{eq:seq-D1-n-3}. 
We deduce from 
\eqref{eq:seq-D1-n-3}
that 
\begin{enumerate}[(1)]
\item
$\mathsf{a}_1(\lfloor wr_{\gamma} \rfloor)
=
\begin{cases}
1 & \text{if} \ w(i+1) \prec w(\overline{i}), \\
0 & \text{if} \ w(i+1) \succ w(\overline{i}),
\end{cases}$
\item
$\mathsf{b}_1(\lfloor wr_{\gamma} \rfloor)
=
\begin{cases}
0 & \text{if} \ w(i+1) \prec w(\overline{i}), \\
w(i+1)-2 & \text{if} \ w(i+1) \succ w(\overline{i}), 
\end{cases}$
\item
for 
$s \in [2,i]$
and 
$t \in [s+1,i]$, 
$t \notin \mathsf{A}_s(\lfloor wr_{\gamma} \rfloor)$
and 
$t-1 \notin \mathsf{A}_{s-1}(w)$, 
\item
for 
$s \in [2,i]$
and 
$t \in [s+1,i]$, 
$t \in \mathsf{B}_s(\lfloor wr_{\gamma} \rfloor)$
if and only if 
$t-1 \in \mathsf{B}_{s-1}(w)$, 
\item
for 
$s \in [2,i]$, 
$i+1 \in \mathsf{B}_s(\lfloor wr_{\gamma} \rfloor)$
if and only if 
$i \in \mathsf{A}_{s-1}(w)$, 
\item
for 
$s \in [2,i]$
and 
$t \in [i+2,n-1]$, 
$t \in \mathsf{A}_s(\lfloor wr_{\gamma} \rfloor)$
if and only if 
$t \in \mathsf{A}_{s-1}(w)$, 
\item
for 
$s \in [2,i]$
and 
$t \in [i+2,n-1]$, 
$t \in \mathsf{B}_s(\lfloor wr_{\gamma} \rfloor)$
if and only if 
$t \in \mathsf{B}_{s-1}(w)$, 
\item
for 
$s \in [2,i]$, 
$n \in \mathsf{A}_s(\lfloor wr_{\gamma} \rfloor)$
if and only if 
$n \in \mathsf{B}_{s-1}(w)$, 
\item
for 
$s \in [2,i]$, 
$n \in \mathsf{B}_s(\lfloor wr_{\gamma} \rfloor)$
if and only if 
$n \in \mathsf{A}_{s-1}(w)$, 
\item
$i+1 \in \mathsf{A}_{s-1}(w)$
for 
$s \in [2,i]$, 
\item
$i+1 \in \mathsf{B}_{s-1}(w)$
if and only if 
$s \in [i-w(i+1)+3,i]$, 
\item
$\mathsf{a}_i(w) = n-i$, \ 
$\mathsf{b}_i(w)
=
\begin{cases}
n-i-1 & \text{if} \ w(i+1) \prec w(\overline{i}), \\
n-i & \text{if} \ w(i+1) \succ w(\overline{i}).
\end{cases}$
\end{enumerate}
Hence
$\ell(\lfloor wr_{\gamma} \rfloor) - \ell(w)
=
2-2n+i$, 
which is our assertion.
\end{proof}

\begin{lem} \label{lem:Q-n-2->D1}
$n \geq 5$, 
$c_{n-2}(\gamma^{\vee}) = 1$, 
and 
$\mathrm{Q}(n-2,w,\gamma)$
imply
(q-D1).
\end{lem}

\begin{proof}
Assume that
$n \geq 5$
and 
$c_{n-2}(\gamma^{\vee}) = 1$, 
and that 
$\mathrm{Q}(n-2,w,\gamma)$
is true. 
By Lemmas \ref{lem:<g,r>} and \ref{lem:J-ad-D}--\ref{lem:<g,r>-D}, 
we have 
$\ell(\lfloor wr_{\gamma} \rfloor) - \ell(w)
=
-n$. 
By Lemmas \ref{lem:LS} and \ref{lem:J-ad-D}, 
we have 
$\gamma^{\vee} 
= 
\varepsilon_{n-2} - \varepsilon_{n-1}
\in 
Q^{\vee,I\setminus \{ n-2 \}}$, 
$r_{\gamma} = (n-2 \ n-1)(\overline{n-2} \ \overline{n-1})$, 
and 
$\lfloor wr_{\gamma} \rfloor 
= 
wr_{\gamma}(z_{\gamma^{\vee}}^{I \setminus \{ n-2 \}})^{-1}$. 
Let 
$I \setminus \{ n-2 \} = I_1 \sqcup I_2' \sqcup I_2''$, 
where 
$I_1 = [n-3]$
is of type 
$A_{n-3}$, 
$I_2' = \{ n-1 \}$
is of type 
$A_1$, 
and 
$I_2'' = \{ n \}$
is of type 
$A_1$.
We see that 
$(n-3,n-1,n) 
\in 
(I_1)_{\af} \times (I_2')_{\af} \times (I_2'')_{\af}$
satisfies the condition for 
$\gamma^{\vee} \in Q^{\vee}$
in 
Lemma \ref{lem:J-ad}; 
note that 
$I_1 \setminus \{ n-3 \}$
is of type 
$A_{n-4}$. 
Hence 
$z_{\gamma^{\vee}}^{I \setminus \{ n-2 \}}
=
w_0^{I_1} w_0^{I_1 \setminus \{ n-3 \}}
w_0^{I_2'} w_0^{I_2''}
=
(1 \ 2 \ \cdots \ n-2)
(\overline{1} \ \overline{2} \ \cdots \ \overline{n-2})
(n-1 \ \overline{n-1})
(n \ \overline{n})$.
Then 
$\lfloor wr_{\gamma} \rfloor$
is given by 
$1 \mapsto w(n-1)$, 
$u \mapsto w(u-1)$
for 
$u \in [2,n-2]$, 
$n-1 \mapsto w(\overline{n-2})$, 
and 
$n \mapsto w(\overline{n})$. 
It follows from Lemma \ref{lem:Gr-D} that 
\begin{align} \label{eq:seq-D1-n-2}
\begin{split}
&
w(n-1)
\prec 
w(1)
\prec 
w(2)
\prec \cdots \prec
w(n-3)
\prec 
w(n-2), \\
&
\max \{
w(\overline{n-2}), 
w(n-1)
\}
\prec 
w(\overline{n})
\prec 
w(\overline{n-1}).
\end{split}
\end{align}
Let 
$\{ a_1 < a_2 < a_3 \}
=
[n] \setminus \{ \| w(u) \| \mid u \in [n-3] \}
=
\{
\| w(\overline{n-2}) \| , 
w(n-1) , 
\| w(\overline{n}) \| \}$. 
We see from 
\eqref{eq:seq-D1-n-2}
that 
$a_1 = 1$
and 
$w(n-1) < \| w(\overline{n}) \|$. 
What is left is to show that 
$w(\overline{n-2}) \preceq n$
and 
$\|w(\overline{n-2})\| < \| w(\overline{n}) \|$. 
We have the following cases:
\begin{enumerate}[(i)]
\item
$w(\overline{n-2}) \prec w(n-1)$, 
\item
$w(\overline{n-2}) \succ w(n-1)$
and 
$w(n-2) \succ w(\overline{n})$, 
\item
$w(\overline{n-2}) \succ w(n-1)$
and 
$w(n-2) \prec w(\overline{n})$;
\end{enumerate}
we will prove that (i) or (ii) holds and these imply (q-D1). 
It follows from 
\eqref{eq:seq-D1-n-2}
that 
\begin{enumerate}[(1)]
\item
$\mathsf{a}_1(\lfloor wr_{\gamma} \rfloor)
=
\begin{cases}
1 & \text{if (i)}, \\
0 & \text{if (ii) or (iii)},
\end{cases}$ 
\item
$\mathsf{b}_1(\lfloor wr_{\gamma} \rfloor)
=
\begin{cases}
w(n-1)-2 & \text{if (i)}, \\
0 & \text{if (ii) or (iii)},
\end{cases}$
\item
for 
$s \in [2,n-2]$, 
$n-1 \in \mathsf{A}_s(\lfloor wr_{\gamma} \rfloor)$
if and only if 
$n-2 \in \mathsf{B}_{s-1}(w)$,
\item
for 
$s \in [2,n-2]$, 
$n \in \mathsf{A}_s(\lfloor wr_{\gamma} \rfloor)$
if and only if 
$n \in \mathsf{B}_{s-1}(w)$,
\item
for 
$s \in [2,n-2]$, 
$n-1 \notin \mathsf{B}_s(\lfloor wr_{\gamma} \rfloor)$
and
$n-1 \in \mathsf{A}_{s-1}(w)$,
\item
for 
$s \in [2,n-2]$, 
$n \in \mathsf{B}_s(\lfloor wr_{\gamma} \rfloor)$
if and only if 
$n \in \mathsf{A}_{s-1}(w)$,
\item
for 
$s \in [2,n-2]$
and 
$t \in [s+1,n-2]$, 
$t \in \mathsf{B}_s(\lfloor wr_{\gamma} \rfloor)$
if and only if 
$t-1 \in \mathsf{B}_{s-1}(w)$,
\item
for 
$s \in [2,n-2]$, 
$n-1 \in \mathsf{B}_{s-1}(w)$
if and only if 
$s \in [n-(w(n-1)-1),n-2]$,
\item
$\mathsf{a}_{n-2}(w) = 2$
and 
$\mathsf{b}_{n-2}(w)
=
\begin{cases}
2 & \text{if (i)}, \\
1 & \text{if (ii)}, \\
0 & \text{if (iii)}.
\end{cases}$
\end{enumerate}

If (i) holds, 
then 
$w(n-2) = \overline{1}$.
Hence 
$w(\overline{n-2}) \preceq n$
and 
$\|w(\overline{n-2})\| < \| w(\overline{n}) \|$. 

If (ii) holds, 
then 
$\lfloor wr_{\gamma} \rfloor (1)
=
w(n-1) 
= 
1$
and 
$w(\overline{n-2})
\prec
\min\{ w(n), w(\overline{n}) \}
\preceq 
n$. 
Hence 
$\|w(\overline{n-2})\| < \| w(\overline{n}) \|$. 

If (iii) holds, 
then 
$\ell(\lfloor wr_{\gamma} \rfloor)
-
\ell(w)
=
1-n
\neq 
-n$, 
a contradiction.
\end{proof}

\begin{lem} \label{lem:D1-n-2->Q}
$n \geq 5$, 
$i = n-2$, 
and 
(q-D1)
imply
$\mathrm{Q}(n-2,w,\gamma)$. 
\end{lem}

\begin{proof}
Assume that 
$n \geq 5$, 
$i = n-2$, 
and 
(q-D1)
hold.
By Lemmas \ref{lem:<g,r>} and \ref{lem:J-ad-D}--\ref{lem:<g,r>-D}, 
$\mathrm{Q}(n-2,w,\gamma)$
is equivalent to 
$\ell(\lfloor wr_{\gamma} \rfloor) - \ell(w)
=
-n$. 
We check at once that 
(q-D1) and $i = n-2$
imply
\eqref{eq:seq-D1-n-2}
and 
(i) or (ii)
in the proof of Lemma \ref{lem:Q-n-2->D1}. 
Then
(1)--(9)
in the proof of Lemma \ref{lem:Q-n-2->D1}
yield
$\ell(\lfloor wr_{\gamma} \rfloor) - \ell(w)
=
-n$. 
\end{proof}

\begin{lem} \label{lem:Q-n=4->D1}
$n = 4$, 
$c_2(\gamma^{\vee}) = 1$, 
and 
$\mathrm{Q}(2,w,\gamma)$
imply
(q-D1).
\end{lem}

\begin{proof}
Assume that 
$n = 4$
and 
$c_2(\gamma^{\vee}) = 1$, 
and that 
$\mathrm{Q}(2,w,\gamma)$
is true.
By Lemmas \ref{lem:LS} and \ref{lem:J-ad-D}, 
we have 
$\gamma^{\vee} 
= 
\varepsilon_2 - \varepsilon_3
\in 
Q^{\vee,I\setminus \{ 2 \}}$, 
$r_{\gamma} = (2 \ 3)(\overline{2} \ \overline{3})$, 
and 
$\lfloor wr_{\gamma} \rfloor = wr_{\gamma}(z_{\gamma^{\vee}}^{I \setminus \{ 2 \}})^{-1}$. 
Let 
$I \setminus \{ 2 \} = I_1 \sqcup I_2' \sqcup I_2''$, 
where 
$I_1 = \{ 1 \}$, 
$I_2' = \{ 3 \}$, 
and 
$I_2'' = \{ 4 \}$
are of type $A_1$.
We see that 
$(1,3,4) \in (I_1)_{\af} \times (I_2')_{\af} \times (I_2'')_{\af}$
satisfies the condition for 
$\gamma^{\vee} \in Q^{\vee}$
in 
Lemma \ref{lem:J-ad}.
Hence 
$z_{\gamma^{\vee}}^{I \setminus \{ 2 \}} 
=
r_1 r_3 r_4
=
(1 \ 2)(\overline{1} \ \overline{2})
(3 \ \overline{3})
(4 \ \overline{4})$. 
Then 
$\lfloor wr_{\gamma} \rfloor$
is given by 
$1 \mapsto w(3)$, 
$2 \mapsto w(1)$, 
$3 \mapsto w(\overline{2})$, 
and 
$4 \mapsto w(\overline{4})$. 
It follows from Lemma \ref{lem:Gr-D} that 
\begin{align} \label{eq:seq-D1-n=4}
w(3) \prec w(1) \prec w(2), \ \ 
w(3) \prec w(4) \prec w(\overline{3}), \ \ 
w(\overline{2}) \prec w(4) \prec w(2).
\end{align}
An easy computation shows that 
\eqref{eq:seq-D1-n=4}
implies 
(q-D1).  
\end{proof}

\begin{lem} \label{lem:D1-n=4->Q}
$n=4$,
$i = 2$, 
and  
(q-D1)
imply
$\mathrm{Q}(2,w,\gamma)$. 
\end{lem}

\begin{proof}
Assume that 
$n=4$,
$i = 2$, 
and  
(q-D1)
hold.
By Lemmas \ref{lem:<g,r>} and \ref{lem:J-ad-D}--\ref{lem:<g,r>-D}, 
$\mathrm{Q}(2,w,\gamma)$
is equivalent to 
$\ell(\lfloor wr_{\gamma} \rfloor) - \ell(w)
=
-4$. 
We see that (q-D1) implies \eqref{eq:seq-D1-n=4}. 
Hence 
$\mathsf{A}_1(\lfloor wr_{\gamma} \rfloor)
=
\{ 3 \}$, 
$\mathsf{A}_2(\lfloor wr_{\gamma} \rfloor)
=
\{ 3,4 \}$, 
$\mathsf{B}_1(\lfloor wr_{\gamma} \rfloor)
=
\{ 2 \}$, 
$\mathsf{B}_2(\lfloor wr_{\gamma} \rfloor)
=
\{ 4 \}$, 
$\mathsf{A}_1(w)
=
\{ 3,4 \}$, 
$\mathsf{A}_2(w)
=
\{ 3,4 \}$, 
$\mathsf{B}_1(w)
=
\{ 2,3,4 \}$, 
and 
$\mathsf{B}_2(w)
=
\{ 3,4 \}$. 
This implies
$\ell(\lfloor wr_{\gamma} \rfloor) - \ell(w)
=
-4$. 
\end{proof}

\begin{lem} \label{lem:Q-2->D2}
$n \geq 5$, 
$c_2(\gamma^{\vee}) = 2$,
and 
$\mathrm{Q}(2,w,\gamma)$
imply
(q-D2).
\end{lem}

\begin{proof}
Assume that 
$n \geq 5$
and 
$c_2(\gamma^{\vee}) = 2$,
and that 
$\mathrm{Q}(2,w,\gamma)$
is true. 
By Lemmas \ref{lem:<g,r>} and \ref{lem:J-ad-D}--\ref{lem:<g,r>-D}, 
we have 
$\ell(\lfloor wr_{\gamma} \rfloor) - \ell(w)
=
7-4n$. 
By Lemmas \ref{lem:LS} and \ref{lem:J-ad-D}, 
we have 
$\gamma^{\vee} 
= 
\varepsilon_1 + \varepsilon_2
\in 
Q^{\vee,I\setminus \{ 2 \}}$, 
$r_{\gamma} = (1 \ \overline{2})(\overline{1} \ 2)$, 
and 
$\lfloor wr_{\gamma} \rfloor 
= 
wr_{\gamma}(z_{\gamma^{\vee}}^{I \setminus \{ 2 \}})^{-1}$. 
Let 
$I \setminus \{ 2 \} = I_1 \sqcup I_2$, 
where 
$I_1 = \{ 1 \}$
is of type 
$A_1$
and 
$I_2 = [3,n]$
is of type 
$D_{n-2}$.
We see that 
$(0,0) \in (I_1)_{\af} \times (I_2)_{\af}$
satisfies the condition for 
$\gamma^{\vee} \in Q^{\vee}$
in 
Lemma \ref{lem:J-ad}. 
Hence 
$\lfloor wr_{\gamma} \rfloor 
=
wr_{\gamma}$
acts by 
$1 \mapsto w(\overline{2})$, 
$2 \mapsto w(\overline{1})$, 
and 
$u \mapsto w(u)$
for 
$u \in [3,n]$. 
The proof will be divided into three steps. 

\begin{proof}[Step 1]
We show that 
$w(1) \prec w(2) \preceq n$
leads to a contradiction. 
Suppose that 
$w(1) \prec w(2) \preceq n$. 
Then 
$w(n) \preceq n$, 
by \eqref{eq:Weyl-grp-D}. 
It follows from Lemma \ref{lem:Gr-D} that 
$\mathsf{a}_1(\lfloor wr_{\gamma} \rfloor)
=
\mathsf{a}_2(\lfloor wr_{\gamma} \rfloor)
=
n-2$, 
$\mathsf{b}_1(\lfloor wr_{\gamma} \rfloor)
=
n - w(2) + 1$, 
$\mathsf{b}_2(\lfloor wr_{\gamma} \rfloor)
=
n - w(1) - 1$, 
$\mathsf{a}_1(w)
=
w(1) - 1$, 
$\mathsf{a}_2(w)
=
w(2) - 2$, 
and 
$\mathsf{b}_1(w)
=
\mathsf{b}_2(w)
=
0$. 
Hence 
$\ell(\lfloor wr_{\gamma} \rfloor)
-
\ell(w)
=
4n-1-2w(1)-2w(2)
>
7-4n$, 
a contradiction. 
\end{proof}

\begin{proof}[Step 2]
We show that 
$w(1) \preceq n$
and 
$\overline{n} \preceq w(2)$
lead to a contradiction. 
Suppose that 
$w(1) \preceq n$
and 
$\overline{n} \preceq w(2)$. 
Then 
$w(\overline{n}) \preceq n$, 
by \eqref{eq:Weyl-grp-D}. 
We have the following cases: 
\begin{enumerate}[(i)]
\item
$w(1) \prec w(\overline{2}) \prec w(\overline{n}) \preceq n$, 
\item
$w(1) \prec w(\overline{n}) \prec w(\overline{2}) \preceq n$, 
\item
$w(\overline{2}) \prec w(1) \prec w(\overline{n}) \preceq n$, 
\item
$w(\overline{2}) \prec w(\overline{n}) \prec w(1) \preceq n$, 
\item
$w(\overline{n}) \prec w(1) \prec w(\overline{2}) \preceq n$, 
\item
$w(\overline{n}) \prec w(\overline{2}) \prec w(1) \preceq n$.
\end{enumerate}
It follows from Lemma \ref{lem:Gr-D} that 
\begin{enumerate}[(1)]
\item
$\mathsf{a}_1(\lfloor wr_{\gamma} \rfloor)
=
\begin{cases}
w(\overline{2}) - 1
&
\text{if (iii) or (iv)}, \\
w(\overline{2}) - 2
&
\text{if (i) or (vi)},  \\
w(\overline{2}) - 3
&
\text{if (ii) or (v)}, 
\end{cases}$ 
\item 
$\mathsf{a}_2(\lfloor wr_{\gamma} \rfloor)
=
\begin{cases}
n-2
&
\text{if (i), (ii), or (iii)}, \\
n-3
&
\text{if (iv), (v), or (vi)},
\end{cases}$
\item
$\mathsf{b}_1(\lfloor wr_{\gamma} \rfloor)
=
\begin{cases}
0
&
\text{if (iii) or (iv)}, \\
1
&
\text{if (i) or (vi)},  \\
2
&
\text{if (ii) or (v)}, 
\end{cases}$ 
\item 
$\mathsf{b}_2(\lfloor wr_{\gamma} \rfloor)
=
\begin{cases}
n-w(1)-1
&
\text{if (i) or (ii)}, \\
n-w(1)
&
\text{if (iii) or (v)}, \\
n-w(1)+1
&
\text{if (iv) or (vi)},
\end{cases}$
\item
$\mathsf{a}_1(w)
=
\begin{cases}
w(1)-3
&
\text{if (iv) or (vi)}, \\
w(1)-2
&
\text{if (iii) or (v)}, \\
w(1)-1
&
\text{if (i) or (ii)}, 
\end{cases}$ 
\item 
$\mathsf{a}_2(w)
=
\begin{cases}
n-3
&
\text{if (ii), (v) or (vi)}, \\
n-2
&
\text{if (i), (iii) or (iv)}, 
\end{cases}$
\item
$\mathsf{b}_1(w)
=
\begin{cases}
0
&
\text{if (i), (ii), or (iii)}, \\
1
&
\text{if (iv), (v) or (iv)}, 
\end{cases}$ 
\item 
$\mathsf{b}_2(w)
=
\begin{cases}
n-w(\overline{2})-1
&
\text{if (iii) or (iv)}, \\
n-w(\overline{2})
&
\text{if (i) or (vi)},  \\
n-w(\overline{2})+1
&
\text{if (ii) or (v)}.
\end{cases}$
\end{enumerate}
Hence 
\begin{align}
\ell(\lfloor wr_{\gamma} \rfloor) - \ell(w)
=
\begin{cases}
2w(\overline{2}) - 2w(1) -1 
&
\text{if (i), (ii), or (v)}, \\
2w(\overline{2}) - 2w(1) +2
&
\text{if (iii), (iv), or (vi)}.
\end{cases}
\end{align}
Since 
$\ell(\lfloor wr_{\gamma} \rfloor) - \ell(w)
=
7-4n$
is odd, 
we have 
$\ell(\lfloor wr_{\gamma} \rfloor) - \ell(w)
\neq 
2w(\overline{2}) - 2w(1) +2$. 
If 
(i), (ii), or (v)
hold, then
$2w(\overline{2}) - 2w(1) -1 > 0$, 
contrary to 
$\ell(\lfloor wr_{\gamma} \rfloor) - \ell(w)
=
7-4n < 0$. 
\end{proof}

\begin{proof}[Step 3]
By Steps 1--2, 
we have 
$\overline{n} \preceq w(1) \prec w(2)$; 
note that 
$w(n) \preceq n$, 
by \eqref{eq:Weyl-grp-D}. 
It remains to prove that 
$w(\overline{1}) = 2$
and 
$w(\overline{2}) = 1$.
It follows from Lemma \ref{lem:Gr-D} that 
$\mathsf{a}_1(\lfloor wr_{\gamma} \rfloor)
=
w(\overline{2}) - 1$, 
$\mathsf{a}_2(\lfloor wr_{\gamma} \rfloor)
=
w(\overline{1}) - 2$, 
$\mathsf{b}_1(\lfloor wr_{\gamma} \rfloor)
=
\mathsf{b}_2(\lfloor wr_{\gamma} \rfloor)
=
0$, 
$\mathsf{a}_1(w)
=
\mathsf{a}_2(w)
=
n-2$, 
$\mathsf{b}_1(w)
=
(n-1) - (w(\overline{1}) - 2)$, 
and 
$\mathsf{b}_2(w)
=
(n-2) - (w(\overline{2}) - 1)$. 
Hence 
$\ell(\lfloor wr_{\gamma} \rfloor)
-
\ell(w)
=
7-4n
+
2(w(\overline{2})-1)
+
2(w(\overline{1})-2)$; 
note that 
$w(\overline{2})-1 \geq 0$
and 
$w(\overline{1})-2 \geq 0$. 
Since 
$\ell(\lfloor wr_{\gamma} \rfloor)
-
\ell(w)
=
7-4n$, 
we conclude that 
$w(\overline{1}) = 2$
and 
$w(\overline{2}) = 1$.
\end{proof}

The proof of 
Lemma \ref{lem:Q-2->D2}
is complete. 
\end{proof}

\begin{lem} \label{lem:D2->Q-2}
$n \geq 5$, 
$i = 2$, 
and 
(q-D2)
imply
$\mathrm{Q}(2,w,\gamma)$.
\end{lem}

\begin{proof}
Assume that 
$n \geq 5$, 
$i = 2$, 
and 
(q-D2)
hold.
By Lemmas \ref{lem:<g,r>} and \ref{lem:J-ad-D}--\ref{lem:<g,r>-D}, 
$\mathrm{Q}(2,w,\gamma)$
is equivalent to 
$\ell(\lfloor wr_{\gamma} \rfloor) - \ell(w)
=
7-4n$. 
It follows from 
Lemma \ref{lem:Gr-D}
that 
$w = r_{\gamma} = (1 \ \overline{2})(\overline{1} \ 2)$
and 
$\lfloor wr_{\gamma} \rfloor = e$. 
We check at once that 
$\ell(\lfloor wr_{\gamma} \rfloor) - \ell(w)
=
7-4n$.
\end{proof}

\begin{lem} \label{lem:Q-n-3->D2}
$n \geq 5$, 
$i \in [3,n-3]$, 
$c_i(\gamma^{\vee}) = 2$, 
and 
$\mathrm{Q}(i,w,\gamma)$
imply
(q-D2).
\end{lem}

\begin{proof}
Assume that 
$n \geq 5$, 
$i \in [3,n-3]$, 
and 
$c_i(\gamma^{\vee}) = 2$, 
and that 
$\mathrm{Q}(i,w,\gamma)$
is true. 
By Lemmas \ref{lem:<g,r>} and \ref{lem:J-ad-D}--\ref{lem:<g,r>-D}, 
we have 
$\ell(\lfloor wr_{\gamma} \rfloor) - \ell(w)
=
3-4n+2i$. 
By Lemmas \ref{lem:LS} and \ref{lem:J-ad-D}, 
we have 
$\gamma^{\vee} 
= 
\varepsilon_{i-1} + \varepsilon_i
\in 
Q^{\vee,I\setminus \{ i \}}$, 
$r_{\gamma} = (i-1 \ \overline{i})(\overline{i-1} \ i)$, 
and 
$\lfloor wr_{\gamma} \rfloor 
= 
wr_{\gamma}(z_{\gamma^{\vee}}^{I \setminus \{ i \}})^{-1}$. 
Let 
$I \setminus \{ i \} = I_1 \sqcup I_2$, 
where 
$I_1 = [i-1]$
is of type 
$A_{i-1}$
and 
$I_2 = [i+1,n]$
is of type 
$D_{n-i}$.
We see that 
$(i-2,0) \in (I_1)_{\af} \times (I_2)_{\af}$
satisfies the condition for 
$\gamma^{\vee} \in Q^{\vee}$
in 
Lemma \ref{lem:J-ad}; 
note that 
$I_1 \setminus \{ i-2 \}$
is of type 
$A_{i-3} \times A_1$. 
Hence 
$z_{\gamma^{\vee}}^{I \setminus \{ i \}}
=
w_0^{I_1}w_0^{I_1 \setminus \{ i-2 \}}$
is given by 
$1 \mapsto i-1$, 
$2 \mapsto i$, 
$u \mapsto u-2$
for 
$u \in [3,i]$, 
and 
$u \mapsto u$
for 
$u \in [i+1,n]$.
Then 
$\lfloor wr_{\gamma} \rfloor$
is given by 
$1 \mapsto w(\overline{i})$, 
$2 \mapsto w(\overline{i-1})$, 
$u \mapsto w(u-2)$
for 
$u \in [3,i]$, 
and 
$u \mapsto w(u)$
for 
$u \in [i+1,n]$. 
It follows from Lemma \ref{lem:Gr-D} that 
\begin{align} \label{eq:seq-D2-n-3}
\begin{split}
&
w(\overline{i})
\prec 
\underbrace{w(\overline{i-1})}_{\preceq n}
\prec 
w(1)
\prec 
w(2)
\prec \cdots \prec
\underbrace{w(i-1)}_{\succeq \overline{n}}
\prec 
w(i), \\
&
w(i+1)
\prec 
w(i+2)
\prec \cdots \prec
\underbrace{w(n-1)}_{\preceq n}
\prec
w(n)
\prec 
w(\overline{n-1}). 
\end{split}
\end{align}
We have the following cases:
\begin{enumerate}[(i)]
\item
$w(n) \prec w(i-1) \prec w(i)$, 
\item
$w(i-1) \prec w(n) \prec w(i)$, 
\item
$w(i-1) \prec w(i) \prec w(n)$; 
\end{enumerate}
we will prove that (i) holds. 
It follows from \eqref{eq:seq-D2-n-3} that 
\begin{enumerate}[(1)]
\item
$\mathsf{a}_1(\lfloor wr_{\gamma} \rfloor) = w(\overline{i})-1$, 
$\mathsf{a}_2(\lfloor wr_{\gamma} \rfloor) = w(\overline{i-1})-2$, 
\item
$\mathsf{b}_1(\lfloor wr_{\gamma} \rfloor) 
= 
\begin{cases}
0 & \text{if (i) or (ii)}, \\
1 & \text{if (iii)},
\end{cases}$ \ 
$\mathsf{b}_2(\lfloor wr_{\gamma} \rfloor) 
= 
\begin{cases}
0 & \text{if (i)}, \\
1 & \text{if (ii) or (iii)},
\end{cases}$
\item
for 
$s \in [3,i]$, 
$\mathsf{a}_s(\lfloor wr_{\gamma} \rfloor)
=
\mathsf{a}_{s-2}(w)$
and 
$\mathsf{b}_s(\lfloor wr_{\gamma} \rfloor)
=
\mathsf{b}_{s-2}(w) - 2$, 
\item
$\mathsf{a}_{i-1}(w) 
= 
\begin{cases}
n-i & \text{if (i)}, \\
n-i-1 & \text{if (ii) or (iii)},
\end{cases}$ \ 
$\mathsf{a}_i(w) 
= 
\begin{cases}
n-i & \text{if (i) or (ii)}, \\
n-i-1 & \text{if (iii)},
\end{cases}$
\item
$\mathsf{b}_{i-1}(w)
=
n-i+1-(w(\overline{i-1})-2)$
and 
$\mathsf{b}_i(w)
=
n-i-(w(\overline{i})-1)$.
\end{enumerate}
Hence 
\begin{align}
\ell(\lfloor wr_{\gamma} \rfloor)
-
\ell(w)
=
\begin{cases}
3-4n+2i+2(w(\overline{i})-1)+2(w(\overline{i-1})-2)
&
\text{if (i)}, \\
5-4n+2i+2(w(\overline{i})-1)+2(w(\overline{i-1})-2)
&
\text{if (ii)}, \\
7-4n+2i+2(w(\overline{i})-1)+2(w(\overline{i-1})-2)
&
\text{if (iii)};
\end{cases}
\end{align}
note that 
$w(\overline{i})-1 \geq 0$
and 
$w(\overline{i-1})-2 \geq 0$. 
Since 
$\ell(\lfloor wr_{\gamma} \rfloor) - \ell(w)
=
3-4n+2i$, 
we have 
(i), 
$w(\overline{i})=1$, 
and 
$w(\overline{i-1})=2$.
This implies (q-D2).
\end{proof}

\begin{lem} \label{lem:D2-n-3->Q}
$n \geq 5$, 
$i \in [3,n-3]$, 
and 
(q-D2)
imply
$\mathrm{Q}(i,w,\gamma)$. 
\end{lem}

\begin{proof}
Assume that
$n \geq 5$, 
$i \in [3,n-3]$, 
and 
(q-D2)
hold. 
By Lemmas \ref{lem:<g,r>} and \ref{lem:J-ad-D}--\ref{lem:<g,r>-D}, 
$\mathrm{Q}(i,w,\gamma)$
is equivalent to 
$\ell(\lfloor wr_{\gamma} \rfloor) - \ell(w)
=
3-4n+2i$. 
We check at once that
(q-D2)
yields
$w(n) \prec w(i-1) \prec w(i)$
and 
\eqref{eq:seq-D2-n-3}. 
As in the proof of Lemma \ref{lem:Q-n-3->D2}, 
we have 
$\ell(\lfloor wr_{\gamma} \rfloor)
-
\ell(w)
=
3-4n+2i+2(w(\overline{1})-1)+2(w(\overline{i-1})-2)$.
Since 
$w(\overline{1})=1$
and 
$w(\overline{i-1})=2$, 
we conclude that 
$\ell(\lfloor wr_{\gamma} \rfloor) - \ell(w)
=
3-4n+2i$. 
\end{proof}

\begin{lem} \label{lem:Q-n-2->D2}
$n \geq 5$, 
$c_{n-2}(\gamma^{\vee}) = 2$,
and 
$\mathrm{Q}(n-2,w,\gamma)$
imply
(q-D2).
\end{lem}

\begin{proof}
Assume that 
$n \geq 5$
and 
$c_{n-2}(\gamma^{\vee}) = 2$,
and that 
$\mathrm{Q}(n-2,w,\gamma)$
is true. 
By Lemmas \ref{lem:<g,r>} and \ref{lem:J-ad-D}--\ref{lem:<g,r>-D}, 
we have 
$\ell(\lfloor wr_{\gamma} \rfloor) - \ell(w)
=
-1-2n$. 
By Lemmas \ref{lem:LS} and \ref{lem:J-ad-D}, 
we have
$\gamma^{\vee} 
= 
\varepsilon_{n-3} + \varepsilon_{n-2}
\in 
Q^{\vee,I\setminus \{ n-2 \}}$, 
$r_{\gamma} = (n-3 \ \overline{n-2})(\overline{n-3} \ n-2)$, 
and 
$\lfloor wr_{\gamma} \rfloor 
= 
wr_{\gamma}(z_{\gamma^{\vee}}^{I \setminus \{ n-2 \}})^{-1}$. 
Let 
$I \setminus \{ n-2 \} = I_1 \sqcup I_2' \sqcup I_2''$, 
where 
$I_1 = [n-3]$
is of type 
$A_{n-3}$, 
$I_2' = \{ n-1 \}$
is of type 
$A_1$, 
and 
$I_2'' = \{ n \}$
is of type 
$A_1$.
We see that 
$(n-4,0,0) 
\in 
(I_1)_{\af} \times (I_2')_{\af} \times (I_2'')_{\af}$
satisfies the condition for 
$\gamma^{\vee} \in Q^{\vee}$
in 
Lemma \ref{lem:J-ad}; 
note that 
$I_1 \setminus \{ n-4 \}$
is of type 
$A_{n-5} \times A_1$. 
Hence 
$z_{\gamma^{\vee}}^{I \setminus \{ n-2 \}}
=
w_0^{I_1} w_0^{I_1 \setminus \{ n-4 \}}$
is given by 
$u \mapsto u+2$
for 
$u \in [n-4]$, 
$n-3 \mapsto 1$, 
$n-2 \mapsto 2$, 
$n-1 \mapsto n-1$, 
and 
$n \mapsto n$. 
Then 
$\lfloor wr_{\gamma} \rfloor$
is given by 
$1 \mapsto w(\overline{n-2})$, 
$2 \mapsto w(\overline{n-3})$, 
$u \mapsto w(u-2)$
for 
$u \in [3,n-2]$, 
$n-1 \mapsto w(n-1)$, 
and 
$n \mapsto w(n)$. 
It follows from Lemma \ref{lem:Gr-D} that 
\begin{align} \label{eq:seq-D2-n-2}
\begin{split}
&
w(\overline{n-2})
\prec 
w(\overline{n-3})
\prec
w(1)
\prec
w(2)
\prec \cdots \prec
w(n-3)
\prec
w(n-2), \\
&
w(n-1)
\prec
w(n)
\prec
w(\overline{n-1}).
\end{split}
\end{align}
Analysis similar to that in the proof of Lemma \ref{lem:Q-n-3->D2} 
shows that 
$w(\overline{n-2}) = 1$
and 
$w(\overline{n-3}) = 2$.
This implies (q-D2).
\end{proof}

\begin{lem} \label{lem:D2-n-2->Q}
$n \geq 5$, 
$i = n-2$, 
and 
(q-D2)
imply
$\mathrm{Q}(n-2,w,\gamma)$. 
\end{lem}

\begin{proof}
Assume that 
$n \geq 5$, 
$i = n-2$, 
and 
(q-D2)
hold.
By Lemmas \ref{lem:<g,r>} and \ref{lem:J-ad-D}--\ref{lem:<g,r>-D}, 
$\mathrm{Q}(n-2,w,\gamma)$
is equivalent to 
$\ell(\lfloor wr_{\gamma} \rfloor) - \ell(w)
=
-1-2n$. 
We see that 
$i = n-2$
and 
(q-D2)
imply
\eqref{eq:seq-D2-n-2}. 
Hence 
$\mathsf{a}_1(\lfloor wr_{\gamma} \rfloor)
=
\mathsf{a}_2(\lfloor wr_{\gamma} \rfloor)
=
\mathsf{b}_1(\lfloor wr_{\gamma} \rfloor)
=
\mathsf{b}_2(\lfloor wr_{\gamma} \rfloor)
=
0$, 
$\mathsf{A}_s(\lfloor wr_{\gamma} \rfloor)
=
\mathsf{A}_{s-2}(w)$
for 
$s \in [3,n-2]$, 
$\mathsf{A}_{n-3}(w)
=
\mathsf{A}_{n-2}(w)
=
\mathsf{B}_{n-2}(w)
=
\{ n-1,n \}$, 
and 
$\mathsf{B}_{n-3}(w)
=
\{ n-2,n-1,n \}$. 
Also, 
for 
$s \in [3,n-2]$, 
we see that 
$n-3,n-2 \in \mathsf{B}_{s-2}(w)$
and the map below is bijective. 
\begin{align}
\mathsf{B}_s(\lfloor wr_{\gamma} \rfloor)
\to 
\mathsf{B}_{s-2}(w), \ 
\begin{cases}
t \mapsto t-2 
& \text{if} \ t \in [s+1,n-2], \\
t \mapsto t
& \text{if} \ t \in \{ n-1,n \}.
\end{cases}
\end{align} 
This implies 
$\mathsf{b}_s(\lfloor wr_{\gamma} \rfloor)
=
\mathsf{b}_{s-2}(w) - 2$
for 
$s \in [3,n-2]$. 
Combining these yields
$\ell(\lfloor wr_{\gamma} \rfloor) - \ell(w)
=
-1-2n$.
\end{proof}

\begin{lem} \label{lem:n=4-Q<->D2}
$n = 4$, 
$c_2(\gamma^{\vee}) = 2$, 
and 
$\mathrm{Q}(2,w,\gamma)$
are equivalent to 
(q-D2).
\end{lem}

\begin{proof}
The assertion follows by the same method as in the proof of 
Lemmas \ref{lem:Q-2->D2}--\ref{lem:D2->Q-2}; 
in this case, we have 
$w = (1 \ \overline{2})(\overline{1} \ 2)$
and 
$\lfloor wr_{\gamma} \rfloor = e$. 
\end{proof}

\begin{lem} \label{lem:Q->D3}
$\mathrm{Q}(n-1,w,\gamma)$
implies
(q-D3).
\end{lem}

\begin{proof}
Assume that 
$\mathrm{Q}(n-1,w,\gamma)$
is true. 
By Lemmas \ref{lem:LS} and \ref{lem:J-ad-D}, 
we have 
$\gamma^{\vee} 
= 
\varepsilon_{n-1} - \varepsilon_n
\in 
Q^{\vee,I\setminus \{ n-1 \}}$, 
$r_{\gamma} = (n-1 \ n)(\overline{n-1} \ \overline{n})$, 
and 
$\lfloor wr_{\gamma} \rfloor 
= 
wr_{\gamma}(z_{\gamma^{\vee}}^{I \setminus \{ n-1 \}})^{-1}$. 
Set 
$J \setminus \{ n-1 \}$; 
note that 
$J$
is of type $A_{n-1}$. 
We see that 
$n-2 \in J_{\af}$
satisfies the condition for 
$\gamma^{\vee} \in Q^{\vee}$
in 
Lemma \ref{lem:J-ad}; 
note that 
$J \setminus \{ n-2 \}$
is of type 
$A_{n-3} \times A_1$. 
Hence 
$z_{\gamma^{\vee}}^{I \setminus \{ n-1 \}}$
is given by 
$u \mapsto u+2$
for 
$u \in [n-3]$, 
$n-2 \mapsto \overline{n}$, 
$n-1 \mapsto 1$, 
and 
$n \mapsto \overline{2}$. 
Then 
$\lfloor wr_{\gamma} \rfloor$
is given by 
$1 \mapsto w(n)$, 
$2 \mapsto w(\overline{n-1})$, 
$u \mapsto w(u-2)$
for 
$u \in [3,n-1]$, 
and 
$n \mapsto w(\overline{n-2})$. 
It follows from Lemma \ref{lem:Gr-D} that 
\begin{align} \label{eq:seq-D3}
\begin{split}
w(n)
\prec 
w(\overline{n-1})
\prec 
w(1)
\prec 
w(2)
\prec \cdots \prec 
w(n-2)
\prec 
w(n-1)
\prec 
w(\overline{n}), 
\end{split}
\end{align}
which implies 
$w(n) = 1$
and 
$w(n-1) = \overline{2}$. 
This completes the proof. 
\end{proof}

\begin{lem} \label{lem:D3->Q}
(q-D3)
implies
$\mathrm{Q}(n-1,w,\gamma)$. 
\end{lem}

\begin{proof}
Assume that (q-D3) is true. 
By Lemmas \ref{lem:<g,r>} and \ref{lem:J-ad-D}--\ref{lem:<g,r>-D}, 
$\mathrm{Q}(n-1,w,\gamma)$
is equivalent to 
$\ell(\lfloor wr_{\gamma} \rfloor) - \ell(w)
=
3-2n$. 
We check at once that 
(q-D3)
implies 
\eqref{eq:seq-D3}.
It follows that 
\begin{enumerate}[(1)]
\item
$\mathsf{a}_1(\lfloor wr_{\gamma} \rfloor)
=
\mathsf{a}_2(\lfloor wr_{\gamma} \rfloor)
=
\mathsf{b}_1(\lfloor wr_{\gamma} \rfloor)
=
\mathsf{b}_2(\lfloor wr_{\gamma} \rfloor)
=
0$,
\item
for 
$s \in [3,n-1]$, 
$\mathsf{A}_s(\lfloor wr_{\gamma} \rfloor
=
\begin{cases}
\emptyset & \text{if} \ w(s-2) \prec w(\overline{n-2}), \\
\{ n \} & \text{if} \ w(s-2) \succ w(\overline{n-2}), \\
\end{cases}$ \ 
$\mathsf{A}_{s-2}(w)
=
\{ n \}$, 
\item
for 
$s \in [3,n-1]$, 
$\mathsf{b}_s(\lfloor wr_{\gamma} \rfloor)
-
\mathsf{b}_{s-2}(w)
=
\begin{cases}
-1 & \text{if} \ w(s-2) \prec w(\overline{n-2}), \\
-2 & \text{if} \ w(s-2) \succ w(\overline{n-2}), \\
\end{cases}$
\item
$\mathsf{A}_{n-2}(w)
=
\mathsf{A}_{n-1}(w)
=
\{ n \}$, 
$\mathsf{B}_{n-2}(w)
=
\{ n-1 \}$, 
$\mathsf{B}_{n-1}(w)
=
\emptyset$.
\end{enumerate}
We give the proof only for (3). 
Let 
$s \in [3,n-1]$. 
We have 
$n-1 \in \mathsf{B}_{s-2}(w)$
and 
$n \notin \mathsf{B}_{s-2}(w)$. 
We see that 
$n-2 \in \mathsf{B}_{s-2}(w)$
if and only if 
$w(s-2) \succ w(\overline{n-2})$. 
Then the map 
\begin{align}
\mathsf{B}_s(\lfloor wr_{\gamma} \rfloor)
\to 
\mathsf{B}_{s-2}(w)
\setminus 
\{ n-2,n-1 \}, \ 
t \mapsto t-2,
\end{align} 
is bijective. 
Combining these yields (3). 
By (1)--(4) above, 
we have 
$\ell(\lfloor wr_{\gamma} \rfloor)
-
\ell(w)
=
3-2n$.
\end{proof}

\begin{lem} \label{lem:Q->D4}
$\mathrm{Q}(n,w,\gamma)$
implies
(q-D4).
\end{lem}

\begin{proof}
Assume that 
$\mathrm{Q}(n,w,\gamma)$
is true. 
By Lemmas \ref{lem:LS} and \ref{lem:J-ad-D}, 
we have 
$\gamma^{\vee} 
= 
\varepsilon_{n-1} + \varepsilon_n
\in 
Q^{\vee,I\setminus \{ n \}}$, 
$r_{\gamma} = (n-1 \ \overline{n})(\overline{n-1} \ n)$, 
and 
$\lfloor wr_{\gamma} \rfloor = wr_{\gamma}(z_{\gamma^{\vee}}^{I \setminus \{ n \}})^{-1}$. 
Set 
$J = I \setminus \{ n \} = [n-1]$; 
note that 
$J$
is of type $A_{n-1}$. 
We see that 
$n-2 \in J_{\af}$
satisfies the condition for 
$\gamma^{\vee} \in Q^{\vee}$
in 
Lemma \ref{lem:J-ad}; 
note that 
$J \setminus \{ n-2 \}$
is of type 
$A_{n-3} \times A_1$. 
Hence 
$z_{\gamma^{\vee}}^{I \setminus \{ n \}}$
is given by 
$u \mapsto u+2$
for 
$u \in [n-2]$, 
$n-1 \mapsto 1$, 
and 
$n \mapsto 2$. 
Then 
$\lfloor wr_{\gamma} \rfloor$
is given by 
$1 \mapsto w(\overline{n})$, 
$2 \mapsto w(\overline{n-1})$, 
and 
$u \mapsto w(u-2)$
for 
$u \in [3,n]$. 
It follows from Lemma \ref{lem:Gr-D} that 
\begin{align} \label{eq:seq-D4}
\begin{split}
w(\overline{n})
\prec 
w(\overline{n-1})
\prec 
w(1)
\prec 
w(2)
\prec \cdots \prec 
w(n-2)
\prec 
w(n-1)
\prec 
w(n), 
\end{split}
\end{align}
which implies 
$w(n) = \overline{1}$
and 
$w(n-1) = \overline{2}$. 
This completes the proof. 
\end{proof}

\begin{lem} \label{lem:D4->Q}
(q-D4)
implies
$\mathrm{Q}(n,w,\gamma)$. 
\end{lem}

\begin{proof}
Assume that (q-D4) is true. 
By Lemmas \ref{lem:<g,r>} and \ref{lem:J-ad-D}--\ref{lem:<g,r>-D}, 
$\mathrm{Q}(n,w,\gamma)$
is equivalent to 
$\ell(\lfloor wr_{\gamma} \rfloor) - \ell(w)
=
3-2n$. 
We check at once that 
(q-D4)
implies 
\eqref{eq:seq-D4}.
Hence 
$\mathsf{A}_s(\lfloor wr_{\gamma} \rfloor)
=
\mathsf{A}_s(w)
=
\emptyset$
for 
$s \in [n]$, 
$\mathsf{B}_1(\lfloor wr_{\gamma} \rfloor)
=
\mathsf{B}_2(\lfloor wr_{\gamma} \rfloor)
=
\mathsf{B}_n(w)
=
\emptyset$, 
$\mathsf{B}_{n-1}(w)
=
\{ n \}$,
and the map 
$\mathsf{B}_s(\lfloor wr_{\gamma} \rfloor)
\to 
\mathsf{B}_{s-2}(w) \setminus \{ n-1,n \}$, 
$t \mapsto t-2$, 
is bijective for 
$s \in [3,n]$. 
Since 
$n-1,n \in \mathsf{B}_{s-2}(w)$
for 
$s \in [3,n]$, 
we have 
$\mathsf{b}_s(\lfloor wr_{\gamma} \rfloor)
=
\mathsf{b}_{s-2}(w) - 2$
for 
$s \in [3,n]$. 
It follows that 
$\ell(\lfloor wr_{\gamma} \rfloor)
-
\ell(w)
=
3-2n$.
\end{proof}

\begin{proof}[Proof of Proposition \ref{prop:tab-cri-D}]
(1) and (3)--(4) follow by the same method as in the proof 
of Proposition \ref{prop:tab-cri-A}. 

We prove (2). 
We can prove the assertion for 
$i \in \{1,n-1,n\}$
by a similar argument to the proof of 
Proposition \ref{prop:tab-cri-C}. 
Assume that
$i \in [2,n-2]$.
Let 
$x,y \in (W^{I \setminus \{ i \}})_{\af}$, 
$\mathcal{Y}_i^{D_n}(x) = (\mathsf{T},c)$, 
and 
$\mathcal{Y}_i^{D_n}(y) = (\mathsf{T}',c')$.
It follows immediately from 
Propositions \ref{prop:Bruhat-D}--\ref{prop:SiB-D} that 
$x \preceq y$
implies 
$c \leq c'$
and 
$\mathsf{T}(u) \preceq \mathsf{T}'(u+c'-c)$
for 
$u \in [i-c'+c]$. 
Hence we may assume that 
$d := c' - c \geq 0$
and 
$\mathsf{T}(u) \preceq \mathsf{T}'(u+d)$
for 
$u \in [i-d]$.
If 
$d$
is even, 
then the assertion follows by the same method 
as in Step 1 of the proof of Proposition \ref{prop:tab-cri-B}. 
In what follows, 
assume that $d \geq 1$ is odd. 
Write 
$a =
\min
\left(
[n]
\setminus 
\{ \| \mathsf{T}(u) \| 
\mid 
u \in [i], \ 
\mathsf{T}(u) \succeq \overline{n} \}
\right)$
and 
$b =
\min
\left(
[n]
\setminus 
\{ \| \mathsf{T}(u) \| 
\mid 
u \in [i] \}
\right)$
(see 
\eqref{eq:SiB-D-n-2}
in Definition 
\ref{def:SiB-D} (1)). 
Let 
$\{ a_1 < a_2 < \cdots < a_{n-i+1} \}
=
[n] \setminus \{ \| \mathsf{T}(u) \| \mid u \in [i-1] \}$. 
We divide the proof into four steps.

\begin{proof}[Step 1]
We show that 
$\mathsf{T}(i) \preceq n$
implies 
$x \prec y$. 
Define 
$\mathsf{T}_1 , \mathsf{T}_2 \in \mathrm{CST}_{D_n}(\varpi_i)$
by 
\begin{align}
\mathsf{T}_1(u)
&=
\begin{cases}
u+1 
&
\text{if} \ u \in [a_2-2], \\
\mathsf{T}(u)
&
\text{if} \ u \in [a_2-1,i-1], \\
\overline{a_2}
&
\text{if} \ u = i, 
\end{cases}
\\
\mathsf{T}_2(u)
&=
\begin{cases}
u 
&
\text{if} \ u \in [a_2-1], \\
\mathsf{T}(u-1)
&
\text{if} \ u \in [a_2,i].
\end{cases}
\end{align}
Let 
$x_1,x_2 \in (W^{I \setminus \{ i \}})_{\af}$
be such that
$\mathcal{Y}_i^{D_n}(x_1)
=
(\mathsf{T}_1,c)$
and 
$\mathcal{Y}_i^{D_n}(x_2)
=
(\mathsf{T}_2,c+1)$. 
Since 
$\mathsf{T}(u) \preceq \mathsf{T}_1(u)$
for 
$u \in [i]$, 
we have 
$x \preceq x_1$
by the assertion for $d = 0$. 
We have 
$x_1 \prec x_2$
by 
Proposition \ref{prop:SiB-D} ($\si$-D8). 
We have 
$x_2 \prec y$
because 
$c'-(c+1) = d-1 \geq 0$
is even 
and 
$\mathsf{T}_2(u) \preceq \mathsf{T}'(u+d-1)$
for 
$u \in [i-d+1]$. 
Combining these we conclude that 
$x \prec y$. 
\end{proof}

\begin{proof}[Step 2]
We show that 
($\mathsf{T}(i) \succeq \overline{n}$,
$a = b$, 
and 
$d > i$)
or 
($\mathsf{T}(i) \succeq \overline{n}$,
$a = b$, 
$d \in [i]$, 
and 
$a \preceq \mathsf{T}'(d)$)
imply
$x \prec y$. 
Assume that 
$\mathsf{T}(i) \succeq \overline{n}$
and 
$a = b$; 
note that 
$a_1 = 1$
holds.
We claim that 
$a \prec \mathsf{T}(1)$. 
Indeed, 
if 
$\mathsf{T}(1) \succeq \overline{n}$, 
then 
$a \prec \mathsf{T}(1)$. 
If 
$\mathsf{T}(1) \preceq n$, 
then 
$\mathsf{T}(1) 
\in 
[n] \setminus \{ \| \mathsf{T}(u) \| \mid u \in [i], \ \mathsf{T}(u) \succeq \overline{n} \}$
and hence 
$a \preceq \mathsf{T}(1)$
by minimality of $a$. 
Since 
$a = b \neq \|\mathsf{T}(u)\|$
for 
$u \in [i]$, 
we conclude that 
$a \prec \mathsf{T}(1)$. 
Thus 
$\mathsf{T}_2 \in \mathrm{CST}_{D_n}(\varpi_i)$ 
below is well-defined.
\begin{align}
\mathsf{T}_1(u)
&=
\begin{cases}
\mathsf{T}(u)
&
\text{if} \ u \in [i-1], \\
\max\{ \overline{a_2}, \mathsf{T}(i) \}
&
\text{if} \ u = i, 
\end{cases}
\\
\mathsf{T}_2(u)
&=
\begin{cases}
a 
&
\text{if} \ u = 1, \\
\mathsf{T}(u-1)
&
\text{if} \ u \in [2,i].
\end{cases}
\end{align}
Let 
$x_1,x_2 \in (W^{I \setminus \{ i \}})_{\af}$
be such that
$\mathcal{Y}_i^{D_n}(x_1)
=
(\mathsf{T}_1,c)$
and 
$\mathcal{Y}_i^{D_n}(x_2)
=
(\mathsf{T}_2,c+1)$. 
Write 
$\{ b_1 < b_2 < \cdots < b_{n-i+1} \}
=
[n] \setminus \{ \|\mathsf{T}_1(u)\| \mid u \in [i-1] \}$; 
note that 
$a_{\nu} = b_{\nu}$
for 
$\nu \in [n-i+1]$. 
Since 
$\mathsf{T}(u) \preceq \mathsf{T}_1(u)$
for 
$u \in [i]$, 
we have 
$x \preceq x_1$
by the assertion for $d = 0$. 
We claim that 
$x_1 \prec x_2$. 
Indeed, 
if 
$\|\mathsf{T}(i)\| = a_1 = 1$, 
then 
$\mathsf{T}_1(i) = \overline{1}$
and  
$\mathsf{T}_2(1) = a = a_2$, 
which imply
$b_1 = \|\mathsf{T}_1(i)\| = 1$
and 
$b_2 = \mathsf{T}_2(1)$. 
If 
$a_2 \leq \|\mathsf{T}(i)\|$, 
then 
$\mathsf{T}_1(i) = \overline{a_2}$
and 
$a = 1$, 
which imply
$b_1 = \mathsf{T}_2(1) = 1$
and 
$b_2 = \| \mathsf{T}_1(i)\|$. 
Proposition \ref{prop:SiB-D} ($\si$-D8)
now yields
$x_1 \prec x_2$
as claimed. 
We claim that 
$x_2 \preceq y$. 
Indeed, 
$c'-(c+1) = d-1 \geq 0$
is even. 
If 
$d > i$, 
then the condition 
\eqref{eq:T(u)<T'(u+d)-D}
is trivial, 
and hence 
$x_2 \prec y$. 
If 
$d \in [i]$
and 
$a \preceq \mathsf{T}'(d)$, 
then 
$\mathsf{T}_2(1) 
= 
a 
\preceq 
\mathsf{T}'(d) 
= 
\mathsf{T}'(1+(d-1))$, 
$\mathsf{T}_2(u)
=
\mathsf{T}(u-1)
\preceq
\mathsf{T}'(u+(d-1))$
for 
$u \in [2,i-d+1]$, 
and hence 
$x_2 \preceq y$
as claimed. 
Combining these we conclude that 
$x \prec y$. 
\end{proof}

\begin{proof}[Step 3]
Assume that 
$\mathsf{T}(i) \succeq \overline{n}$
and 
$a < b$; 
note that 
$\mathsf{T}(1) = a$
holds. 
Let
$k \in [i]$
be such that 
$\mathsf{T}(k) \prec b \prec \mathsf{T}(k+1)$. 
In this case, 
we show that the condition 
(iii)
in Definition \ref{def:SiB-D} (1)
implies 
$x \preceq y$. 
Define 
$\mathsf{T}_1 , \mathsf{T}_2 \in \mathrm{CST}_{D_n}(\varpi_i)$
by 
\begin{align}
\mathsf{T}_1(u)
&=
\begin{cases}
\mathsf{T}(u+1)
&
\text{if} \ u \in [k-1], \\
b
&
\text{if} \ u = k, \\
\mathsf{T}(u)
&
\text{if} \ u \in [k+1,i-1], \\
\max \{ \overline{a_2}, \mathsf{T}(i) \}
&
\text{if} \ u = i,
\end{cases}
\\
\mathsf{T}_2(u)
&=
\begin{cases}
a
&
\text{if} \ u=1, \\
\mathsf{T}_1(u-1)
&
\text{if} \ u \in [2,i].
\end{cases}
\end{align}
If 
$k = 1$, 
then 
$\mathsf{T}_2(2)
=
\mathsf{T}_1(1)
=
b
\succ 
a$. 
If 
$k \in [2,i]$, 
then 
$\mathsf{T}_2(2)
=
\mathsf{T}_1(1)
=
\mathsf{T}(2) 
\succ
\mathsf{T}(1) 
=
a$. 
Therefore
$\mathsf{T}_2$
is well-defined. 
Let 
$x_1,x_2 \in (W^{I \setminus \{ i \}})_{\af}$
be such that
$\mathcal{Y}_i^{D_n}(x_1)
=
(\mathsf{T}_1,c)$
and 
$\mathcal{Y}_i^{D_n}(x_2)
=
(\mathsf{T}_2,c+1)$. 
Write 
$\{ b_1 < b_2 < \cdots < b_{n-i+1} \}
=
[n] \setminus \{ \|\mathsf{T}_1(u)\| \mid u \in [i-1] \}$. 
Since 
$\mathsf{T}(u) \preceq \mathsf{T}_1(u)$
for 
$u \in [i]$, 
we have 
$x \preceq x_1$
by the assertion for $d = 0$. 
We claim that 
$x_1 \prec x_2$. 
Indeed, 
if 
$\|\mathsf{T}(i)\| = a_1 = 1$, 
then 
$b_1 = \|\mathsf{T}_1(i)\| = 1$
and 
$b_2 = \mathsf{T}_2(1)$. 
If 
$\|\mathsf{T}(i)\| = a_1 > 1$, 
then 
$\mathsf{T}_1(i) = \mathsf{T}(i)$, 
$a = 1$, 
and 
$b = a_2$, 
which imply
$b_1 = \mathsf{T}_2(1) = 1$
and 
$b_2 = \|\mathsf{T}_1(i)\|$. 
If 
$\|\mathsf{T}(i)\| \geq a_2$, 
then 
$\mathsf{T}_1(i) = \overline{a_2}$, 
$a = 1$, 
and 
$b = a_1$, 
which imply
$b_1 = \mathsf{T}_2(1) = 1$
and 
$b_2 = \|\mathsf{T}_1(i)\|$. 
Proposition \ref{prop:SiB-D} ($\si$-D8)
now yields
$x_1 \prec x_2$
as claimed. 
We claim that 
$x_2 \preceq y$. 
Indeed, 
$c' - (c+1)
=
d-1 \geq 0$
is even. 
If 
$d > i$, 
then the condition 
\eqref{eq:T(u)<T'(u+d)-D}
is trivial, 
and hence 
$x_2 \prec y$. 
If 
$d \in [i]$, 
$a \preceq \mathsf{T}'(d)$, 
($\mathsf{T}(u) \preceq \mathsf{T}'(u+d-1)$
for 
$u \in [2, \min \{ k , i-d+1 \}]$)
and 
($b \preceq \mathsf{T}'(k+d)$
if 
$k \in [i-d]$), 
then 
$\mathsf{T}_2(1) 
= 
a
\preceq 
\mathsf{T}'(d)
=
\mathsf{T}'(1+(d-1))$, 
$\mathsf{T}_2(u)
=
\mathsf{T}_1(u-1)
=
\mathsf{T}(u)
\preceq
\mathsf{T}'(u+(d-1))$
for 
$u \in [2,\min\{ k,i-d+1\}]$, 
$\mathsf{T}_2(k+1)
=
b
\preceq
\mathsf{T}'(k+d)
=
\mathsf{T}'(k+1+(d-1))$
if 
$k \in [i-d]$, 
$\mathsf{T}_2(u)
=
\mathsf{T}(u-1)
\preceq
\mathsf{T}'(u-1+d)
=
\mathsf{T}'(u+(d-1))$
for 
$u \in [k+2,i-d+1]$, 
and hence 
$x_2 \prec y$
as claimed. 
Combining these we conclude that 
$x \prec y$. 
\end{proof}

\begin{proof}[Step 4]
We show that 
$x \prec y$
and 
$\mathsf{T}(i) \succeq \overline{n}$
imply
$(\mathsf{T},c) \prec (\mathsf{T}',c')$. 
By Proposition \ref{prop:tab-cri-D} (3), 
we may assume that 
$d \in [i]$. 
Let 
$\mathsf{T}_1,\mathsf{T}_2 \in \mathrm{CST}_{D_n}(\varpi_i)$
and 
$x_1,x_2 \in (W^{I \setminus \{ i \}})_{\af}$
be as in Steps 2--3.
By the arguments in Steps 2--3, 
the edge 
$x_1 \to x_2$
is of type 
($\si$-D8). 
Note that 
$(\mathsf{T},c) \prec (\mathsf{T}',c')$
is equivalent to 
$\mathsf{T}_2(u)
\preceq 
\mathsf{T}'(u+d-1)$
for 
$u \in [i-d+1]$
(see (iii) in Definition \ref{def:SiB-D} (1)). 
Since $d$ is odd, 
we see from Proposition \ref{prop:SiB-D} that 
there exists an edge 
$x_3 \xrightarrow{\ \beta \ } x_4$, 
$x_3,x_4 \in (W^{I \setminus \{ i \}})_{\af}$, 
$\beta \in \Delta_{\af}^+$, 
in 
$\mathrm{SiB}^{I \setminus \{ i \}}$
of type ($\si$-D8) 
such that 
$x \preceq x_3$
and 
$x_4 \preceq y$; 
we may assume that there is no edges 
of type ($\si$-D8)
in a directed path from $x$ to $x_3$
in $\mathrm{SiB}^{I \setminus \{ i \}}$. 
Write 
$\mathcal{Y}_i^{D_n}(x_3) = (\mathsf{T}_3,c'')$
and 
$\mathcal{Y}_i^{D_n}(x_4) = (\mathsf{T}_4,c''+1)$. 
We see from Proposition \ref{prop:SiB-D} that 
$c''-c \geq 0$
and 
$c'-(c''+1) \geq 0$
are both even. 
Set 
$d' = c'' - c$
and 
$d'' = c'-(c''+1)$; 
note that 
$d'+d''+1 = d$
and 
$d',d'' \in [0,d-1]$. 
Since 
$x_4 \preceq y$, 
we have 
$\mathsf{T}_4(u)
\preceq 
\mathsf{T}'(u+d'')$
for 
$u \in [i-d'']$. 
Hence 
($\mathsf{T}_2(u)
\preceq 
\mathsf{T}'(u+d-1)$
for 
$u \in [i-d+1]$)
follows from 
($\mathsf{T}_2(u)
\preceq 
\mathsf{T}_4(u+d')$
for 
$u \in [i-d']$). 

We first claim that 
$\mathsf{T}_2(1)
\preceq 
\mathsf{T}_4(1)$
if 
$d' = 0$. 
Assume that 
$d' = 0$. 
We first assume that 
$\mathsf{T}(i) \prec \overline{1}$. 
Then 
$\mathsf{T}_1(i) \prec \overline{1}$
and 
$\mathsf{T}_2(1) = 1$
by 
($\si$-D8). 
Hence 
$\mathsf{T}_2(1) \preceq \mathsf{T}_4(1)$. 
We next assume that 
$\mathsf{T}(i) = \overline{1}$. 
Then 
$\mathsf{T}_3(i) = \overline{1}$, 
because 
$x \preceq x_3$
and 
$d' = 0$. 
If we write 
$\{ c_1 < c_2 < \cdots < c_{n-i+1} \}
=
[n] \setminus \{ \| \mathsf{T}_3(u) \| \mid u \in [i-1] \}$, 
then 
$c_1 = 1$
and 
$\mathsf{T}_4(1) = c_2$
by 
($\si$-D8). 
Also, we have 
$\mathsf{T}_2(1) = a = a_2$
by 
($\si$-D8); 
note that if 
$a < b$, 
then 
$a_2 = b_2$, 
where 
$b_2$
is as in Step 3, 
because 
$\mathsf{T}(1) = a$
and we assume that 
$\mathsf{T}(i) = \overline{1}$. 
By the definition of $a$, 
$\{ \mathsf{T}(u) \mid u \in [i-a_2+1,i] \}
=
\{ \overline{1}, \overline{2}, \ldots , \overline{a_2-1} \}$. 
Hence
$\{ \mathsf{T}_3(u) \mid u \in [i-a_2+1,i] \}
=
\{ \overline{1}, \overline{2}, \ldots , \overline{a_2-1} \}$, 
because 
$x \preceq x_3$
and 
$d' = 0$. 
This implies 
$a_2 \leq c_2$. 
Therefore
$\mathsf{T}_2(1)
\preceq 
\mathsf{T}_4(1)$.

We next claim that 
$\mathsf{T}_2(1)
\preceq 
\mathsf{T}_4(1+d')$
if 
$d' \neq 0$. 
Assume that 
$d' \in [d-1]$
or
$d' = 0$. 
Set 
\begin{align} \label{eq:p-q-D}
\begin{split}
\{ p_1 < p_2 < \cdots < p_{\mu} \}
&=
[b-1]
\setminus 
\{ \| \mathsf{T}(u) \| \mid u \in [i-d'], \ \mathsf{T}(u) \succeq \overline{n} \} , \\
\{ q_1 < q_2 < \cdots < q_{\nu} \}
&=
[b-1]
\setminus 
\{ \| \mathsf{T}_3(u) \| \mid u \in [i], \ \mathsf{T}_3(u) \succeq \overline{n} \} .\end{split}
\end{align}
If 
$[b-1]
\setminus 
\{ \| \mathsf{T}_3(u) \| \mid u \in [i], \ \mathsf{T}_3(u) \succeq \overline{n} \}
= 
\emptyset$, 
then 
$b
\preceq 
\mathsf{T}_3(u)$
for 
$u \in [i]$. 
In particular, 
we have 
$\mathsf{T}_2(1)
=
a
\preceq 
b
\preceq 
\mathsf{T}_3(d')
=
\mathsf{T}_4(1+d')$
if 
$d' \neq 0$, 
which is our claim. 
Therefore we may assume that 
$\nu \geq 1$. 
We see from 
($\si$-D8)
that 
$q_1 
\preceq 
\mathsf{T}_4(1)
\prec 
\mathsf{T}_4(2)
=
\mathsf{T}_3(1)$, 
which implies 
$q_u
\prec 
\mathsf{T}_3(u)$
for 
$u \in [\nu]$. 
In particular, 
we have 
$b \preceq \mathsf{T}_3(u)$
for 
$u \in [\nu,i]$. 
If 
$d' \in [\nu,i-1]$, 
then 
$\mathsf{T}_2(1)
=
a
\preceq 
b
\preceq 
\mathsf{T}_3(d')
=
\mathsf{T}_4(1+d')$, 
which is our claim. 
Therefore we may assume that 
$\nu \geq 2$
and 
$d' \in [0,\nu-1]$. 
Then 
($q_u
\prec 
\mathsf{T}_3(u)$
for 
$u \in [\nu]$)
gives
($q_{u+1}
\preceq 
\mathsf{T}_3(u)$
for 
$u \in [\nu-1]$). 
By minimality of $b$, 
we have 
$[b-1] \subset \{ \| \mathsf{T}(u) \| \mid u \in [i] \}$. 
Hence there exists 
$l \in [0,d']$
such that 
\begin{align}
\{ p_1 < p_2 < \cdots < p_{\mu} \} 
=
\{ \mathsf{T}(u) \mid u \in [\mu-l] \} 
\cup 
\{ \|\mathsf{T}(u)\| \mid u \in [i-l+1,i] \} ,
\end{align}
which implies 
$\mathsf{T}(u)
\preceq 
p_{u+l}$
for 
$u \in [\mu - l]$. 
Since 
$l \leq d'$, 
we have
$\mathsf{T}(u)
\preceq 
p_{u+d'}$
for 
$u \in [\mu - d']$. 
Since 
$x \preceq x_3$, 
we have 
$\| \mathsf{T}(u) \| \geq \| \mathsf{T}_3(u+d') \|$
if 
$u \in [i-d']$
and 
$\mathsf{T}(u) \succeq \overline{n}$. 
Hence 
$\mu \geq \nu$
and 
$p_u \leq q_u$
for 
$u \in [\nu]$. 
Combining these gives
\begin{align} \label{eq:T<p<q<T4-D}
\begin{split}
\mathsf{T}(u)
\preceq 
p_{u+d'}
\preceq 
q_{u+d'}
\preceq 
\mathsf{T}_3(u+d'-1)
&=
\mathsf{T}_4(u+d') \\
&\text{for} 
\ 
u \in [\nu - d']
\ 
\text{if}
\ 
d' \neq 0.
\end{split}
\end{align}
In particular, 
$\mathsf{T}_2(1)
=
a
\preceq 
\mathsf{T}(1)
\preceq 
\mathsf{T}_4(1+d')$
if 
$d' \neq 0$, 
which is our claim.
Similarly, 
we have 
\begin{align} \label{eq:T<p<q<T3-D}
\mathsf{T}(u+1)
\preceq 
\mathsf{T}_3(u)
\ 
\text{for}
\ 
u \in [\nu-1]
\ 
\text{if}
\ 
d' = 0.
\end{align} 

We next claim that 
$\mathsf{T}_1(u)
\preceq 
\mathsf{T}_3(u+d')$
for 
$u \in [i-d'-1]$. 
If 
$u \in [i-d'-1]$
satisfies 
$\mathsf{T}_1(u) = \mathsf{T}(u)$, 
then 
$\mathsf{T}_1(u) \preceq \mathsf{T}_3(u+d')$
since 
$x \preceq x_3$. 
It remains to prove that 
$\mathsf{T}_1(u) \preceq \mathsf{T}_3(u+d')$
for 
$u \in [k] \cap [i-d'-1]$
under the assumption that 
$a < b$, 
where 
$k \in [i]$
is as in Step 3. 
Recall that 
$b \preceq \mathsf{T}_3(u)$
for 
$u \in [\max \{ 1,\nu \},i]$. 
If 
$\nu \in [0,1]$, 
then 
$\mathsf{T}_1(u)
\preceq 
b
\preceq 
\mathsf{T}_3(u+d')$
for 
$u \in [k] \cap [i-d'-1]$.
Therefore we may assume that 
$\nu \geq 2$.
We first assume that 
$d' \in [d-1]$. 
If 
$u \in [k-1] \cap [\nu - d']$, 
then 
$\mathsf{T}_1(u)
=
\mathsf{T}(u+1)
\preceq 
\mathsf{T}_3(u+d')$
by 
\eqref{eq:T<p<q<T4-D}. 
If 
$k \in [\nu - d'-1]$, 
then 
$\mathsf{T}_1(k)
=
b
\prec 
\mathsf{T}(k+1)
\preceq
\mathsf{T}_3(k+d')$
by 
\eqref{eq:T<p<q<T4-D}. 
If 
$k = \nu - d'$, 
then 
$\mathsf{T}_1(k)
=
b
\preceq 
\mathsf{T}_3(\nu)
=
\mathsf{T}_3(k+d')$. 
If 
$u \in [\nu-d'+1,k]$, 
then 
$\mathsf{T}_1(u)
\preceq 
b
\preceq 
\mathsf{T}_3(u+d')$. 
We next assume that 
$d' = 0$. 
If 
$u \in [k-1] \cap [\nu - 1]$, 
then 
$\mathsf{T}_1(u)
=
\mathsf{T}(u+1)
\preceq 
\mathsf{T}_3(u)$
by 
\eqref{eq:T<p<q<T3-D}. 
If 
$k \in [\nu - 1]$, 
then 
$\mathsf{T}_1(k)
=
b
\prec
\mathsf{T}(k+1)
\preceq 
\mathsf{T}_3(k)$
by 
\eqref{eq:T<p<q<T3-D}. 
If 
$u \in [\nu , k]$, 
then 
$\mathsf{T}_1(u)
\preceq 
b
\preceq 
\mathsf{T}_3(u)$. 

We finally claim that 
$\mathsf{T}_2(u)
\preceq 
\mathsf{T}_4(u+d')$
for 
$u \in [2,i-d']$. 
Let 
$u \in [2,i-d']$; 
note that 
$u-1 \in [i-d'-1]$. 
By the above, 
we have 
$\mathsf{T}_2(u)
=
\mathsf{T}_1(u-1)
\preceq 
\mathsf{T}_3(u-1+d')
=
\mathsf{T}_4(u+d')$. 
\end{proof}

The proof of Proposition \ref{prop:tab-cri-D} is complete. 
\end{proof}

\section{Tableau model for crystal bases of level-zero representations} \label{Section:Tab-model}

In this section, 
we apply the results in \S \ref{Section:Tab-cri}
to crystal bases of level-zero representations of  
$\U$ 
of type 
$B_n^{(1)}$, 
$C_n^{(1)}$, 
and 
$D_n^{(1)}$. 
We introduce 
quantum Kashiwara--Nakashima columns 
(see Definitions 
\ref{def:QKN-B} 
and 
\ref{def:QKN-D})
and 
semi-infinite Kashiwara--Nakashima tableaux
(see Definitions \ref{def:SiKN-C}, 
\ref{def:SiKN-B},
and 
\ref{def:SiKN-D}). 
We will see that these tableaux give combinatorial models for 
crystal bases of 
level-zero fundamental representations
and 
level-zero extremal weight modules. 
When $\U$ is of type $B_n^{(1)}$ or $D_n^{(1)}$, 
we give an explicit description of the crystal isomorphisms 
among three different realizations of the crystal basis of 
a level-zero fundamental representation
by
quantum Lakshmibai--Seshadri paths (see \S \ref{Subsection:QLS}), 
quantum Kashiwara--Nakashima columns, 
and 
(ordinary) Kashiwara--Nakashima columns.

\subsection{Quantum Lakshmibai--Seshadri paths} 
\label{Subsection:QLS}

In this subsection, 
we give a brief exposition of 
quantum Lakshmibai--Seshadri paths
(see 
\cite{LNSSS16}
for details). 
Assume that 
$\U$
is of untwisted affine type. 

Let 
$\lambda \in P^+$. 
Recall the notation 
$J_{\lambda}$
in \eqref{eq:J_lambda}. 
For a rational number 
$0 < a \leq 1$, 
define 
$\mathrm{QB}(\lambda;a)$
to be the subgraph of 
$\mathrm{QB}^{J_{\lambda}}$ 
with the same vertex set 
but having only the edges of the form 
\begin{align}
w \xrightarrow{\ \gamma \ } v
\ \ \text{with} \ \ 
a\langle \gamma^{\vee} , \lambda \rangle \in \BZ ;
\end{align}
note that 
$\mathrm{QB}(\lambda ; 1) = \mathrm{QB}^{J_{\lambda}}$. 
A quantum Lakshmibai--Seshadri path of shape $\lambda$ is, by definition, 
a pair 
$(\bm{w}; \bm{a})$
of a sequence 
$\bm{w} : w_1, w_2, \ldots , w_l$ 
of elements in $W^{J_{\lambda}}$
and an increasing sequence
$\bm{a} : 0 = a_0 < a_1 < \cdots < a_l = 1$ 
of rational numbers such that 
there exists a directed path from
$w_{u+1}$ to $w_u$ in $\mathrm{QB}(\lambda;a_u)$ 
for $u \in [l-1]$. 
Let 
$\mathrm{QLS}(\lambda)$ 
denote the set of quantum Lakshmibai--Seshadri paths of shape $\lambda$. 
We call an element 
$(\bm{w};\bm{a}) \in \mathrm{QLS}(\lambda)$
a Lakshmibai--Seshadri path of shape $\lambda$
if there exists a directed path from
$w_{u+1}$ to $w_u$ in $\mathrm{QB}(\lambda;a_u)$ 
not having quantum edges 
for 
$u \in [l-1]$. 
Let 
$\mathrm{LS}(\lambda)$ 
denote the set of Lakshmibai--Seshadri paths of shape $\lambda$. 

In the same manner as in
\S \ref{Subsection:Path-Demazure}
we can define maps 
$\mathrm{wt} : \mathrm{QLS}(\lambda) \to P$, 
$e_i,f_i : \mathrm{QLS}(\lambda) \to \mathrm{QLS}(\lambda) \sqcup \{ \bm{0} \}$, 
and 
$\varepsilon_i,\varphi_i : \mathrm{QLS}(\lambda) \to \BZ_{\geq 0}$
for 
$i \in I_{\af}$. 

\begin{thm}[{\cite{J,K96,L95,LNSSS16}}] \label{thm:QLS=W}
Let $\lambda = \sum_{i \in I} m_i \varpi_i \in P^+$. 
\begin{enumerate}[(1)]
\item
The set 
$\mathrm{QLS}(\lambda)$ 
equipped with the maps 
$\mathrm{wt},e_i,f_i,\varepsilon_i,\varphi_i$, 
$i \in I_{\af}$, 
is a $\U'$-crystal. 
The $\U'$-crystal
$\mathrm{QLS}(\lambda)$ 
is isomorphic to the crystal basis of 
the tensor product
$\bigotimes_{i \in I} W(\varpi_i)^{\otimes m_i}$
of level-zero fundamental representations 
(see \S \ref{Subsection:extremal}).
\item
The set 
$\mathrm{LS}(\lambda)$ 
equipped with the maps 
$\mathrm{wt},e_i,f_i,\varepsilon_i,\varphi_i$, 
$i \in I$, 
is a $\Fg$-crystal. 
The $\Fg$-crystal
$\mathrm{LS}(\lambda)$ 
is isomorphic to the crystal basis of 
the integrable highest weight module of highest weight $\lambda$
over the quantized universal enveloping algebra 
associated with $\Fg$.
\end{enumerate}
\end{thm}

\subsection{Quantum Bruhat graphs and Maya diagrams}
\label{Subsection:QBG-Maya}

Throughout this subsection, 
we assume that 
$\U$
is of type 
$B_n^{(1)}$, 
$C_n^{(1)}$, 
or 
$D_n^{(1)}$. 
The aim of this subsection is to give a description of
$\mathrm{QLS}(\varpi_i)$
and 
$\mathrm{LS}(\varpi_i)$
in terms of Maya diagrams. 

Let 
$i \in I$. 
Note that 
$J_{\varpi_i}
=
I \setminus \{ i \}$. 
We see from 
$c_i(\gamma^{\vee})
=
\langle \gamma^{\vee},\varpi_i\rangle
\in 
\{ 0,1,2 \}$
for 
$\gamma \in \Delta^+$
that the graph
$\mathrm{QB}(\varpi_i;a)$
has at least one edge 
only if 
$a \in \{ 1/2,1\}$.
Hence each element of 
$\mathrm{QLS}(\varpi_i)$
is of the form 
$(w;0,1)$
or 
$(v,w;0,1/2,1)$
for some 
$w,v \in W^{I \setminus \{ i \}}$. 
For simplicity of notation, 
we write 
$(w,w) = (w;0,1)$
and 
$(v,w) = (v,w;0,1/2,1)$. 
The next lemma follows immediately from the results in
\S \ref{Section:Tab-cri}
(see
Propositions 
\ref{prop:Bruhat-C}, 
\ref{prop:Bruhat-B}--\ref{prop:Q=B}, 
and 
\ref{prop:Bruhat-D}--\ref{prop:Q=D}).

\begin{lem} \label{lem:QLS=LS}
\begin{enumerate}[(1)]
\item
If $\U$ is of type $C_n^{(1)}$, 
then 
$\mathrm{QLS}(\varpi_i)
=
\mathrm{LS}(\varpi_i)$.
\item
If 
($\U$ is of type $B_n^{(1)}$
and 
$i = n$), 
($\U$ is of type $C_n^{(1)}$
and 
$i = 1$), 
or 
($\U$ is of type $D_n^{(1)}$
and 
$i \in \{ 1,n-1,n \}$), 
then 
$w = v$
for all 
$(w,v) \in \mathrm{QLS}(\varpi_i)$, 
$\mathrm{QLS}(\varpi_i)
=
\mathrm{LS}(\varpi_i)$, 
and 
$W^{I \setminus \{ i \}}
\to
\mathrm{QLS}(\varpi_i)$, 
$w \mapsto (w;0,1)$, 
is bijective. 
\end{enumerate}
\end{lem}

We say that 
$i \in I$
is minuscule if 
$\langle \gamma^{\vee} , \varpi_i \rangle
\in 
\{ 0,1 \}$
for all
$\gamma \in \Delta^+$; 
we see that 
$i \in I$ 
is minuscule 
if and only if 
it satisfies the assumption of 
Lemma \ref{lem:QLS=LS} (2). 

Unless otherwise stated
we assume that 
$i \in I$
is not minuscule. 
We call a subset 
$J \subset [n]$
a segment if there exist
$j,k \in [n]$
such that 
$j \leq k$
and 
$J = [j,k]$. 
For segments 
$J = [j,k]$
and 
$J' = [j',k']$, 
we write 
$J < J'$
if
$k+1 < j'$. 
Let 
$\mathcal{S}_i$
be the family of all sequences 
$(J_1 < \cdots < J_{\mu})$, 
$\mu \geq 1$, 
of segments such that 
$\sum_{s=1}^{\mu} \# J_s = i$.
It is easy to check that 
for 
$w \in W$
there exists a unique 
$\mathcal{J}(w)
=
(J_1 < \cdots < J_{\mu})
\in \mathcal{S}_i$
such that 
$\{ \| w(u) \| \mid u \in [i] \}
=
J_1 \sqcup \cdots \sqcup J_{\mu}$. 
For 
$\mathcal{J}
=
(J_1 < \cdots < J_{\mu})
\in \mathcal{S}_i$, 
set 
$W^{I \setminus \{ i \}}[\mathcal{J}]
=
\{ w \in W^{I \setminus \{ i \}} \mid 
\mathcal{J}(w) = \mathcal{J} \}$. 
It follows that 
\begin{align} 
W^{I \setminus \{ i \}}
=
\bigsqcup_{\mathcal{J} \in \mathcal{S}_i}
W^{I \setminus \{ i \}}[\mathcal{J}] .
\end{align}
Set 
$2^{\mathcal{J}}
=
\prod_{\nu = 1}^{\mu} 2^{J_{\nu}}$, 
where 
$2^{J_{\nu}}$
denotes the power set of 
$J_{\nu}$. 
We call an element of 
$2^{\mathcal{J}}$
a Maya diagram on 
$\mathcal{J}$. 
We see that the next map is bijective.
\begin{align} \label{eq:W[J]->2^J}
\begin{split}
\mathcal{M} :
W^{I \setminus \{ i \}}[\mathcal{J}]
&\to 
2^{\mathcal{J}}, \\
w 
&\mapsto 
\mathcal{M}(w)
=
(J_{\nu} \cap \{ \|w(u)\| \mid u \in [i], \ 
w(u) \succeq \overline{n} \})_{\nu = 1}^{\mu} .
\end{split}
\end{align}

The next lemma follows immediately from the results in
\S \ref{Section:Tab-cri}
(see
(b-C3),
(b-B3), 
(b-B5), 
(q-B3),
(b-D4),
and 
(q-D2)
in 
Propositions 
\ref{prop:Bruhat-C}, 
\ref{prop:Bruhat-B}--\ref{prop:Q=B}, 
and 
\ref{prop:Bruhat-D}--\ref{prop:Q=D}).

\begin{lem} \label{lem:J(w)=J(v)}
Assume that 
$i \in I$
is not minuscule. 
Let 
$w,v \in W^{I \setminus \{ i \}}$. 
If 
$(v,w) \in \mathrm{QLS}(\varpi_i)$, 
then 
$\mathcal{J}(w) = \mathcal{J}(v)$. 
\end{lem}

Let
$\mathcal{J} \in \mathcal{S}_i$.
Define
$\mathrm{QB}(\varpi_i;1/2)[\mathcal{J}]$
to be the induced subgraph of 
$\mathrm{QB}(\varpi_i;1/2)$
with the vertex set 
$W^{I \setminus \{ i \}}[\mathcal{J}]$. 
It follows from Lemma \ref{lem:J(w)=J(v)} that 
\begin{align} \label{eq:QB=sum-QB[J]}
\mathrm{QB}(\varpi_i;1/2)
=
\bigsqcup_{\mathcal{J} \in \mathcal{S}_i}
\mathrm{QB}(\varpi_i;1/2)[\mathcal{J}].
\end{align}

Let 
$J \subset [n]$
be a segment. 
We concern the following conditions for Maya diagrams
$M,N \in 2^J$. 
\begin{enumerate}[(M1)]
\item
There exists 
$j \in [n-1]$
such that 
$j \in N$, 
$j+1 \in M$, 
and 
$M \setminus \{ j+1 \}
=
N \setminus \{ j \}$.
\item
$n \in N$
and 
$M = N \setminus \{ n \}$.
\item
$n-1,n \in N$
and 
$M = N \setminus \{ n-1,n \}$.
\item
$1,2 \in M$
and 
$N = M \setminus \{ 1,2 \}$.
\end{enumerate}

\begin{define} \label{def:Maya-1/2-J}
Assume that 
$i \in I$
is not minuscule. 
Let 
$J \subset [n]$
be a segment such that 
$\# J \leq i$.
For $W$ of type $B_n,C_n$, or $D_n$, 
define 
$\mathrm{M}(\varpi_i;1/2)[J]$
to be the directed graph 
whose vertex set is 
$2^J$
and edges are given as follows. 
Let 
$M,N \in 2^J$. 
\begin{enumerate}[(1)]
\item
Assume that 
$W$
is of type 
$B_n$. 
There exists an edge 
$M \to N$
if 
(M1), (M2), or (M4)
hold. 
\item
Assume that 
$W$
is of type 
$C_n$. 
There exists an edge 
$M \to N$
if 
(M1)
holds. 
\item
Assume that 
$W$
is of type 
$D_n$. 
There exists an edge 
$M \to N$
if 
(M1), (M3), or (M4)
hold. 
\end{enumerate}
Write 
$M 
\trianglelefteq 
N$
if there exists a directed path 
from $M$ to $N$
in 
$\mathrm{M}(\varpi_i;1/2)[J]$. 
Write 
$M 
\trianglelefteq'
N$
if there exists a directed path 
from $M$ to $N$
in 
$\mathrm{M}(\varpi_i;1/2)[J]$
not having edges of type (M4). 
We see that 
$\trianglelefteq$
and 
$\trianglelefteq'$
define partial orders on $2^J$; 
note that if 
($W$ is of type $C_n$)
or 
($\{ 1,2 \} \not\subset J$), 
then 
$\trianglelefteq$
is identical to
$\trianglelefteq'$. 
\end{define}

The next lemma is an easy consequence 
of the definition of 
partial orders
$\trianglelefteq$
and 
$\trianglelefteq'$
on 
$2^J$. 

\begin{lem} \label{lem:Maya-cri}
Assume that 
$i \in I$
is not minuscule. 
Let 
$J \subset [n]$
be a segment such that 
$\# J \leq i$.
Let 
$M,N \in 2^J$. 
Write 
$M = \{ m_1 < m_2 < \cdots < m_r \}$
and 
$N = \{ n_1 < n_2 < \cdots < n_s \}$. 
\begin{enumerate}[(1)]
\item
Assume that 
($W$
is of type $C_n$), 
($W$
is of type $B_n$, 
$n \notin J$, 
and 
$\{ 1,2 \} \not\subset J$), 
or 
($W$
is of type $D_n$
and 
$\{ 1,2 \} \not\subset J$). 
We have 
$M 
\trianglelefteq 
N$
if and only if 
$r = s$
and 
$m_{\nu} \geq n_{\nu}$
for 
$\nu \in [r]$. 
\item
Assume that 
$W$
is of type $B_n$
and 
$n \in J$. 
We have 
$M 
\trianglelefteq 
N$
if and only if 
$r \leq s$
and 
$m_{\nu} \geq n_{\nu}$
for 
$\nu \in [r]$. 
\item
Assume that 
$W$
is of type $D_n$
and 
$n-1,n \in J$. 
We have 
$M 
\trianglelefteq 
N$
if and only if 
$s - r \in 2\BZ_{\geq 0}$
and 
$m_{\nu} \geq n_{\nu}$
for 
$\nu \in [r]$. 
\item
Assume that 
($W$
is of type $B_n$, 
$n \notin J$, 
and 
$\{ 1,2 \} \subset J$)
or 
($W$
is of type $D_n$, 
$n-1,n \notin J$, 
and 
$\{ 1,2 \} \subset J$). 
We have 
$M 
\trianglelefteq 
N$
if and only if 
$r-s \in 2\BZ_{\geq 0}$
and 
$m_{r-\nu} \geq n_{s-\nu}$
for 
$\nu \in [0,s-1]$. 
We have 
$M 
\trianglelefteq'
N$
if and only if 
$r = s$
and 
$m_{\nu} \geq n_{\nu}$
for 
$\nu \in [r]$. 
\end{enumerate}
\end{lem}

\begin{define} \label{def:Maya-1/2}
Assume that 
$i \in I$
is not minuscule. 
Let 
$\mathcal{J}
=
(J_1 < \cdots < J_{\mu})
\in \mathcal{S}_i$. 
Define
$\mathrm{M}(\varpi_i;1/2)[\mathcal{J}]$
to be the directed graph 
whose vertex set is 
$2^{\mathcal{J}}$
and edges are given as follows: 
for 
$(M_{\nu}),
(N_{\nu}) 
\in 
2^{\mathcal{J}}$, 
set
$(M_{\nu}) \to (N_{\nu})$
if there exists 
$s \in [\mu]$
such that 
$M_s \to N_s$
in 
$\mathrm{M}(\varpi_i;1/2)[J_s]$
and 
$M_t = N_t$
for 
$t \in [\mu] \setminus \{ s \}$. 
Set
$(M_{\nu})
\trianglelefteq
(N_{\nu})$
(resp. 
$(M_{\nu})
\trianglelefteq'
(N_{\nu})$)
if 
$M_{\nu}
\trianglelefteq
N_{\nu}$
(resp. 
$M_{\nu}
\trianglelefteq'
N_{\nu}$)
for all
${\nu} \in [\mu]$. 
\end{define}

The next lemma is an immediate consequence of 
(b-C3),
(b-B3), 
(b-B5), 
(q-B3),
(b-D4),
and 
(q-D2)
in 
Propositions 
\ref{prop:Bruhat-C}, 
\ref{prop:Bruhat-B}--\ref{prop:Q=B}, 
and 
\ref{prop:Bruhat-D}--\ref{prop:Q=D}.

\begin{lem} \label{lem:QB[J]=prod-QB[J]}
Assume that 
$i \in I$
is not minuscule. 
Let 
$\mathcal{J}
=
(J_1 < \cdots < J_{\mu})
\in \mathcal{S}_i$. 
The map 
\eqref{eq:W[J]->2^J}
induces an isomorphism 
\begin{align}
\mathrm{QB}(\varpi_i;1/2)[\mathcal{J}]
\xrightarrow{\ \cong \ }
\mathrm{M}(\varpi_i;1/2)[\mathcal{J}]
\end{align}
of directed graphs. 
\end{lem}

We see that an edge in 
$\mathrm{M}(\varpi_i;1/2)[\mathcal{J}]$
of type 
(M1)--(M3)
(resp. (M4))
corresponds to a 
Bruhat edge
(resp. a quantum edge)
in 
$\mathrm{QB}^{I \setminus \{ i \}}$. 
By 
Lemma \ref{lem:QB[J]=prod-QB[J]}, 
we can define a partial order
$\trianglelefteq$
(resp. 
$\trianglelefteq'$)
on 
$W^{I \setminus \{ i \}}$
as follows. 
Let 
$w,v \in W^{I \setminus \{ i \}}$. 
If 
$\mathcal{J}(w)
\neq 
\mathcal{J}(v)$, 
then 
$w$
and 
$v$
are incomparable to each other. 
Assume that 
$\mathcal{J}(w)
=
\mathcal{J}(v)$. 
Recall the map 
$\mathcal{M}$
in 
\eqref{eq:W[J]->2^J}. 
Set 
$w \trianglelefteq v$
(resp. 
$w \trianglelefteq' v$)
if 
$\mathcal{M}(w) \trianglelefteq \mathcal{M}(v)$
(resp. 
$\mathcal{M}(w) \trianglelefteq' \mathcal{M}(v)$). 
We have thus proved that 
\begin{align}
\mathrm{QLS}(\varpi_i)
&=
\{ (v,w) \mid w,v \in W^{I \setminus \{ i \}}, \ 
w \trianglelefteq v \}, 
\label{eq:QLS}
\\
\mathrm{LS}(\varpi_i)
&=
\{ (v,w) \mid w,v \in W^{I \setminus \{ i \}}, \ 
w \trianglelefteq' v \} .
\label{eq:LS}
\end{align} 
From now on, 
we freely identify 
$w$, 
$\mathsf{T}_w^{(i)}$, 
and 
$\mathcal{M}(w)$
with each other for 
$w \in W^{I \setminus \{ i \}}$.

For 
$(v,w) \in \mathrm{QLS}(\varpi_i)$, 
let 
$d_i(v,w) \in \BZ_{\geq 0}$
be the number of edges of type (M4)
in a directed path from 
$w$
to 
$v$
in 
$\mathrm{QB}(\varpi_i;1/2)$; 
for convenience we define
$d_i(v,w) = 0$
for
$(v,w) \in \mathrm{QLS}(\varpi_i)$
if 
$i \in I$
is minuscule. 
We see from the next lemma that 
$d_i(v,w)$
is independent of the choice of a 
directed path from 
$w$
to 
$v$
in 
$\mathrm{QB}(\varpi_i;1/2)$. 

\begin{lem} \label{lem:d_i(v,w)}
Assume that 
$i \in I$
is not minuscule. 
Let 
$(v,w) \in \mathrm{QLS}(\varpi_i)$; 
we see from Lemma \ref{lem:J(w)=J(v)} that 
$\mathcal{J}(w) = \mathcal{J}(v)$. 
Write 
$\mathcal{J}(w) 
= 
\mathcal{J}(v) 
= 
\{ J_1 < \cdots < J_{\mu} \}$, 
$r
=
\#\{ u \in [i] 
\mid 
\|w(u)\| \in J_1, \ 
w(u) \succeq \overline{n} \}$, 
and 
$s
=
\#\{ u \in [i] 
\mid 
\|v(u)\| \in J_1, \ 
v(u) \succeq \overline{n} \}$;
note that if 
$\mathcal{M}(w) = (M_{\nu})$
and 
$\mathcal{M}(v) = (N_{\nu})$, 
then 
$r = \# M_1$
and 
$s = \# N_1$.
If 
($W$ is of type $C_n$)
or
($\{ 1,2 \} \not\subset J_1$), 
then 
$d_i(v,w) = 0$. 
If 
($W$ is of type $B_n$ or $D_n$)
and 
($\{ 1,2 \} \subset J_1$), 
then 
$d_i(v,w) = (r-s)/2$. 
\end{lem}

\begin{proof}
The assertion follows from Lemma \ref{lem:Maya-cri} (4).
\end{proof}

Let 
$\mathrm{QLS}(\varpi_i)_{\af}
=
\mathrm{QLS}(\varpi_i) \times \BZ$
be the affinization of the $\U'$-crystal 
$\mathrm{QLS}(\varpi_i)$
(see \S\ref{Subsection:Crystals}); 
we have 
$\mathrm{wt}((v,w),c)
=
(1/2)(v\varpi_i + w\varpi_i) - c\delta \in P_{\af}$
for 
$(v,w) \in \mathrm{QLS}(\varpi_i)$
and 
$c \in \BZ$.
The next lemma holds for all 
$i \in I$. 

\begin{lem} \label{lem:SMT-QLSaf}
Let 
$i \in I$
and 
$m \in \BZ_{\geq 0}$. 
\begin{enumerate}[(1)]
\item
We have an isomorphism 
\begin{align} \label{eq:QLSaf->SiLS}
\begin{split}
\mathrm{QLS}(\varpi_i)_{\af}
&\to 
\mathbb{B}^{\si}(\varpi_i), \\
((v,w),c)
&\mapsto 
\left(
vT_{(c+d_i(v,w))\alpha_i^{\vee}}^{I \setminus \{ i \}}, 
wT_{(c-d_i(v,w))\alpha_i^{\vee}}^{I \setminus \{ i \}};
0, \dfrac{1}{2}, 1
\right)
\end{split}
\end{align}
of $\U$-crystals, 
where 
$w,v \in W^{I \setminus \{ i \}}$
and 
$c \in \BZ$; 
we understand that 
$(x,x ; 0,1/2,1)
=
(x;0,1)$
for 
$x \in (W^{I \setminus \{ i \}})_{\af}$. 
\item
The crystal basis
$\mathcal{B}(m \varpi_i)$
is isomorphic to the subcrystal of 
$\mathrm{QLS}(\varpi_i)_{\af}^{\otimes m}$
consisting of the elements
$\bigotimes_{\nu=1}^{m}((v_{\nu},w_{\nu}),c_{\nu})$
such that 
\begin{align}
w_{\nu}T_{(c_{\nu}-d_i(v_{\nu},w_{\nu}))\alpha_i^{\vee}}^{I \setminus \{ i \}}
\succeq
v_{\nu+1}T_{(c_{\nu+1}+d_i(v_{\nu+1},w_{\nu+1}))\alpha_i^{\vee}}^{I \setminus \{ i \}}
\ \text{in} \ 
(W^{I \setminus \{ i \}})_{\af}
\end{align}
for 
$\nu \in [m-1]$.
\end{enumerate}
\end{lem}

\begin{proof}
(1): 
If $i \in I$ is minuscule, 
then the proof is straightforward 
(see also Remark \ref{rem:QLSas=B=SiLS} below). 

Assume that 
$i \in I$
is not minuscule. 
By Lemmas \ref{lem:Q=SiB} and \ref{lem:Maya-cri}--\ref{lem:d_i(v,w)}, 
we check at once that the map 
\eqref{eq:QLSaf->SiLS}
is well-defined, 
and is an isomorphism of 
$\U$-crystals. 

(2): 
The assertion follows from 
(1) 
and 
\cite[Theorem 3.4]{I20}. 
\end{proof}

\begin{rem} \label{rem:QLSas=B=SiLS}
Let 
$i \in I$. 
By Theorem \ref{thm:QLS=W} (1)
and 
\cite[Theorem 5.17 (vii)--(viii)]{K02}, 
we have a unique isomorphism
$\mathrm{QLS}(\varpi_i)_{\af}
\cong 
\mathcal{B}(\varpi_i)$
of 
$\U$-crystals. 
By Theorem \ref{thm:SiLS-isom} (2), 
we have a unique isomorphism 
$\mathcal{B}(\varpi_i)
\cong 
\mathbb{B}^{\si}(\varpi_i)$
of 
$\U$-crystals. 
The map
\eqref{eq:QLSaf->SiLS}
equals the composition of these isomorphisms.
\end{rem}

\subsection{Type $C_n^{(1)}$} \label{Subsection:tab-type-C}

Throughout this subsection, 
we assume that 
$\Fg$ and $W$ are of type $C_n$. 
Recall that 
$I = [n]$, 
$\Delta = \{ \pm(\varepsilon_s \pm \varepsilon_t) \mid s,t \in [n], \ s < t \} \sqcup \{ \pm 2\varepsilon_s \mid s \in [n] \}$, 
and
$\Pi 
=
\{ \alpha_s = \varepsilon_s - \varepsilon_{s+1} \mid s \in [n-1] \} 
\sqcup 
\{ \alpha_n = 2\varepsilon_n \}$. 
The highest root is 
$\theta 
= 
2\varepsilon_1 
= 
2\alpha_1 + \cdots + 2\alpha_{n-1} + \alpha_n$; 
we have 
$\theta^{\vee}
=
\varepsilon_1$. 
We identify
$\varpi_i$
with 
$\varepsilon_1 + \varepsilon_2 + \cdots + \varepsilon_i$
for
$i \in I$.

Let 
$i \in [n]$. 
A map 
$\mathsf{C} : [i] \to \mathcal{C}_n$
is, by definition, 
a Kashiwara--Nakashima $C_n$-column
(KN $C_n$-column for short) 
of shape $\varpi_i$ 
if 
\begin{enumerate}[(KN-C1)]
\item
$\mathsf{C}(1)
\prec
\mathsf{C}(2)
\prec
\cdots 
\prec 
\mathsf{C}(i)$, 
\item
if 
$t = \mathsf{C}(p) = \sigma(\mathsf{C}(q)) \in [n]$
for some 
$p,q \in [i]$, 
then 
$q-p > i-t$.
\end{enumerate}
Let 
$\mathrm{KN}_{C_n}(\varpi_i)$
be the set of 
KN $C_n$-columns 
of shape $\varpi_i$. 
We sometimes identify 
$\mathsf{C} \in \mathrm{KN}_{C_n}(\varpi_i)$
with its image 
$\{ \mathsf{C}(u) \mid u \in [i] \}
\subset 
\mathcal{C}_n$. 

Let 
$\mathsf{C} : [i] \to \mathcal{C}_n$
be a map satisfying 
(KN-C1).
Let 
$I_{\mathsf{C}}
=
\{ z_1 \succ z_2 \succ \cdots \succ z_k \}$
be the set of 
$z \in \mathcal{C}_n$
such that 
$z \preceq n$
and 
$\{ z,\overline{z} \}
\subset 
\mathsf{C}$. 
We say that 
$\mathsf{C}$
can be split if there exists a subset 
$J_{\mathsf{C}}
=
\{ y_1 > y_2 > \cdots > y_k \} \subset [n]$
such that 
\begin{enumerate}[(i)]
\item
$y_1
=
\max
\{ y \in \mathcal{C}_n \mid 
y \prec z_1, \ 
y \notin \mathsf{C}, \ 
\overline{y} \notin \mathsf{C} \}$,
\item
$y_{\nu}
=
\max
\{ y \in \mathcal{C}_n \mid 
y \prec \min \{ y_{\nu -1}, z_{\nu} \}, \ 
y \notin \mathsf{C}, \ 
\overline{y} \notin \mathsf{C} \}$
for 
$\nu \in [2,k]$.
\end{enumerate}
Define
$r\mathsf{C}, 
l\mathsf{C}
\in
\mathrm{CST}_{C_n}(\varpi_i)$
to be such that 
$r\mathsf{C}
=
(\mathsf{C} 
\setminus 
\{ \overline{z} \mid z \in I_{\mathsf{C}} \})
\cup 
\{ \overline{y} \mid y \in J_{\mathsf{C}} \}$
and 
$l\mathsf{C}
=
(\mathsf{C} 
\setminus 
I_{\mathsf{C}})
\cup 
J_{\mathsf{C}}$. 

Define a
$\Fg$-crystal structure on 
$\mathrm{KN}_{C_n}(\varpi_i)$
as follows
(cf. \cite[\S 4]{KN}). 
Let 
$\mathsf{C} \in \mathrm{KN}_{C_n}(\varpi_i)$.
If we set 
$\varepsilon_{\overline{j}} = -\varepsilon_j$
for 
$j \in [n]$, 
then the weight of 
$\mathsf{C}$
is 
\begin{align} \label{eq:KN-wt-C}
\mathrm{wt}(\mathsf{C})
=
\sum_{u \in [i]}
\varepsilon_{\mathsf{C}(u)} .
\end{align}
Let 
$j \in I$. 
Note that only the letters 
$j,j+1,\overline{j+1},\overline{j}$
may be changed in 
$\mathsf{C}$
when we apply
$e_j$
or
$f_j$. 
Moreover, 
the actions of 
$e_j$
and
$f_j$
are uniquely determined from
$\mathsf{C} \cap \{ j,j+1,\overline{j+1},\overline{j} \}$.
Hence, 
by omitting the letters 
not being in 
$\{ j,j+1,\overline{j+1},\overline{j} \}$, 
we can illustrate the actions of 
$f_j$
for 
$j \in [n-1]$
by 
\ytableausetup{mathmode,boxsize=1cm}
\begin{align} \label{eq:f-KN-C-1}
\begin{CD}
\begin{ytableau}
1
\end{ytableau}
@>{f_1}>>
\begin{ytableau}
2
\end{ytableau}
@>{f_2}>>
\cdots 
@>{f_{n-2}}>>
\begin{ytableau}
n-1
\end{ytableau}
@>{f_{n-1}}>>
\begin{ytableau}
n
\end{ytableau}
\end{CD} \ ,
\end{align}
\begin{align} \label{eq:f-KN-C-2}
\begin{CD}
\begin{ytableau}
\overline{1}
\end{ytableau}
@<{f_1}<<
\begin{ytableau}
\overline{2}
\end{ytableau}
@<{f_2}<<
\cdots 
@<{f_{n-2}}<<
\begin{ytableau}
\overline{n-1}
\end{ytableau}
@<{f_{n-1}}<<
\begin{ytableau}
\overline{n}
\end{ytableau}
\end{CD} \ ,
\end{align}
\begin{align} \label{eq:f-KN-C-3}
\begin{CD}
\begin{ytableau}
j \\ \overline{j+1}
\end{ytableau}
@>{f_j}>>
\begin{ytableau}
j+1 \\ \overline{j+1}
\end{ytableau}
@>{f_j}>>
\begin{ytableau}
j+1 \\ \overline{j}
\end{ytableau}
\end{CD} \ ,
\end{align}
\begin{align} \label{eq:f-KN-C-4}
\begin{CD}
\begin{ytableau}
j \\ j+1 \\ \overline{j+1}
\end{ytableau}
@>{f_j}>>
\begin{ytableau}
j \\ j+1 \\ \overline{j}
\end{ytableau}
\end{CD} \ , 
\hspace{1cm}
\begin{CD}
\begin{ytableau}
j \\ \overline{j+1} \\ \overline{j}
\end{ytableau}
@>{f_j}>>
\begin{ytableau}
j+1 \\ \overline{j+1} \\ \overline{j}
\end{ytableau}
\end{CD} \ ;
\end{align}
set 
$f_j \mathsf{C} = \bm{0}$
otherwise. 
Similarly, the action of 
$f_n$
is illustrated by 
\ytableausetup{mathmode,boxsize=7mm}
\begin{align} \label{eq:f-KN-C-5}
\begin{CD}
\begin{ytableau}
n
\end{ytableau}
@>{f_n}>>
\begin{ytableau}
\overline{n}
\end{ytableau}
\end{CD} \ ;
\end{align}
set 
$f_n \mathsf{C} = \bm{0}$
otherwise. 
The maps 
$e_j$, 
$j \in I$, 
are defined to be such that 
the condition (C6) in \S \ref{Subsection:Crystals}
holds. 
For 
$\mathsf{C} \in \mathrm{KN}_{C_n}(\varpi_i)$
and 
$j \in I$, 
set 
$\varepsilon_j(\mathsf{C})
=
\max \{ k \in \BZ_{\geq 0} \mid e_j^k \mathsf{C} \neq \bm{0} \}$
and 
$\varphi_j(\mathsf{C})
=
\max \{ k \in \BZ_{\geq 0} \mid f_j^k \mathsf{C} \neq \bm{0} \}$. 

The next lemma is a reformulation of 
\cite[Theorem A.1]{She}
in terms of Maya diagrams.

\begin{lem}
Assume that 
$\Fg$
and 
$W$
are of type $C_n$. 
For a map
$\mathsf{C} : [i] \to \mathcal{C}_n$
satisfying 
(KN-C1), 
we have 
$\mathsf{C}
\in 
\mathrm{KN}_{C_n}(\varpi_i)$
if and only if 
$\mathsf{C}$
can be split. 
The map 
\begin{align} \label{eq:KN->LS-C}
\mathrm{KN}_{C_n}(\varpi_i)
\to 
\mathrm{LS}(\varpi_i), \ 
\mathsf{C}
\mapsto 
(r\mathsf{C},l\mathsf{C}),
\end{align}
is an isomorphism of $\Fg$-crystals.
The inverse of 
\eqref{eq:KN->LS-C}
is given as follows. 
Let 
$(v,w) 
\in 
\mathrm{LS}(\varpi_i)$, 
$\mathcal{J}
=
\mathcal{J}(w) 
= 
\mathcal{J}(v) 
= 
(J_1 < \cdots < J_{\mu})
\in 
\mathcal{S}_i$
and 
$\mathcal{M}(w) = (M_{\nu}), 
\mathcal{M}(v) = (N_{\nu})
\in 
2^{\mathcal{J}}$. 
The inverse image of 
$(v,w)$
is
\begin{align}
\begin{split}
\mathsf{C} 
&= 
\{ v(u) \mid u \in [i], \ v(u) \preceq n \}
\cup
\{ w(u) \mid u \in [i], \ w(u) \succeq \overline{n} \} \\
&=
\bigcup_{\nu=1}^{\mu}
\left(
(J_{\nu} \setminus N_{\nu})
\cup
\{ \overline{z} \mid z \in M_{\nu} \}
\right) ;
\end{split}
\end{align}
we have 
$I_{\mathsf{C}}
=
\bigcup_{\nu = 1}^{\mu} (M_{\nu} \setminus N_{\nu})$
and
$J_{\mathsf{C}}
=
\bigcup_{\nu = 1}^{\mu} (N_{\nu} \setminus M_{\nu})$. 
\end{lem}

By Lemma \ref{lem:QLS=LS} (1), 
we have 
$\mathrm{QLS}(\varpi_i)
=
\mathrm{LS}(\varpi_i)
\cong 
\mathrm{KN}_{C_n}(\varpi_i)$. 
Hence 
$\mathrm{KN}_{C_n}(\varpi_i)$
inherits a
$\U'$-crystal structure from 
$\mathrm{QLS}(\varpi_i)$. 
We see that only the letters 
$1,\overline{1}$
may be changed in 
$\mathsf{C}$
when we apply
$e_0$
or
$f_0$, 
and the actions of 
$e_0$
and
$f_0$
are uniquely determined from
$\mathsf{C} \cap \{ 1,\overline{1} \}$.
The action of 
$f_0$
is illustrated by 
\begin{align}
\begin{CD}
\begin{ytableau}
\overline{1}
\end{ytableau}
@>{f_0}>>
\begin{ytableau}
1
\end{ytableau}
\end{CD} \ ;
\end{align}
set 
$f_0 \mathsf{C} = \bm{0}$
otherwise. 
The map 
$e_0$
is defined to be such that 
the condition (C6) in \S \ref{Subsection:Crystals}
holds. 

For a tuple 
$(\mathsf{C}_1 , \mathsf{C}_2 , \ldots , \mathsf{C}_m)$
of columns, let 
$\mathsf{C}_1 \mathsf{C}_2 \cdots \mathsf{C}_m$
denote the tableau whose $\nu$-th column is 
$\mathsf{C}_{\nu}$. 
Recall the partial order 
$\preceq$
on 
$\mathrm{CST}_{C_n}(\varpi_i) \times \BZ$
(see Definition \ref{def:SiB-C}).

\begin{define} \label{def:SiKN-C}
Let 
$i \in I$
and 
$m \in \BZ_{\geq 0}$. 
\begin{enumerate}[(1)]
\item
Let 
$\mathbb{T}
=
\left(
\mathsf{T}_1
\mathsf{T}_2
\cdots 
\mathsf{T}_m, 
(c_1,c_2,\ldots ,c_m)
\right)$, 
where
$\mathsf{T}_{\nu}
\in 
\mathrm{KN}_{C_n}(\varpi_1)
=
\mathrm{CST}_{C_n}(\varpi_1)$
and 
$c_{\nu} \in \BZ$
for 
$\nu \in [m]$. 
We call 
$\mathbb{T}$
a semi-infinite KN $C_n$-tableau of shape 
$m\varpi_1$
if
\begin{align}
(\mathsf{T}_{\nu},c_{\nu})
\succeq 
(\mathsf{T}_{\nu+1},c_{\nu+1})
\end{align}
in
$\mathrm{CST}_{C_n}(\varpi_1) \times \BZ$
for 
$\nu \in [m-1]$. 
\item
Assume that 
$i \in [2,n]$. 
Let 
$\mathbb{T}
=
\left(
\mathsf{C}_1
\mathsf{C}_2
\cdots 
\mathsf{C}_m, 
(c_1,c_2,\ldots ,c_m)
\right)$, 
where
$\mathsf{C}_{\nu}
\in 
\mathrm{KN}_{C_n}(\varpi_i)$
and 
$c_{\nu} \in \BZ$
for 
$\nu \in [m]$. 
We call 
$\mathbb{T}$
a semi-infinite KN $C_n$-tableau of shape 
$m\varpi_i$
if 
\begin{align}
(l\mathsf{C}_{\nu},c_{\nu})
\succeq 
(r\mathsf{C}_{\nu+1},c_{\nu+1})
\end{align}
in
$\mathrm{CST}_{C_n}(\varpi_i) \times \BZ$
for 
$\nu \in [m-1]$. 
\end{enumerate}
Let 
$\mathbb{Y}^{\si}_{C_n}(m\varpi_i)$
be the set of 
semi-infinite KN $C_n$-tableaux of shape
$m\varpi_i$. 
For 
$\lambda = \sum_{i \in I} m_i \varpi_i \in P^+$, 
set 
$\mathbb{Y}^{\si}_{C_n}(\lambda)
=
\prod_{i \in I} 
\mathbb{Y}^{\si}_{C_n}(m_i \varpi_i)$. 
We call an element of 
$\mathbb{Y}^{\si}_{C_n}(\lambda)$
a semi-infinite KN $C_n$-tableau
of shape $\lambda$.
\end{define}

Let 
$\mathrm{KN}_{C_n}(\varpi_i)_{\af}$
denote the affinization of the 
$\U'$-crystal 
$\mathrm{KN}_{C_n}(\varpi_i)$
(see \S\ref{Subsection:Crystals}). 
Combining
Theorem \ref{thm:BN-Kashiwara}, 
Proposition \ref{prop:tab-cri-C} (2), 
Lemma \ref{lem:SMT-QLSaf}, 
and 
Definition \ref{def:SiKN-C}
we obtain the following theorem; 
note that 
$d_i(v,w) = 0$
for all 
$i \in I$
and 
$(v,w) 
\in 
\mathrm{QLS}(\varpi_i)$
by 
Lemma \ref{lem:d_i(v,w)}.

\begin{thm} \label{thm:SiKN->B-C}
Assume that 
$\U$
is of type 
$C_n^{(1)}$. 
Let 
$\lambda = \sum_{i \in I} m_i \varpi_i \in P^+$. 
For each 
$i \in I$, 
the image of the map
\begin{align} \label{eq:map-YC->KNaf}
\begin{split}
&\mathbb{Y}^{\si}_{C_n}(m_i \varpi_i)
\to 
\mathrm{KN}_{C_n}(\varpi_i)_{\af}^{\otimes m_i}, \\
&\left(
\mathsf{C}_1
\mathsf{C}_2
\cdots 
\mathsf{C}_{m_i}, 
(c_1,c_2,\ldots ,c_{m_i})
\right)
\mapsto 
\bigotimes_{\nu \in [m_i]}
(\mathsf{C}_{\nu} , c_{\nu}),
\end{split}
\end{align}
is a $\U$-subcrystal. 
Hence we can define a 
$\U$-crystal structure on 
$\mathbb{Y}^{\si}_{C_n}(m_i \varpi_i)$
to be such that the map 
\eqref{eq:map-YC->KNaf}
is a strict embedding of 
$\U$-crystals. 
In particular, 
$\mathbb{Y}^{\si}_{C_n}(\lambda)$
is a $\U$-subcrystal of 
$\bigotimes_{i \in I}
\mathrm{KN}_{C_n}(\varpi_i)_{\af}^{\otimes m_i}$. 
Then 
$\mathbb{Y}^{\si}_{C_n}(\lambda)$
is isomorphic,
as a $\U$-crystal, 
to the crystal basis 
$\mathcal{B}(\lambda)$.
\end{thm}

\subsection{Type $B_n^{(1)}$} \label{Subsection:tab-type-B}

Throughout this subsection, 
we assume that 
$\Fg$ and $W$ are of type $B_n$. 
Recall that 
$I = [n]$, 
$\Delta = \{ \pm(\varepsilon_s \pm \varepsilon_t) \mid s,t \in [n], \ s < t \} \sqcup \{ \pm \varepsilon_s \mid s \in [n] \}$, 
and
$\Pi 
=
\{ \alpha_s = \varepsilon_s - \varepsilon_{s+1} \mid s \in [n-1] \} 
\sqcup 
\{ \alpha_n = \varepsilon_n \}$. 
The highest root is 
$\theta
=
\varepsilon_1 + \varepsilon_2 
= 
\alpha_1 + 2\alpha_2 + \cdots + 2\alpha_n$; 
we have
$\theta^{\vee} = \theta$. 
We identify 
$\varpi_i$
with 
$\varepsilon_1 + \varepsilon_2 + \cdots + \varepsilon_i$
if 
$i \in [1,n-1]$, 
and with 
$\dfrac{1}{2}(\varepsilon_1 + \varepsilon_2 + \cdots + \varepsilon_n)$
if 
$i = n$.

Set 
\begin{align}
\mathcal{B}_n
=
\left\{ 
1 \prec 2 \prec \cdots \prec n 
\prec 0 \prec 
\overline{n} \prec \cdots \prec \overline{2} \prec \overline{1} 
\right\} .
\end{align}
Define 
$\sigma : \mathcal{B}_n \to \mathcal{B}_n$
by 
$\sigma(0) = 0$, 
$\sigma(u) = \overline{u}$, 
and 
$\sigma(\overline{u}) = u$
for 
$u \in [n]$. 

Let 
$i \in [n-1]$. 
A map
$\mathsf{C} : [i] \to \mathcal{B}_n$
is, by definition, 
a Kashiwara--Nakashima $B_n$-column
(KN $B_n$-column for short) 
of shape $\varpi_i$ 
if 
\begin{enumerate}[(KN-B1)]
\item
$\mathsf{C}(1)
\preceq
\mathsf{C}(2)
\preceq
\cdots 
\preceq 
\mathsf{C}(i)$, 
\item
if 
$0 \prec \mathsf{C}(u)$
or 
$\mathsf{C}(u+1) \prec 0$, 
then 
$\mathsf{C}(u) \prec \mathsf{C}(u+1)$
for 
$u \in [i-1]$, 
\item
if 
$t = \mathsf{C}(p) = \sigma(\mathsf{C}(q)) \in [n]$
for some 
$p,q \in [i]$, 
then 
$q-p > i-t$; 
\end{enumerate}
note that 
$0 \in \mathcal{B}_n$ 
is the unique element that may appear in 
$\mathsf{C}$
more than once. 
Let 
$\mathrm{KN}_{B_n}(\varpi_i)$
be the set of 
KN $B_n$-columns of shape $\varpi_i$. 
We sometimes identify 
$\mathsf{C} \in \mathrm{KN}_{B_n}(\varpi_i)$
with the multiset 
$\{ \mathsf{C}(u) \mid u \in [i] \}$. 
For convenience we also denote by 
$\mathrm{KN}_{B_n}(\varpi_n)
=
\mathrm{CST}_{B_n}(\varpi_n)$. 

Let 
$i \in [n-1]$. 
Let 
$\mathsf{C} : [i] \to \mathcal{B}_n$
be a map satisfying (KN-B1)--(KN-B2).
Let 
$I_{\mathsf{C}}
=
\{ z_1 \succeq z_2 \succeq \cdots \succeq z_k \}$
be the multiset of 
$z \in \mathsf{C}$
such that 
$z \preceq 0$
and 
$\{ z,\sigma(z) \}
\subset 
\mathsf{C}$; 
note that 
$\{ \mathsf{C}(u) \mid u \in [i], \mathsf{C}(u) = 0 \}$
is a multisubset of 
$I_{\mathsf{C}}$. 
We say that 
$\mathsf{C}$
can be split if there exists a subset 
$J_{\mathsf{C}}
=
\{ y_1 > y_2 > \cdots > y_k \} \subset [n]$
such that 
\begin{enumerate}[(i)]
\item
$y_1
=
\max
\{ y \in \mathcal{B}_n \mid 
y \prec z_1, \ 
y \notin \mathsf{C}, \ 
\overline{y} \notin \mathsf{C} \}$,
\item
$y_{\nu}
=
\max
\{ y \in \mathcal{B}_n \mid 
y \prec \min \{ y_{\nu -1}, z_{\nu} \}, \ 
y \notin \mathsf{C}, \ 
\overline{y} \notin \mathsf{C} \}$
for 
$\nu \in [2,k]$.
\end{enumerate}
Define
$r\mathsf{C}, 
l\mathsf{C}
\in
\mathrm{CST}_{B_n}(\varpi_i)$
to be such that 
$r\mathsf{C}
=
(\mathsf{C} 
\setminus 
\{ \sigma(z) \mid z \in I_{\mathsf{C}} \})
\cup 
\{ \sigma(y) \mid y \in J_{\mathsf{C}} \}$
and 
$l\mathsf{C}
=
(\mathsf{C} 
\setminus 
I_{\mathsf{C}})
\cup 
J_{\mathsf{C}}$. 

Define a 
$\Fg$-crystal structure on 
$\mathrm{KN}_{B_n}(\varpi_i)$
for 
$i \in [n-1]$
as follows
(cf. \cite[\S 5]{KN}); 
for the definition of a $\Fg$-crystal structure on 
$\mathrm{KN}_{B_n}(\varpi_n)$
we refer the reader to 
\cite[\S 2.3]{Le03}.
The maps 
$\mathrm{wt},\varepsilon_j,\varphi_j$
for 
$j \in I$
and 
$e_j,f_j$
for 
$j \in [n-1]$
are defined in the same manner as 
those for $\mathrm{KN}_{C_n}(\varpi_i)$. 
Note that only the letters 
$n,0,\overline{n}$
may be changed in 
$\mathsf{C}$
when we apply
$e_n$
or
$f_n$. 
Moreover, 
the actions of 
$e_n$
and
$f_n$
are uniquely determined from the multiset 
$\{ \mathsf{C}(u) \mid u \in [i], \ 
\mathsf{C}(u) \in \{ n,0,\overline{n} \} \}$.
The action of 
$f_n$
is illustrated by 
\begin{align}
\begin{CD}
\begin{ytableau}
n
\end{ytableau}
@>{f_n}>>
\begin{ytableau}
0
\end{ytableau}
@>{f_n}>>
\begin{ytableau}
\overline{n}
\end{ytableau}
\end{CD} \ ,
\hspace{1cm}
\begin{CD}
\begin{ytableau}
n \\ 0 \\ \vdots \\ 0
\end{ytableau}
@>{f_n}>>
\begin{ytableau}
0 \\ \vdots \\ \vdots \\ 0
\end{ytableau}
@>{f_n}>>
\begin{ytableau}
0 \\ \vdots \\ 0 \\ \overline{n}
\end{ytableau}
\end{CD} \ ;
\end{align}
set 
$f_n \mathsf{C} = \bm{0}$
otherwise.
The map 
$e_n$
is defined to be such that 
the condition (C6) in \S \ref{Subsection:Crystals}
holds. 

The next lemma is a reformulation of 
\cite[Corollary 3.1.11 and Remark 3.1.13]{Le03}
in terms of Maya diagrams.

\begin{lem} \label{lem:KN->LS-B}
Assume that 
$\Fg$
and 
$W$
are of type $B_n$. 
\begin{enumerate}[(1)]
\item
The map 
$\mathrm{LS}(\varpi_n)
\to 
\mathrm{KN}_{B_n}(\varpi_n)
=
\mathrm{CST}_{B_n}(\varpi_n)$, 
$(w;0,1)
\mapsto 
\mathsf{T}_w^{(n)}$, 
is an isomorphism of 
$\Fg$-crystals. 
\item
Assume that 
$i \in [n-1]$. 
For a map 
$\mathsf{C} : [i] \to \mathcal{B}_n$
satisfying (KN-B1)--(KN-B2), 
we have 
$\mathsf{C}
\in 
\mathrm{KN}_{B_n}(\varpi_i)$
if and only if 
$\mathsf{C}$
can be split. 
The map 
\begin{align} \label{eq:KN->LS-B}
\mathrm{KN}_{B_n}(\varpi_i)
\to 
\mathrm{LS}(\varpi_i), \ 
\mathsf{C}
\mapsto 
(r\mathsf{C},l\mathsf{C}), 
\end{align}
is an isomorphism of $\Fg$-crystals.
The inverse of
\eqref{eq:KN->LS-B}
is given as follows. 
Let 
$(v,w) 
\in 
\mathrm{LS}(\varpi_i)$, 
$\mathcal{J}
=
\mathcal{J}(w) 
= 
\mathcal{J}(v) 
= 
(J_1 < \cdots < J_{\mu})
\in 
\mathcal{S}_i$, 
$\mathcal{M}(w) = (M_{\nu}), 
\mathcal{M}(v) = (N_{\nu})
\in 
2^{\mathcal{J}}$, 
and 
$f = \# N_{\mu} - \# M_{\mu} \in \BZ_{\geq 0}$.
The inverse image of 
$(v,w)$
is
\begin{align}
\begin{split}
\mathsf{C} 
&= 
\{ v(u) \mid u \in [i], \ v(u) \preceq n \}
\cup
\{ w(u) \mid u \in [i], \ w(u) \succeq \overline{n} \}
\cup
\{ \underbrace{0,0,\ldots ,0}_{f \ \text{times}} \} \\
&=
\bigcup_{\nu=1}^{\mu}
\left(
(J_{\nu} \setminus N_{\nu})
\cup
\{ \overline{z} \mid z \in M_{\nu} \}
\right)
\cup
\{ \underbrace{0,0,\ldots ,0}_{f \ \text{times}} \} ;
\end{split}
\end{align}
we have 
$I_{\mathsf{C}}
=
\bigcup_{\nu = 1}^{\mu} (M_{\nu} \setminus N_{\nu})
\cup
\{ \underbrace{0,0,\ldots ,0}_{f \ \text{times}} \}$
and
$J_{\mathsf{C}}
=
\bigcup_{\nu = 1}^{\mu} (N_{\nu} \setminus M_{\nu})$. 
\end{enumerate}
\end{lem}

Set 
\begin{align}
\Tilde{\mathcal{B}}_n
=
\mathcal{B}_n
\cup
\{ \overline{0} \}
=
\left\{ 
1 \prec 2 \prec \cdots \prec n 
\prec 0 \prec 
\overline{n} \prec \cdots \prec \overline{2} \prec \overline{1} 
\prec 
\overline{0}
\right\} .
\end{align}

\begin{define} \label{def:QKN-B}
Let 
$i \in [n-1]$. 
A map
$\Tilde{\mathsf{C}} : [i] \to \Tilde{\mathcal{B}}_n$
is, by definition, 
a quantum Kashiwara--Nakashima $B_n$-column
(QKN $B_n$-column for short) 
of shape $\varpi_i$ 
if 
\begin{enumerate}[(QKN-B1)]
\item
there exists 
$m \in \BZ_{\geq 0}$
such that 
$2m \leq i$
and 
\begin{align}
\begin{split}
\Tilde{\mathsf{C}}(1)
&\preceq
\Tilde{\mathsf{C}}(2)
\preceq
\cdots 
\preceq
\Tilde{\mathsf{C}}(i-2m) \\
&\prec
\overline{0}
=
\Tilde{\mathsf{C}}(i-2m+1)
=
\cdots 
=
\Tilde{\mathsf{C}}(i-1)
=
\Tilde{\mathsf{C}}(i),
\end{split}
\end{align} 
\item
the map 
$\mathsf{C} : [i-2m] \to \mathcal{B}_n$, 
$u \mapsto \Tilde{\mathsf{C}}(u)$, 
is a KN $B_n$-column of shape $\varpi_{i-2m}$; 
\end{enumerate}
in this case, 
we write 
$\Tilde{\mathsf{C}}
=
\mathsf{C}
\cup
\{ \underbrace{\overline{0},\overline{0},\ldots ,\overline{0}}_{2m \ \text{times}} \}$
for brevity.
Let 
$\mathrm{QKN}_{B_n}(\varpi_i)$
be the set of 
QKN $B_n$-columns of shape $\varpi_i$. 
We sometimes identify 
$\Tilde{\mathsf{C}} \in \mathrm{QKN}_{B_n}(\varpi_i)$
with the multiset 
$\{ \Tilde{\mathsf{C}}(u) \mid u \in [i] \}$. 
For convenience we also denote by 
$\mathrm{QKN}_{B_n}(\varpi_n)
=
\mathrm{CST}_{B_n}(\varpi_n)$. 

Let 
$i \in [n-1]$. 
Let
$\Tilde{\mathsf{C}}
\in 
\mathrm{QKN}_{B_n}(\varpi_i)$, 
and let
$\mathsf{C}
\in 
\mathrm{KN}_{B_n}(\varpi_{i-2m})$
be as in 
(QKN-B2). 
Write 
$\{ x_1 < x_2 < \cdots < x_{n-i+2m} \}
=
[n] \setminus \{ \| r\mathsf{C}(u) \| \mid u \in [i-2m] \}$
and set 
$K_{\Tilde{\mathsf{C}}}
=
\{ x_1 < x_2 < \cdots < x_{2m} \}$; 
note that 
$x_{\nu}$, 
$\nu \in [2m]$, 
are uniquely determined by
\begin{enumerate}[(i)]
\item
$x_1 
= 
\min \{ x \in \mathcal{B}_n \mid 
x \succeq 1, \ x \notin r\mathsf{C}, \ \overline{x} \notin r\mathsf{C} \}$,
\item
$x_{\nu} 
= 
\min \{ x \in \mathcal{B}_n \mid 
x \succ x_{\nu - 1}, \ x \notin r\mathsf{C}, \ \overline{x} \notin r\mathsf{C} \}$
for 
$\nu \in [2,2m]$.
\end{enumerate}
Define
$r\Tilde{\mathsf{C}}
=
K_{\Tilde{\mathsf{C}}}
\cup 
r\mathsf{C}$
and 
$l\Tilde{\mathsf{C}}
=
\{ \overline{x} \mid x \in K_{\Tilde{\mathsf{C}}}\}
\cup 
l\mathsf{C}$; 
note that 
$r\Tilde{\mathsf{C}}, 
l\Tilde{\mathsf{C}}
\in 
\mathrm{CST}_{B_n}(\varpi_i)$
(cf. \cite{B}; 
see also 
\cite[Algorithm 4.1]{LeSc}). 
\end{define}

Let 
$i \in [n-1]$. 
Define a $\U'$-crystal structure on
$\mathrm{QKN}_{B_n}(\varpi_i)$
as follows.
Let
$\Tilde{\mathsf{C}}
=
\mathsf{C}
\cup
\{ \overline{0},\ldots ,\overline{0} \}
\in 
\mathrm{QKN}_{B_n}(\varpi_i)$
and 
$j \in I$. 
Set 
$\mathrm{wt}
(\Tilde{\mathsf{C}})
=
\mathrm{wt}
(\mathsf{C})$, 
$e_j \Tilde{\mathsf{C}}
=
e_j\mathsf{C} \cup \{ \overline{0},\ldots ,\overline{0}\}$, 
and 
$f_j \Tilde{\mathsf{C}}
=
f_j\mathsf{C} \cup \{ \overline{0},\ldots ,\overline{0}\}$; 
we understand that
$e_j \Tilde{\mathsf{C}}
=
\bm{0}$
(resp.
$f_j \Tilde{\mathsf{C}}
=
\bm{0}$)
if 
$e_j \mathsf{C}
=
\bm{0}$
(resp.
$f_j \mathsf{C}
=
\bm{0}$).
The actions of 
$e_0$
and 
$f_0$
are uniquely determined from
$\mathsf{C} 
\cap 
\{ 1,2,\overline{2},\overline{1},z_k,\overline{z_k} \}$
and 
$m \in \BZ_{\geq 0}$, 
where 
$z_k = \min I_{\mathsf{C}}$
and 
$m$
is as in (QKN-B1).
Let 
$y_k
=
\min J_{\mathsf{C}}$. 
If 
$m > 0$, 
let 
$x_1 = \min K_{\Tilde{\mathsf{C}}}$
and 
$x_2 = \min (K_{\Tilde{\mathsf{C}}} \setminus \{ x_1 \})$.
The action of 
$f_0$
is illustrated as follows.
\begin{enumerate}[(i)]
\item
Assume that 
$m = 0$
and 
$y_k \notin \{ 1,2 \}$. 
Set 
\begin{align} \label{eq:f0-KN-B-1}
\begin{CD}
\begin{ytableau}
\overline{2}
\end{ytableau}
@>{f_0}>>
\begin{ytableau}
1
\end{ytableau}
\end{CD} \ , 
&&
\begin{CD}
\begin{ytableau}
\overline{1}
\end{ytableau}
@>{f_0}>>
\begin{ytableau}
2
\end{ytableau}
\end{CD} \ ,
&&
\begin{CD}
\begin{ytableau}
\overline{2} \\ \overline{1}
\end{ytableau}
@>{f_0}>>
\begin{ytableau}
\overline{0} \\ \overline{0}
\end{ytableau}
\end{CD} \ ,
\end{align}
and set
$f_j \Tilde{\mathsf{C}} = \bm{0}$
otherwise. 
\item
Assume that 
$m > 0$
and 
$y_k \notin \{ 1,2 \}$. 
Set 
\begin{align} \label{eq:f0-KN-B-2}
\begin{split}
&
\begin{CD}
\begin{ytableau}
\overline{2} \\ \overline{1}
\end{ytableau}
@>{f_0}>>
\begin{ytableau}
\overline{0} \\ \overline{0}
\end{ytableau}
@>{f_0}>>
\begin{ytableau}
1 \\ 2
\end{ytableau}
\end{CD} \ ,
\\[3mm]
&
\begin{CD}
\begin{ytableau}
\overline{2} \\ \overline{0} \\ \overline{0}
\end{ytableau}
@>{f_0}>>
\begin{ytableau}
1 \\ x_2 \\ \overline{x_2}
\end{ytableau}
\end{CD} \ , 
\hspace{1cm}
\begin{CD}
\begin{ytableau}
\overline{1} \\ \overline{0} \\ \overline{0}
\end{ytableau}
@>{f_0}>>
\begin{ytableau}
2 \\ x_2 \\ \overline{x_2}
\end{ytableau}
\end{CD} \ ,
\end{split}
\end{align}
and set 
$f_j \Tilde{\mathsf{C}} = \bm{0}$
otherwise. 
\item
Assume that 
$y_k \in \{ 1,2 \}$. 
Set 
\begin{align} \label{eq:f0-KN-B-3}
\begin{CD}
\begin{ytableau}
z_k \\ \overline{z_k} \\ \overline{2}
\end{ytableau}
@>{f_0}>>
\begin{ytableau}
1 \\ \overline{0} \\ \overline{0}
\end{ytableau}
\end{CD} \ ,
\hspace{1cm}
\begin{CD}
\begin{ytableau}
z_k \\ \overline{z_k} \\ \overline{1}
\end{ytableau}
@>{f_0}>>
\begin{ytableau}
2 \\ \overline{0} \\ \overline{0}
\end{ytableau}
\end{CD} \ ,
\end{align}
and set
$f_j \Tilde{\mathsf{C}} = \bm{0}$
otherwise. 
\end{enumerate}
The map 
$e_0$
is defined to be such that 
the condition (C6) in \S \ref{Subsection:Crystals}
holds. 
For 
$\Tilde{\mathsf{C}} \in \mathrm{QKN}_{B_n}(\varpi_i)$
and 
$j \in I_{\af}$, 
set 
$\varepsilon_j(\Tilde{\mathsf{C}})
=
\max \{ k \in \BZ_{\geq 0} \mid e_j^k \Tilde{\mathsf{C}} \neq \bm{0} \}$
and 
$\varphi_j(\Tilde{\mathsf{C}})
=
\max \{ k \in \BZ_{\geq 0} \mid f_j^k \Tilde{\mathsf{C}} \neq \bm{0} \}$. 
It is Theorem \ref{thm:QKN->QLS-B} (2) below
that makes these definitions allowable.

Similarly, we can define a 
$\U'$-crystal structure on 
$\mathrm{QKN}_{B_n}(\varpi_n)$. 
The maps 
$e_0,f_0$
is given as follows. 
Let 
$\Tilde{\mathsf{C}}
\in 
\mathrm{QKN}_{B_n}(\varpi_n)$.
We see that only the letters 
$1,2,\overline{2},\overline{1}$
may be changed in 
$\Tilde{\mathsf{C}}$
when we apply
$e_0$
or
$f_0$, 
and the actions of 
$e_0$
and
$f_0$
are uniquely determined from
$\Tilde{\mathsf{C}} 
\cap 
\{ 1,2,\overline{2},\overline{1} \}$.
The action of 
$f_0$
is illustrated by
\begin{align}
\begin{CD}
\begin{ytableau}
\overline{2} \\ \overline{1}
\end{ytableau}
@>{f_0}>>
\begin{ytableau}
1 \\ 2
\end{ytableau}
\end{CD} \ ;
\end{align}
set 
$f_0 \Tilde{\mathsf{C}} = \bm{0}$
otherwise. 
The map 
$e_0$
is defined to be such that 
the condition (C6) in \S \ref{Subsection:Crystals}
holds.

\begin{thm} \label{thm:QKN->QLS-B}
Assume that 
$\U$
is of type $B_n^{(1)}$. 
\begin{enumerate}[(1)]
\item
The set 
$\mathrm{QKN}_{B_n}(\varpi_n)$
equipped with the maps 
$\mathrm{wt},e_j,f_j,\varepsilon_j,\varphi_j$, 
$j \in I_{\af}$, 
is a $\U'$-crystal. 
The map 
$\mathrm{QLS}(\varpi_n)
=
\mathrm{LS}(\varpi_n)
\to 
\mathrm{QKN}_{B_n}(\varpi_n)
=
\mathrm{CST}_{B_n}(\varpi_n)$, 
$(w;0,1)
\mapsto 
\mathsf{T}_w^{(n)}$, 
is an isomorphism of 
$\U'$-crystals. 
\item
Assume that 
$i \in [n-1]$.
The set 
$\mathrm{QKN}_{B_n}(\varpi_i)$
equipped with the maps 
$\mathrm{wt},e_j,f_j,\varepsilon_j,\varphi_j$, 
$j \in I_{\af}$, 
is a $\U'$-crystal. 
The map 
\begin{align} \label{eq:QKN->QLS-B}
\mathrm{QKN}_{B_n}(\varpi_i)
\to 
\mathrm{QLS}(\varpi_i), \ 
\Tilde{\mathsf{C}}
\mapsto 
(r\Tilde{\mathsf{C}},l\Tilde{\mathsf{C}}),
\end{align}
is an isomorphism of 
$\U'$-crystals.
The inverse of 
\eqref{eq:QKN->QLS-B}
is given as follows. 
Let 
$(v,w) \in \mathrm{QLS}(\varpi_i)$, 
$\mathcal{J}
=
\mathcal{J}(w) 
= 
\mathcal{J}(v) 
= 
(J_1 < \cdots < J_{\mu})
\in 
\mathcal{S}_i$
and 
$\mathcal{M}(w) = (M_{\nu}), 
\mathcal{M}(v) = (N_{\nu})
\in 
2^{\mathcal{J}}$. 
Set 
$f = \# N_{\mu} - \# M_{\mu} \in \BZ_{\geq 0}$
if 
$n \in J_{\mu}$, 
and set
$f = 0$
otherwise. 
Set
$2m = \# M_1 - \# N_1 \in 2\BZ_{\geq 0}$
if 
$1,2 \in J_1$, 
and set
$m = 0$
otherwise; 
we see from 
Lemma \ref{lem:d_i(v,w)}
that 
$m = d_i(v,w)$. 
Define 
\begin{enumerate}[(i)]
\item
$\{ y^1 < y^2 < \cdots < y^l \}
=
N_1 \setminus M_1$, 
\item
$z^1
=
\min \{ z \in J_1 \mid 
z \succ y^1, \ 
z \in M_1 \setminus N_1 \}$, and 
\item
$z^{\nu}
=
\min \{ z \in J_1 \mid 
z \succ \max \{ y^{\nu} , z^{\nu -1} \}, \ 
z \in M_1 \setminus N_1 \}$
for 
$\nu \in [2,l]$.
\end{enumerate}
The inverse image of 
$(v,w)$
is
\begin{align} \label{eq:QLS->QKN-B}
\begin{split}
\Tilde{\mathsf{C}}
=
(J_1 
&\setminus 
(M_1 \cup N_1))
\cup
\{ z^{\nu}, \overline{z^{\nu}}
\mid 
\nu \in [l] \}
\cup 
\{ \overline{z} \mid z \in M_1 \cap N_1 \} \\[2mm]
&\cup
\bigcup_{\nu=2}^{\mu}
\left(
(J_{\nu} \setminus N_{\nu})
\cup
\{ \overline{z} \mid z \in M_{\nu} \}
\right)
\cup
\{ \underbrace{0,0,\ldots ,0}_{f \ \text{times}} \} 
\cup
\{ \underbrace{\overline{0},\overline{0},\ldots ,\overline{0}}_{2m \ \text{times}} \};
\end{split}
\end{align}
if we set 
$\mathsf{C}
=
\Tilde{\mathsf{C}}
\setminus 
\{ \underbrace{\overline{0},\overline{0},\ldots ,\overline{0}}_{2m \ \text{times}} \}$, 
then 
$\mathsf{C}
\in 
\mathrm{KN}_{B_n}(\varpi_{i-2m})$
and 
\begin{align}
I_{\mathsf{C}}
=
\{ z^{\nu} \mid \nu \in [l] \}
\cup
\bigcup_{\nu = 2}^{\mu} (M_{\nu} \setminus N_{\nu})
\cup
\{ \underbrace{0,0,\ldots ,0}_{f \ \text{times}} \} , \ \ 
J_{\mathsf{C}}
=
\bigcup_{\nu = 1}^{\mu} (N_{\nu} \setminus M_{\nu}).
\end{align}
\item
Assume that 
$i \in [n-1]$. 
The map 
\begin{align} \label{eq:QKN->KN-B}
\mathrm{QKN}_{B_n}(\varpi_i)
\to 
\bigsqcup_{m=0}^{\left\lfloor \frac{i}{2} \right\rfloor}
\mathrm{KN}_{B_n}(\varpi_{i-2m}), \ 
\Tilde{\mathsf{C}}
=
\mathsf{C}
\cup
\{ \underbrace{\overline{0},\overline{0},\ldots ,\overline{0}}_{2m \ \text{times}} \}
\mapsto 
\mathsf{C},
\end{align}
is an isomorphism of 
$\Fg$-crystals, 
where 
$\mathsf{C}$
and 
$m$
are as in (QKN-B1)--(QKN-B2)
for 
$\Tilde{\mathsf{C}}$. 
Here we understand that 
$\mathrm{KN}_{B_n}(\varpi_0)
=
\{ \emptyset \}$
is a $\Fg$-crystal 
isomorphic to the crystal basis of 
the trivial module. 
The inverse image of 
$\mathsf{C}
\in 
\mathrm{KN}_{B_n}(\varpi_{i-2m})$
under the map 
\eqref{eq:QKN->KN-B}
is 
$\Tilde{\mathsf{C}}
=
\mathsf{C}
\cup
\{ \underbrace{\overline{0},\overline{0},\ldots ,\overline{0}}_{2m \ \text{times}} \}$. 
\end{enumerate}
\end{thm}

The proof of (1) and (3) in 
Theorem \ref{thm:QKN->QLS-B}
is straightforward 
(cf. \cite[Lemma 2.7 (i)]{C}).
In 
\S \ref{Subsection:Pr-Thm-B-(2)}, 
we will give the proof for 
Theorem \ref{thm:QKN->QLS-B} (2). 

Recall the partial order 
$\preceq$
on 
$\mathrm{CST}_{B_n}(\varpi_i) \times \BZ$
(see Definition \ref{def:SiB-B}).

\begin{define} \label{def:SiKN-B}
Let 
$i \in I$
and 
$m \in \BZ_{\geq 0}$. 
\begin{enumerate}[(1)]
\item
Let 
$\mathbb{T}
=
\left(
\mathsf{T}_1
\mathsf{T}_2
\cdots 
\mathsf{T}_m, 
(c_1,c_2,\ldots ,c_m)
\right)$, 
where
$\mathsf{T}_{\nu}
\in 
\mathrm{QKN}_{B_n}(\varpi_n)
=
\mathrm{CST}_{B_n}(\varpi_n)$
and 
$c_{\nu} \in \BZ$
for 
$\nu \in [m]$. 
We call 
$\mathbb{T}$
a semi-infinite KN $B_n$-tableau of shape 
$m\varpi_n$
if 
\begin{align}
(\mathsf{T}_{\nu},c_{\nu})
\succeq 
(\mathsf{T}_{\nu+1},c_{\nu+1})
\end{align}
in
$\mathrm{CST}_{B_n}(\varpi_i) \times \BZ$
for 
$\nu \in [m-1]$. 
\item
Assume that 
$i \in [n-1]$. 
Let 
$\mathbb{T}
=
(\Tilde{\mathsf{C}}_1
\Tilde{\mathsf{C}}_2
\cdots 
\Tilde{\mathsf{C}}_m, 
(c_1,c_2,\ldots ,c_m))$, 
where
$\Tilde{\mathsf{C}}_{\nu}
\in 
\mathrm{QKN}_{B_n}(\varpi_i)$
and 
$c_{\nu} \in \BZ$
for 
$\nu \in [m]$. 
We call 
$\mathbb{T}$
a semi-infinite KN $B_n$-tableau of shape 
$m\varpi_i$
if
\begin{align}
(l\Tilde{\mathsf{C}}_{\nu},c_{\nu} 
- 
d_i(r\Tilde{\mathsf{C}}_{\nu},l\Tilde{\mathsf{C}}_{\nu}))
\succeq 
(r\Tilde{\mathsf{C}}_{\nu+1},c_{\nu+1} 
+ 
d_i(r\Tilde{\mathsf{C}}_{\nu+1},l\Tilde{\mathsf{C}}_{\nu+1}))
\end{align}
in
$\mathrm{CST}_{B_n}(\varpi_i) \times \BZ$
for 
$\nu \in [m-1]$.  
\end{enumerate}
Let 
$\mathbb{Y}^{\si}_{B_n}(m\varpi_i)$
be the set of 
semi-infinite KN $B_n$-tableaux of shape
$m\varpi_i$. 
For 
$\lambda = \sum_{i \in I} m_i \varpi_i \in P^+$, 
set 
$\mathbb{Y}^{\si}_{B_n}(\lambda)
=
\prod_{i \in I} 
\mathbb{Y}^{\si}_{B_n}(m_i \varpi_i)$. 
We call an element of 
$\mathbb{Y}^{\si}_{B_n}(\lambda)$
a semi-infinite KN $B_n$-tableau
of shape $\lambda$.
\end{define}

Let 
$\mathrm{QKN}_{B_n}(\varpi_i)_{\af}$
denote the affinization of the 
$\U'$-crystal 
$\mathrm{QKN}_{B_n}(\varpi_i)$
(see \S\ref{Subsection:Crystals}). 
Combining
Theorem \ref{thm:BN-Kashiwara}, 
Proposition \ref{prop:tab-cri-B} (2), 
Lemma \ref{lem:SMT-QLSaf}, 
and 
Definition \ref{def:SiKN-B}
we obtain the following theorem.

\begin{thm} \label{thm:SiKN->B-B}
Assume that 
$\U$
is of type $B_n^{(1)}$. 
Let 
$\lambda = \sum_{i \in I} m_i \varpi_i \in P^+$. 
For each 
$i \in I$, 
the image of the map
\begin{align} \label{eq:map-YB->QKNaf}
\begin{split}
&\mathbb{Y}^{\si}_{B_n}(m_i \varpi_i)
\to 
\mathrm{QKN}_{B_n}(\varpi_i)_{\af}^{\otimes m_i}, \\
&\left(
\mathsf{C}_1
\mathsf{C}_2
\cdots 
\mathsf{C}_{m_i}, 
(c_1,c_2,\ldots ,c_{m_i})
\right)
\mapsto 
\bigotimes_{\nu \in [m_i]}
(\mathsf{C}_{\nu} , c_{\nu}),
\end{split}
\end{align}
is a $\U$-subcrystal. 
Hence we can define a 
$\U$-crystal structure on 
$\mathbb{Y}^{\si}_{B_n}(m_i \varpi_i)$
to be such that the map 
\eqref{eq:map-YB->QKNaf}
is a strict embedding of 
$\U$-crystals. 
In particular, 
$\mathbb{Y}^{\si}_{B_n}(\lambda)$
is a $\U$-subcrystal of 
$\bigotimes_{i \in I}
\mathrm{QKN}_{B_n}(\varpi_i)_{\af}^{\otimes m_i}$. 
Then 
$\mathbb{Y}^{\si}_{B_n}(\lambda)$
is isomorphic,
as a $\U$-crystal, 
to the crystal basis 
$\mathcal{B}(\lambda)$.
\end{thm}

\subsection{Type $D_n^{(1)}$} \label{Subsection:tab-type-D}

Throughout this subsection, 
we assume that 
$\Fg$ and $W$ are of type $D_n$. 
Recall that 
$I = [n]$,  
$\Delta 
= 
\{ \pm(\varepsilon_s \pm \varepsilon_t) \mid s,t \in [n],\ s < t \}$, 
and
$\Pi 
=
\{ \alpha_s = \varepsilon_s - \varepsilon_{s+1} \mid s \in [n-1] \}
\sqcup
\{ \alpha_n = \varepsilon_{n-1} + \varepsilon_n \}$. 
The highest root is 
$\varepsilon_1 + \varepsilon_2
=
\alpha_1 + 2 \alpha_2 + \cdots + 2\alpha_{n-2} + \alpha_{n-1} + \alpha_n$. 
We identify 
$\varpi_i$
with 
$\varepsilon_1 + \varepsilon_2 + \cdots \varepsilon_i$
if 
$i \in [n-2]$, 
with 
$\dfrac{1}{2}(\varepsilon_1 + \varepsilon_2 + \cdots + \varepsilon_{n-1} - \varepsilon_n)$
if 
$i = n-1$, 
and with 
$\dfrac{1}{2}(\varepsilon_1 + \varepsilon_2 + \cdots + \varepsilon_{n-1} + \varepsilon_n)$
if 
$i = n$. 

Let 
$i \in [2,n-2]$. 
A map 
$\mathsf{C} : 
[i] 
\to 
\mathcal{D}_n$
is, by definition, 
a Kashiwara--Nakashima $D_n$-column 
(KN $D_n$-column for short)
of shape $\varpi_i$ 
if 
\begin{enumerate}[(KN-D1)]
\item
$\mathsf{C}(1)
\not\succeq
\mathsf{C}(2)
\not\succeq
\cdots 
\not\succeq
\mathsf{C}(i)$, 
\item
if 
$t = \mathsf{C}(p) = \sigma(\mathsf{C}(q)) \in [n]$
for some 
$p,q \in [i]$, 
then 
$|q-p| > i-t$; 
\end{enumerate}
note that 
$n,\overline{n} \in \mathcal{D}_n$ 
may appear in 
$\mathsf{C}$
more than once. 
Let 
$\mathrm{KN}_{D_n}(\varpi_i)$
be the set of 
KN $D_n$-columns
of shape $\varpi_i$. 

Let 
$i \in [2,n-2]$. 
Let 
$\mathsf{C} : [i] \to \mathcal{D}_n$
be a map satisfying (KN-D1). 
Define 
$\mathsf{C}^{\wedge} : [i] \to \mathcal{B}_n$
to be such that 
$\mathsf{C}^{\wedge}(u)
=
\mathsf{C}^{\wedge}(u+1)
=
0$
if 
$\mathsf{C}(u)
=
\overline{n}$
and 
$\mathsf{C}(u+1)
=
n$
for 
$u \in [i]$, 
and 
$\mathsf{C}^{\wedge}(u)
=
\mathsf{C}(u)$
otherwise. 
We say that 
$\mathsf{C}$
can be split 
if 
$\mathsf{C}^{\wedge}$
can be split. 
If 
$\mathsf{C}
\in 
\mathrm{KN}_{D_n}(\varpi_i)$
can be split, 
then we write 
$r\mathsf{C}
=
r\mathsf{C}^{\wedge}$
and 
$l\mathsf{C}
=
l\mathsf{C}^{\wedge}$; 
note that 
$r\mathsf{C},
l\mathsf{C}
\in 
\mathrm{CST}_{D_n}(\varpi_i)$. 
For convenience we also denote by 
$\mathrm{KN}_{D_n}(\varpi_i)
=
\mathrm{CST}_{D_n}(\varpi_i)$
for 
$i \in \{ 1,n-1,n \}$. 

Define a 
$\Fg$-crystal structure on 
$\mathrm{KN}_{D_n}(\varpi_i)$
for 
$i \in [n-2]$
as follows
(cf. \cite[\S 6]{KN}); 
for the $\Fg$-crystal structure on 
$\mathrm{KN}_{D_n}(\varpi_i)$
for 
$i \in \{ n-1,n \}$
we refer the reader to 
\cite[\S 2.3]{Le03}. 
The maps 
$\mathrm{wt},\varepsilon_j,\varphi_j$
for 
$j \in I$
and 
$e_j,f_j$
for 
$j \in [n-2]$
are defined in the same manner as 
those for $\mathrm{KN}_{C_n}(\varpi_i)$. 
Note that only the letters 
$n-1,n,\overline{n},\overline{n-1}$
may be changed in 
$\mathsf{C}$
when we apply
$e_{n-1}$, 
$e_n$,
$f_{n-1}$, 
or
$f_n$. 
To define the action of 
$f_{n-1}$, 
let
$\{ \mathsf{C}(u) \mid u \in [i], \ 
\mathsf{C}(u) \in \{ n-1,n,\overline{n},\overline{n-1} \} \}
=
\{ 
\mathsf{C}(u_1) 
\not\succ 
\mathsf{C}(u_2) 
\not\succ 
\cdots 
\not\succ 
\mathsf{C}(u_p)
\}$, 
where 
$1 \leq u_1 < u_2 < \cdots < u_p \leq i$, 
and continue deleting a successive pair 
$a \not\succ b$
such that 
$(a,b)
\in 
\{ 
(\overline{n},n), 
(n-1,n), 
(\overline{n},\overline{n-1}), 
(n-1,\overline{n-1})
\}$
from 
$\mathsf{C}(u_1) 
\not\succ 
\mathsf{C}(u_2) 
\not\succ 
\cdots 
\not\succ 
\mathsf{C}(u_p)$
until no such pair exists. 
Let 
$\mathsf{C}(v_1) 
\not\succ 
\mathsf{C}(v_2) 
\not\succ 
\cdots 
\not\succ 
\mathsf{C}(v_q)$
be the resulting sequence.
It follows that
\ytableausetup{mathmode,boxsize=1cm}
\begin{align} \label{eq:f_n-1_tab-D}
\begin{ytableau}
\mathsf{C}(v_1) \\
\vdots \\
\mathsf{C}(v_q)
\end{ytableau}
\in 
\left\{
\emptyset, \ 
\begin{ytableau}
n-1
\end{ytableau} \ , \ 
\begin{ytableau}
n
\end{ytableau} \ , \ 
\begin{ytableau}
\overline{n}
\end{ytableau} \ , \ 
\begin{ytableau}
\overline{n-1}
\end{ytableau} \ , \ 
\begin{ytableau}
n \\ \overline{n}
\end{ytableau} \ , \ 
\begin{ytableau}
n-1 \\ \overline{n}
\end{ytableau} \ , \ 
\begin{ytableau}
n \\ \overline{n-1}
\end{ytableau}
\right\} .
\end{align}
Then the action of 
$f_{n-1}$
is uniquely determined from
\eqref{eq:f_n-1_tab-D}, 
and only one of the entries in 
\eqref{eq:f_n-1_tab-D}
may be changed in 
$\mathsf{C}$
when we apply 
$f_{n-1}$. 
This is illustrated by 
\begin{align}
\begin{split}
&
\begin{CD}
\begin{ytableau}
n-1
\end{ytableau}
@>{f_{n-1}}>>
\begin{ytableau}
n
\end{ytableau}
\end{CD} \ ,
\hspace{1cm}
\begin{CD}
\begin{ytableau}
\overline{n}
\end{ytableau}
@>{f_{n-1}}>>
\begin{ytableau}
\overline{n-1}
\end{ytableau}
\end{CD} \ , 
\\[3mm]
&
\begin{CD}
\begin{ytableau}
n-1 \\ \overline{n}
\end{ytableau}
@>{f_{n-1}}>>
\begin{ytableau}
n \\ \overline{n}
\end{ytableau}
@>{f_{n-1}}>>
\begin{ytableau}
n \\ \overline{n-1}
\end{ytableau}
\end{CD} \ ;
\end{split}
\end{align}
set 
$f_{n-1} \mathsf{C} = \bm{0}$
otherwise. 
We next define the action of 
$f_n$. 
Similarly, continue deleting a successive pair 
$a \not\succ b$
such that 
$(a,b)
\in 
\{ 
(n,\overline{n}), 
(n-1,\overline{n}), 
(n,\overline{n-1}), 
(n-1,\overline{n-1})
\}$
from 
$\mathsf{C}(u_1) 
\not\succ 
\mathsf{C}(u_2) 
\not\succ 
\cdots 
\not\succ 
\mathsf{C}(u_p)$. 
Let 
$\mathsf{C}(v_1') 
\not\succ 
\mathsf{C}(v_2') 
\not\succ 
\cdots 
\not\succ 
\mathsf{C}(v_r')$
be the resulting sequence.
It follows that
\begin{align} \label{eq:f_n_tab-D}
\begin{ytableau}
\mathsf{C}(v_1') \\
\vdots \\
\mathsf{C}(v_r')
\end{ytableau}
\in 
\left\{
\emptyset, \ 
\begin{ytableau}
n-1
\end{ytableau} \ , \ 
\begin{ytableau}
n
\end{ytableau} \ , \ 
\begin{ytableau}
\overline{n}
\end{ytableau} \ , \ 
\begin{ytableau}
\overline{n-1}
\end{ytableau} \ , \ 
\begin{ytableau}
\overline{n} \\ n
\end{ytableau} \ , \ 
\begin{ytableau}
n-1 \\ n
\end{ytableau} \ , \ 
\begin{ytableau}
\overline{n} \\ \overline{n-1}
\end{ytableau}
\right\} .
\end{align}
Then the action of 
$f_n$
is uniquely determined from
\eqref{eq:f_n_tab-D}, 
and only one of the entries in 
\eqref{eq:f_n_tab-D}
may be changed in 
$\mathsf{C}$
when we apply 
$f_n$. 
This is illustrated by 
\begin{align}
\begin{split}
&
\begin{CD}
\begin{ytableau}
n-1
\end{ytableau}
@>{f_n}>>
\begin{ytableau}
\overline{n}
\end{ytableau}
\end{CD} \ ,
\hspace{1cm}
\begin{CD}
\begin{ytableau}
n
\end{ytableau}
@>{f_n}>>
\begin{ytableau}
\overline{n-1}
\end{ytableau}
\end{CD} \ ,
\\[3mm]
&
\begin{CD}
\begin{ytableau}
n-1 \\ n
\end{ytableau}
@>{f_n}>>
\begin{ytableau}
\overline{n} \\ n
\end{ytableau}
@>{f_n}>>
\begin{ytableau}
\overline{n} \\ \overline{n-1}
\end{ytableau}
\end{CD} \ ;
\end{split}
\end{align}
set 
$f_n \mathsf{C} = \bm{0}$
otherwise.

The next lemma is a reformulation of 
\cite[Corollary 3.1.11 and Remark 3.1.13]{Le03}
in terms of Maya diagrams.

\begin{lem}
Assume that 
$\Fg$
and 
$W$
are of type $D_n$. 
\begin{enumerate}[(1)]
\item
Assume that 
$i \in \{ 1,n-1,n \}$. 
The map 
$\mathrm{LS}(\varpi_i)
\to 
\mathrm{KN}_{D_n}(\varpi_i)
=
\mathrm{CST}_{D_n}(\varpi_i)$, 
$(w;0,1)
\mapsto 
\mathsf{T}_w^{(i)}$, 
is an isomorphism of 
$\Fg$-crystals. 
\item
Assume that 
$i \in [2,n-2]$. 
For a map 
$\mathsf{C} : [i] \to \mathcal{D}_n$
satisfying (KN-D1), 
we have 
$\mathsf{C}
\in 
\mathrm{KN}_{D_n}(\varpi_i)$
if and only if 
$\mathsf{C}$
can be split. 
Hence the map 
$\mathrm{KN}_{D_n}(\varpi_i)
\to 
\mathrm{KN}_{B_n}(\varpi_i)$, 
$\mathsf{C} \mapsto \mathsf{C}^{\wedge}$, 
is injective. 
The map 
\begin{align} \label{eq:KN->LS-D}
\mathrm{KN}_{D_n}(\varpi_i)
\to 
\mathrm{LS}(\varpi_i), \ 
\mathsf{C}
\mapsto 
(r\mathsf{C},l\mathsf{C}),
\end{align}
is an isomorphism of $\Fg$-crystals. 
The inverse of
\eqref{eq:KN->LS-D} 
is given as follows. 
Let 
$(v,w) 
\in 
\mathrm{LS}(\varpi_i)$, 
$\mathcal{J}
=
\mathcal{J}(w) 
= 
\mathcal{J}(v) 
= 
(J_1 < \cdots < J_{\mu})
\in 
\mathcal{S}_i$, 
$\mathcal{M}(w) = (M_{\nu}), 
\mathcal{M}(v) = (N_{\nu})
\in 
2^{\mathcal{J}}$, 
and 
$f = 
\# N_{\mu} - \# M_{\mu}
\in 
2\BZ_{\geq 0}$. 
The inverse image of 
$(v,w)$
is the KN $D_n$-column
$\mathsf{C}$
such that 
\begin{align}
\begin{split}
\mathsf{C}^{\wedge} 
&= 
\{ v(u) \mid u \in [i], \ v(u) \preceq n \}
\cup
\{ w(u) \mid u \in [i], \ w(u) \succeq \overline{n} \}
\cup
\{ \underbrace{0,0,\ldots ,0}_{f \ \text{times}} \} \\
&=
\bigcup_{\nu=1}^{\mu}
\left(
(J_{\nu} \setminus N_{\nu})
\cup
\{ \overline{z} \mid z \in M_{\nu} \}
\right) 
\cup
\{ \underbrace{0,0,\ldots ,0}_{f \ \text{times}} \} ;
\end{split}
\end{align}
we have 
$I_{\mathsf{C}^{\wedge}}
=
\bigcup_{\nu = 1}^{\mu} (M_{\nu} \setminus N_{\nu})
\cup
\{ \underbrace{0,0,\ldots ,0}_{f \ \text{times}} \}$
and
$J_{\mathsf{C}^{\wedge}}
=
\bigcup_{\nu = 1}^{\mu} (N_{\nu} \setminus M_{\nu})$. 
\end{enumerate}
\end{lem}

Set 
\begin{align}
\Tilde{\mathcal{D}}_n 
=
\mathcal{D}_n 
\cup 
\{ \overline{0} \}
=
\left\{ 
1 \prec 2 \prec \cdots \prec n-1 
\prec 
\begin{array}{c} n \\ \overline{n} \end{array} 
\prec 
\overline{n-1} \prec \cdots \prec \overline{2} \prec \overline{1}
\prec
\overline{0} 
\right\}.
\end{align}

\begin{define} \label{def:QKN-D}
Let 
$i \in [2,n-2]$. 
A map
$\Tilde{\mathsf{C}} : [i] \to \Tilde{\mathcal{D}}_n$
is, by definition, 
a quantum Kashiwara--Nakashima $D_n$-column
(QKN $D_n$-column for short) 
of shape $\varpi_i$ 
if 
\begin{enumerate}[(QKN-D1)]
\item
there exists 
$m \in \BZ_{\geq 0}$
such that 
$2m \leq i$
and 
\begin{align}
\begin{split}
\Tilde{\mathsf{C}}(1)
&\not\succeq
\Tilde{\mathsf{C}}(2)
\not\succeq
\cdots 
\not\succeq
\Tilde{\mathsf{C}}(i-2m) \\
&
\prec
\overline{0}
=
\Tilde{\mathsf{C}}(i-2m+1)
=
\cdots 
=
\Tilde{\mathsf{C}}(i-1)
=
\Tilde{\mathsf{C}}(i),
\end{split}
\end{align} 
\item
the map 
$\mathsf{C} : 
[i-2m] \to \mathcal{D}_n$, 
$u \mapsto \Tilde{\mathsf{C}}(u)$, 
is a KN $D_n$-column of shape $\varpi_{i-2m}$;
\end{enumerate}
in this case, 
we write 
$\Tilde{\mathsf{C}}
=
\mathsf{C}
\cup
\{ \underbrace{\overline{0},\overline{0},\ldots ,\overline{0}}_{2m \ \text{times}} \}$
for brevity.
Let 
$\mathrm{QKN}_{D_n}(\varpi_i)$
be the set of 
QKN $D_n$-columns of shape $\varpi_i$. 
For convenience we also denote by 
$\mathrm{QKN}_{D_n}(\varpi_i)
=
\mathrm{CST}_{D_n}(\varpi_i)$
for 
$i \in \{ 1,n-1,n \}$. 

For a map 
$\Tilde{\mathsf{C}} : [i] \to \Tilde{\mathcal{D}}_n$, 
define 
$\Tilde{\mathsf{C}}^{\wedge} : [i] \to \Tilde{\mathcal{B}}_n$
to be such that 
$\Tilde{\mathsf{C}}^{\wedge}(u)
=
\Tilde{\mathsf{C}}^{\wedge}(u+1)
=
0$
if 
$\Tilde{\mathsf{C}}(u)
=
\overline{n}$
and 
$\Tilde{\mathsf{C}}(u+1)
=
n$
for 
$u \in [i]$, 
and 
$\Tilde{\mathsf{C}}^{\wedge}(u)
=
\Tilde{\mathsf{C}}(u)$
otherwise. 
We see that if 
$\Tilde{\mathsf{C}}
=
\mathsf{C}
\cup
\{ \underbrace{\overline{0},\overline{0},\ldots ,\overline{0}}_{2m \ \text{times}} \}
\in 
\mathrm{QKN}_{D_n}(\varpi_i)$, 
then 
$\Tilde{\mathsf{C}}^{\wedge}
\in 
\mathrm{QKN}_{B_n}(\varpi_i)$
and, 
in consequence, 
$\Tilde{\mathsf{C}}^{\wedge}
=
\mathsf{C}^{\wedge}
\cup
\{ \underbrace{\overline{0},\overline{0},\ldots ,\overline{0}}_{2m \ \text{times}} \}$. 
For 
$\Tilde{\mathsf{C}}
\in 
\mathrm{QKN}_{D_n}(\varpi_i)$, 
define 
$r\Tilde{\mathsf{C}}
=
r\Tilde{\mathsf{C}}^{\wedge}$
and 
$l\Tilde{\mathsf{C}}
=
l\Tilde{\mathsf{C}}^{\wedge}$; 
note that 
$r\Tilde{\mathsf{C}}, 
l\Tilde{\mathsf{C}}
\in 
\mathrm{CST}_{D_n}(\varpi_i)$. 
\end{define}

Define a 
$\U'$-crystal structure on
$\mathrm{QKN}_{D_n}(\varpi_i)$
in the same manner as that on 
$\mathrm{QKN}_{B_n}(\varpi_i)$. 

\begin{thm} \label{thm:QKN->QLS-D}
Assume that 
$\U$
is of type $D_n^{(1)}$. 
\begin{enumerate}[(1)]
\item
Assume that 
$i \in \{ 1,n-1,n \}$. 
The set 
$\mathrm{QKN}_{D_n}(\varpi_i)$
equipped with the maps 
$\mathrm{wt},e_j,f_j,\varepsilon_j,\varphi_j$, 
$j \in I_{\af}$, 
is a $\U'$-crystal. 
The map 
$\mathrm{QLS}(\varpi_i)
=
\mathrm{LS}(\varpi_i)
\to 
\mathrm{QKN}_{D_n}(\varpi_i)
=
\mathrm{CST}_{D_n}(\varpi_i)$, 
$(w;0,1)
\mapsto 
\mathsf{T}_w^{(i)}$, 
is an isomorphism of 
$\U'$-crystals. 
\item
Assume that 
$i \in [2,n-2]$. 
The set 
$\mathrm{QKN}_{D_n}(\varpi_i)$
equipped with the maps 
$\mathrm{wt},e_j,f_j,\varepsilon_j,\varphi_j$, 
$j \in I_{\af}$, 
is a $\U'$-crystal. 
The map 
\begin{align} \label{eq:QKN->QLS-D}
\mathrm{QKN}_{D_n}(\varpi_i)
\to 
\mathrm{QLS}(\varpi_i), \ 
\Tilde{\mathsf{C}}
\mapsto 
(r\Tilde{\mathsf{C}},l\Tilde{\mathsf{C}}),
\end{align}
is an isomorphism of 
$\U'$-crystals.
The inverse of
\eqref{eq:QKN->QLS-D} 
is given as follows. 
Let 
$(v,w) \in \mathrm{QLS}(\varpi_i)$, 
$\mathcal{J}
=
\mathcal{J}(w) 
= 
\mathcal{J}(v) 
= 
(J_1 < \cdots < J_{\mu})
\in 
\mathcal{S}_i$
and 
$\mathcal{M}(w) = (M_{\nu}), 
\mathcal{M}(v) = (N_{\nu})
\in 
2^{\mathcal{J}}$. 
Set 
$f = \# N_{\mu} - \# M_{\mu} \in 2\BZ_{\geq 0}$
if 
$n-1,n \in J_{\mu}$, 
and set
$f = 0$
otherwise. 
Set
$2m = \# M_1 - \# N_1 \in 2\BZ_{\geq 0}$
if 
$1,2 \in J_1$, 
and set
$m = 0$
otherwise;
we see from 
Lemma \ref{lem:d_i(v,w)}
that 
$m = d_i(v,w)$. 
Define 
\begin{enumerate}[(i)]
\item
$\{ y^1 < y^2 < \cdots < y^l \}
=
N_1 \setminus M_1$, 
\item
$z^1
=
\min \{ z \in J_1 \mid 
z \succ y^1, \ 
z \in M_1 \setminus N_1 \}$, and 
\item
$z^{\nu}
=
\min \{ z \in J_1 \mid 
z \succ \max \{ y^{\nu} , z^{\nu -1} \}, \ 
z \in M_1 \setminus N_1 \}$
for 
$\nu \in [2,l]$.
\end{enumerate}
Let 
$\mathsf{C}$
be the KN $D_n$-column of shape $\varpi_{i-2m}$
such that 
\begin{align}
\begin{split}
\mathsf{C}^{\wedge}
=
(J_1 
\setminus 
(M_1 \cup N_1))
&\cup
\{ z^{\nu}, \overline{z^{\nu}}
\mid 
\nu \in [l] \}
\cup 
\{ \overline{z} \mid z \in M_1 \cap N_1 \} \\[2mm]
&\cup
\bigcup_{\nu=2}^{\mu}
\left(
(J_{\nu} \setminus N_{\nu})
\cup
\{ \overline{z} \mid z \in M_{\nu} \}
\right)
\cup
\{ \underbrace{0,0,\ldots ,0}_{f \ \text{times}} \};
\end{split}
\end{align}
we have
$I_{\mathsf{C}^{\wedge}}
=
\{ z^{\nu} \mid \nu \in [l] \}
\cup
\bigcup_{\nu = 2}^{\mu} (M_{\nu} \setminus N_{\nu}) 
\cup
\{ \underbrace{0,0,\ldots ,0}_{f \ \text{times}} \}$
and 
$J_{\mathsf{C}^{\wedge}}
=
\bigcup_{\nu = 1}^{\mu} (N_{\nu} \setminus M_{\nu})$. 
The inverse image of 
$(v,w)$
is 
$\Tilde{\mathsf{C}}
=
\mathsf{C}
\cup
\{ \underbrace{\overline{0},\overline{0},\ldots ,\overline{0}}_{2m \ \text{times}} \}$.
\item
Assume that 
$i \in [2,n-2]$. 
The map 
\begin{align} \label{eq:QKN->KN-D}
\mathrm{QKN}_{D_n}(\varpi_i)
\to 
\bigsqcup_{m=0}^{\left\lfloor \frac{i}{2} \right\rfloor}
\mathrm{KN}_{D_n}(\varpi_{i-2m}), \ 
\Tilde{\mathsf{C}}
=
\mathsf{C}
\cup
\{ \underbrace{\overline{0},\overline{0},\ldots ,\overline{0}}_{2m \ \text{times}} \}
\mapsto 
\mathsf{C},
\end{align}
is an isomorphism of 
$\Fg$-crystals, 
where 
$\mathsf{C}$
and 
$m$
are as in (QKN-D1)--(QKN-D2)
for 
$\Tilde{\mathsf{C}}$.
Here we understand that 
$\mathrm{KN}_{D_n}(\varpi_0)
=
\{ \emptyset \}$
is a $\Fg$-crystal 
isomorphic to the crystal basis of 
the trivial module. 
The inverse image of 
$\mathsf{C}
\in 
\mathrm{KN}_{D_n}(\varpi_{i-2m})$
under the map 
\eqref{eq:QKN->KN-D}
is 
$\Tilde{\mathsf{C}}
=
\mathsf{C}
\cup
\{ \underbrace{\overline{0},\overline{0},\ldots ,\overline{0}}_{2m \ \text{times}} \}$. 
\end{enumerate}
\end{thm}

The proof of (1) and (3) in 
Theorem \ref{thm:QKN->QLS-D}
is straightforward 
(cf. \cite[Lemma 2.7 (ii)]{C}).
Theorem \ref{thm:QKN->QLS-D} (2)
may be proved in much the same way as 
Theorem \ref{thm:QKN->QLS-B} (2) 
(see \S \ref{Subsection:Pr-Thm-B-(2)}); 
the details are left to the reader. 

Recall the partial order 
$\preceq$
on 
$\mathrm{CST}_{D_n}(\varpi_i) \times \BZ$
(see Definition \ref{def:SiB-D}).

\begin{define} \label{def:SiKN-D}
Let 
$i \in I$
and 
$m \in \BZ_{\geq 0}$. 
\begin{enumerate}[(1)]
\item
Assume that 
$i \in \{ 1,n-1,n \}$. 
Let 
$\mathbb{T}
=
\left(
\mathsf{T}_1
\mathsf{T}_2
\cdots 
\mathsf{T}_m, 
(c_1,c_2,\ldots ,c_m)
\right)$, 
where
$\mathsf{T}_{\nu}
\in 
\mathrm{CST}_{D_n}(\varpi_i)$
and 
$c_{\nu} \in \BZ$
for 
$\nu \in [m]$. 
We call 
$\mathbb{T}$
a semi-infinite KN $D_n$-tableau of shape 
$m\varpi_i$
if 
\begin{align}
(\mathsf{T}_{\nu},c_{\nu})
\succeq 
(\mathsf{T}_{\nu+1},c_{\nu+1})
\end{align}
in
$\mathrm{CST}_{D_n}(\varpi_i) \times \BZ$
for 
$\nu \in [m-1]$. 
\item
Assume that 
$i \in [2,n-2]$. 
Let 
$\mathbb{T}
=
(\Tilde{\mathsf{C}}_1
\Tilde{\mathsf{C}}_2
\cdots 
\Tilde{\mathsf{C}}_m, 
(c_1,c_2,\ldots ,c_m))$, 
where
$\Tilde{\mathsf{C}}_{\nu}
\in 
\mathrm{QKN}_{D_n}(\varpi_i)$
and 
$c_{\nu} \in \BZ$
for 
$\nu \in [m]$. 
We call 
$\mathbb{T}$
a semi-infinite KN $D_n$-tableau of shape 
$m\varpi_i$
if
\begin{align}
(l\Tilde{\mathsf{C}}_{\nu},c_{\nu} 
- 
d_i(r\Tilde{\mathsf{C}}_{\nu},l\Tilde{\mathsf{C}}_{\nu}))
\succeq 
(r\Tilde{\mathsf{C}}_{\nu+1},c_{\nu+1} 
+ 
d_i(r\Tilde{\mathsf{C}}_{\nu+1},l\Tilde{\mathsf{C}}_{\nu+1}))
\end{align}
in
$\mathrm{CST}_{D_n}(\varpi_i) \times \BZ$
for 
$\nu \in [m-1]$.  
\end{enumerate}
Let 
$\mathbb{Y}^{\si}_{D_n}(m\varpi_i)$
be the set of 
semi-infinite KN $D_n$-tableaux of shape
$m\varpi_i$. 
For 
$\lambda = \sum_{i \in I} m_i \varpi_i \in P^+$, 
set 
$\mathbb{Y}^{\si}_{D_n}(\lambda)
=
\prod_{i \in I} 
\mathbb{Y}^{\si}_{D_n}(m_i \varpi_i)$. 
We call an element of 
$\mathbb{Y}^{\si}_{D_n}(\lambda)$
a semi-infinite KN $D_n$-tableau
of shape $\lambda$.
\end{define}

Let 
$\mathrm{QKN}_{D_n}(\varpi_i)_{\af}$
denote the affinization of the 
$\U'$-crystal 
$\mathrm{QKN}_{D_n}(\varpi_i)$
(see \S\ref{Subsection:Crystals}). 
Combining
Theorem \ref{thm:BN-Kashiwara}, 
Proposition \ref{prop:tab-cri-D} (2), 
Lemma \ref{lem:SMT-QLSaf}, 
and 
Definition \ref{def:SiKN-D}
we obtain the following theorem.

\begin{thm} \label{thm:SiKN->B-D}
Let 
$\lambda = \sum_{i \in I} m_i \varpi_i \in P^+$. 
For each 
$i \in I$, 
the image of the map
\begin{align} \label{eq:map-D-Y->QKNaf}
\begin{split}
&\mathbb{Y}^{\si}_{D_n}(m_i \varpi_i)
\to 
\mathrm{QKN}_{D_n}(\varpi_i)_{\af}^{\otimes m_i}, \\
&\left(
\mathsf{C}_1
\mathsf{C}_2
\cdots 
\mathsf{C}_{m_i}, 
(c_1,c_2,\ldots ,c_{m_i})
\right)
\mapsto 
\bigotimes_{\nu \in [m_i]}
(\mathsf{C}_{\nu} , c_{\nu}),
\end{split}
\end{align}
is a $\U$-subcrystal. 
Hence we can define a 
$\U$-crystal structure on 
$\mathbb{Y}^{\si}_{D_n}(m_i \varpi_i)$
to be such that the map 
\eqref{eq:map-D-Y->QKNaf}
is a strict embedding of 
$\U$-crystals. 
In particular, 
$\mathbb{Y}^{\si}_{D_n}(\lambda)$
is a $\U$-subcrystal of 
$\bigotimes_{i \in I}
\mathrm{QKN}_{D_n}(\varpi_i)_{\af}^{\otimes m_i}$. 
Then 
$\mathbb{Y}^{\si}_{D_n}(\lambda)$
is isomorphic,
as a $\U$-crystal, 
to the crystal basis 
$\mathcal{B}(\lambda)$.
\end{thm}

\subsection{Proof of Theorem \ref{thm:QKN->QLS-B} (2)} \label{Subsection:Pr-Thm-B-(2)}

This subsection is devoted to the proof of 
Theorem \ref{thm:QKN->QLS-B} (2). 
We have divided the proof into a sequence of lemmas.

\begin{lem} \label{lem:QKN->QLS-wd-B}
If
$\Tilde{\mathsf{C}}
\in 
\mathrm{QKN}_{B_n}(\varpi_i)$, 
then 
$(r\Tilde{\mathsf{C}},l\Tilde{\mathsf{C}})
\in 
\mathrm{QLS}(\varpi_i)$. 
\end{lem}

\begin{proof}
Let
$\Tilde{\mathsf{C}}
=
\mathsf{C}
\cup
\{ \overline{0},\ldots ,\overline{0} \}
\in 
\mathrm{QKN}_{B_n}(\varpi_i)$, 
with 
$\mathsf{C} \in \mathrm{KN}_{B_n}(\varpi_{i-2m})$. 
Let 
$w,v \in W^{I \setminus \{ i \}}$
and 
$w',v' \in W^{I \setminus \{ i-2m \}}$
be such that 
$\mathsf{T}_w^{(i)} = l\Tilde{\mathsf{C}}$, 
$\mathsf{T}_v^{(i)} = r\Tilde{\mathsf{C}}$, 
$\mathsf{T}_{w'}^{(i-2m)} = l\mathsf{C}$, 
and 
$\mathsf{T}_{v'}^{(i-2m)} = r\mathsf{C}$. 
Obviously, 
$\mathcal{J}(w)
=
\mathcal{J}(v)$. 
Write 
$\mathcal{J}(w)
=
\mathcal{J}(v)
=
(J_1 < \cdots < J_{\mu})
\in 
\mathcal{S}_i$, 
$\mathcal{J}(w')
=
\mathcal{J}(v')
=
(J_1' < \cdots < J_{\kappa}')
\in 
\mathcal{S}_{i-2m}$, 
$\mathcal{M}(w)
=
(M_{\nu}), 
\mathcal{M}(v)
=
(N_{\nu})
\in 
\prod_{\nu=1}^{\mu} 2^{J_{\nu}}$, 
and 
$\mathcal{M}(w')
=
(M_{\nu}'), 
\mathcal{M}(v)
=
(N_{\nu}')
\in 
\prod_{\nu=1}^{\kappa} 2^{J_{\nu}'}$. 
It follows from 
the definition of 
$K_{\Tilde{\mathsf{C}}}$
and 
$\bigcup_{\nu=1}^{\mu}
J_{\nu}
=
K_{\Tilde{\mathsf{C}}}
\cup
\bigcup_{\nu=1}^{\kappa}
J_{\nu}'$
that there exists 
$\tau \in [0,\mu]$
such that 
$J_1
=
K_{\Tilde{\mathsf{C}}}
\cup
\bigcup_{\nu=1}^{\tau}
J_{\nu}'$, 
$J_{\nu}
=
J_{\nu+\tau-1}'$
for 
$\nu \in [2,\mu]$, 
and 
$\mu+\tau-1=\kappa$. 
If 
$n \in J_{\nu}'$
for some 
$\nu \in [\tau]$, 
then 
$1 \notin J_1$, 
$m = 0$,  
and 
$(r\Tilde{\mathsf{C}},l\Tilde{\mathsf{C}})
=
(r\mathsf{C},l\mathsf{C})
\in 
\mathrm{LS}(\varpi_i)
\subset 
\mathrm{QLS}(\varpi_i)$, 
which is our assertion. 
Therefore we may assume that 
$n \notin J_{\nu}'$
for 
$\nu \in [\tau]$. 
Clearly, 
$N_{\nu} \trianglerighteq M_{\nu}$
in 
$2^{J_{\nu}}$
for 
$\nu \in [2,\mu]$. 
Since 
$(N_{\nu}')_{\nu=1}^{\tau}
\trianglerighteq'
(M_{\nu}')_{\nu=1}^{\tau}$
in 
$\prod_{\nu=1}^{\tau} 2^{J_{\nu}'}$, 
we have 
$\# M_{\nu}' = \# N_{\nu}'$
for 
$\nu \in [\tau]$. 
It follows from 
$M_1 
=
K_{\Tilde{\mathsf{C}}}
\cup 
\bigcup_{\nu=1}^{\tau} M_{\nu}'$
and 
$N_1
=
\bigcup_{\nu=1}^{\tau} N_{\nu}'$
that
$\# M_1 - \# N_1
=
\# K_{\Tilde{\mathsf{C}}}
=
2m
\in 
2\BZ_{\geq 0}$.
Write
$\bigcup_{\nu=1}^{\tau} M_{\nu}'
=
\{ m_1' < m_2' < \cdots < m_s' \}
\subset 
M_1 
=
\{ m_1 < m_2 < \cdots < m_r \}$
and 
$N_1
=
\{ n_1 < n_2 < \cdots < n_s \}$, 
where 
$r = s+2m$. 
Then 
$m_{r-\nu}
\geq
m_{s-\nu}'
\geq 
n_{s-\nu}$
for 
$\nu \in [0,s]$. 
Hence 
$N_1 \trianglerighteq M_1$, 
by Lemma \ref{lem:Maya-cri} (4). 
This proves 
$v \trianglerighteq w$, 
and so 
$(r\Tilde{\mathsf{C}},l\Tilde{\mathsf{C}})
\in 
\mathrm{QLS}(\varpi_i)$. 
\end{proof}

\begin{lem} \label{lem:(ii)-(iii)-wd-B}
The elements 
(ii)--(iii)
in Theorem \ref{thm:QKN->QLS-B} (2) 
are well-defined. 
\end{lem}

\begin{proof}
Let 
$(v,w) \in \mathrm{QLS}(\varpi_i)$, 
$\mathcal{J}
=
\mathcal{J}(w) 
= 
\mathcal{J}(v) 
= 
(J_1 < \cdots < J_{\mu})
\in 
\mathcal{S}_i$
and 
$\mathcal{M}(w) = (M_{\nu}), 
\mathcal{M}(v) = (N_{\nu})
\in 
2^{\mathcal{J}}$.
We see from 
Lemma \ref{lem:Maya-cri} (4)
that our assertion follows from
$(N_1 \setminus M_1)
\trianglerighteq
(M_1 \setminus N_1)$
in 
$2^{J_1}$. 
We have 
$\# (M_1 \setminus N_1)
-
\# (N_1 \setminus M_1)
=
\# M_1 - \# N_1
\in 
2\BZ_{\geq 0}$. 
Write
$J_1 = [p,q]$.
Since 
$N_1 \trianglerighteq M_1$, 
we have 
$\#(M_1 \cap [r,q])
\geq 
\#(N_1 \cap [r,q])$
for all 
$r \in [p,q]$. 
Write 
$M_1 \setminus N_1
=
\{ m_1' < \cdots < m_s' \}$
and 
$N_1 \setminus M_1
=
\{ n_1' < \cdots < n_t' \}$. 
On the contrary, 
suppose that 
$(N_1 \setminus M_1)
\not\trianglerighteq
(M_1 \setminus N_1)$
in 
$2^{J_1}$.
By Lemma \ref{lem:Maya-cri} (4), 
there exists
$r \in [0,t-1]$
such that 
$m_{s-r}' < n_{t-r}'$
and 
$m_{s-\nu}' \geq n_{t-\nu}'$
for 
$\nu \in [0,r-1]$. 
This implies 
$\# (M_1 \cap [n_{t-r}',q])
-
\# (N_1 \cap [n_{t-r}',q])
=
-1
<0$,
a contradiction.
\end{proof}

\begin{lem} \label{lem:QLS->QKN-wd-B}
The right-hand side of 
\eqref{eq:QLS->QKN-B}
is a QKN $B_n$-column of shape $\varpi_i$.
\end{lem}

\begin{proof}
Let 
$(v,w) \in \mathrm{QLS}(\varpi_i)$, 
$\mathcal{J}
=
\mathcal{J}(w) 
= 
\mathcal{J}(v) 
= 
(J_1 < \cdots < J_{\mu})
\in 
\mathcal{S}_i$, 
and 
$\mathcal{M}(w) = (M_{\nu}), 
\mathcal{M}(v) = (N_{\nu})
\in 
2^{\mathcal{J}}$. 
Set 
$f = \# N_{\mu} - \# M_{\mu} \in \BZ_{\geq 0}$
if 
$n \in J_{\mu}$, 
and set
$f = 0$
otherwise. 
Set
$2m = \# M_1 - \# N_1 \in 2\BZ_{\geq 0}$
if 
$1,2 \in J_1$, 
and set
$m = 0$
otherwise. 
Let 
$y^{\nu} \in N_1 \setminus M_1$
and 
$z^{\nu} \in M_1 \setminus N_1$, 
$\nu \in [l]$, 
be as 
(i)--(iii)
in Theorem \ref{thm:QKN->QLS-B} (2). 
It suffices to prove that 
$\mathsf{C}$ defined below is a KN $B_n$-column of shape $\varpi_{i-2m}$. 
\begin{align} 
\begin{split}
\mathsf{C}
=
(J_1 
\setminus 
(M_1 \cup N_1))
&\cup
\{ z^{\nu}, \overline{z^{\nu}}
\mid 
\nu \in [l] \}
\cup 
\{ \overline{z} \mid z \in M_1 \cap N_1 \} \\[2mm]
&\cup
\bigcup_{\nu=2}^{\mu}
\left(
(J_{\nu} \setminus N_{\nu})
\cup
\{ \overline{z} \mid z \in M_{\nu} \}
\right)
\cup
\{ \underbrace{0,0,\ldots ,0}_{f \ \text{times}} \} .
\end{split}
\end{align}
By Lemma \ref{lem:KN->LS-B} (2), 
it suffices to prove that 
$\mathsf{C}$
can be split. 
It is easily seen that 
$I_{\mathsf{C}}
=
\{ z^{\nu} \mid \nu \in [l] \}
\cup
\bigcup_{\nu = 2}^{\mu} (M_{\nu} \setminus N_{\nu})
\cup
\{ \underbrace{0,0,\ldots ,0}_{f \ \text{times}} \}$. 
It remains to prove that 
$J_{\mathsf{C}}
=
\bigcup_{\nu = 1}^{\mu} (N_{\nu} \setminus M_{\nu})$. 
By Lemma \ref{lem:KN->LS-B} (2)
and 
$N_{\nu}
\trianglerighteq
M_{\nu}$
for 
$\nu \in [2,\mu]$, 
we may assume that 
$\mu = 1$
and 
$I_{\mathsf{C}}
=
\{ z^{\nu} \mid \nu \in [l] \}$. 
Then we only need to show that 
$J_{\mathsf{C}}
=
N_1 \setminus M_1
=
\{ y^{\nu} \mid \nu \in [l] \}$. 
We see that (ii)--(iii) in Theorem \ref{thm:QKN->QLS-B} (2)
imply 
\begin{enumerate}[(i)]
\item
$y^l
=
\max
\{ y \in J_1 \mid y \prec z^l, \ y \in N_1 \setminus M_1 \}$,
\item
$y^{\nu}
=
\max
\{ y \in J_1 \mid y \prec \min \{ y^{\nu+1}, z^l \}, \ y \in N_1 \setminus M_1 \}$
for 
$\nu \in [l-1]$. 
\end{enumerate}
This is our assertion. 
\end{proof}

By Lemmas 
\ref{lem:QKN->QLS-wd-B}--\ref{lem:QLS->QKN-wd-B}, 
we obtain the maps 
$\Phi_i : \mathrm{QKN}_{B_n}(\varpi_i)
\to 
\mathrm{QLS}(\varpi_i)$, 
$\Tilde{\mathsf{C}} 
\mapsto
(r\Tilde{\mathsf{C}},l\Tilde{\mathsf{C}})$, 
and 
$\Psi_i : \mathrm{QLS}(\varpi_i)
\to
\mathrm{QKN}_{B_n}(\varpi_i)$, 
$(v,w) \mapsto \Tilde{\mathsf{C}}$, 
where 
$\Tilde{\mathsf{C}}$
is defined as 
\eqref{eq:QLS->QKN-B}.

\begin{lem}
The maps 
$\Phi_i$
and 
$\Psi_i$
are inverses of each other.
\end{lem}

\begin{proof}
By Lemma \ref{lem:KN->LS-B} (2), 
the map 
$\Phi_i$
is injective. 
The proof is completed by showing that 
$(\Phi_i \circ \Psi_i)(v,w) = (v,w)$
for 
$(v,w) \in \mathrm{QLS}(\varpi_i)$. 
Let 
$\Tilde{\mathsf{C}} = \Psi_i(v,w)
\in \mathrm{QKN}_{B_n}(\varpi_i)$. 
Let 
$\mathsf{C}$
be as in (QKN-B2) for 
$\Tilde{\mathsf{C}}$.
By Lemma \ref{lem:KN->LS-B} (2), 
we may assume that 
$\mathcal{J}(w) 
= 
\mathcal{J}(v) 
= 
(J_1)
\in 
\mathcal{S}_i$, 
$\mathcal{M}(w) = M_1, 
\mathcal{M}(v) = N_1
\in 
2^{J_1}$, 
and 
$1,2 \in J_1$. 
Let 
$y^{\nu},z^{\nu}$, 
$\nu \in [l]$, 
be as 
(ii)--(iii)
in Theorem \ref{thm:QKN->QLS-B} (2). 
Then 
$I_{\mathsf{C}}
=
\{ z^{\nu} \mid \nu \in [l] \}$. 
We check at once that 
$M_1
=
\{ z^{\nu} \mid \nu \in [l] \}
\cup
(M_1 \cap N_1)
\cup
K_{\Tilde{\mathsf{C}}}$
and 
$N_1
=
\{ y^{\nu} \mid \nu \in [l] \}
\cup
(M_1 \cap N_1)$. 
Since 
$y^{\nu},z^{\nu}$, 
$\nu \in [l]$, 
satisfy
(i)--(ii)
in the proof of Lemma \ref{lem:QLS->QKN-wd-B}, 
we conclude that 
$\Phi_i(\Tilde{\mathsf{C}})
=
(r\Tilde{\mathsf{C}},l\Tilde{\mathsf{C}})
=
(N_1,M_1)
=
(v,w)$.
\end{proof}

The proof of the next lemma is straightforward.

\begin{lem} \label{lem:C1C2-B}
Let 
$\Tilde{\mathsf{C}} \in \mathrm{QKN}_{B_n}(\varpi_i)$, 
$(v,w)
=
\Phi_i(\Tilde{\mathsf{C}})
\in 
\mathrm{QLS}(\varpi_i)$, 
$\mathcal{J}
=
\mathcal{J}(w) 
= 
\mathcal{J}(v) 
= 
(J_1 < \cdots < J_{\mu})
\in 
\mathcal{S}_i$,  
and 
$\mathcal{M}(w) = (M_{\nu}),
\mathcal{M}(v) = (N_{\nu})
\in 
2^{\mathcal{J}}$; 
note that 
$(N_1,M_1) \in \mathrm{QLS}(\varpi_{\# J_1})$
and 
$((N_{\nu})_{\nu=2}^{\mu},(N_{\nu})_{\nu=2}^{\mu})
\in 
\mathrm{LS}(\varpi_{i-\#J_1})$
by 
Definition \ref{def:Maya-1/2}
and 
\eqref{eq:QLS}--\eqref{eq:LS}. 
Set 
$p = \min J_1$, 
$\Tilde{\mathsf{C}}_1
=
\Psi_{\# J_1}(N_1,M_1)
\in \mathrm{QKN}_{B_n}(\varpi_{\# J_1})$, 
and 
$\Tilde{\mathsf{C}}_2
=
\Psi_{i-\# J_1}((N_{\nu})_{\nu=2}^{\mu},(N_{\nu})_{\nu=2}^{\mu})
\in \mathrm{KN}_{B_n}(\varpi_{i - \# J_1})$. 
Then 
$\Tilde{\mathsf{C}}
=
\Tilde{\mathsf{C}}_1
\sqcup
\Tilde{\mathsf{C}}_2$
and 
\begin{align}
f_j
\Tilde{\mathsf{C}}
=
\begin{cases}
f_j\Tilde{\mathsf{C}}_1 \sqcup \Tilde{\mathsf{C}}_2 
&
\text{if} \ 
j \in J_1 \cup \{ 0,p-1 \}, \\
\Tilde{\mathsf{C}}_1 \sqcup f_j\Tilde{\mathsf{C}}_2
&
\text{otherwise};
\end{cases}
\end{align}
we understand that 
$f_j
\Tilde{\mathsf{C}}
=
\bm{0}$
if 
($j \in J_1 \cup \{ 0,p-1 \}$
and 
$f_j\Tilde{\mathsf{C}}_1 = \bm{0}$)
or 
($j \notin J_1 \cup \{ 0,p-1 \}$
and 
$f_j\Tilde{\mathsf{C}}_2 = \bm{0}$).
Similarly, 
if we write 
$f_j(N_1,M_1)
=
(\mathcal{N}_1,\mathcal{M}_1)$
and 
$f_j((N_{\nu})_{\nu=2}^{\mu},(M_{\nu})_{\nu=2}^{\mu})
=
(\mathcal{N}_2,\mathcal{M}_2)$, 
then 
\begin{align}
f_j(v,w)
=
\begin{cases}
((\mathcal{N}_1 , (N_{\nu})_{\nu=2}^{\mu}),
(\mathcal{M}_1 , (M_{\nu})_{\nu=2}^{\mu}))
&
\text{if} \ 
j \in J_1 \cup \{ 0,p-1 \}, \\
((N_1, \mathcal{N}_2),
(M_1 , \mathcal{M}_2))
&
\text{otherwise}.
\end{cases}
\end{align}

\end{lem}

\begin{lem}
$\mathrm{QKN}_{B_n}(\varpi_i)$
is a $\U'$-crystal, 
and the map
$\Phi_i$
is a morphism of 
$\U'$-crystals. 
\end{lem}

\begin{proof}
We need to show that the set 
$\mathrm{QKN}_{B_n}(\varpi_i) \sqcup \{ \bm{0} \}$
is stable under the maps
$e_j,f_j$, 
$j \in I_{\af}$, 
and that the map 
$\Phi_i$
satisfies the conditions 
(CM2)--(CM3)
in 
\S \ref{Subsection:Crystals}. 
We give the proof only for the equality 
$f_j \Tilde{\mathsf{C}}
=
\Psi_i(f_j \Phi_i(\Tilde{\mathsf{C}}))$
for 
$\Tilde{\mathsf{C}}
=
\mathsf{C}
\cup
\{ \underbrace{\overline{0},\overline{0},\ldots ,\overline{0}}_{2m \ \text{times}} \}
\in 
\mathrm{QKN}_{B_n}(\varpi_i)$
and 
$j \in I_{\af}$, 
where we understand that 
$\Phi_i(\bm{0}) = \bm{0}$
and 
$\Psi_i(\bm{0}) = \bm{0}$; 
the other statements are left to the reader. 

We continue to use the notation in Lemma \ref{lem:C1C2-B}.
By Lemma \ref{lem:KN->LS-B} (2), 
we have 
$\Phi_{i-\# J_1}(f_j \Tilde{\mathsf{C}}_2)
=
(\mathcal{N}_2,\mathcal{M}_2)$
for 
$j \in I$. 
By Lemma \ref{lem:C1C2-B}, 
this implies 
$f_j \Tilde{\mathsf{C}}
=
\Psi_i(f_j\Phi_i(\Tilde{\mathsf{C}}))$
for 
$j \notin J_1 \cup \{ 0,p-1 \}$. 
Therefore, 
without loss of generality 
we can assume that
$J_1 = [p,p+i-1]$, 
$\Tilde{\mathsf{C}}
=
\Tilde{\mathsf{C}}_1$, 
$\Phi_i(\Tilde{\mathsf{C}})
=
(v,w)
=
(N_1,M_1)$, 
and 
$j \in J_1 \cup \{ 0,p-1 \}$.
Write 
$\Psi_i(f_j\Phi_i(\Tilde{\mathsf{C}}))
=
\Tilde{\mathsf{C}}'
=
\mathsf{C}'
\cup
\{ \overline{0},\overline{0},\ldots ,\overline{0} \}
\in 
\mathrm{QKN}_{B_n}(\varpi_i) \sqcup \{ \bm{0} \}$
and 
$f_j(v,w)
=
(v',w')
\in 
\mathrm{QLS}(\varpi_i) \sqcup \{ \bm{0} \}$.

We first assume that 
$j \neq 0$. 
Write 
$[\Tilde{\mathsf{C}}]_j
=
\Tilde{\mathsf{C}}
\cap
\{ j,j+1,\overline{j+1},\overline{j} \}$
and 
$[\Tilde{w}]_j
=
\mathsf{T}_{\Tilde{w}}^{(i)}
\cap
\{ j,j+1,\overline{j+1},\overline{j} \}$
for 
$\Tilde{w} \in W^{I \setminus \{ i \}}$. 
If 
$p \geq 2$, 
then 
$m = 0$, 
$\Tilde{\mathsf{C}}
=
\mathsf{C}
\in 
\mathrm{KN}_{B_n}(\varpi_i)$, 
and the assertion follows from 
Lemma \ref{lem:KN->LS-B}. 
Therefore we can assume that 
$p = 1$
and 
$J_1 = [i]$. 
It follows from 
$J_1 = [i]$
and 
Lemma \ref{lem:J(w)=J(v)}
that 
$\{ \|u\| \mid u \in [w]_j \}
=
\{ \|u\| \mid u \in [v]_j \}
\in
\{ \{ j \} , \{ j,j+1 \} \}$.

\begin{proof}[Case 1]
Assume that 
$[\Tilde{\mathsf{C}}]_j
=
\emptyset$.
Then 
$f_j \Tilde{\mathsf{C}}
=
\bm{0}$. 
We have 
$([v]_j,[w]_j)
\in 
\{ 
(\{ j \} , \{ \overline{j} \}), 
(\{ \overline{j},\overline{j+1} \} , \{ j,j+1 \}), 
(\{ j,\overline{j+1} \} , \{ \overline{j},j+1 \})
\}$.
Hence 
$f_j(v,w) = \bm{0}$. 
\end{proof}

\begin{proof}[Case 2]
Assume that 
$[\Tilde{\mathsf{C}}]_j
=
\{ j+1 \}$.
Then 
$f_j \Tilde{\mathsf{C}}
=
\bm{0}$. 
We have 
$([v]_j,[w]_j)
\in 
\{ 
(\{ j+1,\overline{j} \} , \{ j,j+1 \}), 
(\{ j,j+1 \} , \{ j+1,\overline{j} \})
\}$.
Hence 
$f_j(v,w) = \bm{0}$. 
\end{proof}

\begin{proof}[Case 3]
Assume that 
$[\Tilde{\mathsf{C}}]_j
=
\{ \overline{j} \}$.
Then 
$f_j \Tilde{\mathsf{C}}
=
\bm{0}$. 
We have 
$([v]_j,[w]_j)
\in 
\{ 
(\{ \overline{j} \} , \{ \overline{j} \}), 
(\{ \overline{j+1},\overline{j} \} , \{ j+1,\overline{j} \}),
(\{ j+1,\overline{j} \} , \{ \overline{j+1},\overline{j} \})
\}$.
Hence 
$f_j(v,w) = \bm{0}$. 
\end{proof}

\begin{proof}[Case 4]
Assume that 
$[\Tilde{\mathsf{C}}]_j
=
\{ j,j+1 \}$.
Then 
$f_j \Tilde{\mathsf{C}}
=
\bm{0}$. 
We have 
$([v]_j,[w]_j)
=
(\{ j,j+1 \} , \{ j,j+1 \})$.
Hence 
$f_j(v,w) = \bm{0}$. 
\end{proof}

\begin{proof}[Case 5]
Assume that 
$[\Tilde{\mathsf{C}}]_j
=
\{ j,\overline{j} \}$.
Then 
$f_j \Tilde{\mathsf{C}}
=
\bm{0}$. 
We have 
$([v]_j,[w]_j)
\in 
\{ 
(\{ j \} , \{ \overline{j} \}), 
(\{ j,\overline{j+1} \} , \{ j+1,\overline{j} \}), 
(\{ j,j+1 \} , \{ \overline{j+1},\overline{j} \})
\}$.
Hence 
$f_j(v,w) = \bm{0}$. 
\end{proof}

\begin{proof}[Case 6]
Assume that 
$[\Tilde{\mathsf{C}}]_j
=
\{ j+1, \overline{j} \}$.
Then 
$f_j \Tilde{\mathsf{C}}
=
\bm{0}$. 
We have 
$([v]_j,[w]_j)
=  
(\{ j+1, \overline{j} \} , \{ j+1, \overline{j} \})$.
Hence 
$f_j(v,w) = \bm{0}$. 
\end{proof}

\begin{proof}[Case 7]
Assume that 
$[\Tilde{\mathsf{C}}]_j
=
\{ \overline{j+1}, \overline{j} \}$.
Then 
$f_j \Tilde{\mathsf{C}}
=
\bm{0}$. 
We have 
$([v]_j,[w]_j)
=  
(\{ \overline{j+1}, \overline{j} \} , 
\{ \overline{j+1}, \overline{j} \})$.
Hence 
$f_j(v,w) = \bm{0}$. 
\end{proof}

\begin{proof}[Case 8]
Assume that 
$[\Tilde{\mathsf{C}}]_j
=
\{ j,j+1, \overline{j} \}$.
Then 
$f_j \Tilde{\mathsf{C}}
=
\bm{0}$. 
We have 
$([v]_j,[w]_j)
=  
(\{ j,j+1 \} , 
\{ j+1, \overline{j} \})$.
Hence 
$f_j(v,w) = \bm{0}$. 
\end{proof}

\begin{proof}[Case 9]
Assume that 
$[\Tilde{\mathsf{C}}]_j
=
\{ j+1, \overline{j+1} , \overline{j} \}$.
Then 
$f_j \Tilde{\mathsf{C}}
=
\bm{0}$. 
We have 
$([v]_j,[w]_j)
=  
(\{ j+1,\overline{j} \} , 
\{ \overline{j+1}, \overline{j} \})$.
Hence 
$f_j(v,w) = \bm{0}$. 
\end{proof}

\begin{proof}[Case 10]
Assume that 
$[\Tilde{\mathsf{C}}]_j
=
\{ j,j+1, \overline{j+1} , \overline{j} \}$.
Then 
$f_j \Tilde{\mathsf{C}}
=
\bm{0}$. 
We have 
$([v]_j,[w]_j)
=  
(\{ j,j+1 \} , 
\{ \overline{j+1}, \overline{j} \})$.
Hence 
$f_j(v,w) = \bm{0}$. 
\end{proof}

\begin{proof}[Case 11]
Assume that 
$[\Tilde{\mathsf{C}}]_j
=
\{ j \}$.
Then 
$f_j \Tilde{\mathsf{C}}
=
(\Tilde{\mathsf{C}} \setminus \{ j \}) \cup \{ j+1 \}$,
by 
\eqref{eq:f-KN-C-1}. 
We have 
$([v]_j,[w]_j)
\in 
\{  
(\{ j \} , \{ j \}),
(\{ j,\overline{j+1} \} , \{ j,j+1 \}),
(\{ j,j+1 \} , \{ j,\overline{j+1} \})
\}$.

If 
$([v]_j,[w]_j)
=
(\{ j \} , \{ j \})$, 
then 
$([v']_j,[w']_j)
=
(\{ j+1 \} , \{ j+1 \})$, 
$I_{\mathsf{C}}
=
I_{\mathsf{C}'}$, 
$J_{\mathsf{C}}
=
J_{\mathsf{C}'}$, 
and 
$K_{\Tilde{\mathsf{C}}}
=
K_{\Tilde{\mathsf{C}}'}$; 
note that 
$j = i$. 
Hence 
$f_j \Tilde{\mathsf{C}}
=
\Psi_i(v',w')$.

If 
$([v]_j,[w]_j)
=
(\{ j,\overline{j+1} \} , \{ j,j+1 \})$, 
then 
$([v']_j,[w']_j)
=
(\{ j+1,\overline{j} \} , \{ j,j+1 \})$, 
$I_{\mathsf{C}}
=
I_{\mathsf{C}'}$, 
$j+1 \in J_{\mathsf{C}}$, 
$j \in J_{\mathsf{C}'}$, 
$J_{\mathsf{C}} \setminus \{ j+1 \}
=
J_{\mathsf{C}'} \setminus \{ j \}$, 
and 
$K_{\Tilde{\mathsf{C}}}
=
K_{\Tilde{\mathsf{C}}'}$. 
Hence 
$f_j \Tilde{\mathsf{C}}
=
\Psi_i(v',w')$.

If 
$([v]_j,[w]_j)
=
(\{ j,j+1 \} , \{ j,\overline{j+1} \})$, 
then 
$([v']_j,[w']_j)
=
(\{ j,j+1 \} , \{ j+1,\overline{j} \})$, 
$I_{\mathsf{C}}
=
I_{\mathsf{C}'}$, 
$J_{\mathsf{C}}
=
J_{\mathsf{C}'}$, 
$j+1 \in K_{\Tilde{\mathsf{C}}}$, 
$j \in K_{\Tilde{\mathsf{C}}'}$, 
and 
$K_{\Tilde{\mathsf{C}}} \setminus \{ j+1 \}
=
K_{\Tilde{\mathsf{C}}'} \setminus \{ j \}$. 
Hence 
$f_j \Tilde{\mathsf{C}}
=
\Psi_i(v',w')$.
\end{proof}

\begin{proof}[Case 12]
Assume that 
$[\Tilde{\mathsf{C}}]_j
=
\{ \overline{j+1} \}$.
Then 
$f_j \Tilde{\mathsf{C}}
=
(\Tilde{\mathsf{C}} \setminus \{ \overline{j+1} \}) 
\cup 
\{ \overline{j} \}$, 
by 
\eqref{eq:f-KN-C-2}. 
We have 
$([v]_j,[w]_j)
\in 
\{  
(\{ \overline{j},\overline{j+1} \} , \{ j,\overline{j+1} \}),
(\{ j,\overline{j+1} \} , \{ \overline{j},\overline{j+1} \})
\}$.

If 
$([v]_j,[w]_j)
=
(\{ \overline{j},\overline{j+1} \} , \{ j,\overline{j+1} \})$, 
then 
$([v']_j,[w']_j)
=
(\{ \overline{j},\overline{j+1} \} , \{ j+1,\overline{j} \})$, 
$I_{\mathsf{C}}
=
I_{\mathsf{C}'}$, 
$j \in J_{\mathsf{C}}$, 
$j+1 \in J_{\mathsf{C}'}$, 
$J_{\mathsf{C}} \setminus \{ j \}
=
J_{\mathsf{C}'} \setminus \{ j+1 \}$, 
and 
$K_{\Tilde{\mathsf{C}}}
=
K_{\Tilde{\mathsf{C}}'}$. 
Hence 
$f_j \Tilde{\mathsf{C}}
=
\Psi_i(v',w')$.

If 
$([v]_j,[w]_j)
=
(\{ j,\overline{j+1} \} , \{ \overline{j},\overline{j+1} \})$, 
then 
$([v']_j,[w']_j)
=
(\{ j+1,\overline{j} \} , \{ \overline{j},\overline{j+1} \})$, 
$I_{\mathsf{C}}
=
I_{\mathsf{C}'}$, 
$J_{\mathsf{C}}
=
J_{\mathsf{C}'}$, 
$j \in K_{\Tilde{\mathsf{C}}}$, 
$j+1 \in K_{\Tilde{\mathsf{C}}'}$,
and 
$K_{\Tilde{\mathsf{C}}} \setminus \{ j \}
=
K_{\Tilde{\mathsf{C}}'} \setminus \{ j+1 \}$. 
Hence 
$f_j \Tilde{\mathsf{C}}
=
\Psi_i(v',w')$.
\end{proof}

\begin{proof}[Case 13]
Assume that 
$[\Tilde{\mathsf{C}}]_j
=
\{ j,\overline{j+1} \}$.
Then 
$f_j \Tilde{\mathsf{C}}
=
(\Tilde{\mathsf{C}} \setminus \{ j \}) 
\cup 
\{ j+1 \}$,
by 
\eqref{eq:f-KN-C-3}. 
We have 
$([v]_j,[w]_j)
= 
(\{ j,\overline{j+1} \} , \{ j,\overline{j+1} \})$, 
$([v']_j,[w']_j)
=
(\{ j+1,\overline{j} \} , \{ j,\overline{j+1} \})$, 
$j+1 \notin I_{\mathsf{C}}$, 
$I_{\mathsf{C}}
\cup
\{ j+1 \}
=
I_{\mathsf{C}'}$, 
$j \notin J_{\mathsf{C}}$, 
$J_{\mathsf{C}} \cup \{ j \}
=
J_{\mathsf{C}'}$, 
and 
$K_{\Tilde{\mathsf{C}}}
=
K_{\Tilde{\mathsf{C}}'}$. 
Hence 
$f_j \Tilde{\mathsf{C}}
=
\Psi_i(v',w')$.
\end{proof}

\begin{proof}[Case 14]
Assume that 
$[\Tilde{\mathsf{C}}]_j
=
\{ j+1,\overline{j+1} \}$.
Then 
$f_j \Tilde{\mathsf{C}}
=
(\Tilde{\mathsf{C}} \setminus \{ \overline{j+1} \}) 
\cup 
\{ \overline{j} \}$,
by 
\eqref{eq:f-KN-C-3}. 
We have 
$([v]_j,[w]_j)
= 
(\{ j+1,\overline{j} \} , \{ j,\overline{j+1} \})$, 
$([v']_j,[w']_j)
=
(\{ j+1,\overline{j} \} , \{ j+1,\overline{j} \})$, 
$j+1 \in I_{\mathsf{C}}$,
$I_{\mathsf{C}} \setminus \{ j+1 \}
=
I_{\mathsf{C}'}$,  
$j \in J_{\mathsf{C}}$, 
$J_{\mathsf{C}} \setminus \{ j \}
=
J_{\mathsf{C}'}$, 
and 
$K_{\Tilde{\mathsf{C}}}
=
K_{\Tilde{\mathsf{C}}'}$.
Hence 
$f_j \Tilde{\mathsf{C}}
=
\Psi_i(v',w')$.
\end{proof}

\begin{proof}[Case 15]
Assume that 
$[\Tilde{\mathsf{C}}]_j
=
\{ j,j+1,\overline{j+1} \}$.
Then 
$f_j \Tilde{\mathsf{C}}
=
(\Tilde{\mathsf{C}} \setminus \{ \overline{j+1} \}) 
\cup 
\{ \overline{j} \}$,
by 
\eqref{eq:f-KN-C-4}. 
We have 
$([v]_j,[w]_j)
= 
(\{ j,j+1 \} , \{ j,\overline{j+1} \})$, 
$([v']_j,[w']_j)
=
(\{ j,j+1 \} , \{ j+1,\overline{j} \})$, 
$j+1 \in I_{\mathsf{C}}$, 
$j \in I_{\mathsf{C}'}$, 
$I_{\mathsf{C}} \setminus \{ j+1 \}
=
I_{\mathsf{C}'} \setminus \{ j \}$,
$J_{\mathsf{C}}
=
J_{\mathsf{C}'}$,
and 
$K_{\Tilde{\mathsf{C}}}
=
K_{\Tilde{\mathsf{C}}'}$. 
Hence 
$f_j \Tilde{\mathsf{C}}
=
\Psi_i(v',w')$.
\end{proof}

\begin{proof}[Case 16]
Assume that 
$[\Tilde{\mathsf{C}}]_j
=
\{ j,\overline{j+1},\overline{j} \}$.
Then 
$f_j \Tilde{\mathsf{C}}
=
(\Tilde{\mathsf{C}} \setminus \{ j \}) 
\cup 
\{ j+1 \}$,
by 
\eqref{eq:f-KN-C-4}. 
We have 
$([v]_j,[w]_j)
= 
(\{ j,\overline{j+1} \} , \{ \overline{j},\overline{j+1} \})$, 
$([v']_j,[w']_j)
=
(\{ j+1,\overline{j} \} , \{ \overline{j},\overline{j+1} \})$, 
$j \in I_{\mathsf{C}}$, 
$j+1 \in I_{\mathsf{C}'}$,
$I_{\mathsf{C}} \setminus \{ j \}
=
I_{\mathsf{C}'} \setminus \{ j+1 \}$,
$J_{\mathsf{C}}
=
J_{\mathsf{C}'}$,
and 
$K_{\Tilde{\mathsf{C}}}
=
K_{\Tilde{\mathsf{C}}'}$. 
Hence 
$f_j \Tilde{\mathsf{C}}
=
\Psi_i(v',w')$.
\end{proof}

We next assume that 
$j = 0$. 
Obviously, 
$f_0 \Tilde{\mathsf{C}} = \bm{0}$
and 
$f_0 (v,w) = \bm{0}$
if 
$p > 2$
and 
$J_1 = [p,p+i-1]$. 
Therefore we can assume that 
$J_1 = [i]$
or 
$J_1 = [2,i+1]$; 
note that if 
$J_1 = [2,i+1]$, 
then 
$m = 0$. 
Write 
$[\Tilde{\mathsf{C}}]_0
=
\Tilde{\mathsf{C}}
\cap
\{ 1,2,\overline{2},\overline{1} \}$,
$[w]_0
=
l\Tilde{\mathsf{C}}
\cap
\{ 1,2,\overline{2},\overline{1} \}$, 
and 
$[v]_0
=
r\Tilde{\mathsf{C}}
\cap
\{ 1,2,\overline{2},\overline{1} \}$. 
Let 
$z_k
=
\min I_{\mathsf{C}}$
and 
$y_k
=
\min J_{\mathsf{C}}$.
If 
$m > 0$, 
let 
$x_1
=
\min K_{\Tilde{\mathsf{C}}}$
and 
$x_2
=
\min (K_{\Tilde{\mathsf{C}}} \setminus \{ x_1 \})$. 

\begin{proof}[Case 17]
Assume that 
$[\Tilde{\mathsf{C}}]_0
=
\emptyset$
and 
$m = 0$.
Then 
$f_0 \Tilde{\mathsf{C}}
=
\bm{0}$. 
We have 
$([v]_0,[w]_0)
\in  
\{ 
(\emptyset,\emptyset), 
(\{ \overline{2} \} , \{ 2 \})
\}$. 
Hence 
$f_0 (v,w) = \bm{0}$. 
\end{proof}

\begin{proof}[Case 18]
Assume that 
$[\Tilde{\mathsf{C}}]_0
=
\{ 1 \}$. 
Then 
$f_0 \Tilde{\mathsf{C}}
=
\bm{0}$. 
We have 
$([v]_0,[w]_0)
\in  
\{ 
(\{ 1 \} ,\{ 1 \}), 
(\{ 1,\overline{2} \} , \{ 1,2 \}),
\\
(\{ 1,2 \} , \{ 1,\overline{2} \})
\}$. 
Hence 
$f_0 (v,w) = \bm{0}$. 
\end{proof}

\begin{proof}[Case 19]
Assume that 
$[\Tilde{\mathsf{C}}]_0
=
\{ 2 \}$. 
Then
$f_0 \Tilde{\mathsf{C}}
=
\bm{0}$. 
We have 
$([v]_0,[w]_0)
\in  
\{ 
(\{ 2 \} ,\{ 2 \}), 
(\{ \overline{1},2 \} , \{ 1,2 \}),
\\
(\{ 1,2 \} , \{ 2,\overline{1} \})
\}$. 
Hence 
$f_0 (v,w) = \bm{0}$. 
\end{proof}

\begin{proof}[Case 20]
Assume that 
$[\Tilde{\mathsf{C}}]_0
=
\{ 2,\overline{2} \}$. 
Then
$f_0 \Tilde{\mathsf{C}}
=
\bm{0}$
and 
$1 \in J_{\mathsf{C}}$. 
We have 
$([v]_0,[w]_0)
=
(\{ 2,\overline{1} \} ,\{ 1,\overline{2} \})$. 
Hence 
$f_0 (v,w) = \bm{0}$. 
\end{proof}

\begin{proof}[Case 21]
Assume that 
$[\Tilde{\mathsf{C}}]_0
=
\{ 1,2 \}$. 
Then
$f_0 \Tilde{\mathsf{C}}
=
\bm{0}$. 
We have 
$([v]_0,[w]_0)
=
(\{ 1,2 \} ,\{ 1,2 \})$. 
Hence 
$f_0 (v,w) = \bm{0}$. 
\end{proof}

\begin{proof}[Case 22]
Assume that 
$[\Tilde{\mathsf{C}}]_0
=
\{ 1,\overline{2} \}$. 
Then
$f_0 \Tilde{\mathsf{C}}
=
\bm{0}$. 
We have 
$([v]_0,[w]_0)
=
(\{ 1,\overline{2} \} ,\{ 1,\overline{2} \})$. 
Hence 
$f_0 (v,w) = \bm{0}$. 
\end{proof}

\begin{proof}[Case 23]
Assume that 
$[\Tilde{\mathsf{C}}]_0
=
\{ 2,\overline{1} \}$. 
Then
$f_0 \Tilde{\mathsf{C}}
=
\bm{0}$. 
We have 
$([v]_0,[w]_0)
=
(\{ 2,\overline{1} \} ,\{ 2,\overline{1} \})$. 
Hence 
$f_0 (v,w) = \bm{0}$. 
\end{proof}

\begin{proof}[Case 24]
Assume that 
$[\Tilde{\mathsf{C}}]_0
=
\emptyset$, 
$m > 0$, 
and 
$y_k \notin \{ 1,2 \}$. 
Then
$f_0 \Tilde{\mathsf{C}}
=
(\Tilde{\mathsf{C}} \setminus \{ \overline{0},\overline{0} \})
\cup
\{ 1,2 \}$, 
by 
\eqref{eq:f0-KN-B-2}. 
We have 
$([v]_0,[w]_0)
=
(\{ 1,2 \} ,\{ \overline{2},\overline{1} \})$,
$([v']_0,[w']_0)
=
(\{ 1,2 \} , \{ 1,2 \})$, 
$I_{\mathsf{C}}
=
I_{\mathsf{C}'}$, 
$J_{\mathsf{C}}
=
J_{\mathsf{C}'}$, 
$1,2 \in K_{\Tilde{\mathsf{C}}}$, 
and 
$K_{\Tilde{\mathsf{C}}} \setminus \{ 1,2 \}
=
K_{\Tilde{\mathsf{C}}'}$.
Hence 
$f_0 \Tilde{\mathsf{C}}
=
\Psi_i(v',w')$. 
\end{proof}

\begin{proof}[Case 25]
Assume that 
$[\Tilde{\mathsf{C}}]_0
=
\{ \overline{2} \}$, 
$m = 0$, 
and 
$y_k \notin \{ 1,2 \}$. 
Then
$f_0 \Tilde{\mathsf{C}}
=
(\Tilde{\mathsf{C}} \setminus \{ \overline{2} \})
\cup
\{ 1 \}$, 
by 
\eqref{eq:f0-KN-B-1}. 
We have 
$([v]_0,[w]_0)
=(\{ \overline{2} \} ,\{ \overline{2} \})$, 
$([v']_0,[w']_0)
=
(\{ 1 \} , \{ 1 \})$, 
$I_{\mathsf{C}}
=
I_{\mathsf{C}'}$, 
$J_{\mathsf{C}}
=
J_{\mathsf{C}'}$, 
and
$K_{\Tilde{\mathsf{C}}}
=
K_{\Tilde{\mathsf{C}}'}
=
\emptyset$.
Hence 
$f_0 \Tilde{\mathsf{C}}
=
\Psi_i(v',w')$. 
\end{proof}

\begin{proof}[Case 26]
Assume that 
$[\Tilde{\mathsf{C}}]_0
=
\{ \overline{2} \}$, 
$m > 0$, 
and 
$y_k \notin \{ 1,2 \}$. 
Then
$x_1 = 1$
and 
$f_0 \Tilde{\mathsf{C}}
=
(\Tilde{\mathsf{C}} \setminus \{ \overline{2},\overline{0},\overline{0} \})
\cup
\{ 1,x_2,\overline{x_2} \}$, 
by 
\eqref{eq:f0-KN-B-2}. 
We have 
$([v]_0,[w]_0)
=(\{ 1,\overline{2} \} ,\{ \overline{2},\overline{1} \})$, 
$([v']_0,[w']_0)
=
(\{ 1,\overline{2} \} , \{ 1,2 \})$, 
$I_{\mathsf{C}} \cup \{ x_2 \}
=
I_{\mathsf{C}'}$, 
$J_{\mathsf{C}} \cup \{ 2 \}
=
J_{\mathsf{C}'}$, 
and
$K_{\Tilde{\mathsf{C}}} \setminus \{ 1,x_2 \}
=
K_{\Tilde{\mathsf{C}}'}$.
Hence 
$f_0 \Tilde{\mathsf{C}}
=
\Psi_i(v',w')$. 
\end{proof}

\begin{proof}[Case 27]
Assume that 
$[\Tilde{\mathsf{C}}]_0
=
\{ \overline{2} \}$
and 
$y_k = 1$. 
Then
$f_0 \Tilde{\mathsf{C}}
=
(\Tilde{\mathsf{C}} \setminus \{ z_k,\overline{z_k},\overline{2} \})
\cup
\{ 1,\overline{0},\overline{0} \}$, 
by 
\eqref{eq:f0-KN-B-3}. 
We have 
$([v]_0,[w]_0)
=(\{ \overline{2},\overline{1} \} ,\{ 1,\overline{2} \})$, 
$([v']_0,[w']_0)
=
(\{ 1,2 \} , \{ 1,\overline{2} \})$, 
$I_{\mathsf{C}} \setminus \{ z_k \}
=
I_{\mathsf{C}'}$, 
$J_{\mathsf{C}} \setminus \{ 1 \}
=
J_{\mathsf{C}'}$, 
and
$K_{\Tilde{\mathsf{C}}} \cup \{ 2,z_k \}
=
K_{\Tilde{\mathsf{C}}'}$.
Hence 
$f_0 \Tilde{\mathsf{C}}
=
\Psi_i(v',w')$. 
\end{proof}

\begin{proof}[Case 28]
Assume that 
$[\Tilde{\mathsf{C}}]_0
=
\{ \overline{1} \}$, 
$m = 0$, 
and 
$y_k \notin \{ 1,2 \}$. 
Then
$J_1 = \{ 1 \}$
and 
$f_0 \Tilde{\mathsf{C}}
=
(\Tilde{\mathsf{C}} \setminus \{ \overline{1} \})
\cup
\{ 2 \}$, 
by 
\eqref{eq:f0-KN-B-1}. 
We have 
$([v]_0,[w]_0)
=(\{ \overline{1} \} ,\{ \overline{1} \})$, 
$([v']_0,[w']_0)
=
(\{ 2 \} , \{ 2 \})$
and 
$I_{\mathsf{C}}
=
I_{\mathsf{C}'}
=
J_{\mathsf{C}}
=
J_{\mathsf{C}'}
=
K_{\Tilde{\mathsf{C}}}
=
K_{\Tilde{\mathsf{C}}'}
=
\emptyset$.
Hence 
$f_0 \Tilde{\mathsf{C}}
=
\Psi_i(v',w')$. 
\end{proof}

\begin{proof}[Case 29]
Assume that 
$[\Tilde{\mathsf{C}}]_0
=
\{ \overline{1} \}$, 
$m > 0$, 
and 
$y_k \notin \{ 1,2 \}$. 
Then
$x_1 = 2$
and 
$f_0 \Tilde{\mathsf{C}}
=
(\Tilde{\mathsf{C}} \setminus \{ \overline{1},\overline{0},\overline{0} \})
\cup
\{ 2,x_2,\overline{x_2} \}$, 
by 
\eqref{eq:f0-KN-B-2}. 
We have 
$([v]_0,[w]_0)
=(\{ 2,\overline{1} \} ,\{ \overline{2},\overline{1} \})$, 
$([v']_0,[w']_0)
=
(\{ 2,\overline{1} \} , \{ 1,2 \})$, 
$I_{\mathsf{C}} \cup \{ x_2 \}
=
I_{\mathsf{C}'}$, 
$J_{\mathsf{C}} \cup \{ 1 \}
=
J_{\mathsf{C}'}$, 
and
$K_{\Tilde{\mathsf{C}}} \setminus \{ 2,x_2 \}
=
K_{\Tilde{\mathsf{C}}'}$.
Hence 
$f_0 \Tilde{\mathsf{C}}
=
\Psi_i(v',w')$. 
\end{proof}

\begin{proof}[Case 30]
Assume that 
$[\Tilde{\mathsf{C}}]_0
=
\{ \overline{1} \}$
and 
$y_k = 2$. 
Then
$f_0 \Tilde{\mathsf{C}}
=
(\Tilde{\mathsf{C}} \setminus \{ z_k,\overline{z_k},\overline{1} \})
\cup
\{ 2,\overline{0},\overline{0} \}$, 
by 
\eqref{eq:f0-KN-B-3}. 
We have 
$([v]_0,[w]_0)
=(\{ \overline{2},\overline{1} \} ,\{ 2,\overline{1} \})$, 
$([v']_0,[w']_0)
=
(\{ 1,2 \} , \{ 2,\overline{1} \})$, 
$I_{\mathsf{C}} \setminus \{ z_k \}
=
I_{\mathsf{C}'}$, 
$J_{\mathsf{C}} \setminus \{ 2 \}
=
J_{\mathsf{C}'}$, 
and
$K_{\Tilde{\mathsf{C}}} \cup \{ 1,z_k \}
=
K_{\Tilde{\mathsf{C}}'}$.
Hence 
$f_0 \Tilde{\mathsf{C}}
=
\Psi_i(v',w')$. 
\end{proof}

\begin{proof}[Case 31]
Assume that 
$[\Tilde{\mathsf{C}}]_0
=
\{ \overline{2},\overline{1} \}$. 
Then
$f_0 \Tilde{\mathsf{C}}
=
(\Tilde{\mathsf{C}} \setminus \{ \overline{2},\overline{1} \})
\cup
\{ \overline{0},\overline{0} \}$, 
by 
\eqref{eq:f0-KN-B-1}--\eqref{eq:f0-KN-B-2}. 
We have 
$([v]_0,[w]_0)
=(\{ \overline{2},\overline{1} \} , \{ \overline{2},\overline{1} \})$, 
$([v']_0,[w']_0)
=
(\{ 1,2 \} , \{ \overline{2},\overline{1} \})$, 
$I_{\mathsf{C}}
=
I_{\mathsf{C}'}$, 
$J_{\mathsf{C}}
=
J_{\mathsf{C}'}$, 
and
$K_{\Tilde{\mathsf{C}}} \cup \{ 1,2 \}
=
K_{\Tilde{\mathsf{C}}'}$.
Hence 
$f_0 \Tilde{\mathsf{C}}
=
\Psi_i(v',w')$. 
\end{proof}

The proof is complete.
\end{proof}


{\small
\setlength{\baselineskip}{13pt}
\renewcommand{\refname}{References}


\ \\
\textsc{Motohiro Ishii}, 
Department of Mathematics, Cooperative Faculty of Education, Gunma University, 

Maebashi, Gunma 371-8510, Japan

{\it Email address}: {\tt m.ishii@gunma-u.ac.jp}

\end{document}